\documentclass[10pt,fleqn]{article}

\usepackage[includeheadfoot]{geometry}
\usepackage{times}
\usepackage{graphics}
\usepackage{graphicx}
\usepackage{amsbsy}
\usepackage{amssymb}
\usepackage{amsfonts}
\usepackage{amstext}
\usepackage{amsmath}
\usepackage{psfrag}
\usepackage[parfill]{parskip}
\usepackage{sectsty}
\usepackage{multicol}
\usepackage{setspace}

\usepackage{algorithm}
\usepackage{algpseudocode}
\algnewcommand\algorithmicgoto{\textbf{Go to }line }
\algnewcommand\algorithmicdownto{\textbf{ down to }}
\algnewcommand\algorithmicto{\textbf{ to }}
\algnewcommand\algorithmicand{\textbf{ and }}
\algnewcommand\algorithmicnot{\textbf{not }}
\algnewcommand\algorithmicexit{\textbf{exit algorithm}}

\usepackage{fancyhdr}
\fancypagestyle{plain}{% 
\fancyhead[L]{}
\fancyhead[C]{}
\fancyhead[R]{}
\fancyfoot[L]{\small\sf Davide Giglio}
\fancyfoot[C]{\small\sf DIBRIS -- University of Genova}
\fancyfoot[R]{\sf\thepage}

}

\usepackage{amsthm}
\newtheoremstyle{thstyle}{6pt}{3pt}{}{}{\bf}{.}{.5em}{}

\theoremstyle{thstyle}

\newtheorem{myalgorithm}{Algorithm}
\newtheorem{mylemma}{Lemma}
\newtheorem{mydefinition}{Definition}

\title{{\sc\large TECHNICAL REPORT}\\\vspace{25pt}Fundamental lemmas for the determination of optimal control strategies for a class of single machine family scheduling problems}
\author{Davide Giglio\vspace{6pt}\\
{\small Department of Informatics, Bioengineering, Robotics and Systems Engineering (DIBRIS)}\vspace{-2pt}\\
{\small University of Genova}\vspace{-2pt}\\
{\small Via Opera Pia 13, 16145 -- Genova, Italy}\vspace{-2pt}\\
{\small davide.giglio@unige.it}%
\vspace{20pt}}
\date{\includegraphics[scale=.2]{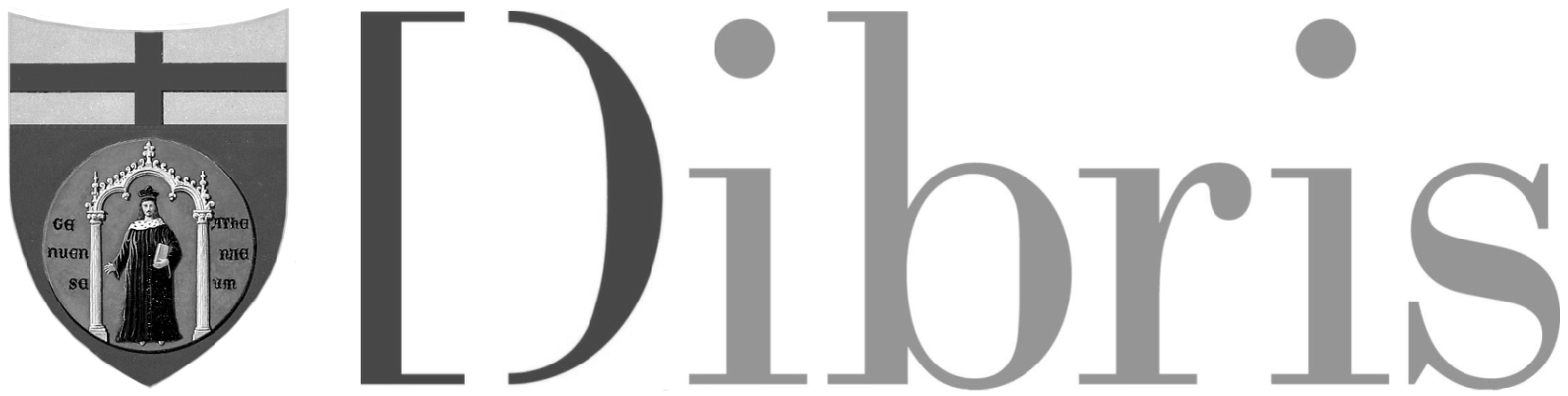}}

\begin{document}

%\pagestyle{fancy}
%
%\maketitle
%
%\tableofcontents
%
%\newpage

%\setlength{\belowdisplayskip}{9pt}
%\setlength{\belowdisplayshortskip}{0pt}
%\setlength{\abovedisplayskip}{9pt}
%\setlength{\abovedisplayshortskip}{0pt}

\pagestyle{fancy}

\maketitle

\vspace{20pt}

\begin{abstract}
Four lemmas, which constitute the theoretical foundation necessary to determine optimal control strategies for a class of single machine family scheduling problems, are presented in this technical report. The scheduling problem is characterized by the presence of sequence-dependent batch setup and controllable processing times; moreover, the generalized due-date model is adopted in the problem. The lemmas are employed within a constructive procedure (proposed by the Author and based on the application of dynamic programming) that allows determining the decisions which optimally solve the scheduling problem as functions of the system state. Two complete examples of single machine family scheduling problem are included in the technical report with the aim of illustrating the application of the fundamental lemmas in the proposed approach.
\end{abstract}

\newpage

\section{Introduction} \label{sec:introduction}

In~\cite{Aicardi2008}, a class of single machine family scheduling problems (mainly characterized by multiclass jobs, generalized due-dates, and controllable processing times) has been formalized as an optimal control problem. Its solution consists of optimal control strategies which are functions of the system state, and therefore they are able to provide the optimal decisions for any actual machine behavior (the single machine is assumed to be unreliable and then perturbations, such as breakdowns, generic unavailabilities, and slowdowns, may affect the nominal behavior of the system). However, the scheduling problem in~\cite{Aicardi2008} has been solved under the assumption that, for each class of jobs, any unitary tardiness cost is greater than the unitary cost related to the deviation from the nominal service time. In order to remove such a strong hypothesis and to extend the scheduling model by adding setup times and, especially, setup costs, new fundamental lemmas have been defined. They are employed within the constructive procedure proposed in~\cite{GiglioJOSH} that solves, from a control-theoretic perspective, a single machine scheduling problem with sequence-dependent batch setup and controllable processing times.

This technical report is organized as follows. Some preliminary definitions are reported in section~\ref{sec:definitions}. The four new lemmas are presented in sections~\ref{sec:lemmas}, together with their complete proofs. Nine numerical examples aiming at illustrating how lemmas~\ref{lem:xopt} and~\ref{lem:h(t)} work are in section~\ref{sec:examples}. Finally, sections~\ref{sec:app1} and~\ref{sec:app2} present two complete example which explain the application of the procedure proposed in~\cite{GiglioJOSH} to two single machine family scheduling problems (the latter with setup).

\section{Definitions} \label{sec:definitions}

\begin{mydefinition} \label{def:f(x)}
Consider a function $f(x)$ which is continuous, nondecreasing, and piece-wise linear function of the independent variable $x$. Let $f(x)$ be characterized by $M \geq 1$ changes of slopes and let $\gamma_{i}$, $i = 1, \ldots, M$, be the values of the horizontal axis at which the slope changes ($\gamma_{i+1} > \gamma_{i}$, $\forall \, i = 1, \ldots, M-1$). In this connection, let $\mu_{0}$ be the slope in interval $(-\infty,\gamma_{1})$, $\mu_{i}$ be the slope in interval $[\gamma_{i},\gamma_{i+1})$, $i = 1, \ldots, M-1$, and $\mu_{M}$ be the slope in interval $[\gamma_{M},+\infty)$ ($\mu_{i+1} \neq \mu_{i}$, $\forall \, i = 0, \ldots, M-1$). Moreover, it is assumed $f(x) = 0$ for any $x \leq \gamma_{1}$; then, $\mu_{0} = 0$. An example of function $f(x)$ following this definition is in figure~\ref{fig:f(x)}.
\end{mydefinition}

\begin{figure}[h]
\centering
\psfrag{f(x)}[cl][Bl][1][0]{$f(x)$}
\psfrag{x}[bc][Bl][1][0]{$x$}
\psfrag{G1}[Bc][Bl][1][0]{$\gamma_{1}$}
\psfrag{G2}[Bc][Bl][1][0]{$\gamma_{2}$}
\psfrag{G3}[Bc][Bl][1][0]{$\gamma_{3}$}
\psfrag{G4}[Bc][Bl][1][0]{$\gamma_{4}$}
\psfrag{G5}[Bc][Bl][1][0]{$\gamma_{5}$}
\psfrag{G6}[Bc][Bl][1][0]{$\gamma_{6}$}
\psfrag{G7}[Bc][Bl][1][0]{$\gamma_{7}$}
\psfrag{G8}[Bc][Bl][1][0]{$\gamma_{8}$}
\psfrag{G9}[Bc][Bl][1][0]{$\gamma_{9}$}
\psfrag{m0}[bc][Bl][.8][0]{$\mu_{0}$}
\psfrag{m1}[cr][Bl][.8][0]{$\mu_{1}$}
\psfrag{m2}[cr][Bl][.8][0]{$\mu_{2}$}
\psfrag{m3}[cr][Bl][.8][0]{$\mu_{3}$}
\psfrag{m4}[cr][Bl][.8][0]{$\mu_{4}$}
\psfrag{m5}[cr][Bl][.8][0]{$\mu_{5}$}
\psfrag{m6}[cr][Bl][.8][0]{$\mu_{6}$}
\psfrag{m7}[cr][Bl][.8][0]{$\mu_{7}$}
\psfrag{m8}[cr][Bl][.8][0]{$\mu_{8}$}
\psfrag{m9}[cr][Bl][.8][0]{$\mu_{9}$}
\includegraphics[scale=.35]{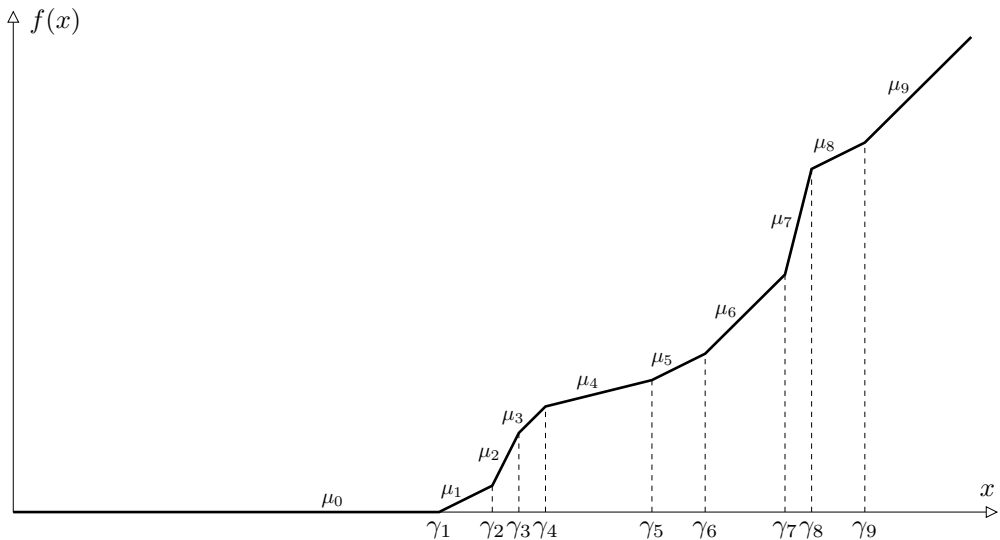}
\caption{Example of function $f(x)$.}
\label{fig:f(x)}
\end{figure}

\vspace{12pt}

\begin{mydefinition} \label{def:f(x+t)}
With reference to $f(x)$, as defined by definition~\ref{def:f(x)}, let $f(x+t)$ be a continuous, nondecreasing, and piece-wise linear function of the independent variable $x$, parameterized by the real value $t$. An example of function $f(x+t)$ following this definition is in figure~\ref{fig:f(x+t)}.
\end{mydefinition}

\begin{figure}[h]
\centering
\psfrag{f(x)}[cl][Bl][1][0]{$f(x+t)$}
\psfrag{x}[bc][Bl][1][0]{$x$}
\psfrag{G1}[Bl][Bl][1][-60]{$\gamma_{1} - t$}
\psfrag{G2}[Bl][Bl][1][-60]{$\gamma_{2} - t$}
\psfrag{G3}[Bl][Bl][1][-60]{$\gamma_{3} - t$}
\psfrag{G4}[Bl][Bl][1][-60]{$\gamma_{4} - t$}
\psfrag{G5}[Bl][Bl][1][-60]{$\gamma_{5} - t$}
\psfrag{G6}[Bl][Bl][1][-60]{$\gamma_{6} - t$}
\psfrag{G7}[Bl][Bl][1][-60]{$\gamma_{7} - t$}
\psfrag{G8}[Bl][Bl][1][-60]{$\gamma_{8} - t$}
\psfrag{G9}[Bl][Bl][1][-60]{$\gamma_{9} - t$}
\psfrag{m0}[bc][Bl][.8][0]{$\mu_{0}$}
\psfrag{m1}[cr][Bl][.8][0]{$\mu_{1}$}
\psfrag{m2}[cr][Bl][.8][0]{$\mu_{2}$}
\psfrag{m3}[cr][Bl][.8][0]{$\mu_{3}$}
\psfrag{m4}[cr][Bl][.8][0]{$\mu_{4}$}
\psfrag{m5}[cr][Bl][.8][0]{$\mu_{5}$}
\psfrag{m6}[cr][Bl][.8][0]{$\mu_{6}$}
\psfrag{m7}[cr][Bl][.8][0]{$\mu_{7}$}
\psfrag{m8}[cr][Bl][.8][0]{$\mu_{8}$}
\psfrag{m9}[cr][Bl][.8][0]{$\mu_{9}$}
\includegraphics[scale=.35]{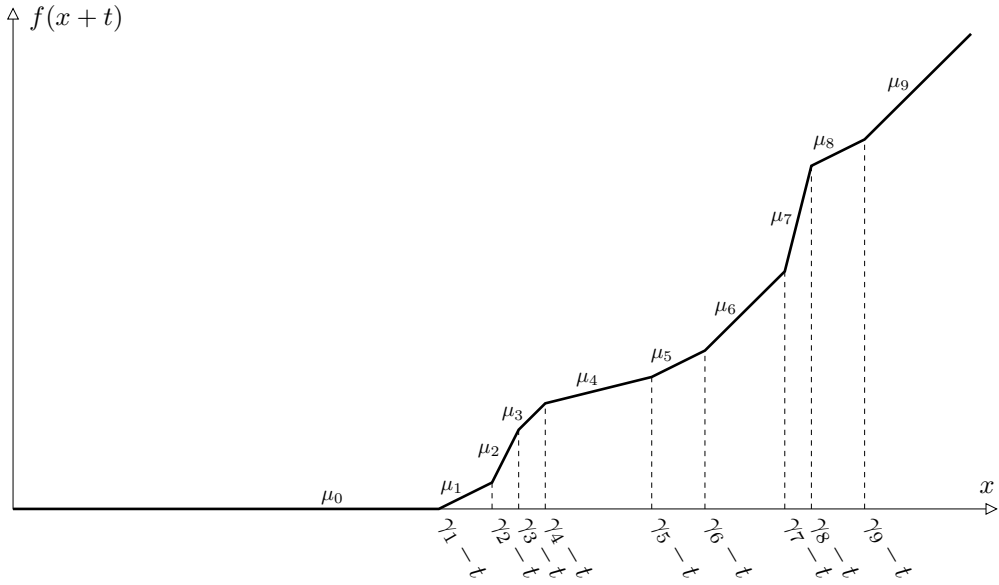}
\vspace{18pt}
\caption{Example of function $f(x+t)$.}
\label{fig:f(x+t)}
\end{figure}

\vspace{12pt}

\begin{mydefinition} \label{def:g(x)}
Consider a function $g(x)$ which is noncontinuous, nonincreasing, and piece-wise linear function of the independent variable $x$. Let $g(x)$ be defined as
\begin{equation}
g(x) = \left\{ \begin{array}{ll}
- \nu (x - x_{2}) & x \in [x_{1},x_{2})\\
0 & x \notin [x_{1},x_{2})
\end{array} \right.
\end{equation}
 An example of function $g(x)$ following this definition is in figure~\ref{fig:g(x)}.
 \end{mydefinition}

\begin{figure}[h]
\centering
\psfrag{g(x)}[cl][Bl][1][0]{$g(x)$}
\psfrag{x}[bc][Bl][1][0]{$x$}
\psfrag{G1}[Bc][Bl][1][0]{$x_{1}$}
\psfrag{G2}[Bc][Bl][1][0]{$x_{2}$}
\psfrag{G3}[Bc][Bl][1][0]{$\gamma_{3}$}
\psfrag{G4}[Bc][Bl][1][0]{$\gamma_{4}$}
\psfrag{G5}[Bc][Bl][1][0]{$\gamma_{5}$}
\psfrag{G6}[Bc][Bl][1][0]{$\gamma_{6}$}
\psfrag{G7}[Bc][Bl][1][0]{$\gamma_{7}$}
\psfrag{G8}[Bc][Bl][1][0]{$\gamma_{8}$}
\psfrag{G9}[Bc][Bl][1][0]{$\gamma_{9}$}
\psfrag{m0}[cl][Bl][.8][0]{$-\nu$}
\psfrag{m1}[cr][Bl][.8][0]{$\mu_{1}$}
\psfrag{m2}[cr][Bl][.8][0]{$\mu_{2}$}
\psfrag{m3}[cr][Bl][.8][0]{$\mu_{3}$}
\psfrag{m4}[cr][Bl][.8][0]{$\mu_{4}$}
\psfrag{m5}[cr][Bl][.8][0]{$\mu_{5}$}
\psfrag{m6}[cr][Bl][.8][0]{$\mu_{6}$}
\psfrag{m7}[cr][Bl][.8][0]{$\mu_{7}$}
\psfrag{m8}[cr][Bl][.8][0]{$\mu_{8}$}
\psfrag{m9}[cr][Bl][.8][0]{$\mu_{9}$}
\includegraphics[scale=.35]{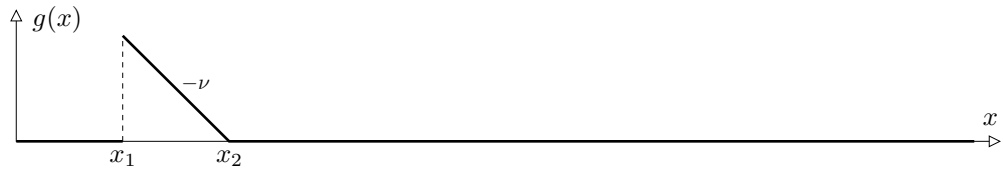}
\caption{Example of function $g(x)$.}
\label{fig:g(x)}
\end{figure}

\section{Lemmas} \label{sec:lemmas}

In connection with functions $f(x)$ and $g(x)$ as defined by definitions~\ref{def:f(x)} and~\ref{def:g(x)}, let:
\begin{itemize}
\item $A$ be the set of indices $i$, $i \in \{ 1, \ldots, M \}$, such that $\mu_{i-1} < \nu$ and $\mu_{i} \geq \nu$; in this connection, let $\lvert A \rvert$ be the cardinality of set $A$ and, if $\lvert A \rvert > 0$, let $a_{j}$, $j = 1, \ldots, \lvert A \rvert$, be the generic element of set $A$; thus, $\gamma_{a_{j}}$, $j = 1, \ldots, \lvert A \rvert$, are the value of the horizontal axis at which the slope of $f(x)$ changes from a value less than $\nu$ to a value greater than or equal to $\nu$;
\item $B$ be the set of indices $i$, $i \in \{ 1, \ldots, M \}$, such that $\mu_{i-1} \geq \nu$ and $\mu_{i} < \nu$; in this connection, let $\lvert B \rvert$ be the cardinality of set $B$ and, if $\lvert B \rvert > 0$ let $b_{j}$, $j = 1, \ldots, \lvert B \rvert$, be the generic element of set $B$; thus, $\gamma_{b_{j}}$, $j = 1, \ldots, \lvert B \rvert$, are the value of the horizontal axis at which the slope of $f(x)$ changes from a value greater than or equal to $\nu$ to a value less than $\nu$.
\end{itemize}

Since it has been assumed $\mu_{0} = 0$, then $a_{j} < b_{j}$ $\forall \, j = 1, \ldots, \lvert B \rvert$ and $b_{j} < a_{j+1}$ $\forall \, j = 1, \ldots, \lvert A \rvert - 1$. Moreover, $\lvert A \rvert - \lvert B \rvert \leq 1$ being $\lvert A \rvert = \lvert B \rvert$ if $\mu_{M} < \nu$ and $\lvert A \rvert = \lvert B \rvert + 1$ if $\mu_{M} \geq \nu$.

\vspace{12pt}

\begin{mylemma} \label{lem:xopt}
Let $f(x+t)$ be a continuous nondecreasing piece-wise linear function of $x$, parameterized by $t$, as defined by definition~\ref{def:f(x+t)}, and let $g(x)$ be a noncontinuous function of $x$, as defined by definition~\ref{def:g(x)}.

In case $\lvert B \rvert \geq 1$, let $\Omega$ be the set of time instants $\{ \omega_{1} , \ldots, \omega_{j}, \ldots, \omega_{\lvert B \rvert} \}$ in which any value $\omega_{j}$, $j = 1, \ldots, \lvert B \rvert$, is obtained by executing algorithm~\ref{alg:tstar}. Each value $\omega_{j}$, $j = 1, \ldots, \lvert B \rvert$, is either finite or nonfinite. Let $T$ be the set of time instants $\{ t^{\star}_{1} , \ldots, t^{\star}_{q}, \ldots, t^{\star}_{Q} \}$ which is obtained from $\Omega$ by removing all nonfinite values from it, that is
\begin{equation} \label{equ:Tq}
T = \Omega \setminus \big\{ \omega_{j} \, : \, \omega_{j} = +\infty \, , \, j = 1, \ldots, \lvert B \rvert \big\}
\end{equation}
Let $Q$ be the cardinality of set $T$; it is obviously $1 \leq Q \leq \lvert B \rvert$. In case $\lvert B \rvert = 0$, it is $T = \emptyset$ and $Q = 0$.

Then, the function of $t$
\begin{equation} \label{equ:xopt}
x^{\circ}(t) = \arg \min_{\substack{x\\x_{1} \leq x \leq x_{2}}} \big\{ f(x+t) + g(x) \big\}
\end{equation}
is a nonincreasing, possibly noncontinuous, piece-wise linear function of $t$ defined as
\begin{subequations}
\begin{equation} \label{equ:xopt_1}
\text{if $Q = 0$} \, : \quad x^{\circ}(t) = x_{\mathrm{e}}(t)
\end{equation}
\begin{equation} \label{equ:xopt_2}
\text{if $Q = 1$} \, : \quad x^{\circ}(t) = \left\{ \begin{array}{ll}
x_{\mathrm{s}}(t) & t < t^{\star}_{1}\\
x_{\mathrm{e}}(t) & t \geq t^{\star}_{1}\\
\end{array} \right.
\end{equation}
\begin{equation} \label{equ:xopt_3}
\text{if $Q > 1$} \, : \quad x^{\circ}(t) = \left\{ \begin{array}{ll}
x_{\mathrm{s}}(t) & t < t^{\star}_{1}\\
x_{q}(t) & t^{\star}_{q} \leq t < t^{\star}_{q+1} , \quad q = 1, \ldots, Q - 1\\
x_{\mathrm{e}}(t) & t \geq t^{\star}_{Q}
\end{array} \right.
\end{equation}
\end{subequations}
where $x_{\mathrm{s}}(t)$, $x_{q}(t)$, and $x_{\mathrm{e}}(t)$ are the following functions of $t$:
\begin{itemize}
\item $x_{\mathrm{s}}(t)$ is a continuous nonincreasing piece-wise linear functions of $t$ defined as:
\begin{subequations} \label{equ:xs}
\begin{equation} \label{equ:xs_1}
\left. \begin{array}{l}
\text{if $t^{\star}_{1} > \gamma_{a_{1}} - x_{1}$}
\end{array} \right. \rightarrow \quad x_{\mathrm{s}}(t) = \left\{ \begin{array}{ll}
x_{2} & t < \gamma_{a_{1}} - x_{2}\\
-t + \gamma_{a_{1}} & \gamma_{a_{1}} - x_{2} \leq t < \gamma_{a_{1}} - x_{1}\\
x_{1} & \gamma_{a_{1}} - x_{1} \leq t < t^{\star}_{1}
\end{array} \right.
\end{equation}
\begin{equation} \label{equ:xs_2}
\left. \begin{array}{l}
\text{if $t^{\star}_{1} \leq \gamma_{a_{1}} - x_{1}$}
\end{array} \right. \rightarrow \quad x_{\mathrm{s}}(t) = \left\{ \begin{array}{ll}
x_{2} & t < \gamma_{a_{1}} - x_{2}\\
-t + \gamma_{a_{1}} & \gamma_{a_{1}} - x_{2} \leq t < t^{\star}_{1}
\end{array} \right.
\end{equation}
\end{subequations}
\item $x_{q}(t)$ is a continuous nonincreasing piece-wise linear functions of $t$ defined as:
\begin{subequations} \label{equ:xj}
\begin{equation} \label{equ:xj_1}
\left. \begin{array}{l}
\text{if $t^{\star}_{q} < \gamma_{a_{l(q)+1}} - x_{2}$}\\
\text{and $t^{\star}_{q+1} > \gamma_{a_{l(q)+1}} - x_{1}$}
\end{array} \right. \hspace{-0.1cm} \rightarrow \  x_{q}(t) = \left\{ \begin{array}{ll}
x_{2} & t^{\star}_{q} \leq t < \gamma_{a_{l(q)+1}} - x_{2}\\
-t + \gamma_{a_{l(q)+1}} & \gamma_{a_{l(q)+1}} - x_{2} \leq t < \gamma_{a_{l(q)+1}} - x_{1}\\
x_{1} & \gamma_{a_{l(q)+1}} - x_{1} \leq t < t^{\star}_{q+1}
\end{array} \right.
\end{equation}
\begin{equation} \label{equ:xj_2}
\left. \begin{array}{l}
\text{if $t^{\star}_{q} \geq \gamma_{a_{l(q)+1}} - x_{2}$}\\
\text{and $t^{\star}_{q+1} > \gamma_{a_{l(q)+1}} - x_{1}$}
\end{array} \right. \hspace{-0.1cm} \rightarrow \  x_{q}(t) = \left\{ \begin{array}{ll}
-t + \gamma_{a_{l(q)+1}} & t^{\star}_{q} \leq t < \gamma_{a_{l(q)+1}} - x_{1}\\
x_{1} & \gamma_{a_{l(q)+1}} - x_{1} \leq t < t^{\star}_{q+1}
\end{array} \right.
\end{equation}
\begin{equation} \label{equ:xj_3}
\left. \begin{array}{l}
\text{if $t^{\star}_{q} < \gamma_{a_{l(q)+1}} - x_{2}$}\\
\text{and $t^{\star}_{q+1} \leq \gamma_{a_{l(q)+1}} - x_{1}$}
\end{array} \right. \hspace{-0.1cm} \rightarrow \  x_{q}(t) = \left\{ \begin{array}{ll}
x_{2} & t^{\star}_{q} \leq t < \gamma_{a_{l(q)+1}} - x_{2}\\
-t + \gamma_{a_{l(q)+1}} & \gamma_{a_{l(q)+1}} - x_{2} \leq t < t^{\star}_{q+1}
\end{array} \right.
\end{equation}
\begin{equation} \label{equ:xj_4}
\left. \begin{array}{l}
\text{if $t^{\star}_{q} \geq \gamma_{a_{l(q)+1}} - x_{2}$}\\
\text{and $t^{\star}_{q+1} \leq \gamma_{a_{l(q)+1}} - x_{1}$}
\end{array} \right. \hspace{-0.1cm} \rightarrow \  x_{q}(t) = -t + \gamma_{a_{l(q)+1}}
\end{equation}
\end{subequations}
\item $x_{\mathrm{e}}(t)$ is a continuous nonincreasing piece-wise linear functions of $t$ defined as:
\begin{subequations} \label{equ:xe}
\begin{equation} \label{equ:xe_1}
\left. \begin{array}{l}
\text{if $l(Q) < \lvert A \rvert$}\\
\text{and $t^{\star}_{Q} < \gamma_{a_{l(Q)+1}} - x_{2}$} 
\end{array} \right. \hspace{-0.1cm} \rightarrow \  x_{\mathrm{e}}(t) = \left\{ \begin{array}{ll}
x_{2} & t^{\star}_{Q} \leq t < \gamma_{a_{l(Q)+1}} - x_{2}\\
-t + \gamma_{a_{l(Q)+1}} & \gamma_{a_{l(Q)+1}} - x_{2} \leq t < \gamma_{a_{l(Q)+1}} - x_{1}\\
x_{1} & t \geq \gamma_{a_{l(Q)+1}} - x_{1}
\end{array} \right.
\end{equation}
\begin{equation} \label{equ:xe_2}
\left. \begin{array}{l}
\text{if $l(Q) < \lvert A \rvert$}\\
\text{and $t^{\star}_{Q} \geq \gamma_{a_{l(Q)+1}} - x_{2}$} 
\end{array} \right. \hspace{-0.1cm} \rightarrow \  x_{\mathrm{e}}(t) = \left\{ \begin{array}{ll}
-t + \gamma_{a_{l(Q)+1}} & t^{\star}_{Q} \leq t < \gamma_{a_{l(Q)+1}} - x_{1}\\
x_{1} & t \geq \gamma_{a_{l(Q)+1}} - x_{1}
\end{array} \right.
\end{equation}
\begin{equation} \label{equ:xe_3}
\left. \begin{array}{l}
\text{if $l(Q) = \lvert A \rvert$} 
\end{array} \right. \qquad\qquad \rightarrow \  x_{\mathrm{e}}(t) = x_{2}
\end{equation}
\end{subequations}
having assumed (for notational convenience) $t^{\star}_{Q} = -\infty$ when $Q = 0$.
\end{itemize}
In~\eqref{equ:xj} and~\eqref{equ:xe}, $l(q)$, $q = 1, \ldots, Q$, is a mapping function which  provides the index $j \in \{ 1, \ldots, \lvert B \rvert \}$ of the value $t^{\star}_{q}$ in the set $\Omega$, that is, $l(q) = j \Leftrightarrow \omega_{j} = t^{\star}_{q}$. In this connection, it is always $l(Q) = \lvert B \rvert$ and, in case $\lvert A \rvert > \lvert B \rvert$, it turns out $l(Q)+1 = \lvert A \rvert$. Moreover, it is assumed, for notational convenience, $l(0) = 0$.

\vspace{12pt}

\begin{myalgorithm} \label{alg:tstar}
Determination of the time instant $\omega_{j}$, $j = 1, \ldots, \lvert B \rvert$, at which, in case $\omega_{j} < +\infty$, the function $x^{\circ}(t)$ jumps in an upward direction.
\begin{multicols}{2}

{\it SECTION A -- INITIALIZATION}

\begin{algorithmic}[1]
\State $\gamma_{0} = -\infty$
\State $h \geq 0 \, : \, \gamma_{h} \leq \gamma_{b_{j}} - (x_{2} - x_{1}) < \gamma_{h+1}$ \label{alg:h}
\State $i = b_{j}$
\State $\gamma_{M+1} = +\infty$
\State $k \leq M \, : \, \gamma_{k} < \gamma_{b_{j}} + (x_{2} - x_{1}) \leq \gamma_{k+1}$ \label{alg:k}
\If{$j = \lvert B \rvert \algorithmicand \lvert A \rvert = \lvert B \rvert$}
\State $a_{j+1} = M+1$
\EndIf
\For{$p = h \algorithmicto k$}
\State $\tilde{\mu}_{p} = \mu_{p} - \nu$
\EndFor
\State $\tau = \gamma_{b_{j}} - (x_{2} - x_{1})$
\State $\theta = \gamma_{b_{j}}$
\State $d = \max \{ 0 , \tilde{\mu}_{h} ( \gamma_{h+1} - \tau ) \}$
\If{$h < b_{j}-1$}
\For{$p = h+1 \algorithmicto b_{j}-1$}
\State $d = \max \{ 0 , d + \tilde{\mu}_{p} ( \gamma_{p+1} - \gamma_{p} ) \}$
\EndFor
\EndIf
\State $\lambda = h$
\State $\xi = i$
\algstore{bkbreak}
\end{algorithmic}

{\it SECTION B -- FIRST LOOP}

\begin{algorithmic}[1]
\algrestore{bkbreak}
\While{$h < b_{j} \algorithmicand i < a_{j+1}$}
\State $\psi = \min \{ \gamma_{h+1} - \tau , \gamma_{i+1} - \theta \}$ \label{alg:psi}

\If{$\gamma_{h+1} - \tau \leq \gamma_{i+1} - \theta$}
\State $\lambda = h + 1$
\EndIf
\If{$\gamma_{h+1} - \tau \geq \gamma_{i+1} - \theta$}
\State $\xi = i+1$
\EndIf

% calcolo di delta nel primo caso
\State $\delta = \max \{ 0 , \tilde{\mu}_{\lambda} [ \gamma_{\lambda+1} - (\tau + \psi) ] \}$
\If{$\lambda < b_{j}-1$}
\For{$p = \lambda+1 \algorithmicto b_{j}-1$}
\State $\delta = \max \{ 0 , \delta + \tilde{\mu}_{p} ( \gamma_{p+1} - \gamma_{p} ) \}$
\EndFor
\EndIf
\If{$\xi = b_{j}$}
\State $\delta = \delta + \tilde{\mu}_{\xi} [ (\theta + \psi) - \gamma_{\xi} ]$
\ElsIf{$\xi = a_{j+1}$}
\State $\delta = \delta + \sum_{p=b_{j}}^{\xi-1} \tilde{\mu}_{p} ( \gamma_{p+1} - \gamma_{p} )$
\Else
\State $\delta = \delta + \sum_{p=b_{j}}^{\xi-1} \tilde{\mu}_{p} ( \gamma_{p+1} - \gamma_{p} ) +$
\Statex $\qquad\qquad\qquad\qquad\qquad\qquad + \tilde{\mu}_{\xi} [ (\theta + \psi) - \gamma_{\xi} ]$
\EndIf

% determinazione di tj (e uscita) nel primo caso
\If{$\delta \leq 0$} \label{alg:conddrhodelta}
\State $a_{0} = 0$
\State $r \geq 1 \, : \, a_{r-1} \leq h < a_{r}$
\If{$r \leq j$}
\For{$q = r \algorithmicto j$}
\State $\chi = \tilde{\mu}_{h} ( \gamma_{h+1} - \tau )$
\If{$h < a_{q}-1$}
\State $\chi = \chi + \sum_{p=h+1}^{a_{q}-1} \tilde{\mu}_{p} ( \gamma_{p+1} - \gamma_{p} )$
\EndIf
\If{$q = r$}
\State $m = \chi$
\Else
\State $m = \min \{ m, \chi \}$
\EndIf
\EndFor
\If{$m \leq 0$}
\State $\omega_{j} = \tau - x_{1} - \frac{d}{\tilde{\mu}_{i}}$
\ElsIf{$-\frac{d-m}{\tilde{\mu}_{i}} \leq \frac{m}{\tilde{\mu}_{h}}$}
\State $\omega_{j} = \tau - x_{1} + \frac{d}{\tilde{\mu}_{h} - \tilde{\mu}_{i}}$
\State \algorithmicexit
\Else
\State $\omega_{j} = \tau - x_{1} - \frac{d-m}{\tilde{\mu}_{i}}$
\State \algorithmicexit
\EndIf
\Else
\State $\omega_{j} = \tau - x_{1} + \frac{d}{\tilde{\mu}_{h} - \tilde{\mu}_{i}}$
\State \algorithmicexit
\EndIf
\Else

% aggiornamento valori
\State $h = \lambda$
\State $i = \xi$
\State $\tau = \tau + \psi$
\State $\theta = \theta + \psi$
\State $d = \delta$
\EndIf

\EndWhile
%\State $t^{\star}_{j} = +\infty$
\algstore{bkbreak}
\end{algorithmic}

{\it SECTION C -- SECOND LOOP}

\begin{algorithmic}[1]
\algrestore{bkbreak}
\While{$h < b_{j}$} \label{alg:secondloopinizio}
\State $\psi = \gamma_{h+1} - \tau$
\State $\lambda = h+1$

% calcolo di delta nel secondo caso
\If{$\lambda < b_{j}$}
\State $\delta = \max \{ 0 , \tilde{\mu}_{\lambda} [ \gamma_{\lambda+1} - (\tau + \psi) ] \}$
\If{$\lambda < b_{j}-1$}
\For{$p = \lambda+1 \algorithmicto b_{j}-1$}
\State $\delta = \max \{ 0 , \delta + \tilde{\mu}_{p} ( \gamma_{p+1} - \gamma_{p} ) \}$
\EndFor
\EndIf
\Else
\State $\delta = 0$
\EndIf
\State $\delta = \delta + \sum_{p=b_{j}}^{a_{j+1}-1} \tilde{\mu}_{p} ( \gamma_{p+1} - \gamma_{p} )$ \label{alg:delta}

\If{$\delta \leq 0$} \label{alg:deltamin0inizio}
\State $\tau = \tau + \frac{d}{\tilde{\mu}_{h}}$
\State $\theta = \tau + (x_{2} - x_{1})$
\State $k \leq M \, : \, \gamma_{k} < \theta \leq \gamma_{k+1}$
\State $r = a_{j+1}$
\State $\phi = 0$
\While{$r \leq k$}
\If{$r < k$}
\State $\phi = \phi + \tilde{\mu}_{r} ( \gamma_{r+1} - \gamma_{r} )$
\Else
\State $\phi = \phi + \tilde{\mu}_{r} ( \theta - \gamma_{r} )$
\EndIf
\If{$\phi < 0$}
\State $\omega_{j} = +\infty$
\State \algorithmicexit
\Else
\State $r = r +1$
\EndIf
\EndWhile
\State $\omega_{j} = \tau - x_{1}$
\State \algorithmicexit \label{alg:deltamin0fine}
\Else \label{alg:deltamax0inizio}
\State $h = \lambda$
\State $\tau = \tau + \psi$
\State $d = \delta$
\EndIf \label{alg:deltamax0fine}
\EndWhile \label{alg:secondloopfine}
\end{algorithmic}

\end{multicols}
\end{myalgorithm}

\end{mylemma}

\vspace{12pt}

\begin{proof}
The function $x^{\circ}(t)$ can be obtained by analyzing the shape of the function $f(x+t) + g(x)$ in the interval $[x_{1},x_{2}]$, with $t$ moving from $-\infty$ to $+\infty$. The proof consists of seven parts:
\begin{enumerate}
\item in the first part, it is proven that, when $\lvert B \rvert = 0$, $x^{\circ}(t)$ has the structure provided by~\eqref{equ:xopt_1}, with $x_{\mathrm{e}}(t)$ provided by~\eqref{equ:xe_1} (if $l(Q) < \lvert A \rvert$) or~\eqref{equ:xe_3} (if $l(Q) = \lvert A \rvert$);

\item in the second part, it is proven that, when $\gamma_{b_{j}} - \gamma_{a_{j}} > (x_{2}-x_{1})$, $\forall \, j = 1, \ldots, \lvert B \rvert$, $\lvert B \rvert > 0$, and $\gamma_{a_{j+1}} - \gamma_{b_{j}} > (x_{2}-x_{1})$, $\forall \, j = 1, \ldots, \lvert A \rvert - 1$, $\lvert A \rvert > 1$, $x^{\circ}(t)$ has the structure provided by~\eqref{equ:xopt_2} or~\eqref{equ:xopt_3}, with $x_{\mathrm{s}}(t)$ provided by~\eqref{equ:xs_1}, $x_{q}(t)$, $q = 1, \ldots, Q-1$, $Q > 1$, provided by~\eqref{equ:xj_1}, and $x_{\mathrm{e}}(t)$ provided by~\eqref{equ:xe_1} (if $l(Q) < \lvert A \rvert$) or~\eqref{equ:xe_3} (if $l(Q) = \lvert A \rvert$);

\item in the third part, it is shown that the number of jump discontinuities in $x^{\circ}(t)$ may be less than $\lvert B \rvert$, that is, they are $Q \leq \lvert B \rvert$, and the conditions for which a jump discontinuity does not exist in connection with a specific abscissa $\gamma_{b_{j}} - t$, $j \in \{ 1, \ldots, \lvert B \rvert \}$ are provided;

\item in the fourth part, it is proven that, even if the assumptions considered in the second part do not hold for some $j \in \{ 1, \ldots, \lvert B \rvert \}$ or $j \in \{ 1, \ldots, \lvert A \rvert - 1\}$, it is sufficient that $t^{\star}_{1} > \gamma_{a_{1}} - x_{1}$, or $t^{\star}_{q} < \gamma_{a_{l(q)+1}} - x_{2}$ and $t^{\star}_{q+1} > \gamma_{a_{l(q)+1}} - x_{1}$, $q \in \{ 1, \ldots, Q-1 \}$, $Q > 1$, or $t^{\star}_{Q} < \gamma_{a_{l(Q)+1}} - x_{2}$, to guarantee that $x_{\mathrm{s}}(t)$ has the structure provided by~\eqref{equ:xs_1}, $x_{q}(t)$ has the structure provided by~\eqref{equ:xj_1}, and $x_{\mathrm{e}}(t)$ has the structure provided by~\eqref{equ:xe_1} (if $l(Q) < \lvert A \rvert$), respectively;

\item in the fifth part, it is proven that, if $t^{\star}_{q} \geq \gamma_{a_{l(q)+1}} - x_{2}$, $q \in \{ 1, \ldots, Q-1 \}$ or $t^{\star}_{Q} \geq \gamma_{a_{l(Q)+1}} - x_{2}$, then there is at $t = t^{\star}_{q}$ a discontinuity in $x^{\circ}(t)$ at which it jumps upwardly from $x_{1}$ to $-t + \gamma_{a_{l(q)+1}} \leq x_{2}$, $q \in \{ 1, \ldots, Q \}$, that is, $x_{q}(t)$ in~\eqref{equ:xopt_3} has the structure of~\eqref{equ:xj_2} or~\eqref{equ:xj_4} or $x_{\mathrm{e}}(t)$ in~\eqref{equ:xopt_2} and~\eqref{equ:xopt_3} has the structure of~\eqref{equ:xe_2};

\item in the sixth part, it is proven that, if $t^{\star}_{1} \leq \gamma_{a_{1}} - x_{1}$ or $t^{\star}_{q+1} \leq \gamma_{a_{l(q)+1}} - x_{1}$, $q \in \{ 1, \ldots, Q-1 \}$, then there is at $t = t^{\star}_{q+1}$ a discontinuity in $x^{\circ}(t)$ at which it jumps upwardly from $-t + \gamma_{a_{l(q)+1}} \geq x_{1}$ to $x_{2}$, $q \in \{ 1, \ldots, Q \}$, that is, $x_{\mathrm{s}}(t)$ in~\eqref{equ:xopt_2} and~\eqref{equ:xopt_3} has the structure of~\eqref{equ:xs_2} or $x_{q}(t)$ in~\eqref{equ:xopt_3} has the structure of~\eqref{equ:xj_3} or~\eqref{equ:xj_4};

\item in the seventh and last part, algorithm~\ref{alg:tstar}, which allows determining time instants $\omega_{j}$, $j = 1, \ldots, \lvert B \rvert$, is described.
\end{enumerate}

\vspace{12pt}

\underline{First part}

Consider the case $\lvert B \rvert = 0$, which implies $\Omega = \emptyset$ and then $T = \emptyset$ and $Q = 0$. Moreover, $l(0) = 0$. If $\lvert A \rvert = 0$ as well, all the slopes of $f(x+t)$ are less than $\nu$; then, $f(x+t) + g(x)$ is a strictly decreasing function of $x$, $\forall \, t$. In this case, the minimum of the function $f(x+t) + g(x)$, with respect to $[x_{1},x_{2}]$, is always obtained at $x_{2}$. Thus, in this case, $x^{\circ}(t)$ has the structure provided by~\eqref{equ:xopt_1}, with $x_{\mathrm{e}}(t)$ provided by~\eqref{equ:xe_3}, being $l(Q) = \lvert A \rvert$.

If $\lvert A \rvert > 0$, it is definitely $\lvert A \rvert = 1 = l(Q)+1$. In this case, the slopes of $f(x+t)$ are less than $\nu$ in the interval $(-\infty,\gamma_{a_{l(Q)+1}})$ and greater than or equal to $\nu$ in $[\gamma_{a_{l(Q)+1}},+\infty)$. This case is very similar to that considered in lemma~1 of~\cite{Aicardi2008}. When $t$ is such that $x_{2} < \gamma_{a_{l(Q)+1}} - t$ (that is, $t < \gamma_{a_{l(Q)+1}} - x_{2}$), the minimum with respect to $x$, $x_{1} \leq x \leq x_{2}$, of $f(x+t) + g(x)$ is obtained at $x_{2}$. When $t$ is such that $x_{1} < \gamma_{a_{l(Q)+1}} - t \leq x_{2}$ (that is, $\gamma_{a_{l(Q)+1}} - x_{2} \leq t < \gamma_{a_{l(Q)+1}} - x_{1}$), the function $f(x+t) + g(x)$ is strictly decreasing in $[x_{1},\gamma_{a_{l(Q)+1}} - t]$ and nondecreasing in $[\gamma_{a_{l(Q)+1}} - t , x_{2}]$; then, it has a minimum, with respect to $[x_{1},x_{2}]$, in $\gamma_{a_{j}} - t$; when $t$ increases in the interval $[\gamma_{a_{l(Q)+1}} - x_{2} , \gamma_{a_{l(Q)+1}} - x_{1})$, the minimum decreases (with unitary speed) from $x_{2}$ to $x_{1}$. Finally, when $t$ is such that $x_{1} \geq \gamma_{a_{l(Q)+1}} - t$ (that is, $t \geq \gamma_{a_{l(Q)+1}} - x_{1}$), the minimum is obtained at $x_{1}$. Thus, in this case, $x^{\circ}(t)$ has the structure provided by~\eqref{equ:xopt_1}, with $x_{\mathrm{e}}(t)$ provided by~\eqref{equ:xe_1}, being $l(Q) < \lvert A \rvert$. Note that, since $t^{\star}_{Q} = -\infty$ when $Q = 0$, it is $t^{\star}_{Q} < \gamma_{a_{l(Q)+1}} - x_{2}$ for sure.

\vspace{12pt}

\underline{Second part}

Consider the case $\lvert A \rvert > 1$ and $\lvert B \rvert > 0$, and assume $\gamma_{b_{j}} - \gamma_{a_{j}} > (x_{2}-x_{1})$, $\forall \, j = 1, \ldots, \lvert B \rvert$, and $\gamma_{a_{j+1}} - \gamma_{b_{j}} > (x_{2}-x_{1})$, $\forall j = 1, \ldots, \lvert A \rvert - 1$. Under such hypotheses, the function $x^{\circ}(t)$ is defined as follows. 

\begin{enumerate}
\item \label{rul:shape1} 
When $t$ is such that the slopes of $f(x+t)$ in the interval $[x_{1},x_{2}]$ are less than $\nu$, that is, $\forall \, t < \gamma_{a_{1}} - x_{2}$, $\forall \, t \in [ \gamma_{b_{j}} - x_{1} , \gamma_{a_{j+1}} - x_{2})$, $j = 1, \ldots, \lvert A \rvert - 1$, and $\forall \, t \geq \gamma_{b_{\lvert B \rvert}} - x_{1}$ if $\lvert A \rvert = \lvert B \rvert$, the minimum of the function $f(x+t) + g(x)$, with respect to $[x_{1},x_{2}]$, is obtained at $x_{2}$, since $f(x+t) + g(x)$ is strictly decreasing in $[x_{1},x_{2}]$.

\item \label{rul:shape2} 
When $t$ is such that $\gamma_{a_{j}} - t \in [x_{1},x_{2}]$, $j = 1, \ldots, \lvert A \rvert$, that is, $\forall \, t \in [ \gamma_{a_{j}} - x_{2} , \gamma_{a_{j}} - x_{1} )$, $j = 1, \ldots, \lvert A \rvert$, the function $f(x+t) + g(x)$ is strictly decreasing in $[x_{1},\gamma_{a_{j}} - t]$ and nondecreasing in $[\gamma_{a_{j}} - t , x_{2}]$; then, it has a minimum, with respect to $[x_{1},x_{2}]$, in $\gamma_{a_{j}} - t$; when $t$ increases in the interval $[\gamma_{a_{j}} - x_{2} , \gamma_{a_{j}} - x_{1})$, the minimum decreases (with unitary speed) from $x_{2}$ to $x_{1}$.

\item \label{rul:shape3} 
When $t$ is such that the slope of $f(x+t)$ in the interval $[x_{1},x_{2}]$ is grater than or equal to $\nu$, that is, $\forall \, t \in [ \gamma_{a_{j}} - x_{1} , \gamma_{b_{j}} - x_{2})$, $j = 1, \ldots, \lvert B \rvert$, and $\forall \, t \geq \gamma_{a_{\lvert A \rvert}} - x_{1}$ if $\lvert A \rvert > \lvert B \rvert$, the minimum of the function $f(x+t) + g(x)$, with respect to $[x_{1},x_{2}]$, is obtained at $x_{1}$, since $f(x+t) + g(x)$ is nondecreasing in $[x_{1},x_{2}]$.

\item \label{rul:shape4} 
When $t$ is such that $\gamma_{b_{j}} - t \in [x_{1},x_{2}]$, $j = 1, \ldots, \lvert B \rvert$, that is, $\forall \, t \in [ \gamma_{b_{j}} - x_{2} , \gamma_{b_{j}} - x_{1} )$, $j = 1, \ldots, \lvert B \rvert$, the function $f(x+t) + g(x)$ is nondecreasing in $[x_{1},\gamma_{b_{j}} - t]$ and strictly decreasing in $[\gamma_{b_{j}} - t , x_{2}]$; then, it has a maximum, with respect to $[x_{1},x_{2}]$, in $x = \gamma_{b_{j}} - t$, and the minimum is obtained either at $x_{1}$ or $x_{2}$, depending on the values $f(x_{1}+t) + g(x_{1})$ and $f(x_{2}+t) + g(x_{2})$, $t \in [ \gamma_{b_{j}} - x_{2} , \gamma_{b_{j}} - x_{1} )$ (the minimum is obtained at $x^{\circ} = x_{1}$ if $f(x_{1}+t) + g(x_{1}) < f(x_{2}+t) + g(x_{2})$ and at $x^{\circ} = x_{2}$ otherwise). In this connection, note that:
\begin{itemize}
\item when $t = \gamma_{b_{j}} - x_{2}$, it is certainly $f(x_{1}+t) + g(x_{1}) \leq f(x_{2}+t) + g(x_{2})$;
\item when $t$ increases in the interval $(\gamma_{b_{j}} - x_{2} , \gamma_{b_{j}} - x_{1})$, the value of $f(x_{1}+t) + g(x_{1})$ increases or remains constant and the value of $f(x_{2}+t) + g(x_{2})$ decreases;
\item when $t = \gamma_{b_{j}} - x_{1}$, it is certainly $f(x_{1}+t) + g(x_{1}) > f(x_{2}+t) + g(x_{2})$.
\end{itemize}
This means that it certainly exists $\omega_{j} \in [ \gamma_{b_{j}} - x_{2} , \gamma_{b_{j}} - x_{1} )$ such that $f(x_{1}+t) + g(x_{1}) \leq f(x_{2}+t) + g(x_{2})$ $\forall \, t \in [ \gamma_{b_{j}} - x_{2}, \omega_{j})$, $f(x_{1}+\omega_{j}) + g(x_{1}) = f(x_{2}+\omega_{j}) + g(x_{2})$, and $f(x_{1}+t) + g(x_{1}) > f(x_{2}+t) + g(x_{2})$ $\forall \, t \in ( \omega_{j} , \gamma_{b_{j}} - x_{1} )$; then, the minimum is obtained at $x_{1}$ $\forall \, t \in [ \gamma_{b_{j}} - x_{2}, \omega_{j})$, ``jumps'' from $x_{1}$ to $x_{2}$ at $\omega_{j}$, and is obtained at $x_{2}$ $\forall \, t \in [ \omega_{j} , \gamma_{b_{j}} - x_{1} )$.
\end{enumerate}
Thus, according to the previous ``rules'', since $\mu_{0} = 0$ the function $x^{\circ}(t)$ is $x_{2}$ at the beginning (rule~\ref{rul:shape1}), decreases with slope $-1$ in the interval $[ \gamma_{a_{1}} - x_{2} , \gamma_{a_{1}} - x_{1} )$ (rule~\ref{rul:shape2}), is equal to $x_{1}$ from $\gamma_{a_{1}} - x_{1}$ to $\omega_{1}$, at which it jumps to $x_{2}$ (rules~\ref{rul:shape3} and~\ref{rul:shape4}); $x^{\circ}(t)$ remains equal to $x_{2}$ from $\omega_{1}$ up to $\gamma_{a_{2}} - x_{2}$ (rules~\ref{rul:shape4} and~\ref{rul:shape1}), then it descreases with slope $-1$ in the interval $[ \gamma_{a_{2}} - x_{2} , \gamma_{a_{2}} - x_{1} )$ (rule~\ref{rul:shape2}), is equal to $x_{1}$ from $\gamma_{a_{2}} - x_{1}$ to $\omega_{2}$, at which jumps to $x_{2}$ (rules~\ref{rul:shape3} and~\ref{rul:shape4}), and so on. In its last part, the function $x^{\circ}(t)$ is $x_{1}$ if $\lvert A \rvert > \lvert B \rvert$ or $x_{2}$ if $\lvert A \rvert = \lvert B \rvert$, in accordance with rules~\ref{rul:shape3} and~\ref{rul:shape1}.

\begin{figure}[h]
\centering
\psfrag{f(x)}[cl][Bl][1][0]{$f(x+t) + g(x)$ when $t = \gamma_{b_{j}} - x_{2}$}
\psfrag{x}[bc][Bl][1][0]{$x$}
\psfrag{G2}[Bl][Bl][1][-60]{$x_{1}$}
\psfrag{G3}[Bl][Bl][1][-60]{$\gamma_{b_{j} - 3} - t$}
\psfrag{G4}[Bl][Bl][1][-60]{$\gamma_{b_{j} - 2} - t$}
\psfrag{G5}[Bl][Bl][1][-60]{$\gamma_{b_{j} - 1} - t$}
\psfrag{G6}[Bl][Bl][1][-60]{$\gamma_{b_{j}} - t \equiv x_{2}$}
\psfrag{G7}[Bl][Bl][1][-60]{$\gamma_{b_{j} + 1} - t$}
\psfrag{G8}[Bl][Bl][1][-60]{$\gamma_{b_{j} + 2} - t$}
\psfrag{G9}[Bl][Bl][1][-60]{$\gamma_{b_{j} + 3} - t$}
\psfrag{G10}[Bl][Bl][1][-60]{$\gamma_{b_{j} + 4} - t$}
\psfrag{G12}[Bl][Bl][1][-60]{$2 x_{2} - x_{1}$}
\psfrag{m0}[cr][Bl][.8][-30]{$\mu_{b_{j} - 4} - \nu$}
\psfrag{m1}[cr][Bl][.8][-30]{$\mu_{b_{j} - 3} - \nu$}
\psfrag{m2}[cr][Bl][.8][-30]{$\mu_{b_{j} - 2} - \nu$}
\psfrag{m3}[cr][Bl][.8][-30]{$\mu_{b_{j} - 1} - \nu$}
\psfrag{m4}[cl][Bl][.8][30]{$\mu_{b_{j}} - \nu$}
\psfrag{m5}[cl][Bl][.8][30]{$\mu_{b_{j} + 1} - \nu$}
\psfrag{m6}[cl][Bl][.8][30]{$\mu_{b_{j} + 2} - \nu$}
\psfrag{m7}[cl][Bl][.8][30]{$\mu_{b_{j} + 3} - \nu$}
\psfrag{m8}[cl][Bl][.8][30]{$\mu_{b_{j} + 4} - \nu$}
\psfrag{D}[bc][Bl][1][90]{$d$}
\includegraphics[scale=.4]{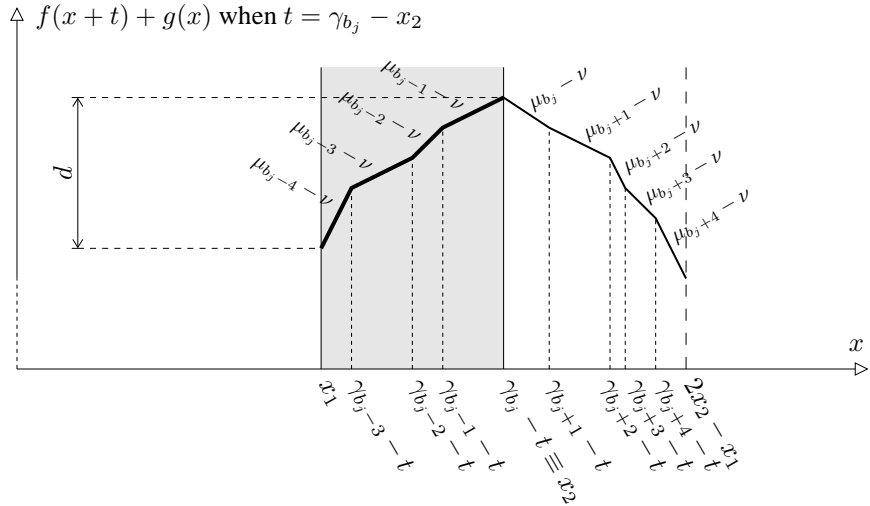}
\vspace{36pt}
\caption{Example of function $f(x+t) + g(x)$, when $t = \gamma_{b_{j}} - x_{2}$.}
\label{fig:det_tstar_1}
\end{figure}

Time instants $\omega_{j} \in [ \gamma_{b_{j}} - x_{2} , \gamma_{b_{j}} - x_{1} )$, $j = 1, \ldots, \lvert B \rvert$, are those for which it results $f(x_{1}+\omega_{j}) + g(x_{1}) = f(x_{2}+\omega_{j}) + g(x_{2})$. They can be determined through a simple procedure which analyzes the values within the interval $[x_{1},x_{2}]$ of the piece-wise linear function $f(x+t) + g(x)$, during its leftward movement (when $t$ increases from $t = \gamma_{b_{j}} - x_{2}$ up to $t = \gamma_{b_{j}} - x_{1}$). Consider figure~\ref{fig:det_tstar_1} which illustrates an example of function $f(x+t) + g(x)$ when $t = \gamma_{b_{j}} - x_{2}$ (only the part which belongs to the interval $[ x_{2} - (x_{2}-x_{1}) , x_{2} + (x_{2}-x_{1}) ] \equiv [x_{1},2 x_{2} - x_{1}]$ is reported, and note also that the vertical value is not meaningful in the search of the abscissa at which the minimum is obtained). Let $d = f(x_{2} + t) + g(x_{2}) - f(x_{1} + t) + g(x_{1})$. Without considering the upward movement of the function (which is not important for the determination of $\omega_{j}$), when $t$ increases the function moves leftward and $d$ is reduced. As an example, in figure~\ref{fig:det_tstar_2} the same function $f(x+t) + g(x)$ is illustrated when $t = \gamma_{b_{j}-3} - x_{1}$. It is evident that $\omega_{j}$ is the time instant at which $d$ is null.

\begin{figure}[h]
\centering
\psfrag{f(x)}[cl][Bl][1][0]{$f(x+t) + g(x)$ when $t = \gamma_{b_{j}-3} - x_{1}$}
\psfrag{x}[bc][Bl][1][0]{$x$}
\psfrag{G2}[Bl][Bl][1][-60]{$x_{1}$}
\psfrag{G3}[Bl][Bl][1][-60]{$\gamma_{b_{j} - 3} - t \equiv x_{1}$}
\psfrag{G4}[Bl][Bl][1][-60]{$\gamma_{b_{j} - 2} - t$}
\psfrag{G5}[Bl][Bl][1][-60]{$\gamma_{b_{j} - 1} - t$}
\psfrag{G6}[Bl][Bl][1][-60]{$\gamma_{b_{j}} - t \equiv x_{2}$}
\psfrag{G7}[Bl][Bl][1][-60]{$\gamma_{b_{j} + 1} - t$}
\psfrag{G8}[Bl][Bl][1][-60]{$\gamma_{b_{j} + 2} - t$}
\psfrag{G9}[Bl][Bl][1][-60]{$\gamma_{b_{j} + 3} - t$}
\psfrag{G10}[Bl][Bl][1][-60]{$\gamma_{b_{j} + 4} - t$}
\psfrag{G11}[Bl][Bl][1][-60]{$x_{2}$}
\psfrag{G12}[Bl][Bl][1][-60]{$2 x_{2} - x_{1}$}
\psfrag{m0}[cr][Bl][.8][-30]{$\mu_{b_{j} - 4} - \nu$}
\psfrag{m1}[cr][Bl][.8][-30]{$\mu_{b_{j} - 3} - \nu$}
\psfrag{m2}[cr][Bl][.8][-30]{$\mu_{b_{j} - 2} - \nu$}
\psfrag{m3}[cr][Bl][.8][-30]{$\mu_{b_{j} - 1} - \nu$}
\psfrag{m4}[cl][Bl][.8][30]{$\mu_{b_{j}} - \nu$}
\psfrag{m5}[cl][Bl][.8][30]{$\mu_{b_{j} + 1} - \nu$}
\psfrag{m6}[cl][Bl][.8][30]{$\mu_{b_{j} + 2} - \nu$}
\psfrag{m7}[cl][Bl][.8][30]{$\mu_{b_{j} + 3} - \nu$}
\psfrag{m8}[cl][Bl][.8][30]{$\mu_{b_{j} + 4} - \nu$}
\psfrag{D}[bc][Bl][1][90]{$d$}
\includegraphics[scale=.4]{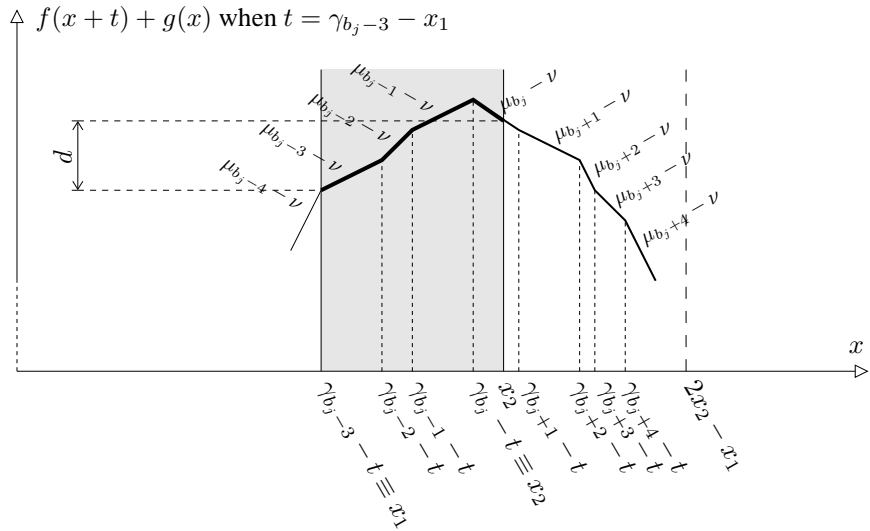}
\vspace{48pt}
\caption{Example of function $f(x+t) + g(x)$, when $t = \gamma_{b_{j}-3} - x_{1}$.}
\label{fig:det_tstar_2}
\end{figure}

On the basis of such considerations, with the considered assumptions, to compute $\omega_{j}$ it is possible to use the following algorithm (which is not formally described, but the reader can refer to the description of algorithm~\ref{alg:tstar}, which generalizes the following one).
 
 \begin{multicols}{2}
{\it SECTION A -- INITIALIZATION}

\begin{algorithmic}[1]
\State $h \in \{Êa_{j}, \ldots, b_{j}-1 \} \, :$
\Statex $\qquad\qquad\qquad \gamma_{h} \leq \gamma_{b_{j}} - (x_{2} - x_{1}) < \gamma_{h+1}$
\State $i = b_{j}$
\If{$j = \lvert B \rvert \algorithmicand \lvert A \rvert = \lvert B \rvert$}
\State $a_{j+1} = M+1$
\State $\gamma_{M+1} = +\infty$
\EndIf
\State $k \in \{Êb_{j}, \ldots, a_{j+1}-1 \} \, :$
\Statex $\qquad\qquad\qquad \gamma_{k} < \gamma_{b_{j}} + (x_{2} - x_{1}) \leq \gamma_{k+1}$
\For{$p = h \algorithmicto k$}
\State $\tilde{\mu}_{p} = \mu_{p} - \nu$
\EndFor
\State $\tau = \gamma_{b_{j}} - (x_{2} - x_{1})$
\State $\theta = \gamma_{b_{j}}$
\State $d = \tilde{\mu}_{h} ( \gamma_{h+1} - \tau )$
\If{$h < b_{j}-1$}
\State $d = d + \sum_{p=h+1}^{b_{j}-1} \tilde{\mu}_{p} ( \gamma_{p+1} - \gamma_{p} )$
\EndIf
\State $\lambda = h$
\State $\xi = i$
\algstore{bkbreak}
\end{algorithmic}

{\it SECTION B -- LOOP}

\begin{algorithmic}[1]
\algrestore{bkbreak}
\While{$h < b_{j} \algorithmicand i < k+1$}
\State $\psi = \min \{ \gamma_{h+1} - \tau , \gamma_{i+1} - \theta \}$
\If{$\gamma_{h+1} - \tau \leq \gamma_{i+1} - \theta$}
\State $\lambda = h + 1$
\EndIf
\If{$\gamma_{h+1} - \tau \geq \gamma_{i+1} - \theta$}
\State $\xi = i + 1$
\EndIf
\State $\delta = \tilde{\mu}_{\lambda} [ \gamma_{\lambda+1} - (\tau+\psi) ]$
\If{$\lambda < \xi - 1$}
\State $\delta = \delta + \sum_{p=\lambda+1}^{\xi-1} \tilde{\mu}_{p} ( \gamma_{p+1} - \gamma_{p} )$
\EndIf
\State $\delta = \delta + \tilde{\mu}_{\xi} [ (\theta+\psi) - \gamma_{\xi} ]$
\If{$\delta \leq 0$}
\State $\omega_{j} = \tau - x_{1} + \frac{d}{\tilde{\mu}_{h} - \tilde{\mu}_{i}}$
\State \algorithmicexit
\Else
\State $h = \lambda$
\State $i = \xi$
\State $\tau = \tau + \psi$
\State $\theta = \theta + \psi$
\State $d = \delta$
\EndIf
\EndWhile
\Statex
\end{algorithmic}
\end{multicols}

This algorithm provides, for any $j = 1, \ldots, \lvert B \rvert$, the time instant $\omega_{j}$ at which a jump discontinuity in $x^{\circ}(t)$ occurs. Since $\omega_{j} < +\infty$ $\forall \, j = 1, \ldots, \lvert B \rvert$, then $T = \Omega$, being $\Omega = \{ \omega_{1} , \ldots, \omega_{j}, \ldots, \omega_{\lvert B \rvert} \}$. Moreover, $Q = \lvert B \rvert$, $t^{\star}_{q} = \omega_{q}$, and $l(q) = q$, $\forall \, q = 1, \ldots , Q$. Then, it is possible to write $t^{\star}_{q} \in [\gamma_{b_{l(q)}} - x_{2}, \gamma_{b_{l(q)}} - x_{1})$, $\gamma_{b_{l(q)}} - \gamma_{a_{l(q)}} > (x_{2}-x_{1})$, $\forall \, q = 1, \ldots, Q$, and $\gamma_{a_{l(q)+1}} - \gamma_{b_{l(q)}} > (x_{2}-x_{1})$, $\forall \, q = 1, \ldots, Q$ (if $l(Q) < \lvert A \rvert$) or $\forall \, q = 1, \ldots, Q-1$ (if $l(Q) = \lvert A \rvert$), which imply $t^{\star}_{q} > \gamma_{a_{l(q)}} - x_{1}$ $\forall \, q = 1, \ldots, Q$ and $t^{\star}_{q} < \gamma_{a_{l(q)+1}} - x_{2}$ $\forall \, q = 1, \ldots, Q$ (if $l(Q) < \lvert A \rvert$) or $\forall \, q = 1, \ldots, Q-1$ (if $l(Q) = \lvert A \rvert$).

Then, $x^{\circ}(t)$ has the structure provided by~\eqref{equ:xopt_2} or~\eqref{equ:xopt_3}, with $x_{\mathrm{s}}(t)$ provided by~\eqref{equ:xs_1}, $x_{q}(t)$, $q = 1, \ldots, Q-1$, $Q > 1$, provided by~\eqref{equ:xj_1}, and $x_{\mathrm{e}}(t)$ provided by~\eqref{equ:xe_1} (if $l(Q) < \lvert A \rvert$) or~\eqref{equ:xe_3} (if $l(Q) = \lvert A \rvert$).

\vspace{12pt}

\underline{Third part}

It has been shown in the second part of the proof that, under the assumptions $\gamma_{b_{j}} - \gamma_{a_{j}} > (x_{2}-x_{1})$, $\forall \, j = 1, \ldots, \lvert B \rvert$, $\lvert B \rvert > 0$, and $\gamma_{a_{j+1}} - \gamma_{b_{j}} > (x_{2}-x_{1})$, $\forall j = 1, \ldots, \lvert A \rvert - 1$, $\lvert A \rvert > 1$, there exists, for each value $b_{j}$, $j = 1, \ldots, \lvert B \rvert$, a finite value $\omega_{j} \in [ \gamma_{b_{j}} - x_{2} , \gamma_{b_{j}} - x_{1} )$ at which $x^{\circ}(t)$ jumps to $x_{2}$. In other words, there are $\lvert B \rvert$ points of discontinuity in the function $x^{\circ}(t)$. In presence of a narrower intervals, this is not necessarily true.

As a matter of fact, in connection with two consecutive time intervals $[\gamma_{b_{j}}-t,\gamma_{a_{j+1}}-t )$ and $[\gamma_{a_{j+1}}-t,\gamma_{b_{j+1}}-t )$ which are such that $\gamma_{b_{j+1}} - \gamma_{b_{j}} < (x_{2}-x_{1})$, when, for any $t \in [ \gamma_{a_{j+1}} - x_{2} , \min \{ \gamma_{b_{j}} - x_{1} , \gamma_{a_{j+2}} - x_{2} \} )$, at least one of the two conditions $f(x_{1}+t) + g(x_{1}) < f(\gamma_{a_{j+1}}) + g(\gamma_{a_{j+1}}-t)$ and $f(x_{2}+t) + g(x_{2}) < f(\gamma_{a_{j+1}}) + g(\gamma_{a_{j+1}}-t)$ is satisfied, then the local minimum at $\gamma_{a_{j+1}} - t$ is never the absolute minimum in the interval $[x_{1},x_{2}]$. Then, in this case, the presence of the abscissa $\gamma_{b_{j}}$, at which the slope of $f(x+t)$ changes from a value greater than or equal to $\nu$ to a value less than $\nu$, does not cause the function $x^{\circ}(t)$ to jump in an upward direction.

\begin{figure}[h]
\centering
\psfrag{f1(x)}[cl][Bl][.8][0]{$f(x+t) + g(x)$ when $t = \gamma_{b_{j}} - x_{2}$}
\psfrag{f2(x)}[cl][Bl][.8][0]{$f(x+t) + g(x)$ when $t = \gamma_{a_{j+1}} - x_{2}$}
\psfrag{f3(x)}[cl][Bl][.8][0]{$f(x+t) + g(x)$ when $t : f(x_{1}+t)+g(x_{1}) = f(x_{2}+t)+g(x_{2})$}
\psfrag{x}[bc][Bl][.8][0]{$x$}
\psfrag{G1}[Bl][Bl][.8][-45]{$x_{1}$}
\psfrag{G2}[Bl][Bl][.8][-45]{$x_{2} \equiv \gamma_{b_{j}} - t$}
\psfrag{G3}[Bl][Bl][.8][-45]{$\gamma_{a_{j+1}} - t$}
\psfrag{G4}[Bl][Bl][.8][-45]{$\gamma_{b_{j+1}} - t$}
\psfrag{G5}[Bl][Bl][.8][-45]{$\gamma_{a_{j+2}} - t$}
\psfrag{G6}[Bl][Bl][.8][-45]{$\gamma_{b_{j}} - t$}
\psfrag{G7}[Bl][Bl][.8][-45]{$x_{2} \equiv \gamma_{a_{j+1}} - t$}
\psfrag{G8}[Bl][Bl][.8][-45]{$x_{2}$}
\psfrag{A}[cc][Bl][1][0]{(a)}
\psfrag{B}[cc][Bl][1][0]{(b)}
\psfrag{C}[cc][Bl][1][0]{(c)}
\includegraphics[scale=.25]{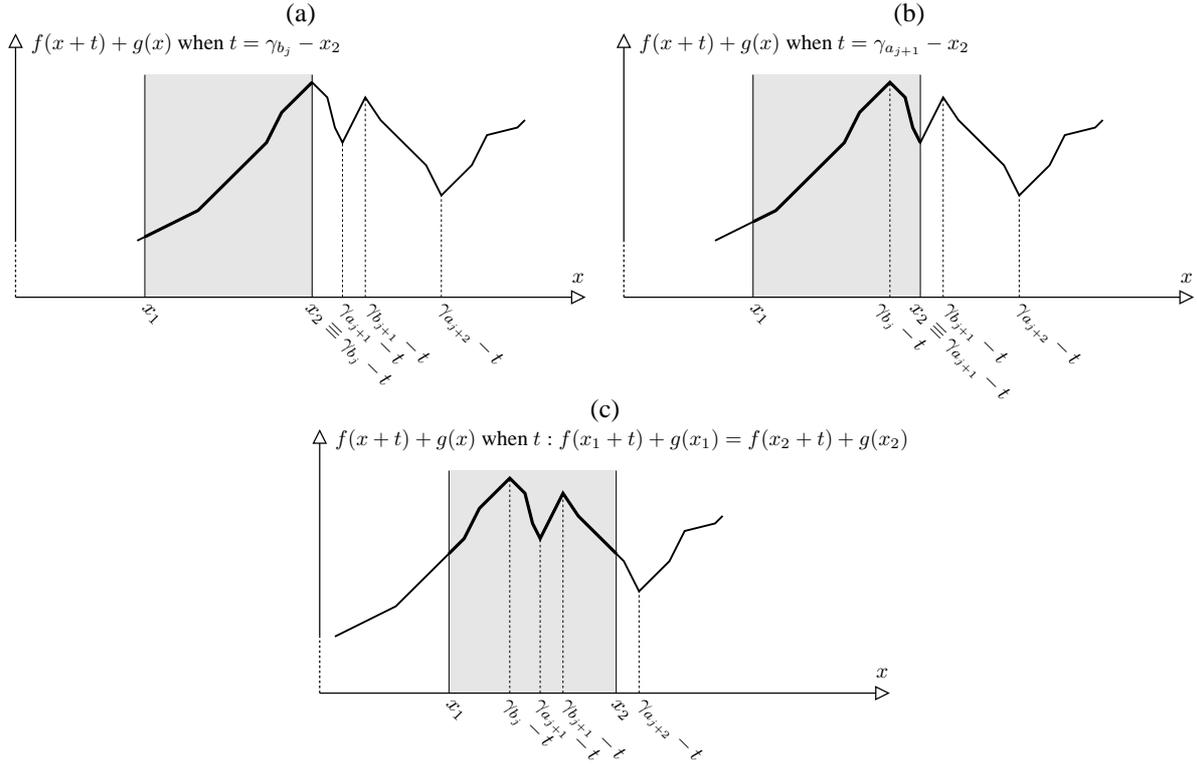}
\vspace{18pt}
\caption{Example of function $f(x+t) + g(x)$, (a) when $t = \gamma_{b_{j}} - x_{2}$, (b) when $t = \gamma_{a_{j+1}} - x_{2}$, and (c) when $t : f(x_{1}+t)+g(x_{1}) = f(x_{2}+t)+g(x_{2})$.}
\label{fig:det_tstar_7}
\end{figure}

To show this, consider the example of function $f(x+t) + g(x)$ illustrated in figure~\ref{fig:det_tstar_7}(a), when $t = \gamma_{b_{j}} - x_{2}$, in which it is $\gamma_{b_{j+1}} - \gamma_{b_{j}} < (x_{2}-x_{1})$. When $t = \gamma_{b_{j}} - x_{2}$, the minimum with respect to $[x_{1},x_{2}]$ is obtained at $x_{1}$. If $f(x_{1}+t) + g(x_{1}) < f(\gamma_{a_{j+1}}) + g(\gamma_{a_{j+1}}-t)$ when $t = \gamma_{a_{j+1}} - x_{2}$ (see figure~\ref{fig:det_tstar_7}(b)), that is, if $f(\gamma_{a_{j+1}} - x_{2} + x_{1}) + g(x_{1}) < f(\gamma_{a_{j+1}})$, then the minimum certainly remains at $x_{1}$ when $t$ increases in the interval $[ \gamma_{b_{j}} - x_{2}, \gamma_{a_{j+1}} - x_{2})$, since $f(x+t) + g(x)$ is strictly decreasing in $[\gamma_{b_{j}}-t,\gamma_{a_{j+1}}-t )$.

When $t$ increases in the interval $[ \gamma_{a_{j+1}} - x_{2}, \gamma_{b_{j+1}} - x_{2})$, if $f(x_{1}+t) + g(x_{1}) < f(\gamma_{a_{j+1}}) + g(\gamma_{a_{j+1}}-t)$ for all $t$ in such an interval, then the minimum is once more at $x_{1}$, since $f(x+t) + g(x)$ is nondecreasing in $[\gamma_{a_{j+1}}-t,\gamma_{b_{j+1}}-t )$. When $t$ increases in the interval $[ \gamma_{b_{j+1}} - x_{2}, \gamma_{a_{j+2}} - x_{2})$, if at least one of the two conditions $f(x_{1}+t) + g(x_{1}) < f(\gamma_{a_{j+1}}) + g(\gamma_{a_{j+1}}-t)$ and $f(x_{2}+t) + g(x_{2}) < f(\gamma_{a_{j+1}}) + g(\gamma_{a_{j+1}}-t)$ is satisfied, then the minimum is at $x_{1}$ (if $f(x_{1}+t) + g(x_{1}) < f(x_{2}+t) + g(x_{2})$) or at $x_{2}$ or $\gamma_{a_{j+2}} - t$ (if $f(x_{1}+t) + g(x_{1}) \geq f(x_{2}+t) + g(x_{2})$). More specifically, the minimum jumps from $x_{1}$ to $x_{2}$ (or to $\gamma_{a_{j+2}} - t$) when $t$ is such that $f(x_{1}+t) + g(x_{1}) = f(x_{2}+t) + g(x_{2})$ (see figure~\ref{fig:det_tstar_7}(c)); however, such a jump in an upward direction has to be associated with abscissa $\gamma_{b_{j+1}}$ and not with $\gamma_{b_{j}}$.

Note that, assumption $\gamma_{b_{j+1}} - \gamma_{b_{j}} < (x_{2}-x_{1})$ is a necessary condition, because in case $\gamma_{b_{j+1}} - \gamma_{b_{j}} \geq (x_{2}-x_{1})$ there is definitely a time instant $t$ at which the local minimum at $\gamma_{a_{j+1}} - t$ is the absolute minimum in the interval $[x_{1},x_{2}]$.

Algorithm~\ref{alg:tstar} determines in the section C (``second loop'') if the local minimum at $\gamma_{a_{j+1}} - t$ is the absolute minimum in the interval $[x_{1},x_{2}]$. Such a part of the algorithm (rows~\ref{alg:secondloopinizio}$\div$~\ref{alg:delta} and~\ref{alg:deltamax0inizio}$\div$~\ref{alg:secondloopfine}) moves the function $f(x+t) + g(x)$ in a leftward direction (by increasing the time variable $\tau$) until that the value of the function at the local minimum $\gamma_{a_{j+1}} - t$ is lower than or equal to the value of the function at $x_{1}$ or, equivalently, until that $f(x_{1}+t) + g(x_{1}) \geq f(\gamma_{a_{j+1}}) + g(\gamma_{a_{j+1}}-t)$. When this happens, it results $\delta \leq 0$. At that point, the algorithm determines (at rows~\ref{alg:deltamin0inizio}$\div$~\ref{alg:deltamin0fine}) if there is a value of the function $f(x+t) + g(x)$ in $(\gamma_{a_{j+1}} - t,x_{2}] \subset [x_{1},x_{2}]$ which is lower than the value of the function at the local minimum $\gamma_{a_{j+1}} - t$ or, equivalently, if it exists $t$ such that $f(x_{2}+t) + g(x_{2}) < f(\gamma_{a_{j+1}}) + g(\gamma_{a_{j+1}}-t)$. In the algorithm, such a lower value exists when $\phi < 0$; in this case, $\omega_{j}$ is set to the nonfinite value $+\infty$ and the algorithm ends since there is no more the possibility that the local minimum becomes an absolute minimum (since $f(x+t) + g(x)$ is strictly decreasing in $[ \gamma_{b_{j+1}} - t , \gamma_{a_{j+2}} - t )$).

In conclusion, it has been shown in this part of the proof that some of the abscissae $\gamma_{b_{j}} - t$, $j \in \{ 1, \ldots, \lvert B \rvert \}$, may not cause a jump discontinuity in $x^{\circ}(t)$. Then, $x^{\circ}(t)$ has a number of discontinuities (at which it jumps in an upward direction) equal to $Q \leq \lvert B \rvert$. In accordance with the notation adopted in algorithm~\ref{alg:tstar}, values $\omega_{j} < +\infty$, $j \in \{ 1, \ldots, \lvert B \rvert \}$, are those actually corresponding to jump discontinuities. These $Q$ values are denoted as $t^{\star}_{q}$, $q = 1, \ldots, Q$. The link between values $\omega_{j}$ and $t^{\star}_{q}$ is represented by the mapping function $l(q) = j$, $j \in \{ 1, \ldots, \lvert B \rvert \}$, $q = 1, \ldots, Q$ (that is, $l(q) = j \Leftrightarrow \omega_{j} = t^{\star}_{q}$).

From now on, only the $Q$ abscissae $\gamma_{b_{l(q)}} - t$, $q = 1, \ldots, Q$, and the $Q+1$ abscissae $\gamma_{a_{1}}$ and $\gamma_{a_{l(q)+1}} - t$, $q = 1, \ldots, Q$, will be taken into account, to prove that, when $Q \geq 1$, $x^{\circ}(t)$ has the structure provided by~\eqref{equ:xopt_2} or~\eqref{equ:xopt_3}, with $x_{\mathrm{s}}(t)$ provided by one of the~\eqref{equ:xs}, $x_{q}(t)$, $q = 1, \ldots, Q-1$, $Q > 1$, provided by one of the~\eqref{equ:xj}, and $x_{\mathrm{e}}(t)$ provided by one of the~\eqref{equ:xe}.

\vspace{12pt}

\underline{Fourth part}

Consider the case $Q \geq 1$, and assume that $\gamma_{a_{l(q)+1}} - \gamma_{b_{l(q)}} < (x_{2}-x_{1})$ and $\gamma_{b_{l(q)+1}} - \gamma_{a_{l(q)+1}} < (x_{2}-x_{1})$ for some $q \in \{ 1, \ldots, Q \}$. An example of such a case is illustrated in figure~\ref{fig:det_tstar_10}. Note that, in accordance with the considerations made in the third part of the proof, regarding the time instants which actually produce a jump discontinuity in $x^{\circ}(t)$, such assumptions imply $l(q)+1 = l(q+1)$ and $l(q)+2 = l(q+1) + 1 = l(q+2)$.

\begin{figure}[h]
\centering
\psfrag{f1(x)}[cl][Bl][.8][0]{$f(x+t) + g(x)$ when $t = t^{\star}_{q}$}
\psfrag{f2(x)}[cl][Bl][.8][0]{$f(x+t) + g(x)$ when $t = t^{\star}_{q+1}$}
\psfrag{x}[bc][Bl][.8][0]{$x$}
\psfrag{G1}[Bl][Bl][.8][-45]{$x_{1}$}
\psfrag{G2}[Bl][Bl][.8][-45]{$\gamma_{b_{l(q)}} - t$}
\psfrag{G3}[Bl][Bl][.8][-45]{$x_{2}$}
\psfrag{G4}[Bl][Bl][.8][-45]{$\gamma_{a_{l(q)+1}} - t$}
\psfrag{G5}[Bl][Bl][.8][-45]{$\gamma_{b_{l(q)+1}} - t$}
\psfrag{G6}[Bl][Bl][.8][-45]{$\gamma_{a_{l(q)+2}} - t$}
\psfrag{A}[cc][Bl][1][0]{(a)}
\psfrag{B}[cc][Bl][1][0]{(b)}
\includegraphics[scale=.25]{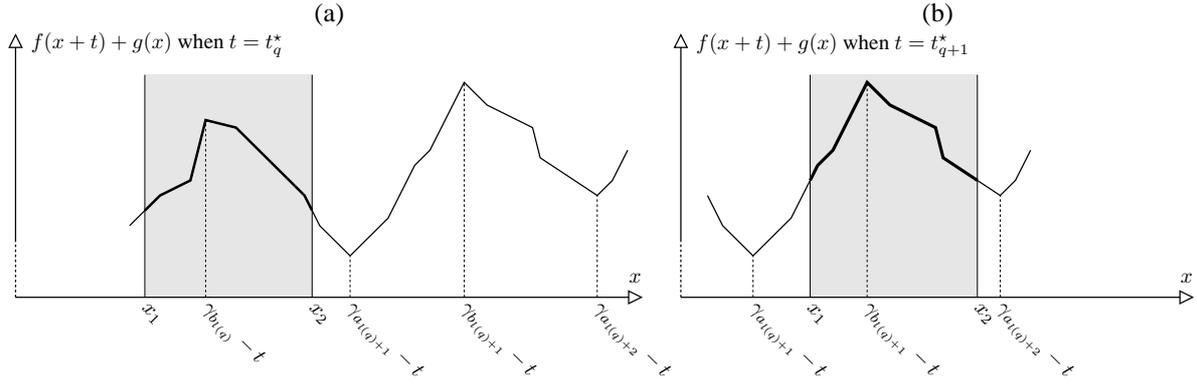}
\vspace{18pt}
\caption{Example of function $f(x+t) + g(x)$, (a) when $t = t^{\star}_{q}$ and (b) when $t = t^{\star}_{q+1}$.}
\label{fig:det_tstar_10}
\end{figure}

The condition $t^{\star}_{q} < \gamma_{a_{l(q)+1}} - x_{2}$ means that a discontinuity occurs at $t = t^{\star}_{q}$, at which $x^{\circ}(t)$ jumps to $x_{2}$, from either $x_{1}$ or $-t^{\star}_{q} + \gamma_{a_{l(q)}}$ (depending if $t^{\star}_{q} > \gamma_{a_{l(q)}} - x_{1}$ or $t^{\star}_{q} \leq \gamma_{a_{l(q)}} - x_{1}$, respectively). In figure~\ref{fig:det_tstar_10}, $f(x+t) + g(x)$ when $t = t^{\star}_{q}$ is illustrated. In accordance with the rules discussed in the second part of the proof, when $t$ increases from $t^{\star}_{q}$ to $\gamma_{a_{l(q)+1}} - x_{2}$, the minimum remains at $x_{2}$. Moreover, when $t$ increases from $\gamma_{a_{l(q)+1}} - x_{2}$ on, the minimum decreases with unitary speed from $x_{2}$ towards $x_{1}$.

The condition $t^{\star}_{q+1} > \gamma_{a_{l(q)+1}} - x_{1}$ means that the minimum, which is decreasing, reaches $x_{1}$ when $t = \gamma_{a_{l(q)+1}} - x_{1}$, and remains at $x_{1}$ in the interval $[\gamma_{a_{l(q)+1}} - x_{1},t^{\star}_{q+1})$. At $t = t^{\star}_{q+1}$ a discontinuity occurs, at which $x^{\circ}(t)$ jumps from $x_{1}$, to either $x_{2}$ or $-t^{\star}_{q+1} + \gamma_{a_{l(q)+2}}$ (depending if $t^{\star}_{q+1} < \gamma_{a_{l(q)+2}} - x_{2}$ or $t^{\star}_{q+1} \geq \gamma_{a_{l(q)+2}} - x_{2}$, respectively).

Then, in case $t^{\star}_{q} < \gamma_{a_{l(q)+1}} - x_{2}$ and $t^{\star}_{q+1} > \gamma_{a_{l(q)+1}} - x_{1}$, the function $x^{\circ}(t)$ between time instants $t^{\star}_{q}$ and $t^{\star}_{q+1}$ has the structure provided by~\eqref{equ:xj_1}.

When $q = 0$, the function $f(x+t)+g(x)$ is strictly decreasing in $(-\infty,\gamma_{a_{1}}-t)$; then, the minimum is obtained at $x_{2}$ for all $t < \gamma_{a_{1}} - x_{2}$. When $t$ increases from $\gamma_{a_{1}} - x_{2}$ on, the minimum decreases with unitary speed from $x_{2}$ towards $x_{1}$. The condition $t^{\star}_{1} > \gamma_{a_{1}} - x_{1}$ means that the minimum, which is decreasing, reaches $x_{1}$ when $t = \gamma_{a_{1}} - x_{1}$, and remains at $x_{1}$ in the interval $[\gamma_{a_{1}} - x_{1},t^{\star}_{1})$. At $t = t^{\star}_{1}$ a discontinuity occurs, at which $x^{\circ}(t)$ jumps from $x_{1}$, to either $x_{2}$ or $-t^{\star}_{1} + \gamma_{a_{2}}$ (depending if $t^{\star}_{1} < \gamma_{a_{2}} - x_{2}$ or $t^{\star}_{1} \geq \gamma_{a_{2}} - x_{2}$, respectively). Then, in case $t^{\star}_{1} > \gamma_{a_{1}} - x_{1}$, the function $x^{\circ}(t)$ before time instant $t^{\star}_{1}$ has the structure provided by~\eqref{equ:xs_1}.

When $q = Q$, if $t^{\star}_{Q} < \gamma_{a_{l(Q)+1}} - x_{2}$, then a discontinuity occurs at $t = t^{\star}_{Q}$, at which $x^{\circ}(t)$ jumps to $x_{2}$, from either $x_{1}$ or $-t^{\star}_{Q} + \gamma_{a_{l(Q)}}$ (depending if $t^{\star}_{Q} > \gamma_{a_{l(Q)}} - x_{1}$ or $t^{\star}_{Q} \leq \gamma_{a_{l(Q)}} - x_{1}$, respectively). In accordance with the previous considerations, the minimum remains at $x_{2}$ in the interval $[ t^{\star}_{Q} , \gamma_{a_{l(Q)+1}} - x_{2} )$, decreases with unitary speed in the interval $[ \gamma_{a_{l(Q)+1}} - x_{2} , \gamma_{a_{l(Q)+1}} - x_{1} )$, and remains at $x_{1}$ from $t = \gamma_{a_{l(Q)+1}} - x_{1}$ on, since $f(x+t)+g(x)$ is nondecreasing in $[\gamma_{a_{l(Q)+1}} - x_{1},+\infty)$. Then, in case $t^{\star}_{Q} < \gamma_{a_{l(Q)+1}} - x_{2}$, the function $x^{\circ}(t)$ after time instant $t^{\star}_{Q}$ has the structure provided by~\eqref{equ:xe_1}.

\vspace{12pt}

\underline{Fifth part}

Consider the case $Q \geq 1$, and assume that $\gamma_{a_{l(q)+1}} - \gamma_{b_{l(q)}} < (x_{2}-x_{1})$ and $\gamma_{a_{l(q)+1)}} - \gamma_{a_{l(q)}} > (x_{2}-x_{1})$ for some $q \in \{ 1, \ldots, Q-1 \}$. If $t^{\star}_{q} \geq \gamma_{a_{l(q)+1}} - x_{2}$, then $f(x_{1}+t)+g(x_{1}) \leq f(x_{2}+t)+g(x_{2})$ when $t = \gamma_{a_{l(q)+1}} - x_{2}$, that is, $f(\gamma_{a_{l(q)+1}} - x_{2}+x_{1})+g(x_{1}) \leq f(\gamma_{a_{l(q)+1}})$, as in the case illustrated in figure~\ref{fig:det_tstar_14}(a). When $t$ increases from $\gamma_{a_{l(q)+1}} - x_{2}$ on, the local minimum at $\gamma_{a_{l(q)+1}} - t$ decreases with unitary speed from $x_{2}$ towards $x_{1}$. Thus, $t^{\star}_{q}$, which corresponds to the finite value $\omega_{l(q)}$, is the time instant at which $f(x_{1}+t)+g(x_{1}) = f(\gamma_{a_{l(q)+1}}) + g(\gamma_{a_{l(q)+1}} - t)$, as it is illustrated in figure~\ref{fig:det_tstar_14}(b). At $t^{\star}_{q}$, the minimum within $[x_{1},x_{2}]$ jumps from $x_{1}$ to $-t^{\star}_{q} + \gamma_{a_{l(q)+1}}$. Then, $x_{q}(t)$ in~\eqref{equ:xopt_3} has the structure of~\eqref{equ:xj_2} or~\eqref{equ:xj_4} (depending on the value $t^{\star}_{q+1}$, as discussed in the following part of the proof).

\begin{figure}[h]
\centering
\psfrag{f1(x)}[cl][Bl][.8][0]{$f(x+t) + g(x)$ when $t = \gamma_{a_{l(q)+1}} - x_{2}$}
\psfrag{f2(x)}[cl][Bl][.8][0]{$f(x+t) + g(x)$ when $t = t^{\star}_{q}$}
\psfrag{x}[bc][Bl][.8][0]{$x$}
\psfrag{Gx1}[Bl][Bl][.8][-45]{$x_{1}$}
\psfrag{G1}[Bl][Bl][.8][-45]{$\gamma_{b_{l(q)}} - t$}
\psfrag{Gx2}[Bl][Bl][.8][-45]{$x_{2} \equiv \gamma_{a_{l(q)+1}} - t$}
\psfrag{G2}[Bl][Bl][.8][-45]{$\gamma_{b_{l(q+1)}} - t$}
\psfrag{Gx3}[Bl][Bl][.8][-45]{$x_{1}$}
\psfrag{G3}[Bl][Bl][.8][-45]{$\gamma_{b_{l(q)}} - t$}
\psfrag{G4}[Bl][Bl][.8][-45]{$\gamma_{a_{l(q)+1}} - t$}
\psfrag{Gx4}[Bl][Bl][.8][-45]{$x_{2}$}
\psfrag{G5}[Bl][Bl][.8][-45]{$\gamma_{b_{l(q+1)}} - t$}
\psfrag{A}[cc][Bl][1][0]{(a)}
\psfrag{B}[cc][Bl][1][0]{(b)}
\includegraphics[scale=.25]{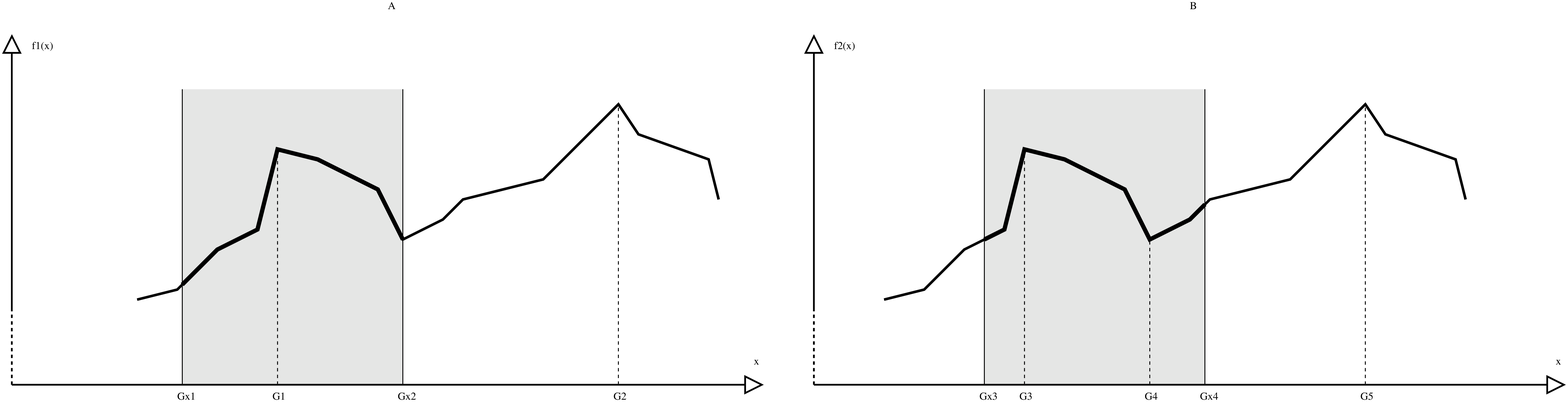}
\vspace{30pt}
\caption{Example of function $f(x+t) + g(x)$, (a) when $t = \gamma_{a_{l(q)+1}} - x_{2}$ and (b) when $t = t^{\star}_{q}$.}
\label{fig:det_tstar_14}
\end{figure}

Consider now the same case in which $\gamma_{a_{l(q)+1}} - \gamma_{b_{l(q)}} < (x_{2}-x_{1})$, for some $q \in \{ 1, \ldots, Q-1 \}$, but without any assumption about the interval $[\gamma_{a_{l(q)}} , \gamma_{a_{l(q)+1}} )$. In this case, when $t = \gamma_{a_{l(q)+1}} - x_{2}$, one or more local minima are present in the interval $[x_{1},x_{2}]$, as in the cases illustrated in figure~\ref{fig:det_tstar_16}(a). In accordance with the considerations made in the third part of the proof, regarding the time instants which actually produce a jump discontinuity in $x^{\circ}(t)$, if $t^{\star}_{q} \geq \gamma_{a_{l(q)+1}} - x_{2}$, then, when $t = \gamma_{a_{l(q)+1}} - x_{2}$, the global minimum in $[x_{1},x_{2}]$ is at abscissa $\gamma_{a_{l(q-1)+1}} - t$ (see again figure~\ref{fig:det_tstar_16}(a)). As before, when $t$ increases from $\gamma_{a_{l(q)+1}} - x_{2}$ on, the local minimum at $\gamma_{a_{l(q)+1}} - t$ decreases with unitary speed from $x_{2}$ towards $x_{1}$, and $t^{\star}_{q}$ is the time instant at which $f(x_{1}+t)+g(x_{1}) = f(\gamma_{a_{l(q)+1}}) + g(\gamma_{a_{l(q)+1}} - t)$, as it is illustrated in figure~\ref{fig:det_tstar_14}(b). At $t^{\star}_{q}$, the minimum within $[x_{1},x_{2}]$ jumps from $x_{1}$ to $-t^{\star}_{q} + \gamma_{a_{l(q)+1}}$. Then, $x_{q}(t)$ in~\eqref{equ:xopt_3} has the structure of~\eqref{equ:xj_2} or~\eqref{equ:xj_4} (depending on the value $t^{\star}_{q+1}$).

\begin{figure}[h]
\centering
\psfrag{f1(x)}[cl][Bl][.8][0]{$f(x+t) + g(x)$ when $t = \gamma_{a_{l(q)+1}} - x_{2}$}
\psfrag{f2(x)}[cl][Bl][.8][0]{$f(x+t) + g(x)$ when $t = t^{\star}_{q}$}
\psfrag{x}[bc][Bl][.8][0]{$x$}
\psfrag{Gx1}[Bl][Bl][.8][-45]{$x_{1}$}
\psfrag{G1}[Bl][Bl][.8][-45]{$\gamma_{a_{l(q-1)+1}} - t$}
\psfrag{G21}[Bl][Bl][.8][-45]{$\gamma_{a_{l(q)}} - t$}
\psfrag{G2}[Bl][Bl][.8][-45]{$\gamma_{b_{l(q)}} - t$}
\psfrag{Gx2}[Bl][Bl][.8][-45]{$x_{2} \equiv \gamma_{a_{l(q)+1}} - t$}
\psfrag{G3}[Bl][Bl][.8][-45]{$\gamma_{b_{l(q+1)}} - t$}
\psfrag{G4}[Bl][Bl][.8][-45]{$\gamma_{a_{l(q-1)+1}} - t$}
\psfrag{Gx3}[Bl][Bl][.8][-45]{$x_{1}$}
\psfrag{G5}[Bl][Bl][.8][-45]{$\gamma_{a_{l(q)}} - t$}
\psfrag{G56}[Bl][Bl][.8][-45]{$\gamma_{b_{l(q)}} - t$}
\psfrag{G6}[Bl][Bl][.8][-45]{$\gamma_{a_{l(q)+1}} - t$}
\psfrag{Gx4}[Bl][Bl][.8][-45]{$x_{2}$}
\psfrag{G7}[Bl][Bl][.8][-45]{$\gamma_{b_{l(q+1)}} - t$}
\psfrag{A}[cc][Bl][1][0]{(a)}
\psfrag{B}[cc][Bl][1][0]{(b)}
\includegraphics[scale=.25]{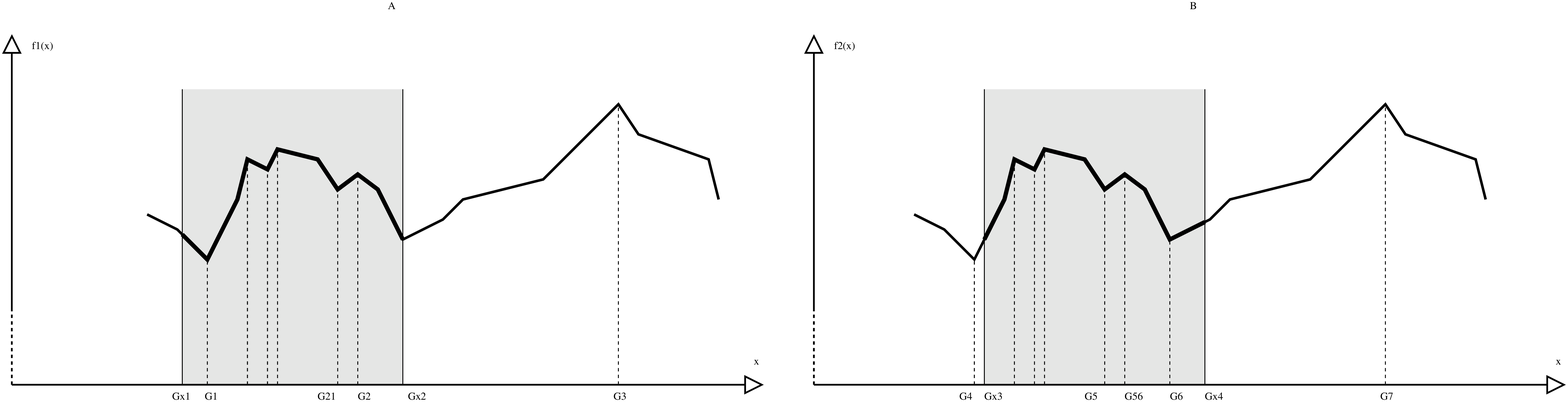}
\vspace{30pt}
\caption{Example of function $f(x+t) + g(x)$, (a) when $t = \gamma_{a_{l(q)+1}} - x_{2}$ and (b) when $t = t^{\star}_{q}$.}
\label{fig:det_tstar_16}
\end{figure}

The same considerations can be made when $q = Q$, in the case $l(Q) < \lvert A \rvert$. If $t^{\star}_{Q} \geq \gamma_{a_{l(Q)+1}} - x_{2}$, at $t^{\star}_{Q}$ the minimum within $[x_{1},x_{2}]$ jumps from $x_{1}$ to $-t^{\star}_{Q} + \gamma_{a_{l(Q)+1}}$, and then $x_{\mathrm{e}}(t)$ in~\eqref{equ:xopt_2} or~\eqref{equ:xopt_3} has the structure of~\eqref{equ:xe_2}.

\underline{Sixth part}

Consider the case $Q \geq 1$, and assume that $\gamma_{b_{l(q+1)}} - \gamma_{a_{l(q)+1}} < (x_{2}-x_{1})$ and $\gamma_{a_{l(q+1)+1}} - \gamma_{a_{l(q+1)}} > (x_{2}-x_{1})$ for some $q \in \{ 1, \ldots, Q-1 \}$. When $t = \gamma_{b_{l(q+1)}} - x_{2}$, the minimum within $[x_{1},x_{2}]$ (which is decreasing with unitary speed since $t$ was equal to $\gamma_{a_{l(q)+1}} - x_{2}$) is at $\gamma_{a_{l(q)+1}} - t$, as in the case illustrated in figure~\ref{fig:det_tstar_17}(a). If $t^{\star}_{q+1} \leq \gamma_{a_{l(q)+1}} - x_{1}$, then the minimum jumps to $x_{2}$ before than (or exactly when) it reaches $x_{1}$, that is, the minimum jumps from $\gamma_{a_{l(q)+1}} - t \geq x_{1}$ to $x_{2}$. $t^{\star}_{q+1}$ is the time instant at which $f(\gamma_{a_{l(q)+1}}) + g(\gamma_{a_{l(q)+1}} - t) = f(x_{2}+t)+g(x_{2}) $, as it is illustrated in figure~\ref{fig:det_tstar_14}(b). Then, $x_{q}(t)$ in~\eqref{equ:xopt_3} has the structure of~\eqref{equ:xj_3} or~\eqref{equ:xj_4} (depending on the value $t^{\star}_{q}$, as discussed in the previous part of the proof).

\begin{figure}[h]
\centering
\psfrag{f1(x)}[cl][Bl][.8][0]{$f(x+t) + g(x)$ when $t = \gamma_{b_{l(q+1)}} - x_{2}$}
\psfrag{f2(x)}[cl][Bl][.8][0]{$f(x+t) + g(x)$ when $t = t^{\star}_{q+1}$}
\psfrag{x}[bc][Bl][.8][0]{$x$}
\psfrag{Gx1}[Bl][Bl][.8][-45]{$x_{1}$}
\psfrag{G1}[Bl][Bl][.8][-45]{$\gamma_{b_{l(q)}} - t$}
\psfrag{G2}[Bl][Bl][.8][-45]{$\gamma_{a_{l(q)+1}} - t$}
\psfrag{Gx2}[Bl][Bl][.8][-45]{$x_{2} \equiv \gamma_{b_{l(q+1)}} - t$}
\psfrag{G3}[Bl][Bl][.8][-45]{$\gamma_{a_{l(q+1)+1}} - t$}
\psfrag{G4}[Bl][Bl][.8][-45]{$\gamma_{b_{l(q)}} - t$}
\psfrag{Gx3}[Bl][Bl][.8][-45]{$x_{1}$}
\psfrag{G5}[Bl][Bl][.8][-45]{$\gamma_{a_{l(q)+1}} - t$}
\psfrag{G6}[Bl][Bl][.8][-45]{$\gamma_{b_{l(q+1)}} - t$}
\psfrag{Gx4}[Bl][Bl][.8][-45]{$x_{2}$}
\psfrag{G7}[Bl][Bl][.8][-45]{$\gamma_{a_{l(q+1)+1}} - t$}
\psfrag{A}[cc][Bl][1][0]{(a)}
\psfrag{B}[cc][Bl][1][0]{(b)}
\includegraphics[scale=.25]{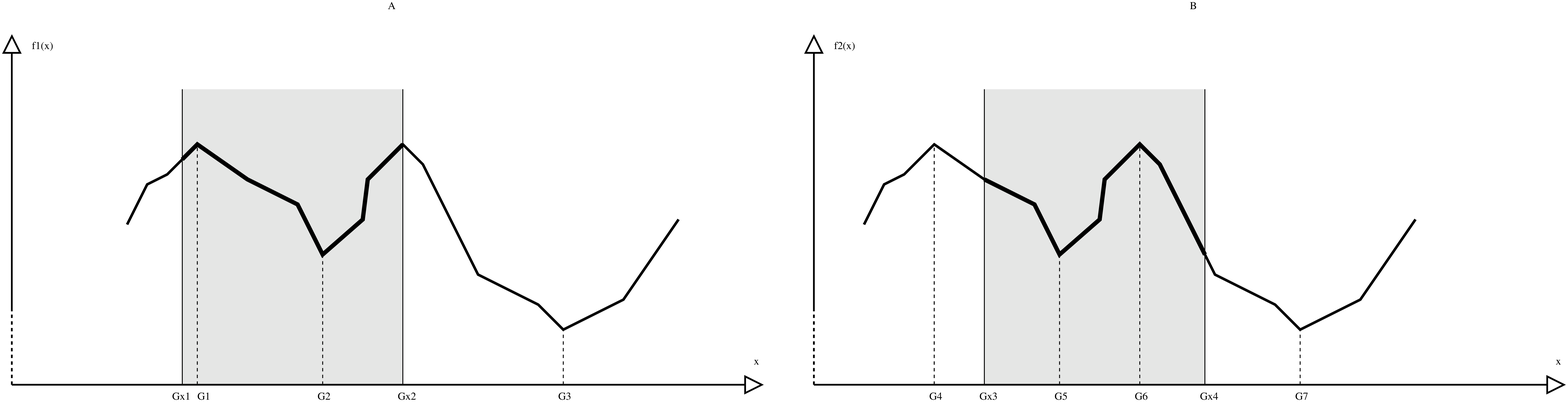}
\vspace{30pt}
\caption{Example of function $f(x+t) + g(x)$, (a) when $t = \gamma_{b_{l(q+1)}} - x_{2}$ and (b) when $t = t^{\star}_{q+1}$.}
\label{fig:det_tstar_17}
\end{figure}

\begin{figure}[h]
\centering
\psfrag{f1(x)}[cl][Bl][.8][0]{$f(x+t) + g(x)$ when $t = \gamma_{b_{l(q+1)}} - x_{2}$}
\psfrag{f2(x)}[cl][Bl][.8][0]{$f(x+t) + g(x)$ when $t = t^{\star}_{q+1}$}
\psfrag{x}[bc][Bl][.8][0]{$x$}
\psfrag{Gx1}[Bl][Bl][.8][-45]{$x_{1}$}
\psfrag{G1}[Bl][Bl][.8][-45]{$\gamma_{b_{l(q)}} - t$}
\psfrag{G2}[Bl][Bl][.8][-45]{$\gamma_{a_{l(q)+1}} - t$}
\psfrag{Gx2}[Bl][Bl][.8][-45]{$x_{2} \equiv \gamma_{b_{l(q+1)}} - t$}
\psfrag{G3}[Bl][Bl][.8][-45]{$\gamma_{a_{l(q+1)+1}} - t$}
\psfrag{Gx3}[Bl][Bl][.8][-45]{$x_{1}$}
\psfrag{G4}[Bl][Bl][.8][-45]{$\gamma_{b_{l(q)}} - t$}
\psfrag{G5}[Bl][Bl][.8][-45]{$\gamma_{a_{l(q)+1}} - t$}
\psfrag{G6}[Bl][Bl][.8][-45]{$\gamma_{b_{l(q+1)}} - t$}
\psfrag{Gx4}[Bl][Bl][.8][-45]{$x_{2}$}
\psfrag{G7}[Bl][Bl][.8][-45]{$\gamma_{a_{l(q+1)+1}} - t$}
\psfrag{A}[cc][Bl][1][0]{(a)}
\psfrag{B}[cc][Bl][1][0]{(b)}
\includegraphics[scale=.25]{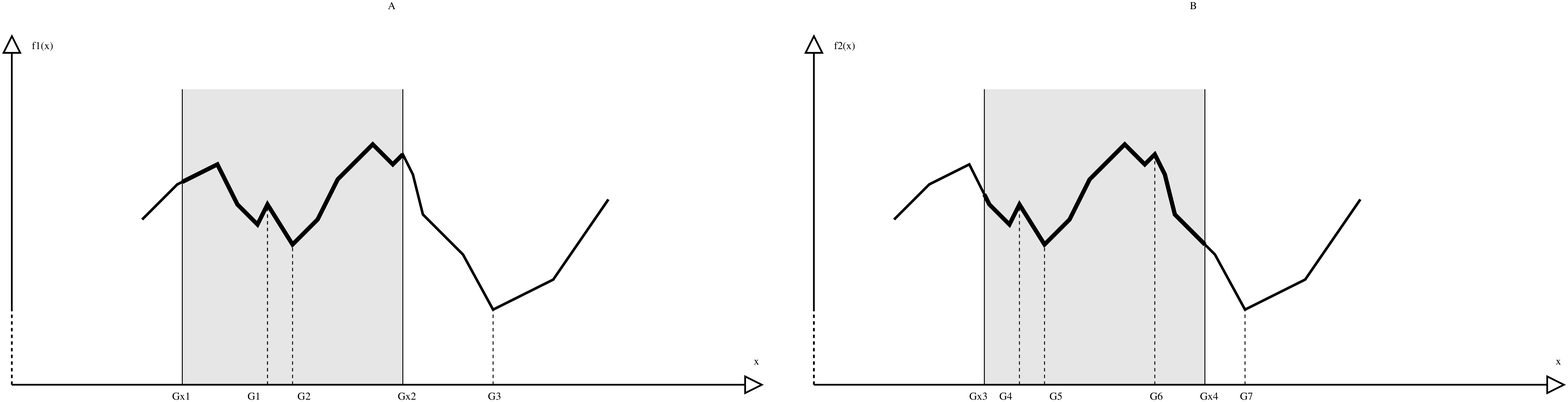}
\vspace{30pt}
\caption{Example of function $f(x+t) + g(x)$, (a) when $t = \gamma_{b_{l(q+1)}} - x_{2}$ and (b) when $t = t^{\star}_{q+1}$.}
\label{fig:det_tstar_18}
\end{figure}

Consider now the same case in which $\gamma_{b_{l(q+1)}} - \gamma_{a_{l(q)+1}} < (x_{2}-x_{1})$, for some $q \in \{ 1, \ldots, Q-1 \}$, but without any assumption about the interval $[ \gamma_{a_{l(q+1)}} , \gamma_{a_{l(q+1)+1}} )$. In this case, when $t = \gamma_{b_{l(q+1)}} - x_{2}$, one or more local minima are present in the interval $[x_{1},x_{2}]$, as in the cases illustrated in figure~\ref{fig:det_tstar_18}(a). In accordance with the considerations made in the third part of the proof, regarding the time instants which actually produce a jump discontinuity in $x^{\circ}(t)$, the global minimum in $[x_{1},x_{2}]$ is at abscissa $\gamma_{a_{l(q)+1}} - t$, as before. Then, also in this case, if $t^{\star}_{q+1} \leq \gamma_{a_{l(q)+1}} - x_{1}$, then the minimum jumps, at $t^{\star}_{q+1}$, from $\gamma_{a_{l(q)+1}} - t \geq x_{1}$ to $x_{2}$. In conclusion, $x_{q}(t)$ in~\eqref{equ:xopt_3} has the structure of~\eqref{equ:xj_3} or~\eqref{equ:xj_4} (depending on the value $t^{\star}_{q}$).

The same considerations can be made in connection with time instant $t^{\star}_{1}$, when $\gamma_{b_{l(1)}} - \gamma_{a_{1}} < (x_{2}-x_{1})$. In this case, if $t^{\star}_{1} \leq \gamma_{a_{1}} - x_{1}$, then the minimum jumps from $\gamma_{a_{1}} - t \geq x_{1}$ to $x_{2}$, and then $x_{\mathrm{s}}(t)$ in~\eqref{equ:xopt_2} or~\eqref{equ:xopt_3} has the structure of~\eqref{equ:xs_2}. 

\vspace{12pt}

\underline{Seventh part}

The algorithm which computes the value $\omega_{j}$, in correspondence with abscissa $\gamma_{b_{j}}$, considers the function $f(x)+g(x)$ and the ``window'' $[\gamma_{b_{j}} - (x_{2} - x_{1}), \gamma_{b_{j}}]$, which is moved rightward to find the instant at which the minimum of the function within the window ``jumps'' from the left bound to the right bound, as discussed in the previous parts of the proof. Note that, considering the function $f(x)+g(x)$ and the window $[\gamma_{b_{j}} - (x_{2} - x_{1}), \gamma_{b_{j}}]$ is equivalent to consider the function $f(x+t)+g(x)$ with $t=\gamma_{b_{j}}-x_{2}$ and the window $[x_{1},x_{2}]$.
 
Basically, to determine the time instant at which the minimum within the window jumps in an upward direction, if it exists (as discussed in the third part of the proof), the algorithm moves the window rightward until the difference $d$ between the value of $f(x)+g(x)$ at the right bound $\theta$ of the window (or at the local minimum which is the nearest to the right bound) and its value at left bound $\tau$ of the same window (or at the current global minimum within $[\tau,\theta]$) becomes null. Since $f(x)+g(x)$ is a piece-wise linear function, the window is repeatedly moved of intervals whose lengths $\psi$ correspond to the lengths on the abscissae axis of the segments of the function. At each step, the new difference $\delta$ is computed and, if $\delta$ turns out to be null or negative, then the minimum has jumped to the right bound; this also means that the time instant $\omega_{j}$ is within the last rightward movement, that is, $\omega_{j} \in [\tau-x_{1},\tau-x_{1}+\psi)$.

In the ``Section A -- Initialization'' part of the algorithm, the segments of the piece-wise linear function $f(x)+g(x)$ which are included in the interval $[\gamma_{b_{j}} - (x_{2} - x_{1}), \gamma_{b_{j}}]$, and those of the interval $[\gamma_{b_{j}} , \gamma_{b_{j}} + (x_{2} - x_{1})]$ that could ``enter'' the window when it moves rightward, are determined (rows 1$\div$5); the slopes of $f(x)+g(x)$ are computed for any of those segments (rows 6$\div$11); the initial values of the left and right bounds $\tau$ and $\theta$ are set (rows 12$\div$13), and the initial value of $d$ is calculated (rows 14$\div$19). Note that, the $\min$ operator in the determination of $d$ is necessary to compute $d$ when the minimum within $[\gamma_{b_{j}} - (x_{2} - x_{1}), \gamma_{b_{j}}]$ (that is, at the beginning) is not obtained at the left bound but is obtained at an abscissa greater than $\gamma_{b_{j}} - (x_{2} - x_{1})$ (as for example, in the cases illustrated in figures~\ref{fig:det_tstar_17}(a) and~\ref{fig:det_tstar_18}(a)).

In the ``Section B -- First loop'' of the algorithm, the while loop allows moving, segment-by-segment, the window $[\tau,\theta]$ leftward. At each step of the while loop, the length $\psi$ of the next rightward movement is determined (row 23) and the new difference $\delta$ is computed (rows 30$\div$42). If $\delta \leq 0$, then $\omega_{j}$ can be determined (through one of the equation at rows 59, 61, 64 and 68); otherwise all values and indexes are updated (rows 72$\div$76) and another step of the loop is executed. It is worth noting that several equations to compute $\omega_{j}$ must be provided because of the possible presence of local minima within the moving window $[\tau,\theta]$; in this connection, values $\chi$ (rows 48$\div$51) are the relative value at the local minima (relative with respect to the value at the left bound of the window), and $m$ (rows 52$\div$56) is the relative value at the global minimum; note also that all local minima, if present, are before $\gamma_{b_{j}} \in [\tau,\theta]$.

In the first loop, the window is moved until its right bound reaches abscissa $\gamma_{a_{j+1}}$. This means that, if $\omega_{j}$ is determined within the first loop, then the new minimum is definitely obtained at $x_{2}$, since $f(x)+g(x)$ is strictly decreasing in $[\gamma_{b_{j}},\gamma_{a_{j+1}})$. In case the minimum did not jump during the first loop (or, equivalently, if $\omega_{j}$ has not been determined during the first loop), then the algorithm executes another loop in which, again, the window is moved rightward; the difference with respect to the first loop is that now the local minimum of $f(x)+g(x)$ at $\gamma_{a_{j+1}}$ is within the window.

In the ``Section C -- Second loop'' of the algorithm, as before, the while loop allows moving, segment-by-segment, the window $[\tau,\theta]$ leftward and, at each step of the while loop, the length $\psi$ of the next rightward movement is determined (row 80) and the new difference $\delta$ is computed (rows 82$\div$92).  If $\delta \leq 0$, then a nonfinite or a finite value of $\omega_{j}$ is determined (respectively at rows 106 or 112); otherwise all values and indexes are updated (rows 115$\div$117) and another step of the loop is executed. In this second loop, the window is moved until its left bound reaches abscissa $\gamma_{b_{j}}$, but the algorithm certainly exits before then.

It is important to observe that when it results $\delta \leq 0$, it is necessary to analyze the shape of $f(x) + g(x)$ in the last part of the window, that is, from $\gamma_{a_{j+1}}$ to $\theta$; as a matter of fact, it is possible that, when the value of $f(x)+g(x)$ at the abscissa $\gamma_{a_{j+1}}$ becomes lower than or equal to all the values in $[\tau,\gamma_{a_{j+1}})$, it is not the global minimum in $[\tau,\theta]$ because a lower value is obtained in $(\gamma_{a_{j+1}},\theta]$ (such a lower value exists when $\phi$, determined at rows 98$\div$104, is negative); this is the case in which the presence of a local maximum at $\gamma_{b_{j}}$ do not produce a jump discontinuity in $x^{\circ}(t)$, as discussed in the third part of the proof. In this case, $\omega_{j}$ is conventionally set to $+\infty$.

This concludes the proof.
\end{proof}

Lemma~\ref{lem:xopt} is still valid when $f(x) = c \neq 0$ for any $x \leq \gamma_{1}$. Moreover, Lemma~\ref{lem:xopt} can be easily extended to consider the more general case in which the slope of function $f(x)$ is not null at the beginning, that is, $\mu_{0} \neq 0$.

\vspace{12pt}

\begin{mylemma} \label{lem:h(t)}
With reference to the functions $f(x+t)$ and $g(x)$, as considered in Lemma~\ref{lem:xopt}, and to the function $x^{\circ}(t) = \arg \min_{x} \{ f(x+t) + g(x) \}$, $x_{1} \leq x \leq x_{2}$, provided by Lemma~\ref{lem:xopt} itself, the function
\begin{equation}
h(t) = f \big( x^{\circ}(t) + t \big) + g \big( x^{\circ}(t) \big)
\end{equation}
is a continuous, nondecreasing, and piece-wise linear function of the independent variable $t$, that can be obtained by $f(x+t)$ and $x^{\circ}(t)$ as follows:
\begin{itemize}
\item $h(t)$ is equal to $f(x_{2}+t)$ for all $t$ in which $x^{\circ}(t) = x_{2}$;
\item $h(t)$ is a linear segment with slope $\nu$ for all $t$ in which $x^{\circ}(t)$ decreases with slope $-1$; the vertical alignments of such segments are such that $h(t)$ is a continuous function;
\item $h(t)$ is equal to $f(x_{1}+t) + \nu ( x_{2} - x_{1} )$ for all $t$ in which $x^{\circ}(t) = x_{1}$.
\end{itemize}
Then:
\begin{subequations}
\begin{equation} \label{equ:h_1}
\text{if $Q = 0$} \, : \quad h(t) = \left\{ \begin{array}{ll}
\vspace{3pt} f(x_{2}+t) & \forall \, t \, : \, x^{\circ}(t) = x_{2}\\
\vspace{3pt} \nu t + \big[ f(\gamma_{a_{1}}) - \nu (\gamma_{a_{1}} - x_{2}) \big] & \forall \, t \, : \, x^{\circ}(t) \neq \{ x_{1} , x_{2} \}\\
f(x_{1}+t) + \nu ( x_{2} - x_{1} ) & \forall \, t \, : \, x^{\circ}(t) = x_{1}\\
\end{array} \right.
\end{equation}
\begin{equation} \label{equ:h_2}
\text{if $Q = 1$} \, : \quad h(t) = \left\{ \begin{array}{ll}
\vspace{3pt} f(x_{2}+t) & \forall \, t \, : \, x^{\circ}(t) = x_{2}\\
\vspace{3pt} \nu t + \big[ f(\gamma_{a_{1}}) - \nu (\gamma_{a_{1}} - x_{2}) \big] & \forall \, t < t^{\star}_{1} : x^{\circ}(t) \neq \{ x_{1} , x_{2} \}\\
\vspace{3pt} \nu t + \big[ f(\gamma_{a_{l(Q)+1}}) - \nu (\gamma_{a_{l(Q)+1}} - x_{2}) \big] & \forall \, t \geq t^{\star}_{1} : x^{\circ}(t) \neq \{ x_{1} , x_{2} \}\\
f(x_{1}+t) + \nu ( x_{2} - x_{1} ) & \forall \, t \, : \, x^{\circ}(t) = x_{1}\\
\end{array} \right.
\end{equation}
\begin{equation} \label{equ:h_3}
\text{if $Q > 1$} \, : \quad h(t) = \left\{ \begin{array}{ll}
\vspace{3pt} f(x_{2}+t) & \forall \, t \, : \, x^{\circ}(t) = x_{2}\\
\vspace{3pt} \nu t + \big[ f(\gamma_{a_{1}}) - \nu (\gamma_{a_{1}} - x_{2}) \big] & \forall \, t < t^{\star}_{1} : x^{\circ}(t) \neq \{ x_{1} , x_{2} \}\\
\nu t + \big[ f(\gamma_{a_{l(q)+1}}) - \nu (\gamma_{a_{l(q)+1}} - x_{2}) \big] & \forall \, t \in [ t^{\star}_{q} , t^{\star}_{q+1} ) : x^{\circ}(t) \neq \{ x_{1} , x_{2} \}\\
\vspace{3pt} & \qquad\qquad\qquad q = 1, \ldots, Q-1\\
\vspace{3pt} \nu t + \big[ f(\gamma_{a_{l(Q)+1}}) - \nu (\gamma_{a_{l(Q)+1}} - x_{2}) \big] & \forall \, t \geq t^{\star}_{Q} : x^{\circ}(t) \neq \{ x_{1} , x_{2} \}\\
f(x_{1}+t) + \nu ( x_{2} - x_{1} ) & \forall \, t \, : \, x^{\circ}(t) = x_{1}\\
\end{array} \right.
\end{equation}
\end{subequations}

\end{mylemma}

\begin{proof}
When $x = x_{2}$, it is $g(x) = 0$ (see figure~\ref{fig:g(x)}); then, when $x^{\circ}(t) = x_{2}$, function $h(t) = f(x_{2}+t)$. Instead, when $x = x_{1}$, it is $g(x) = \nu (x_{2} - x_{1})$ (see again figure~\ref{fig:g(x)}); then, when $x^{\circ}(t) = x_{1}$, function $h(t) = f(x_{1}+t) + \nu (x_{2} - x_{1})$. 

When $\gamma_{a_{1}} - x_{2} < t < \gamma_{a_{1}} - x_{1}$, with regards to~\eqref{equ:xs_1}, $x^{\circ}(t)$ passes linearly (with unitary speed) from the value $x_{2}$ at $t = \gamma_{a_{1}} - x_{2}$ to the value $x_{1}$ at $t = \gamma_{a_{1}} - x_{1}$; then, in the same interval, function $h(t)$ passes, with the same dynamics (that is, linearly), from the value $f(x_{2} + \gamma_{a_{1}} - x_{2}) = f(\gamma_{a_{1}})$ (at $t = \gamma_{a_{1}} - x_{2}$) to the value $f(x_{1} + \gamma_{a_{1}} - x_{1}) + \nu ( x_{2} - x_{1} ) = f(\gamma_{a_{1}}) + \nu ( x_{2} - x_{1} )$ (at $t = \gamma_{a_{1}} - x_{1}$); the segment which joins such values belongs to the line $\nu t + [ f(\gamma_{a_{1}}) - \nu (\gamma_{a_{1}} - x_{2}) ]$. In the same way, when $\gamma_{a_{1}} - x_{2} < t < t^{\star}_{1}$, with regards to~\eqref{equ:xs_2}, $x^{\circ}(t)$ passes linearly (with unitary speed) from the value $x_{2}$ at $t = \gamma_{a_{1}} - x_{2}$ to the value $-t^{\star}_{1}+\gamma_{a_{1}}$ at $t = t^{\star}_{1}$; then, in the same interval, function $h(t)$ passes linearly from the value $f(\gamma_{a_{1}})$ (at $t = \gamma_{a_{1}} - x_{2}$) to the value $f(\gamma_{a_{1}}) + \nu ( x_{2} + t^{\star}_{1} - \gamma_{a_{1}} )$ (at $t = t^{\star}_{1}$); the segment which joins such values belongs again to the line $\nu t + [ f(\gamma_{a_{1}}) - \nu (\gamma_{a_{1}} - x_{2}) ]$. This proves that, when $t : x^{\circ}(t) \neq \{ x_{1} , x_{2} \}$ in~\eqref{equ:h_1}, and when $t < t^{\star}_{1} : x^{\circ}(t) \neq \{ x_{1} , x_{2} \}$ in~\eqref{equ:h_2} and~\eqref{equ:h_3}, function $h(t) = \nu t + [ f(\gamma_{a_{1}}) - \nu (\gamma_{a_{1}} - x_{2}) ]$.

In analogous way, when $\gamma_{a_{l(q)+1}} - x_{2} < t < \gamma_{a_{l(q)+1}} - x_{1}$, $q = 1, \ldots, Q$, with regards to~\eqref{equ:xj_1} or~\eqref{equ:xe_1}, $x^{\circ}(t)$ passes linearly (with unitary speed) from the value $x_{2}$ at $t = \gamma_{a_{l(q)+1}} - x_{2}$ to the value $x_{1}$ at $t = \gamma_{a_{l(q)+1}} - x_{1}$, function $h(t)$ passes linearly from the value $f(\gamma_{a_{l(q)+1}})$ (at $t = \gamma_{a_{l(q)+1}} - x_{2}$) to the value $f(\gamma_{a_{l(q)+1}}) + \nu ( x_{2} - x_{1} )$ (at $t = \gamma_{a_{l(q)+1}} - x_{1}$), and the segment which joins such values belongs to the line $\nu t + [ f(\gamma_{a_{l(q)+1}}) - \nu (\gamma_{a_{l(q)+1}} - x_{2}) ]$. In the same way, when $t^{\star}_{q} < t < \gamma_{a_{l(q)+1}} - x_{1}$, $q = 1, \ldots, Q$, with regards to~\eqref{equ:xj_2} or~\eqref{equ:xe_2}, when $\gamma_{a_{l(q)+1}} - x_{2} < t < t^{\star}_{q+1}$, $q = 1, \ldots, Q-1$, with regards to~\eqref{equ:xj_3}, and when $t^{\star}_{q} < t < t^{\star}_{q+1}$, $q = 1, \ldots, Q-1$, with regards to~\eqref{equ:xj_4}, function $h(t)$ passes linearly from two values which are connected through the line $\nu t + [ f(\gamma_{a_{l(q)+1}}) - \nu (\gamma_{a_{l(q)+1}} - x_{2}) ]$. This proves that, when $t \geq t^{\star}_{1} : x^{\circ}(t) \neq \{ x_{1} , x_{2} \}$ in~\eqref{equ:h_2}, when $t \in [ t^{\star}_{q} , t^{\star}_{q+1} ) : x^{\circ}(t) \neq \{ x_{1} , x_{2} \}$, $q = 1, \ldots, Q-1$, in~\eqref{equ:h_3}, and when $t \geq t^{\star}_{Q} : x^{\circ}(t) \neq \{ x_{1} , x_{2} \}$ in~\eqref{equ:h_3}, function $h(t) = \nu t + [ f(\gamma_{a_{l(q)+1}}) - \nu (\gamma_{a_{l(q)+1}} - x_{2}) ]$, $q = 1, \ldots, Q$.
\end{proof}

Note that, when $l(Q) = \lvert A \rvert$ there isn't any $t \geq t^{\star}_{Q}$ such that $x^{\circ}(t) \neq \{ x_{1} , x_{2} \}$ (since $x_{\mathrm{e}}(t) = x_{2}$ for any $t \geq t^{\star}_{Q}$, in accordance with~\eqref{equ:xe_3}). Then, in this case, the term $\nu t + [ f(\gamma_{a_{1}}) - \nu (\gamma_{a_{1}} - x_{2}) ]$ in~\eqref{equ:h_1} and the term $\nu t + [ f(\gamma_{a_{l(Q)+1}}) - \nu (\gamma_{a_{l(Q)+1}} - x_{2}) ]$ in~\eqref{equ:h_2} and~\eqref{equ:h_3} have not to be considered as a part of the function $h(t)$.

\vspace{12pt}

\begin{mylemma} \label{lem:sum}
Let $f_{1}(x+t)$ and $f_{2}(x+t)$ be two continuous nondecreasing piece-wise linear functions of $x$, parameterized by $t$, as defined by definition~\ref{def:f(x+t)}. The sum function
\begin{equation}
s(x+t) = f_{1}(x+t) + f_{2}(x+t)
\end{equation}
is still a continuous nondecreasing piece-wise linear functions of $x$, parameterized by $t$, which is in accordance with definition~\ref{def:f(x+t)}.
\end{mylemma}

\begin{proof}
It is evident that the sum of two continuous piece-wise linear functions of the same argument is a continuous piece-wise linear function of that argument as well; moreover, since all slopes in $f_{1}(x+t)$ and $f_{2}(x+t)$ are nonnegative, the slope in a generic segment of $s(x+t)$ is nonnegative as well, because it is the sum of two specific (nonnegative) slopes of $f_{1}(x+t)$ and $f_{2}(x+t)$; finally, since the initial slope of both $f_{1}(x+t)$ and $f_{2}(x+t)$ is null, also $s(x+t)$ has initial slope null. Then, $s(x+t)$ is a continuous nondecreasing piece-wise linear functions of $x$, parameterized by $t$ (which is in accordance with definition~\ref{def:f(x+t)}). This concludes the proof.
\end{proof}

\vspace{12pt}

\begin{mylemma} \label{lem:min}
Let $f_{1}(x)$ and $f_{2}(x)$ be two continuous nondecreasing piece-wise linear functions of $x$, as defined by definition~\ref{def:f(x)}. The min function
\begin{equation}
m(x) = \min \big\{ f_{1}(x) , f_{2}(x) \big\}
\end{equation}
is still a continuous nondecreasing piece-wise linear functions of $x$, which is in accordance with definition~\ref{def:f(x)}.
\end{mylemma}

\begin{proof}
It is evident that the minimum of two continuous piece-wise linear functions of the same argument is a continuous piece-wise linear function of that argument as well; moreover, since all slopes in $f_{1}(x)$ and $f_{2}(x)$ are nonnegative, the slope in a generic segment of $m(x)$ is nonnegative as well, because it corresponds to the slope of one segment of $f_{1}(x)$ or one segment of $f_{2}(x)$; finally, since the initial slope of both $f_{1}(x)$ and $f_{2}(x)$ is null, also $m(x)$ has initial slope null. Then, $m(x)$ is a continuous nondecreasing piece-wise linear functions of $x$ (which is in accordance with definition~\ref{def:f(x)}). This concludes the proof.
\end{proof}

%========================================
%========================================
%          ESEMPI
%========================================
%========================================

\newpage

\section{Examples} \label{sec:examples}

\subsection{Example 1}

Consider the following functions $f(x+t)$ and $g(x)$ (depicted in the same graphic).

\begin{figure}[h]
\centering
\psfrag{f(x)}[cl][Bl][.8][0]{$f(x+t)$}
\psfrag{g(x)}[cl][Bl][.8][0]{$g(x)$}
\psfrag{x}[bc][Bl][.8][0]{$x$}
\psfrag{X1}[tc][Bl][.7][0]{$4$}
\psfrag{X2}[tc][Bl][.7][0]{$8$}
\psfrag{G1}[tc][Bl][.7][0]{$16$}
\psfrag{G2}[tc][Bl][.7][0]{$18$}
\psfrag{G3}[tc][Bl][.7][0]{$19$}
\psfrag{G4}[tc][Bl][.7][0]{$20$}
\psfrag{G5}[tc][Bl][.7][0]{$24$}
\psfrag{G6}[tc][Bl][.7][0]{$26$}
\psfrag{G7}[tc][Bl][.7][0]{$29$}
\psfrag{G8}[tc][Bl][.7][0]{$30$}
\psfrag{G9}[tc][Bl][.7][0]{$32$}
\psfrag{n}[cl][Bl][.6][0]{$-1$}
\psfrag{m1}[Bc][Bl][.6][0]{$0.5$}
\psfrag{m2}[cr][Bl][.6][0]{$2$}
\psfrag{m3}[cr][Bl][.6][0]{$1$}
\psfrag{m4}[Bc][Bl][.6][0]{$0.25$}
\psfrag{m5}[Bc][Bl][.6][0]{$0.5$}
\psfrag{m6}[cr][Bl][.6][0]{$1$}
\psfrag{m7}[cr][Bl][.6][0]{$4$}
\psfrag{m8}[Bc][Bl][.6][0]{$0.5$}
\psfrag{m9}[cr][Bl][.6][0]{$1$}
\includegraphics[scale=.25]{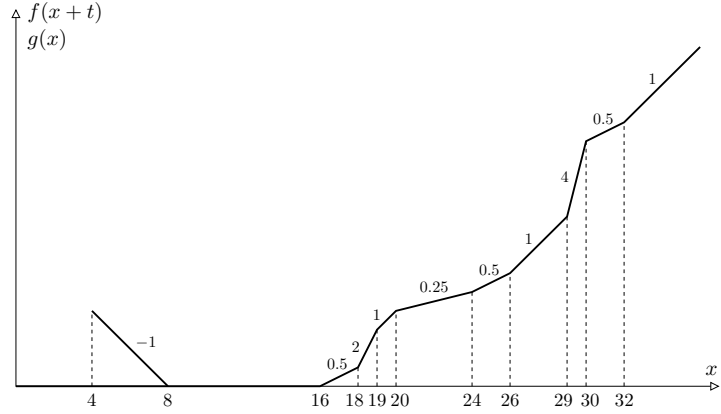}
\caption{Example 1 -- Functions $f(x+t)$ and $g(x)$.}
\label{fig:es1_fg}
\end{figure}

Algorithm~\ref{alg:tstar} provides $\omega_{1} = 13.\overline{3}$ and $\omega_{2} = 25.\overline{6}$. Then, by applying lemma~\ref{lem:xopt} (taking into account $f(x+t)$, instead of $f(x)$, and $g(x)$) the following function $x^{\circ}(t)$ is obtained.
\begin{equation*}
x^{\circ}(t) = \left\{ \begin{array}{ll}
x_{\mathrm{s}}(t) & t < 13.\overline{3}\\
x_{1}(t) & 13.\overline{3} \leq t < 25.\overline{6}\\
x_{\mathrm{e}}(t) & t \geq 25.\overline{6}
\end{array} \right.
\end{equation*}
with
\begin{equation*}
x_{\mathrm{s}}(t) = \left\{ \begin{array}{ll}
8 & t < 10\\
-t + 18 & 10 \leq t < 13.\overline{3}
\end{array} \right. \qquad x_{1}(t) = \left\{ \begin{array}{ll}
8 & 13.\overline{3} \leq t < 18\\
-t + 26 & 18 \leq t < 22\\
4 & 22 \leq t < 25.\overline{6}
\end{array} \right.
\end{equation*}
\begin{equation*}
x_{\mathrm{e}}(t) = \left\{ \begin{array}{ll}
-t + 32 & 25.\overline{6} \leq t < 28\\
4 & t \geq 28
\end{array} \right.
\end{equation*}

Note that, $T = \{ 13.\overline{3} , 25.\overline{6} \}$, that is, $t^{\star}_{1} = 13.\overline{3}$ and $t^{\star}_{2} = 25.\overline{6}$, and $Q = 2$. Since $T = \Omega$, the mapping function is basically $l(1) = 1$ and $l(2) = 2$. Moreover, $t^{\star}_{1} \leq \gamma_{a_{1}} - x_{1} = 14$ (then, $x_{\mathrm{s}}(t)$ has the structure of~\eqref{equ:xs_2}), $t^{\star}_{1} < \gamma_{a_{2}} - x_{2} = 18$ and $t^{\star}_{2} > \gamma_{a_{2}} - x_{1} = 22$ (then, $x_{1}(t)$ has the structure of~\eqref{equ:xj_1}), and $t^{\star}_{2} \geq \gamma_{a_{3}} - x_{2} = 24$ (then, $x_{\mathrm{e}}(t)$ has the structure of~\eqref{equ:xe_2}). The graphical representation of $x^{\circ}(t)$ is the following.

\begin{figure}[h]
\centering
\psfrag{xopt(t)}[cl][Bl][.8][0]{$x^{\circ}(t)$}
\psfrag{x}[bc][Bl][.8][0]{$t$}
\psfrag{XS}[bc][Bl][.7][0]{$x_{\mathrm{s}}(t)$}
\psfrag{X1}[bc][Bl][.7][0]{$x_{1}(t)$}
\psfrag{XE}[bc][Bl][.7][0]{$x_{\mathrm{e}}(t)$}
\psfrag{Y1}[cr][Bl][.7][0]{$4$}
\psfrag{Y2}[cr][Bl][.7][0]{$8$}
\psfrag{T1}[bc][Bl][.7][0]{$10$}
\psfrag{T2}[bc][Bl][.7][0]{$13.\overline{3}$}
\psfrag{T3}[bc][Bl][.7][0]{$18$}
\psfrag{T4}[bc][Bl][.7][0]{$22$}
\psfrag{T5}[bc][Bl][.7][0]{$25.\overline{6}$}
\psfrag{T6}[bc][Bl][.7][0]{$28$}
\includegraphics[scale=.25]{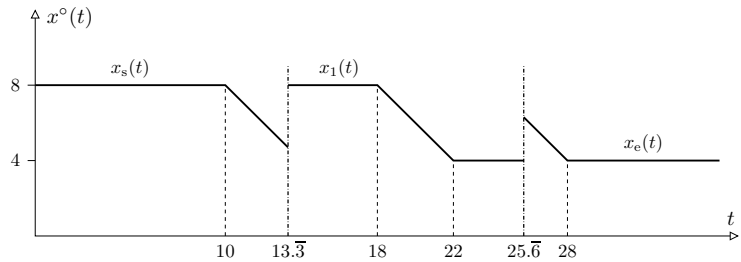}
\caption{Example 1 -- Functions $x^{\circ}(t)$.}
\label{fig:es1_xopt}
\end{figure}

By applying lemma~\ref{lem:h(t)} the following function $h(t) = f \big( x^{\circ}(t) + t \big) + g \big( x^{\circ}(t) \big)$ is obtained.

\begin{figure}[h]
\centering
\psfrag{h(t)}[cl][Bl][.8][0]{$h(t)$}
\psfrag{x}[bc][Bl][.8][0]{$t$}
\psfrag{m1}[Bc][Bl][.6][0]{$0.5$}
\psfrag{m2}[cr][Bl][.6][0]{$1$}
\psfrag{m3}[Bc][Bl][.6][0]{$0.25$}
\psfrag{m4}[Bc][Bl][.6][0]{$0.5$}
\psfrag{m5}[cr][Bl][.6][0]{$1$}
\psfrag{m6}[cr][Bl][.6][0]{$4$}
\psfrag{m7}[cr][Bl][.6][0]{$1$}
\psfrag{T1}[bc][Bl][.7][0]{$8$}
\psfrag{T2}[bc][Bl][.7][0]{$10$}
\psfrag{T3}[bc][Bl][.7][0]{$13.\overline{3}$}
\psfrag{T4}[bc][Bl][.7][0]{$16$}
\psfrag{T5}[bc][Bl][.7][0]{$18$}
\psfrag{T6}[br][Bl][.7][0]{$25$}
\psfrag{T7}[bl][Bl][.7][0]{$25.\overline{6}$}
\includegraphics[scale=.25]{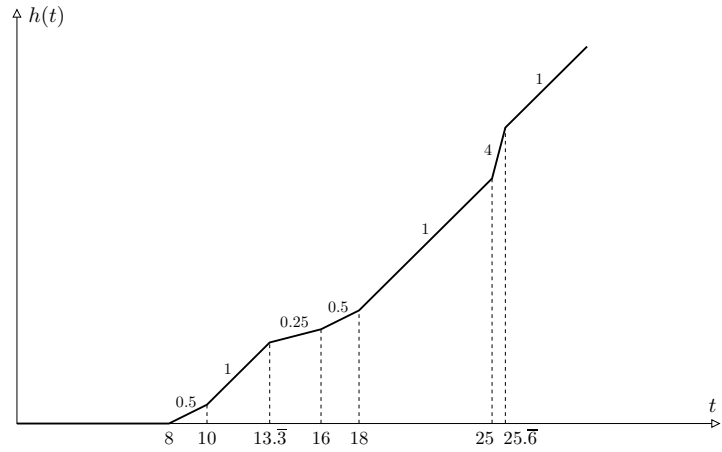}
\caption{Example 1 -- Function $h(t) = f \big( x^{\circ}(t) + t \big) + g \big( x^{\circ}(t) \big)$.}
\label{fig:es1_h}
\end{figure}

\vspace{10pt}

In accordance with lemma~\ref{lem:h(t)} and, in particular, with~\eqref{equ:h_3}, function $h(t)$ is
\begin{equation*}
h(t) = \left\{ \begin{array}{ll}
\vspace{3pt} f(8+t) & \forall \, t \, : \, x^{\circ}(t) = 8\\
\vspace{3pt} t-9 & \forall \, t < 13.\overline{3} : x^{\circ}(t) \neq \{ 4,8 \} \quad (\Rightarrow 10 < t < 13.\overline{3})\\
\vspace{3pt} t-12 & \forall \, t \in [ 13.\overline{3} , 25.\overline{6} ) : x^{\circ}(t) \neq \{ 4,8 \} \quad (\Rightarrow 18 < t < 22)\\
\vspace{3pt} t - 10 & \forall \, t \geq 25.\overline{6} : x^{\circ}(t) \neq \{ 4,8 \} \quad (\Rightarrow 25.\overline{6} < t < 28)\\
f(4+t) + 4 & \forall \, t \, : \, x^{\circ}(t) = 4\\
\end{array} \right.
\end{equation*}

\clearpage

\subsection{Example 2}

Consider the following functions $f(x+t)$ and $g(x)$ (depicted in the same graphic).

\begin{figure}[h]
\centering
\psfrag{f(x)}[cl][Bl][.8][0]{$f(x+t)$}
\psfrag{g(x)}[cl][Bl][.8][0]{$g(x)$}
\psfrag{x}[bc][Bl][.8][0]{$x$}
\psfrag{X1}[tc][Bl][.7][0]{$4$}
\psfrag{X2}[tc][Bl][.7][0]{$8$}
\psfrag{G1}[tc][Bl][.7][0]{$16$}
\psfrag{G2}[tc][Bl][.7][0]{$17$}
\psfrag{G3}[tc][Bl][.7][0]{$19$}
\psfrag{G4}[tc][Bl][.7][0]{$20$}
\psfrag{G5}[tc][Bl][.7][0]{$26$}
\psfrag{G6}[tc][Bl][.7][0]{$31$}
\psfrag{G7}[tc][Bl][.7][0]{$35$}
\psfrag{n}[cl][Bl][.6][0]{$-1$}
\psfrag{m1}[cr][Bl][.6][0]{$2$}
\psfrag{m2}[Bc][Bl][.6][0]{$0.5$}
\psfrag{m3}[cr][Bl][.6][0]{$2$}
\psfrag{m4}[Bc][Bl][.6][0]{$0.5$}
\psfrag{m5}[cr][Bl][.6][0]{$1$}
\psfrag{m6}[Bc][Bl][.6][0]{$0.5$}
\psfrag{m7}[cr][Bl][.6][0]{$2$}
\includegraphics[scale=.25]{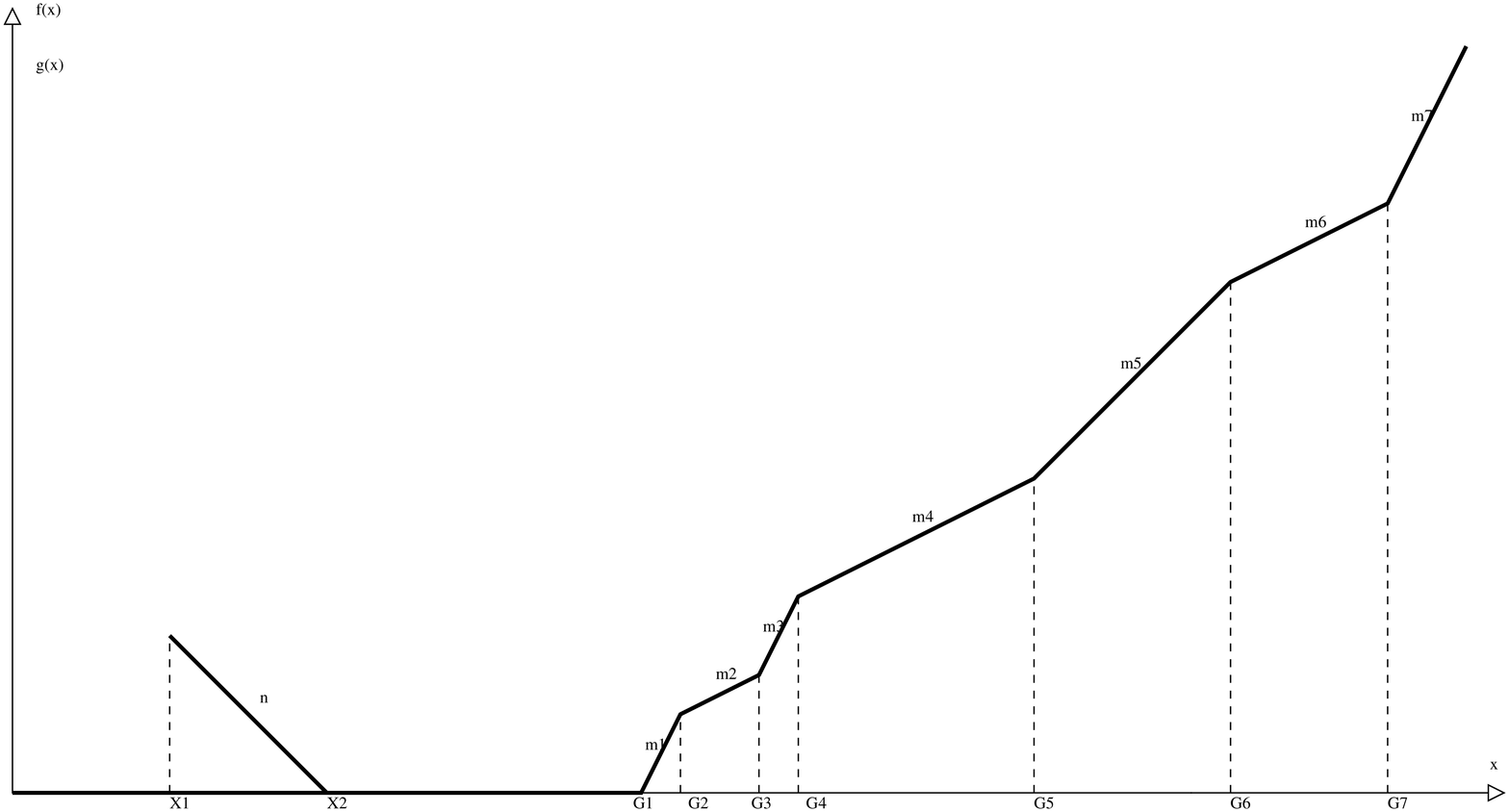}
\caption{Example 2 -- Functions $f(x+t)$ and $g(x)$.}
\label{fig:es2_fg}
\end{figure}

Algorithm~\ref{alg:tstar} provides $\omega_{1} = 11$, $\omega_{2} = 14$, and $\omega_{3} = 23$. Then, by applying lemma~\ref{lem:xopt} (taking into account $f(x+t)$, instead of $f(x)$, and $g(x)$) the following function $x^{\circ}(t)$ is obtained.
\begin{equation*}
x^{\circ}(t) = \left\{ \begin{array}{ll}
x_{\mathrm{s}}(t) & t < 11\\
x_{1}(t) & 11 \leq t < 14\\
x_{2}(t) & 14 \leq t < 23\\
x_{\mathrm{e}}(t) & t \geq 23
\end{array} \right.
\end{equation*}
with
\begin{equation*}
x_{\mathrm{s}}(t) = \left\{ \begin{array}{ll}
8 & t < 8\\
-t + 16 & 8 \leq t < 11
\end{array} \right. \qquad x_{1}(t) = -t + 19
\end{equation*}
\begin{equation*}
x_{2}(t) = \left\{ \begin{array}{ll}
8 & 14 \leq t < 18\\
-t + 26 & 18 \leq t < 22\\
4 & 22 \leq t < 23
\end{array} \right. \qquad x_{\mathrm{e}}(t) = \left\{ \begin{array}{ll}
8 & 23 \leq t < 27\\
-t + 35 & 27 \leq t < 31\\
4 & t \geq 31
\end{array} \right.
\end{equation*}

Note that, $T = \{ 11, 14 , 23 \}$, that is, $t^{\star}_{1} = 11$, $t^{\star}_{2} = 14$, and $t^{\star}_{3} = 23$, and $Q = 3$. Since $T = \Omega$, the mapping function is basically $l(1) = 1$, $l(2) = 2$, and $l(3) = 3$. Moreover, $t^{\star}_{1} \leq \gamma_{a_{1}} - x_{1} = 12$ (then, $x_{\mathrm{s}}(t)$ has the structure of~\eqref{equ:xs_2}), $t^{\star}_{1} \geq \gamma_{a_{2}} - x_{2} = 11$ and $t^{\star}_{2} \leq \gamma_{a_{2}} - x_{1} = 15$ (then, $x_{1}(t)$ has the structure of~\eqref{equ:xj_4}), $t^{\star}_{2} < \gamma_{a_{3}} - x_{2} = 18$ and $t^{\star}_{3} > \gamma_{a_{3}} - x_{1} = 22$ (then, $x_{2}(t)$ has the structure of~\eqref{equ:xj_1}), and $t^{\star}_{3} < \gamma_{a_{4}} - x_{2} = 27$ (then, $x_{\mathrm{e}}(t)$ has the structure of~\eqref{equ:xe_1}). The graphical representation of $x^{\circ}(t)$ is the following.

\begin{figure}[h]
\centering
\psfrag{xopt(t)}[cl][Bl][.8][0]{$x^{\circ}(t)$}
\psfrag{x}[bc][Bl][.8][0]{$t$}
\psfrag{XS}[bc][Bl][.7][0]{$x_{\mathrm{s}}(t)$}
\psfrag{X1}[bc][Bl][.7][0]{$x_{1}(t)$}
\psfrag{X2}[bc][Bl][.7][0]{$x_{2}(t)$}
\psfrag{XE}[bc][Bl][.7][0]{$x_{\mathrm{e}}(t)$}
\psfrag{Y1}[cr][Bl][.7][0]{$4$}
\psfrag{Y2}[cr][Bl][.7][0]{$8$}
\psfrag{T1}[tc][Bl][.7][0]{$8$}
\psfrag{T2}[tc][Bl][.7][0]{$11$}
\psfrag{T3}[tc][Bl][.7][0]{$14$}
\psfrag{T4}[tc][Bl][.7][0]{$18$}
\psfrag{T5}[tc][Bl][.7][0]{$22$}
\psfrag{T6}[tc][Bl][.7][0]{$23$}
\psfrag{T7}[tc][Bl][.7][0]{$27$}
\psfrag{T8}[tc][Bl][.7][0]{$31$}
\includegraphics[scale=.25]{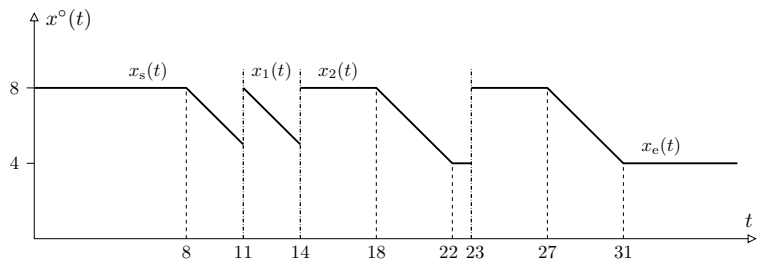}
\caption{Example 2 -- Functions $x^{\circ}(t)$.}
\label{fig:es2_xopt}
\end{figure}

By applying lemma~\ref{lem:h(t)} the following function $h(t) = f \big( x^{\circ}(t) + t \big) + g \big( x^{\circ}(t) \big)$ is obtained.

\begin{figure}[h]
\centering
\psfrag{h(t)}[cl][Bl][.8][0]{$h(t)$}
\psfrag{x}[bc][Bl][.8][0]{$t$}
\psfrag{m1}[cr][Bl][.6][0]{$1$}
\psfrag{m2}[Bc][Bl][.6][0]{$0.5$}
\psfrag{m3}[cr][Bl][.6][0]{$1$}
\psfrag{m4}[Bc][Bl][.6][0]{$0.5$}
\psfrag{m5}[cr][Bl][.6][0]{$1$}
\psfrag{m6}[cr][Bl][.6][0]{$2$}
\psfrag{G1}[tc][Bl][.7][0]{$8$}
\psfrag{G2}[tc][Bl][.7][0]{$11$}
\psfrag{G3}[tc][Bl][.7][0]{$14$}
\psfrag{G4}[tc][Bl][.7][0]{$18$}
\psfrag{G5}[tc][Bl][.7][0]{$22$}
\psfrag{G6}[tc][Bl][.7][0]{$23$}
\psfrag{G7}[tc][Bl][.7][0]{$27$}
\psfrag{G8}[tc][Bl][.7][0]{$31$}
\includegraphics[scale=.25]{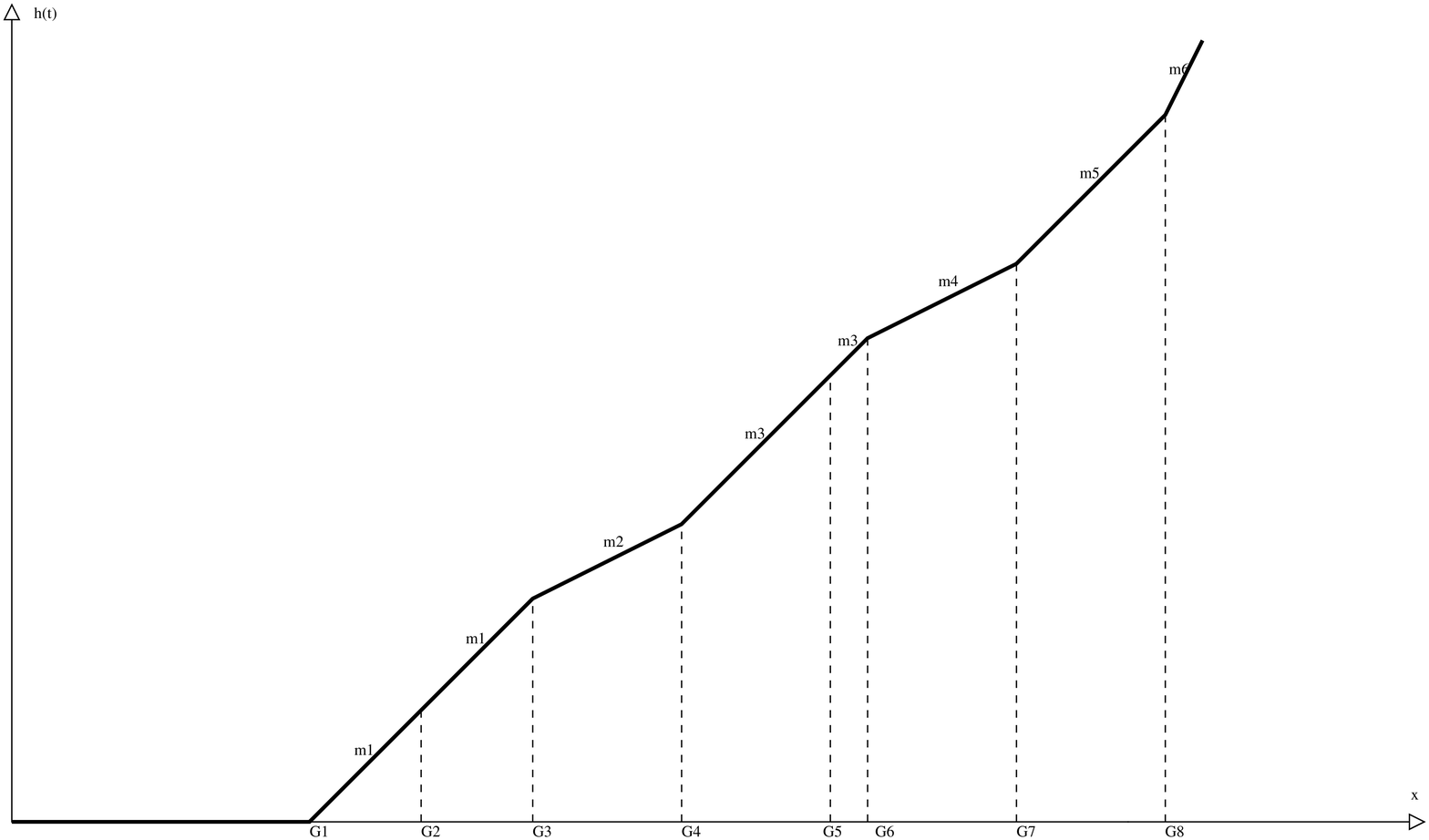}
\caption{Example 2 -- Function $h(t) = f \big( x^{\circ}(t) + t \big) + g \big( x^{\circ}(t) \big)$.}
\label{fig:es2_h}
\end{figure}

In accordance with lemma~\ref{lem:h(t)} and, in particular, with~\eqref{equ:h_3}, function $h(t)$ is
\begin{equation*}
h(t) = \left\{ \begin{array}{ll}
\vspace{3pt} f(8+t) & \forall \, t \, : \, x^{\circ}(t) = 8\\
\vspace{3pt} t-8 & \forall \, t < 11 : x^{\circ}(t) \neq \{ 4,8 \} \quad (\Rightarrow 8 < t < 11)\\
\vspace{3pt} t-8 & \forall \, t \in [ 11 , 14 ) : x^{\circ}(t) \neq \{ 4,8 \} \quad (\Rightarrow 11 < t < 14)\\
\vspace{3pt} t-10 & \forall \, t \in [ 14 , 23 ) : x^{\circ}(t) \neq \{ 4,8 \} \quad (\Rightarrow 18 < t < 22)\\
\vspace{3pt} t -12 & \forall \, t \geq 23 : x^{\circ}(t) \neq \{ 4,8 \} \quad (\Rightarrow 27 < t < 31)\\
f(4+t) + 4 & \forall \, t \, : \, x^{\circ}(t) = 4\\
\end{array} \right.
\end{equation*}

\clearpage

\subsection{Example 3}

Consider the following functions $f(x+t)$ and $g(x)$ (depicted in the same graphic).

\begin{figure}[h]
\centering
\psfrag{f(x)}[cl][Bl][.8][0]{$f(x+t)$}
\psfrag{g(x)}[cl][Bl][.8][0]{$g(x)$}
\psfrag{x}[bc][Bl][.8][0]{$x$}
\psfrag{X1}[tc][Bl][.7][0]{$4$}
\psfrag{X2}[tc][Bl][.7][0]{$8$}
\psfrag{G1}[tc][Bl][.7][0]{$16$}
\psfrag{G2}[tc][Bl][.7][0]{$17$}
\psfrag{G3}[tc][Bl][.7][0]{$19$}
\psfrag{G4}[tc][Bl][.7][0]{$20$}
\psfrag{G5}[tc][Bl][.7][0]{$26$}
\psfrag{G6}[tc][Bl][.7][0]{$31$}
\psfrag{G7}[tc][Bl][.7][0]{$35$}
\psfrag{n}[cl][Bl][.6][0]{$-1.5$}
\psfrag{m1}[cr][Bl][.6][0]{$2$}
\psfrag{m2}[Bc][Bl][.6][0]{$0.5$}
\psfrag{m3}[cr][Bl][.6][0]{$2$}
\psfrag{m4}[Bc][Bl][.6][0]{$0.5$}
\psfrag{m5}[cr][Bl][.6][0]{$1$}
\psfrag{m6}[Bc][Bl][.6][0]{$0.5$}
\psfrag{m7}[cr][Bl][.6][0]{$2$}
\includegraphics[scale=.25]{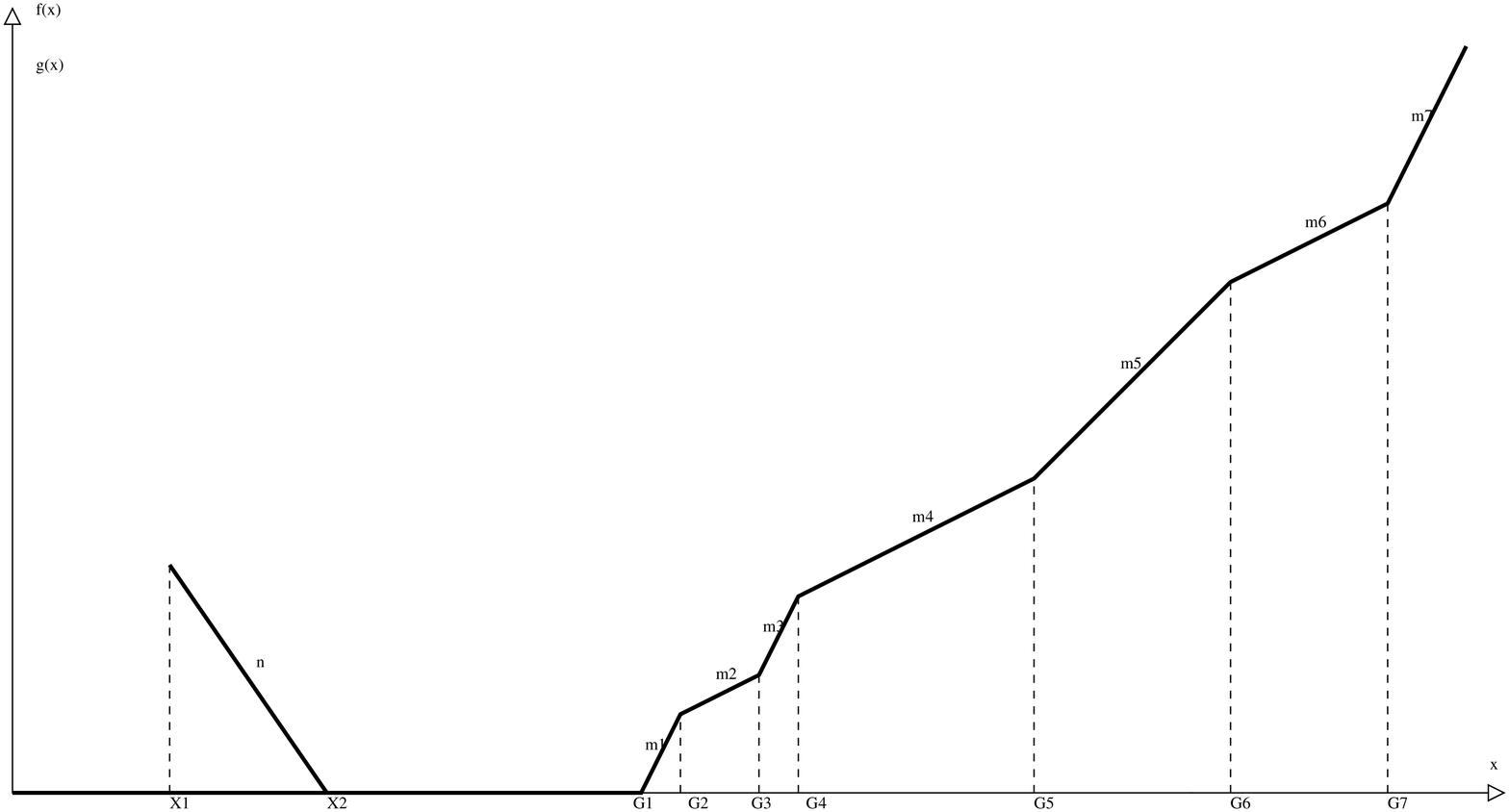}
\caption{Example 3 -- Functions $f(x+t)$ and $g(x)$.}
\label{fig:es3_fg}
\end{figure}

Algorithm~\ref{alg:tstar} provides $\omega_{1} = 9.5$ and $\omega_{2} = 12.5$. Then, by applying lemma~\ref{lem:xopt} (taking into account $f(x+t)$, instead of $f(x)$, and $g(x)$) the following function $x^{\circ}(t)$ is obtained.
\begin{equation*}
x^{\circ}(t) = \left\{ \begin{array}{ll}
x_{\mathrm{s}}(t) & t < 9.5\\
x_{1}(t) & 9.5 \leq t < 12.5\\
x_{\mathrm{e}}(t) & t \geq 12.5
\end{array} \right.
\end{equation*}
with
\begin{equation*}
x_{\mathrm{s}}(t) = \left\{ \begin{array}{ll}
8 & t < 8\\
-t + 16 & 8 \leq t < 9.5
\end{array} \right. \qquad x_{1}(t) = \left\{ \begin{array}{ll}
8 & 9.5 \leq t < 11\\
-t + 19 & 11 \leq t < 12.5
\end{array} \right.
\end{equation*}
\begin{equation*}
x_{\mathrm{e}}(t) = \left\{ \begin{array}{ll}
8 & 12.5 \leq t < 27\\
-t + 35 & 27 \leq t < 31\\
4 & t \geq 31
\end{array} \right.
\end{equation*}

Note that, $T = \{ 9.5 , 12.5 \}$, that is, $t^{\star}_{1} = 9.5$ and $t^{\star}_{2} = 12.5$, and $Q = 2$. Since $T = \Omega$, the mapping function is basically $l(1) = 1$ and $l(2) = 2$. Moreover, $t^{\star}_{1} \leq \gamma_{a_{1}} - x_{1} = 12$ (then, $x_{\mathrm{s}}(t)$ has the structure of~\eqref{equ:xs_2}), $t^{\star}_{1} < \gamma_{a_{2}} - x_{2} = 11$ and $t^{\star}_{2} \leq \gamma_{a_{2}} - x_{1} = 15$ (then, $x_{1}(t)$ has the structure of~\eqref{equ:xj_3}), and $t^{\star}_{2} < \gamma_{a_{3}} - x_{2} = 27$ (then, $x_{\mathrm{e}}(t)$ has the structure of~\eqref{equ:xe_1}). The graphical representation of $x^{\circ}(t)$ is the following.

\begin{figure}[h]
\centering
\psfrag{xopt(t)}[cl][Bl][.8][0]{$x^{\circ}(t)$}
\psfrag{x}[bc][Bl][.8][0]{$t$}
\psfrag{XS}[bc][Bl][.7][0]{$x_{\mathrm{s}}(t)$}
\psfrag{X1}[bc][Bl][.7][0]{$x_{1}(t)$}
\psfrag{X2}[bc][Bl][.7][0]{$x_{2}(t)$}
\psfrag{XE}[bc][Bl][.7][0]{$x_{\mathrm{e}}(t)$}
\psfrag{Y1}[cr][Bl][.7][0]{$4$}
\psfrag{Y2}[cr][Bl][.7][0]{$8$}
\psfrag{T1}[tc][Bl][.7][0]{$8$}
\psfrag{T2}[tc][Bl][.7][0]{$9.5$}
\psfrag{T3}[tc][Bl][.7][0]{$11$}
\psfrag{T4}[tc][Bl][.7][0]{$12.5$}
\psfrag{T5}[tc][Bl][.7][0]{$27$}
\psfrag{T6}[tc][Bl][.7][0]{$31$}
\includegraphics[scale=.25]{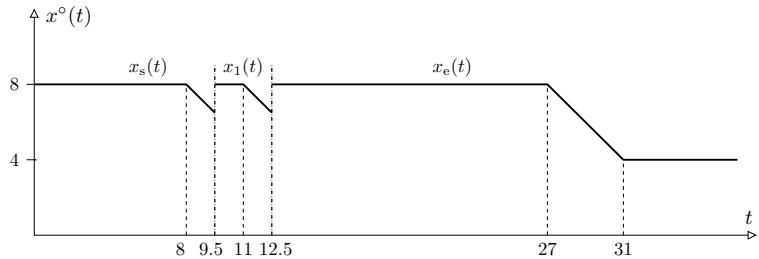}
\caption{Example 3 -- Functions $x^{\circ}(t)$.}
\label{fig:es3_xopt}
\end{figure}

By applying lemma~\ref{lem:h(t)} the following function $h(t) = f \big( x^{\circ}(t) + t \big) + g \big( x^{\circ}(t) \big)$ is obtained.

\begin{figure}[h]
\centering
\psfrag{h(t)}[cl][Bl][.8][0]{$h(t)$}
\psfrag{x}[bc][Bl][.8][0]{$t$}
\psfrag{m1}[cr][Bl][.6][0]{$1.5$}
\psfrag{m2}[Bc][Bl][.6][0]{$0.5$}
\psfrag{m3}[cr][Bl][.6][0]{$1$}
\psfrag{m4}[Bc][Bl][.6][0]{$0.5$}
\psfrag{m5}[cr][Bl][.6][0]{$1.5$}
\psfrag{m6}[cr][Bl][.6][0]{$2$}
\psfrag{G1}[tc][Bl][.7][0]{$8$}
\psfrag{G2}[tc][Bl][.7][0]{$9.5$}
\psfrag{G3}[tc][Bl][.7][0]{$11$}
\psfrag{G4}[tc][Bl][.7][0]{$12.5$}
\psfrag{G5}[tc][Bl][.7][0]{$18$}
\psfrag{G6}[tc][Bl][.7][0]{$23$}
\psfrag{G7}[tc][Bl][.7][0]{$27$}
\psfrag{G8}[tc][Bl][.7][0]{$31$}
\includegraphics[scale=.25]{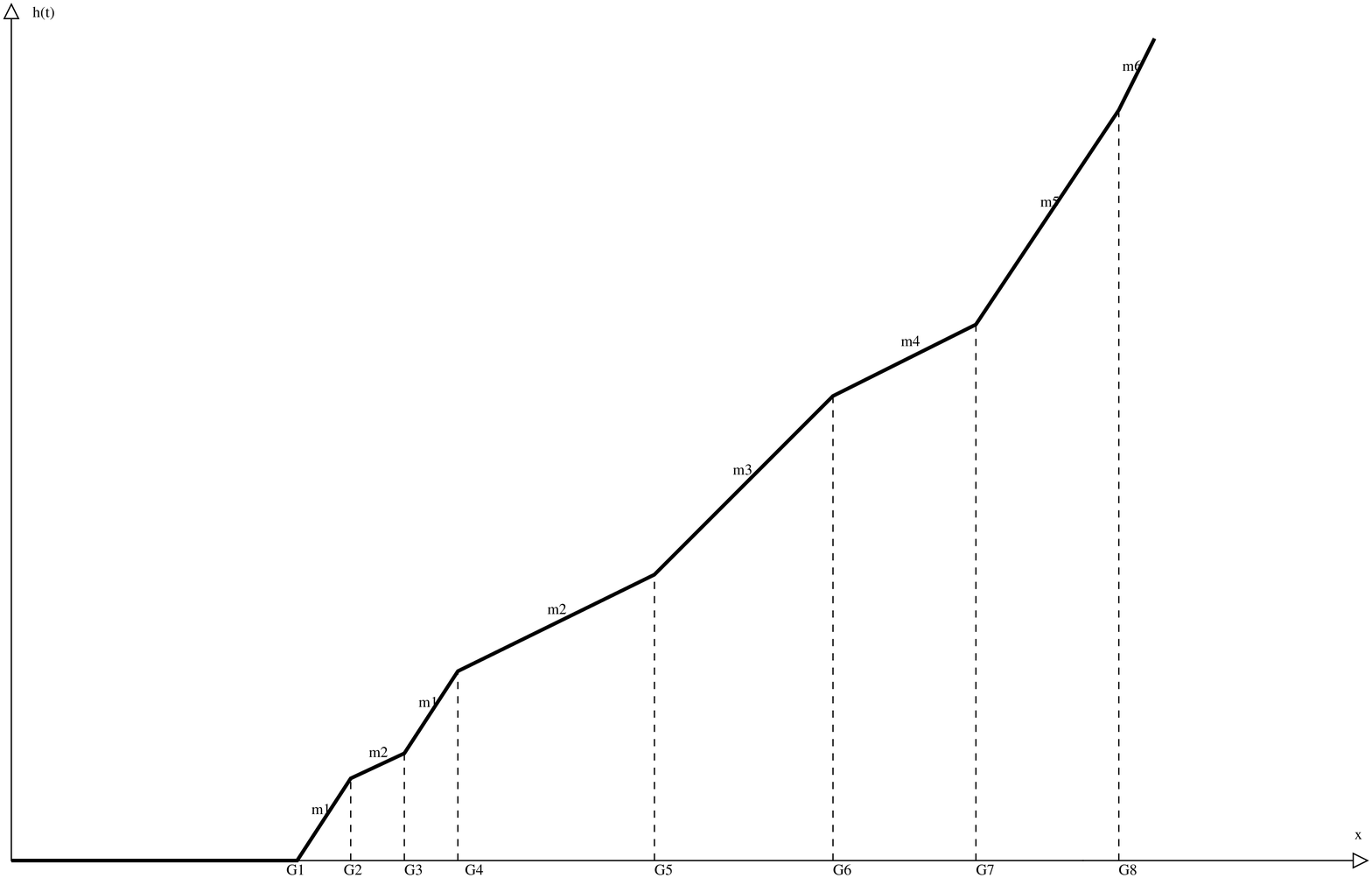}
\caption{Example 3 -- Function $h(t) = f \big( x^{\circ}(t) + t \big) + g \big( x^{\circ}(t) \big)$.}
\label{fig:es3_h}
\end{figure}

\vspace{10pt}

In accordance with lemma~\ref{lem:h(t)} and, in particular, with~\eqref{equ:h_3}, function $h(t)$ is
\begin{equation*}
h(t) = \left\{ \begin{array}{ll}
\vspace{3pt} f(8+t) & \forall \, t \, : \, x^{\circ}(t) = 8\\
\vspace{3pt} 1.5 \, t - 12 & \forall \, t < 9.5 : x^{\circ}(t) \neq \{ 4,8 \} \quad (\Rightarrow 8 < t < 9.5)\\
\vspace{3pt} 1.5 \, t - 13.5 & \forall \, t \in [ 9.5 , 12.5 ) : x^{\circ}(t) \neq \{ 4,8 \} \quad (\Rightarrow 11 < t < 12.5)\\
\vspace{3pt} 1.5 \, t - 25.5 & \forall \, t \geq 12.5 : x^{\circ}(t) \neq \{ 4,8 \} \quad (\Rightarrow 27 < t < 31)\\
f(4+t) + 6 & \forall \, t \, : \, x^{\circ}(t) = 4\\
\end{array} \right.
\end{equation*}

\clearpage

\subsection{Example 4}

Consider the following functions $f(x+t)$ and $g(x)$ (depicted in the same graphic).

\begin{figure}[h]
\centering
\psfrag{f(x)}[cl][Bl][.8][0]{$f(x+t)$}
\psfrag{g(x)}[cl][Bl][.8][0]{$g(x)$}
\psfrag{x}[bc][Bl][.8][0]{$x$}
\psfrag{X1}[tc][Bl][.7][0]{$3$}
\psfrag{X2}[tc][Bl][.7][0]{$8$}
\psfrag{G1}[tc][Bl][.7][0]{$14$}
\psfrag{G2}[tc][Bl][.7][0]{$16$}
\psfrag{G3}[tc][Bl][.7][0]{$17$}
\psfrag{G4}[tc][Bl][.7][0]{$21$}
\psfrag{G5}[tc][Bl][.7][0]{$23$}
\psfrag{G6}[tc][Bl][.7][0]{$24$}
\psfrag{G7}[tc][Bl][.7][0]{$25$}
\psfrag{G8}[tc][Bl][.7][0]{$26$}
\psfrag{G9}[tc][Bl][.7][0]{$28$}
\psfrag{G10}[tc][Bl][.7][0]{$30$}
\psfrag{G11}[tc][Bl][.7][0]{$34$}
\psfrag{G12}[tc][Bl][.7][0]{$38$}
\psfrag{n}[cl][Bl][.6][0]{$-1$}
\psfrag{m1}[Bc][Bl][.6][0]{$0.5$}
\psfrag{m2}[cr][Bl][.6][0]{$2$}
\psfrag{m3}[Bc][Bl][.6][0]{$0.25$}
\psfrag{m4}[cr][Bl][.6][0]{$1$}
\psfrag{m5}[cr][Bl][.6][0]{$3$}
\psfrag{m6}[cr][Bl][.6][0]{$4$}
\includegraphics[scale=.25]{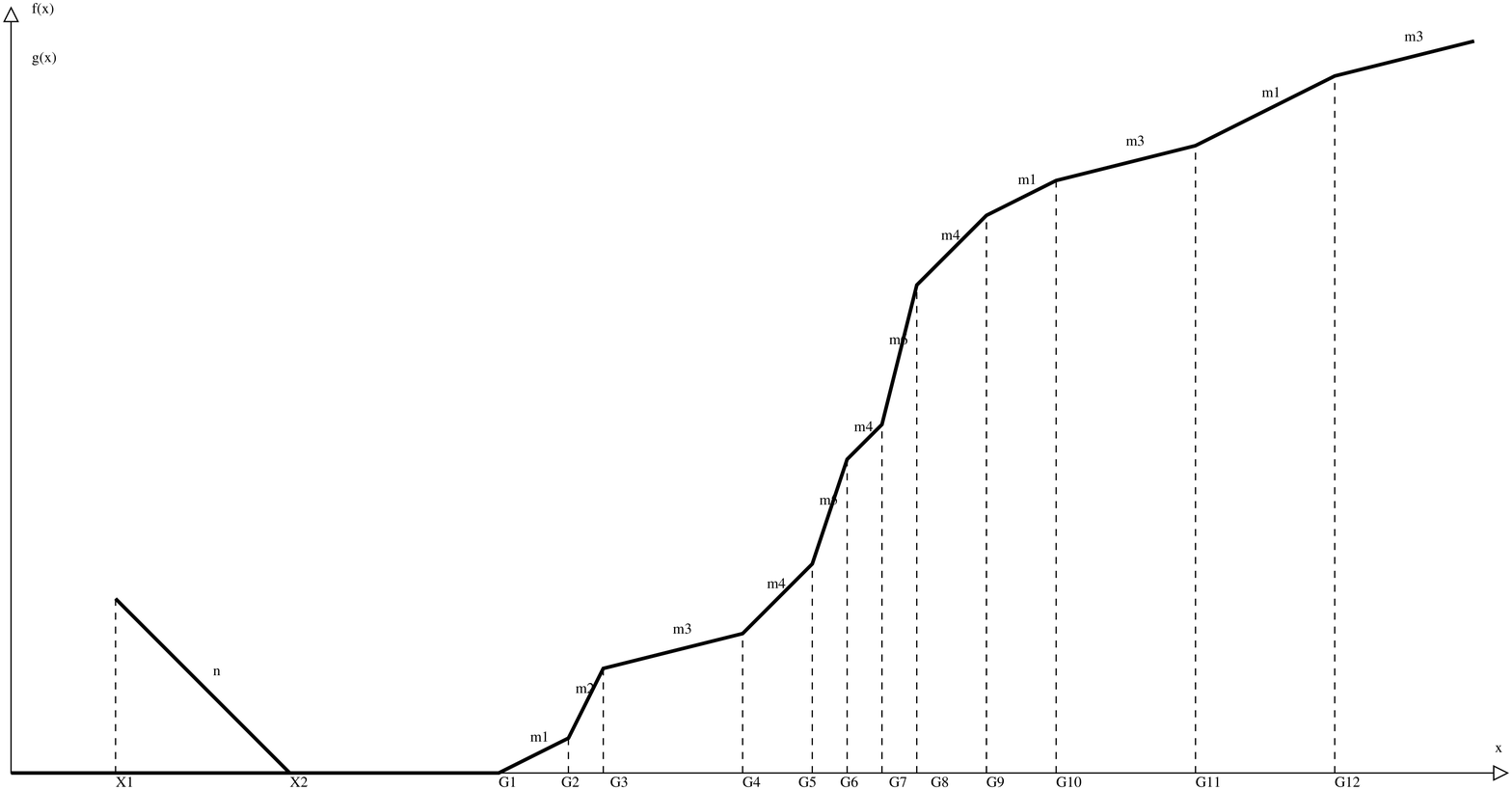}
\caption{Example 4 -- Functions $f(x+t)$ and $g(x)$.}
\label{fig:es4_fg}
\end{figure}

Algorithm~\ref{alg:tstar} provides $\omega_{1} = 10.\overline{3}$ and $\omega_{2} = 22.5\overline{3}$. Then, by applying lemma~\ref{lem:xopt} (taking into account $f(x+t)$, instead of $f(x)$, and $g(x)$) the following function $x^{\circ}(t)$ is obtained.
\begin{equation*}
x^{\circ}(t) = \left\{ \begin{array}{ll}
x_{\mathrm{s}}(t) & t < 10.\overline{3}\\
x_{1}(t) & 10.\overline{3} \leq t < 22.5\overline{3}\\
x_{\mathrm{e}}(t) & t \geq 22.5\overline{3}
\end{array} \right.
\end{equation*}
with
\begin{equation*}
x_{\mathrm{s}}(t) = \left\{ \begin{array}{ll}
8 & t < 8\\
-t + 16 & 8 \leq t < 10.\overline{3}
\end{array} \right. \qquad x_{1}(t) = \left\{ \begin{array}{ll}
8 & 10.\overline{3} \leq t < 13\\
-t + 21 & 13 \leq t < 18\\
3 & 18 \leq t < 22.5\overline{3}
\end{array} \right. \qquad x_{\mathrm{e}}(t) = 8
\end{equation*}

Note that, $T = \{ 10.\overline{3} , 22.5\overline{3} \}$, that is, $t^{\star}_{1} = 10.\overline{3}$ and $t^{\star}_{2} = 22.5\overline{3}$, and $Q = 2$. Since $T = \Omega$, the mapping function is basically $l(1) = 1$ and $l(2) = 2$. Moreover, $t^{\star}_{1} \leq \gamma_{a_{1}} - x_{1} = 13$ (then, $x_{\mathrm{s}}(t)$ has the structure of~\eqref{equ:xs_2}), $t^{\star}_{1} < \gamma_{a_{2}} - x_{2} = 13$ and $t^{\star}_{2} > \gamma_{a_{2}} - x_{1} = 18$ (then, $x_{1}(t)$ has the structure of~\eqref{equ:xj_1}); since $l(Q) = \lvert A \rvert = 2$, the function $x_{\mathrm{e}}(t)$ has the structure of~\eqref{equ:xe_3}. The graphical representation of $x^{\circ}(t)$ is the following.

\begin{figure}[h]
\centering
\psfrag{xopt(t)}[cl][Bl][.8][0]{$x^{\circ}(t)$}
\psfrag{x}[bc][Bl][.8][0]{$t$}
\psfrag{XS}[bc][Bl][.7][0]{$x_{\mathrm{s}}(t)$}
\psfrag{X1}[bc][Bl][.7][0]{$x_{1}(t)$}
\psfrag{X2}[bc][Bl][.7][0]{$x_{2}(t)$}
\psfrag{XE}[bc][Bl][.7][0]{$x_{\mathrm{e}}(t)$}
\psfrag{Y1}[cr][Bl][.7][0]{$3$}
\psfrag{Y2}[cr][Bl][.7][0]{$8$}
\psfrag{T1}[bc][Bl][.7][0]{$8$}
\psfrag{T2}[bc][Bl][.7][0]{$10.\overline{3}$}
\psfrag{T3}[bc][Bl][.7][0]{$13$}
\psfrag{T4}[bc][Bl][.7][0]{$18$}
\psfrag{T5}[bc][Bl][.7][0]{$22.5\overline{3}$}
\includegraphics[scale=.25]{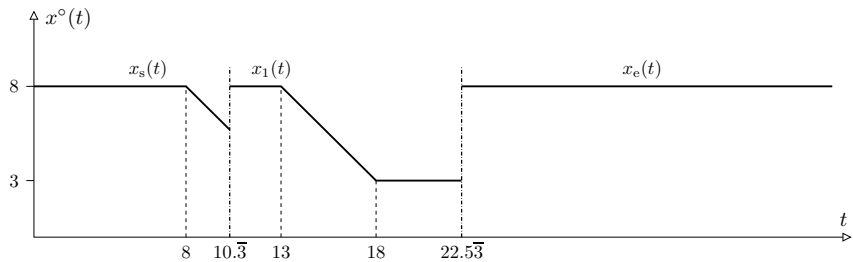}
\caption{Example 4 -- Functions $x^{\circ}(t)$.}
\label{fig:es4_xopt}
\end{figure}

By applying lemma~\ref{lem:h(t)} the following function $h(t) = f \big( x^{\circ}(t) + t \big) + g \big( x^{\circ}(t) \big)$ is obtained.

\begin{figure}[h]
\centering
\psfrag{h(t)}[cl][Bl][.8][0]{$h(t)$}
\psfrag{x}[bc][Bl][.8][0]{$t$}
\psfrag{m1}[Bc][Bl][.6][0]{$0.5$}
\psfrag{m2}[cr][Bl][.6][0]{$2$}
\psfrag{m3}[Bc][Bl][.6][0]{$0.25$}
\psfrag{m4}[cr][Bl][.6][0]{$1$}
\psfrag{m5}[cr][Bl][.6][0]{$3$}
\psfrag{m6}[cr][Bl][.6][0]{$4$}
\psfrag{G1}[bc][Bl][.7][0]{$6$}
\psfrag{G2}[bc][Bl][.7][0]{$8$}
\psfrag{G3}[bc][Bl][.7][0]{$10.\overline{3}$}
\psfrag{G4}[bc][Bl][.7][0]{$13$}
\psfrag{G5}[bc][Bl][.7][0]{$18$}
\psfrag{G6}[bc][Bl][.7][0]{$20$}
\psfrag{G7}[bc][Bl][.7][0]{$21$}
\psfrag{G8}[bc][Bl][.7][0]{$22$}
\psfrag{G9}[bc][Bl][.7][0]{$22.5\overline{3}$}
\psfrag{G10}[bc][Bl][.7][0]{$26$}
\psfrag{G11}[bc][Bl][.7][0]{$30$}
\includegraphics[scale=.25]{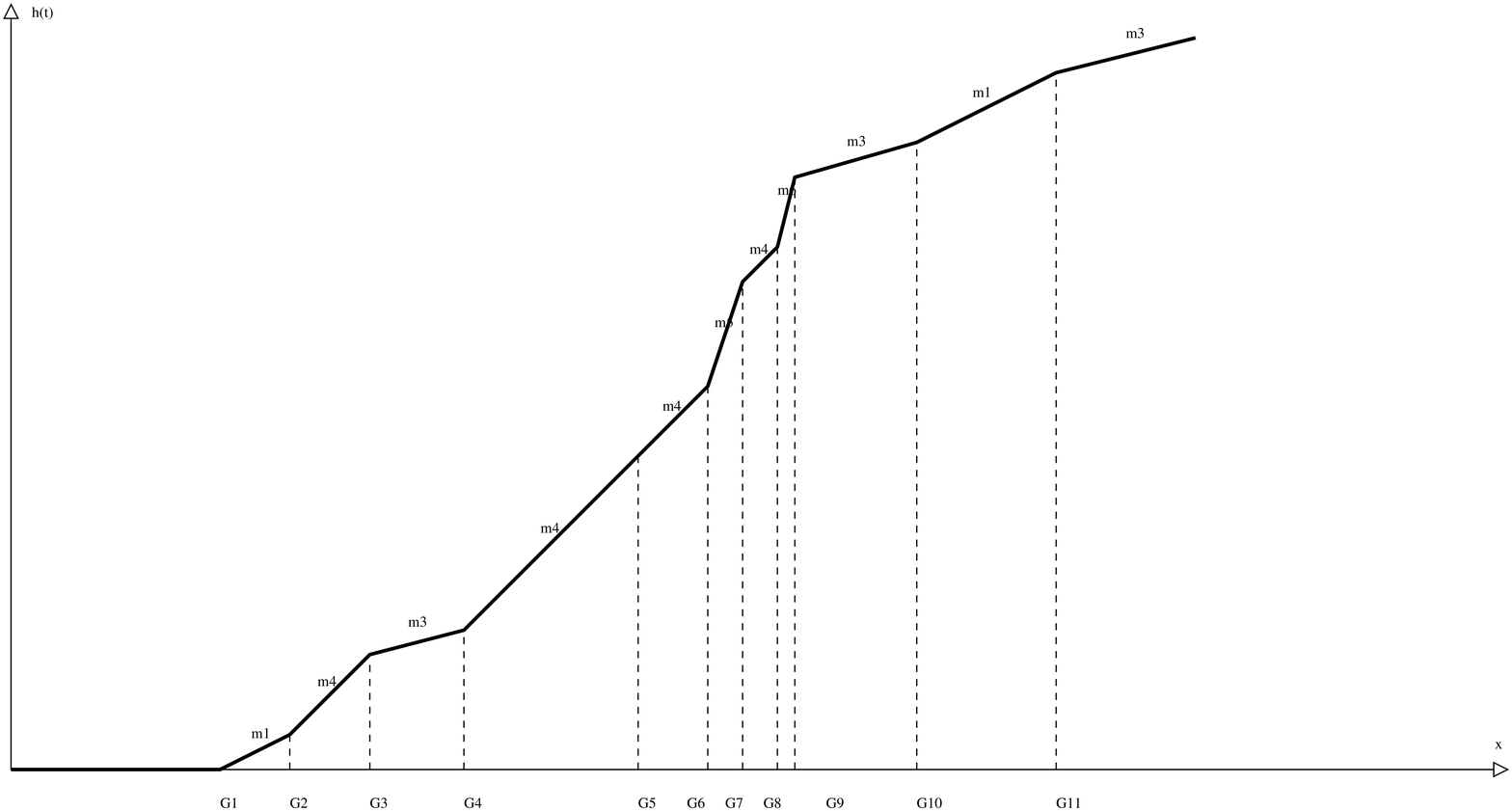}
\caption{Example 4 -- Function $h(t) = f \big( x^{\circ}(t) + t \big) + g \big( x^{\circ}(t) \big)$.}
\label{fig:es4_h}
\end{figure}

\vspace{36pt}

In accordance with lemma~\ref{lem:h(t)} and, in particular, with~\eqref{equ:h_3}, function $h(t)$ is
\begin{equation*}
h(t) = \left\{ \begin{array}{ll}
\vspace{3pt} f(8+t) & \forall \, t \, : \, x^{\circ}(t) = 8\\
\vspace{3pt} t-7 & \forall \, t < 10.\overline{3} : x^{\circ}(t) \neq \{ 3,8 \} \quad (\Rightarrow 8 < t < 10.\overline{3})\\
\vspace{3pt} t-9 & \forall \, t \in [ 10.\overline{3} , 22.5\overline{3} ) : x^{\circ}(t) \neq \{ 3,8 \} \quad (\Rightarrow 13 < t < 18)\\
f(3+t) + 5 & \forall \, t \, : \, x^{\circ}(t) = 3\\
\end{array} \right.
\end{equation*}

\clearpage

\subsection{Example 5}

Consider the following functions $f(x+t)$ and $g(x)$ (depicted in the same graphic).

\begin{figure}[h]
\centering
\psfrag{f(x)}[cl][Bl][.8][0]{$f(x+t)$}
\psfrag{g(x)}[cl][Bl][.8][0]{$g(x)$}
\psfrag{x}[bc][Bl][.8][0]{$x$}
\psfrag{X1}[tc][Bl][.7][0]{$4$}
\psfrag{X2}[tc][Bl][.7][0]{$8$}
\psfrag{G1}[tc][Bl][.7][0]{$17$}
\psfrag{G2}[tc][Bl][.7][0]{$18$}
\psfrag{G3}[tc][Bl][.7][0]{$19$}
\psfrag{G4}[tc][Bl][.7][0]{$23$}
\psfrag{G5}[tc][Bl][.7][0]{$24$}
\psfrag{G6}[tc][Bl][.7][0]{$25$}
\psfrag{n}[cl][Bl][.6][0]{$-0.5$}
\psfrag{m1}[cr][Bl][.6][0]{$1$}
\psfrag{m2}[Bc][Bl][.6][0]{$0$}
\psfrag{m3}[Bc][Bl][.6][0]{$0.25$}
\psfrag{m4}[cr][Bl][.6][0]{$2$}
\includegraphics[scale=.25]{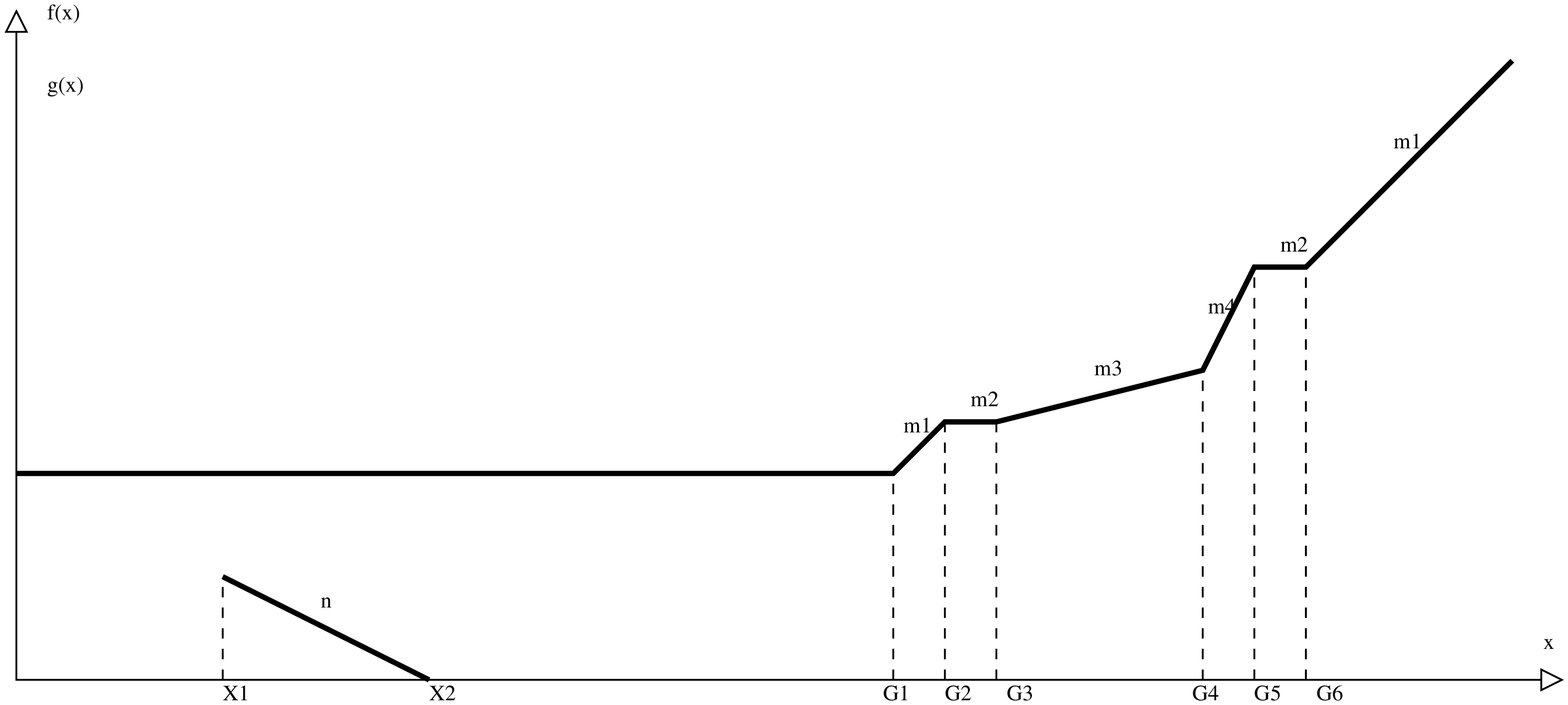}
\caption{Example 5 -- Functions $f(x+t)$ and $g(x)$.}
\label{fig:es5_fg}
\end{figure}

Algorithm~\ref{alg:tstar} provides $\omega_{1} = 11$ and $\omega_{2} = 19.\overline{6}$. Then, by applying lemma~\ref{lem:xopt} (taking into account $f(x+t)$, instead of $f(x)$, and $g(x)$) the following function $x^{\circ}(t)$ is obtained.
\begin{equation*}
x^{\circ}(t) = \left\{ \begin{array}{ll}
x_{\mathrm{s}}(t) & t < 11\\
x_{1}(t) & 11 \leq t < 19.\overline{6}\\
x_{\mathrm{e}}(t) & t \geq 19.\overline{6}
\end{array} \right.
\end{equation*}
with
\begin{equation*}
x_{\mathrm{s}}(t) = \left\{ \begin{array}{ll}
8 & t < 9\\
-t + 17 & 9 \leq t < 11
\end{array} \right. \qquad x_{1}(t) = \left\{ \begin{array}{ll}
8 & 11 \leq t < 15\\
-t + 23 & 15 \leq t < 19\\
4 & 19 \leq t < 19.\overline{6}
\end{array} \right.
\end{equation*}
\begin{equation*}
x_{\mathrm{e}}(t) = \left\{ \begin{array}{ll}
-t + 25 & 19.\overline{6} \leq t < 21\\
4 & t \geq 21
\end{array} \right.
\end{equation*}

Note that, $T = \{ 11 , 19.\overline{6} \}$, that is, $t^{\star}_{1} = 11$ and $t^{\star}_{2} = 19.\overline{6}$, and $Q = 2$. Since $T = \Omega$, the mapping function is basically $l(1) = 1$ and $l(2) = 2$. Moreover, $t^{\star}_{1} \leq \gamma_{a_{1}} - x_{1} = 13$ (then, $x_{\mathrm{s}}(t)$ has the structure of~\eqref{equ:xs_2}), $t^{\star}_{1} < \gamma_{a_{2}} - x_{2} = 15$ and $t^{\star}_{2} > \gamma_{a_{2}} - x_{1} = 19$ (then, $x_{1}(t)$ has the structure of~\eqref{equ:xj_1}), and $t^{\star}_{2} \geq \gamma_{a_{3}} - x_{2} = 17$ (then, $x_{\mathrm{e}}(t)$ has the structure of~\eqref{equ:xe_2}). The graphical representation of $x^{\circ}(t)$ is the following.

\begin{figure}[h]
\centering
\psfrag{xopt(t)}[cl][Bl][.8][0]{$x^{\circ}(t)$}
\psfrag{x}[bc][Bl][.8][0]{$t$}
\psfrag{XS}[bc][Bl][.7][0]{$x_{\mathrm{s}}(t)$}
\psfrag{X1}[bc][Bl][.7][0]{$x_{1}(t)$}
\psfrag{XE}[bc][Bl][.7][0]{$x_{\mathrm{e}}(t)$}
\psfrag{Y1}[cr][Bl][.7][0]{$4$}
\psfrag{Y2}[cr][Bl][.7][0]{$8$}
\psfrag{T1}[bc][Bl][.7][0]{$9$}
\psfrag{T2}[bc][Bl][.7][0]{$11$}
\psfrag{T3}[bc][Bl][.7][0]{$15$}
\psfrag{T4}[bc][Bl][.7][0]{$19$}
\psfrag{T5}[bc][Bl][.7][0]{$19.\overline{6}$}
\psfrag{T6}[bc][Bl][.7][0]{$21$}
\includegraphics[scale=.25]{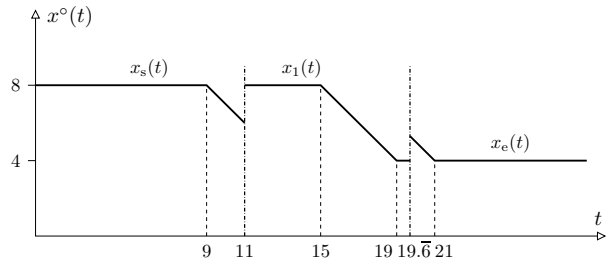}
\caption{Example 5 -- Functions $x^{\circ}(t)$.}
\label{fig:es5_xopt}
\end{figure}

By applying lemma~\ref{lem:h(t)} the following function $h(t) = f \big( x^{\circ}(t) + t \big) + g \big( x^{\circ}(t) \big)$ is obtained.

\begin{figure}[h]
\centering
\psfrag{h(t)}[cl][Bl][.8][0]{$h(t)$}
\psfrag{x}[bc][Bl][.8][0]{$t$}
\psfrag{m1}[cr][Bl][.6][0]{$1$}
\psfrag{m2}[Bc][Bl][.6][0]{$0.5$}
\psfrag{m3}[Bc][Bl][.6][0]{$0.25$}
\psfrag{m4}[cr][Bl][.6][0]{$2$}
\psfrag{G1}[bc][Bl][.7][0]{$9$}
\psfrag{G2}[bc][Bl][.7][0]{$11$}
\psfrag{G3}[bc][Bl][.7][0]{$15$}
\psfrag{G4}[bc][Bl][.7][0]{$19$}
\psfrag{G5}[bc][Bl][.7][0]{$19.\overline{6}$}
\psfrag{G6}[bc][Bl][.7][0]{$21$}
\includegraphics[scale=.25]{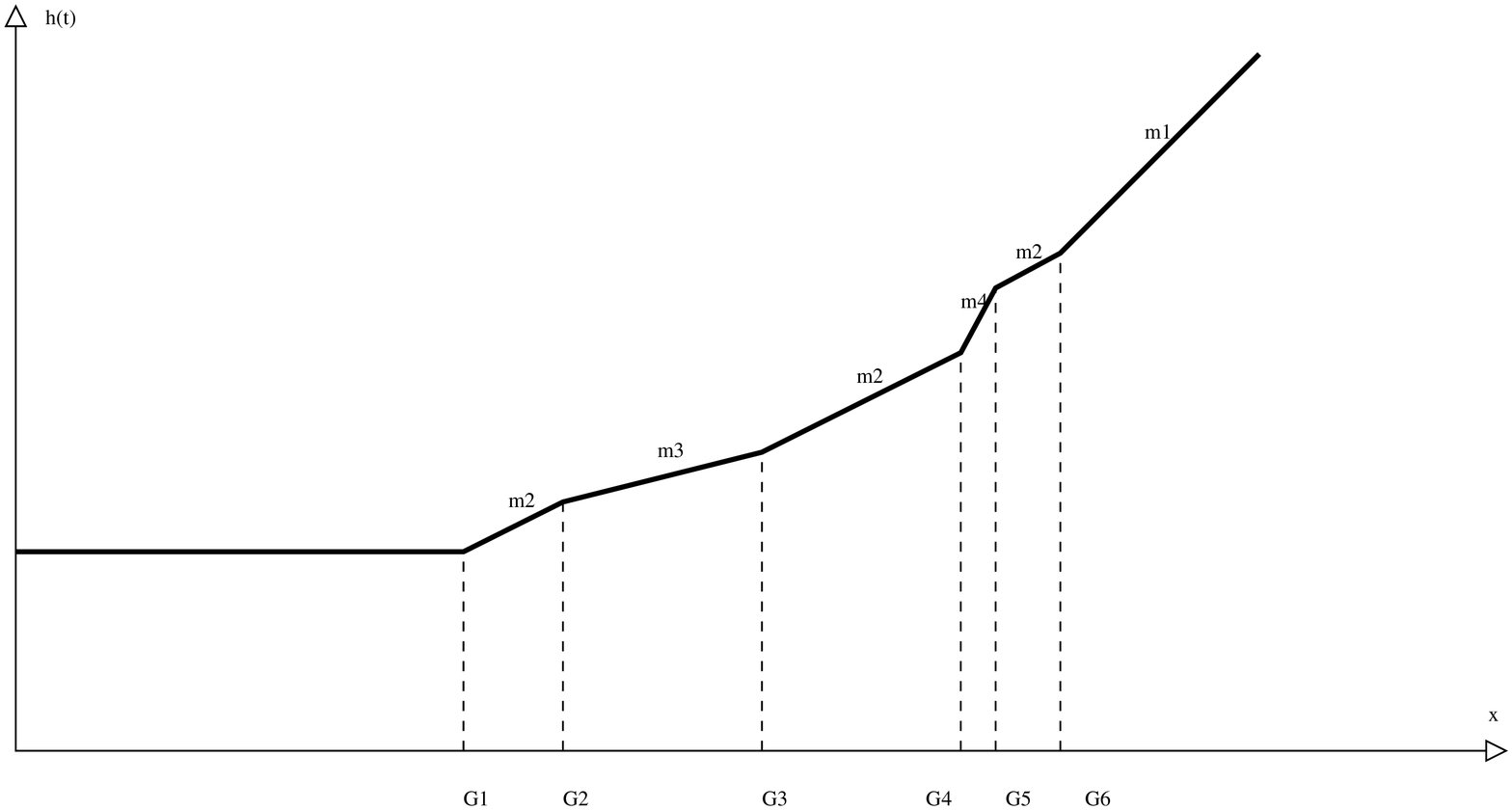}
\caption{Example 5 -- Function $h(t) = f \big( x^{\circ}(t) + t \big) + g \big( x^{\circ}(t) \big)$.}
\label{fig:es5_h}
\end{figure}

\vspace{60pt}

In accordance with lemma~\ref{lem:h(t)} and, in particular, with~\eqref{equ:h_3}, function $h(t)$ is
\begin{equation*}
h(t) = \left\{ \begin{array}{ll}
\vspace{3pt} f(8+t) & \forall \, t \, : \, x^{\circ}(t) = 8\\
\vspace{3pt} 0.5 \, t - 0.5 & \forall \, t < 11 : x^{\circ}(t) \neq \{ 4,8 \} \quad (\Rightarrow 9 < t < 11)\\
\vspace{3pt} 0.5 \, t - 1.5 & \forall \, t \in [ 11 , 19.\overline{6} ) : x^{\circ}(t) \neq \{ 4,8 \} \quad (\Rightarrow 15 < t < 19)\\
\vspace{3pt} 0.5 \, t - 0.5 & \forall \, t \geq 19.\overline{6} : x^{\circ}(t) \neq \{ 4,8 \} \quad (\Rightarrow 19.\overline{6} < t < 21)\\
f(4+t) + 2 & \forall \, t \, : \, x^{\circ}(t) = 4\\
\end{array} \right.
\end{equation*}

\clearpage

\subsection{Example 6}

Consider the following functions $f(x)$ and $g(x)$ (depicted in the same graphic).

\begin{figure}[h]
\centering
\psfrag{f(x)}[cl][Bl][.8][0]{$f(x)$}
\psfrag{g(x)}[cl][Bl][.8][0]{$g(x)$}
\psfrag{x}[bc][Bl][.8][0]{$x$}
\psfrag{X1}[tc][Bl][.7][0]{$4$}
\psfrag{X2}[tc][Bl][.7][0]{$8$}
\psfrag{G1}[tc][Bl][.7][0]{$20$}
\psfrag{G2}[tc][Bl][.7][0]{$24$}
\psfrag{G3}[tc][Bl][.7][0]{$25$}
\psfrag{G4}[tc][Bl][.7][0]{$26$}
\psfrag{G5}[tc][Bl][.7][0]{$36$}
\psfrag{n}[cl][Bl][.6][0]{$-1$}
\psfrag{m1}[cr][Bl][.6][0]{$3$}
\psfrag{m2}[Bc][Bl][.6][0]{$0$}
\psfrag{m3}[cr][Bl][.6][0]{$1.5$}
\psfrag{m4}[Bc][Bl][.6][0]{$0$}
\psfrag{m5}[cr][Bl][.6][0]{$1.5$}
\includegraphics[scale=.25]{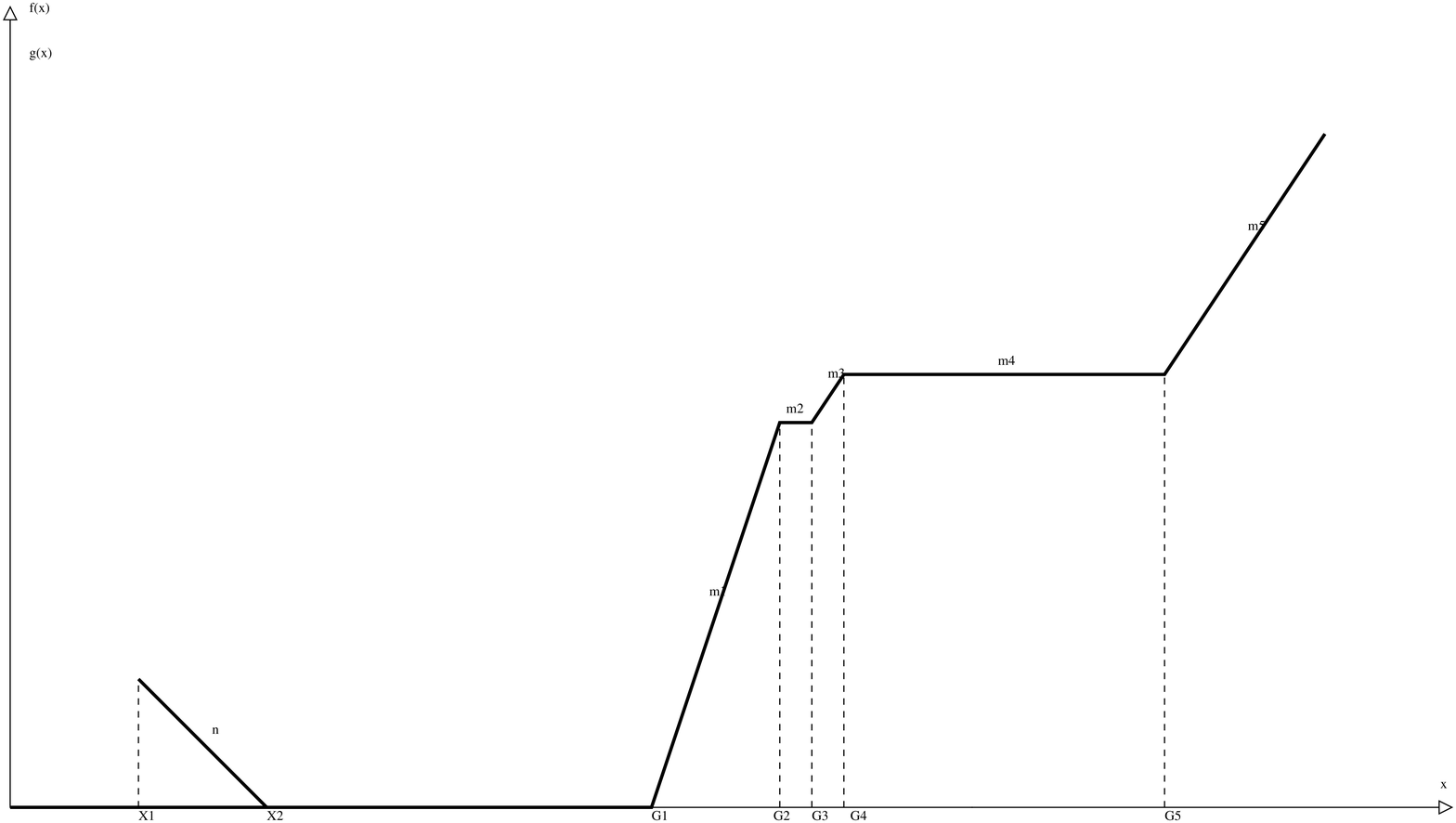}
\caption{Example 6 -- Functions $f(x)$ and $g(x)$.}
\label{fig:es6_fg}
\end{figure}

Algorithm~\ref{alg:tstar} provides $\omega_{1} = +\infty$ and $\omega_{2} = 19.1\overline{6}$. Then, by applying lemma~\ref{lem:xopt} (taking into account $f(x+t)$, instead of $f(x)$, and $g(x)$) the following function $x^{\circ}(t)$ is obtained.
\begin{equation*}
x^{\circ}(t) = \left\{ \begin{array}{ll}
x_{\mathrm{s}}(t) & t < 19.1\overline{6}\\
x_{\mathrm{e}}(t) & t \geq 19.1\overline{6}
\end{array} \right.
\end{equation*}
with
\begin{equation*}
x_{\mathrm{s}}(t) = \left\{ \begin{array}{ll}
8 & t < 12\\
-t + 20 & 12 \leq t < 16\\
4 & 16 \leq t < 19.1\overline{6}
\end{array} \right. \qquad x_{\mathrm{e}}(t) = \left\{ \begin{array}{ll}
8 & 19.1\overline{6} \leq t < 28\\
-t + 36 & 28 \leq t < 32\\
4 & t \geq 32
\end{array} \right.
\end{equation*}

Note that, $T = \{ 19.1\overline{6} \}$, that is, $t^{\star}_{1} = 19.1\overline{6}$, and $Q = 1$. The mapping function provides $l(1) = 2$. Moreover, $t^{\star}_{1} > \gamma_{a_{1}} - x_{1} = 16$ (then, $x_{\mathrm{s}}(t)$ has the structure of~\eqref{equ:xs_1}) and $t^{\star}_{1} < \gamma_{a_{3}} - x_{2} = 28$ (then, $x_{\mathrm{e}}(t)$ has the structure of~\eqref{equ:xe_1}). The graphical representation of $x^{\circ}(t)$ is the following.

\begin{figure}[h]
\centering
\psfrag{xopt(t)}[cl][Bl][.8][0]{$x^{\circ}(t)$}
\psfrag{x}[bc][Bl][.8][0]{$t$}
\psfrag{XS}[bc][Bl][.7][0]{$x_{\mathrm{s}}(t)$}
\psfrag{XE}[bc][Bl][.7][0]{$x_{\mathrm{e}}(t)$}
\psfrag{Y1}[cr][Bl][.7][0]{$4$}
\psfrag{Y2}[cr][Bl][.7][0]{$8$}
\psfrag{T1}[bc][Bl][.7][0]{$12$}
\psfrag{T2}[bc][Bl][.7][0]{$16$}
\psfrag{T3}[bc][Bl][.7][0]{$19.1\overline{6}$}
\psfrag{T4}[bc][Bl][.7][0]{$28$}
\psfrag{T5}[bc][Bl][.7][0]{$32$}
\includegraphics[scale=.25]{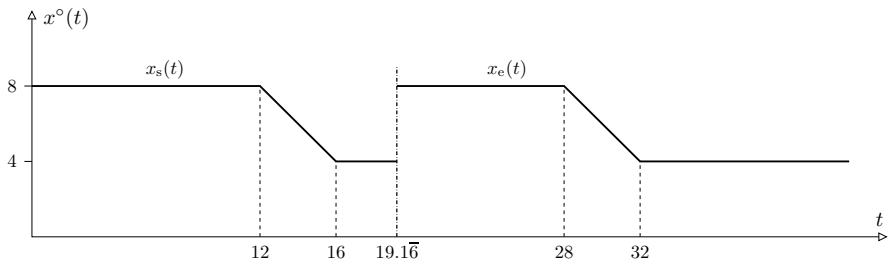}
\caption{Example 6 -- Functions $x^{\circ}(t)$.}
\label{fig:es6_xopt}
\end{figure}

By applying lemma~\ref{lem:h(t)} the following function $h(t) = f \big( x^{\circ}(t) + t \big) + g \big( x^{\circ}(t) \big)$ is obtained.

\begin{figure}[h]
\centering
\psfrag{h(t)}[cl][Bl][.8][0]{$h(t)$}
\psfrag{x}[bc][Bl][.8][0]{$t$}
\psfrag{m1}[cr][Bl][.6][0]{$1$}
\psfrag{m2}[cr][Bl][.6][0]{$3$}
\psfrag{m3}[Bc][Bl][.6][0]{$0$}
\psfrag{m4}[cr][Bl][.6][0]{$1$}
\psfrag{m5}[cr][Bl][.6][0]{$1.5$}
\psfrag{G1}[bc][Bl][.7][0]{$12$}
\psfrag{G2}[bc][Bl][.7][0]{$16$}
\psfrag{G3}[bc][Bl][.7][0]{$19.1\overline{6}$}
\psfrag{G4}[bc][Bl][.7][0]{$28$}
\psfrag{G5}[bc][Bl][.7][0]{$32$}
\includegraphics[scale=.25]{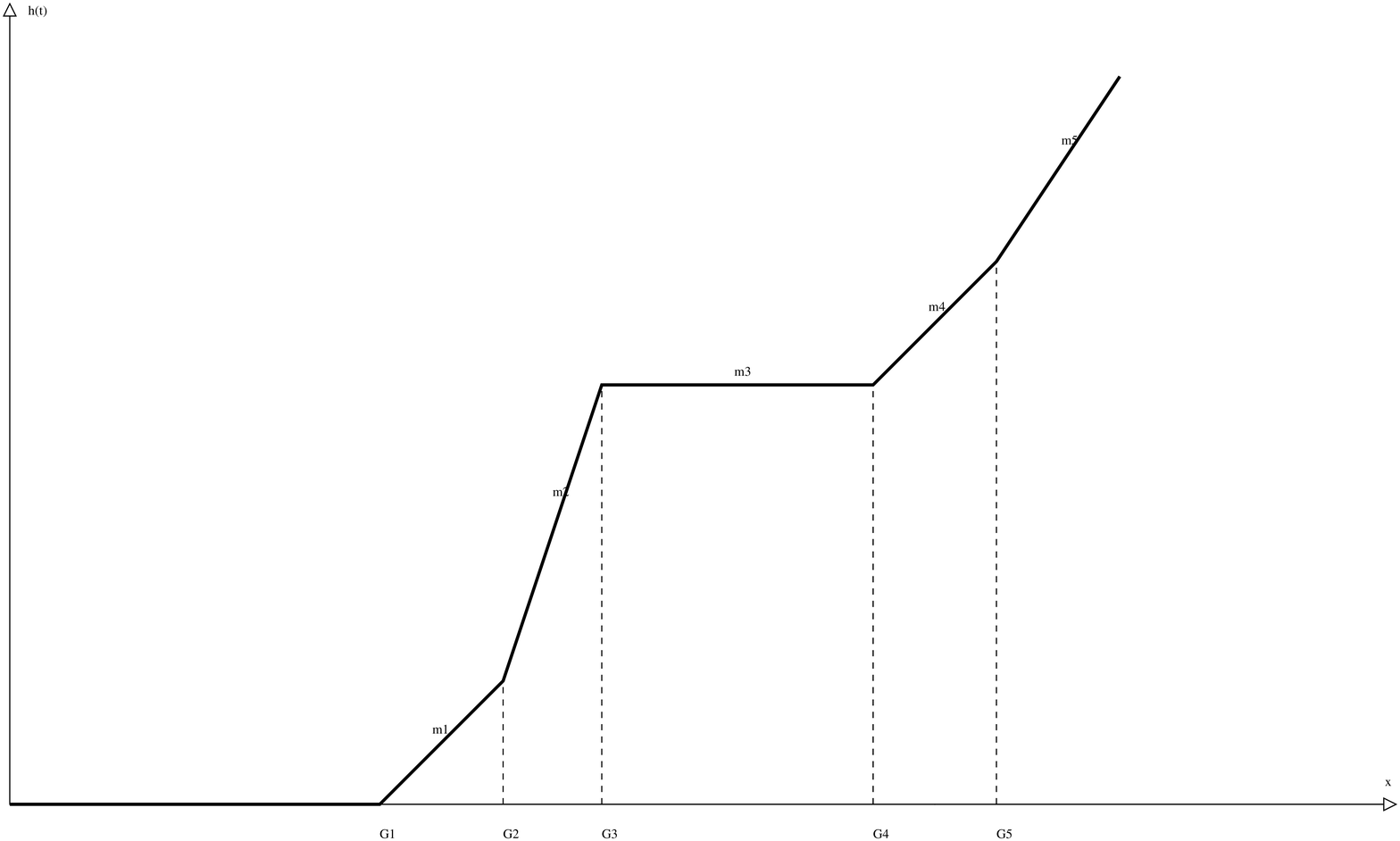}
\caption{Example 6 -- Function $h(t) = f \big( x^{\circ}(t) + t \big) + g \big( x^{\circ}(t) \big)$.}
\label{fig:es6_h}
\end{figure}

\vspace{36pt}

In accordance with lemma~\ref{lem:h(t)} and, in particular, with~\eqref{equ:h_2}, function $h(t)$ is
\begin{equation*}
h(t) = \left\{ \begin{array}{ll}
\vspace{3pt} f(8+t) & \forall \, t \, : \, x^{\circ}(t) = 8\\
\vspace{3pt} t-12 & \forall \, t < 19.1\overline{6} : x^{\circ}(t) \neq \{ 4,8 \} \quad (\Rightarrow 12 < t < 16)\\
\vspace{3pt} t - 14.5 & \forall \, t \geq 19.1\overline{6} : x^{\circ}(t) \neq \{ 4,8 \} \quad (\Rightarrow 28 < t < 32)\\
f(4+t) + 4 & \forall \, t \, : \, x^{\circ}(t) = 4\\
\end{array} \right.
\end{equation*}

\clearpage

\subsection{Example 7}

Consider the following functions $f(x)$ and $g(x)$ (depicted in the same graphic).

\begin{figure}[h]
\centering
\psfrag{f(x)}[cl][Bl][.8][0]{$f(x)$}
\psfrag{g(x)}[cl][Bl][.8][0]{$g(x)$}
\psfrag{x}[bc][Bl][.8][0]{$x$}
\psfrag{X1}[tc][Bl][.7][0]{$4$}
\psfrag{X2}[tc][Bl][.7][0]{$8$}
\psfrag{G1}[tc][Bl][.7][0]{$20$}
\psfrag{G2}[tc][Bl][.7][0]{$24$}
\psfrag{G3}[tc][Bl][.7][0]{$25$}
\psfrag{G4}[tc][Bl][.7][0]{$26$}
\psfrag{G5}[tc][Bl][.7][0]{$36$}
\psfrag{n}[cl][Bl][.6][0]{$-1$}
\psfrag{m1}[cr][Bl][.6][0]{$3$}
\psfrag{m2}[Bc][Bl][.6][0]{$0$}
\psfrag{m3}[cr][Bl][.6][0]{$1.5$}
\psfrag{m4}[cr][Bl][.6][0]{$0.9$}
\psfrag{m5}[cr][Bl][.6][0]{$1.5$}
\includegraphics[scale=.25]{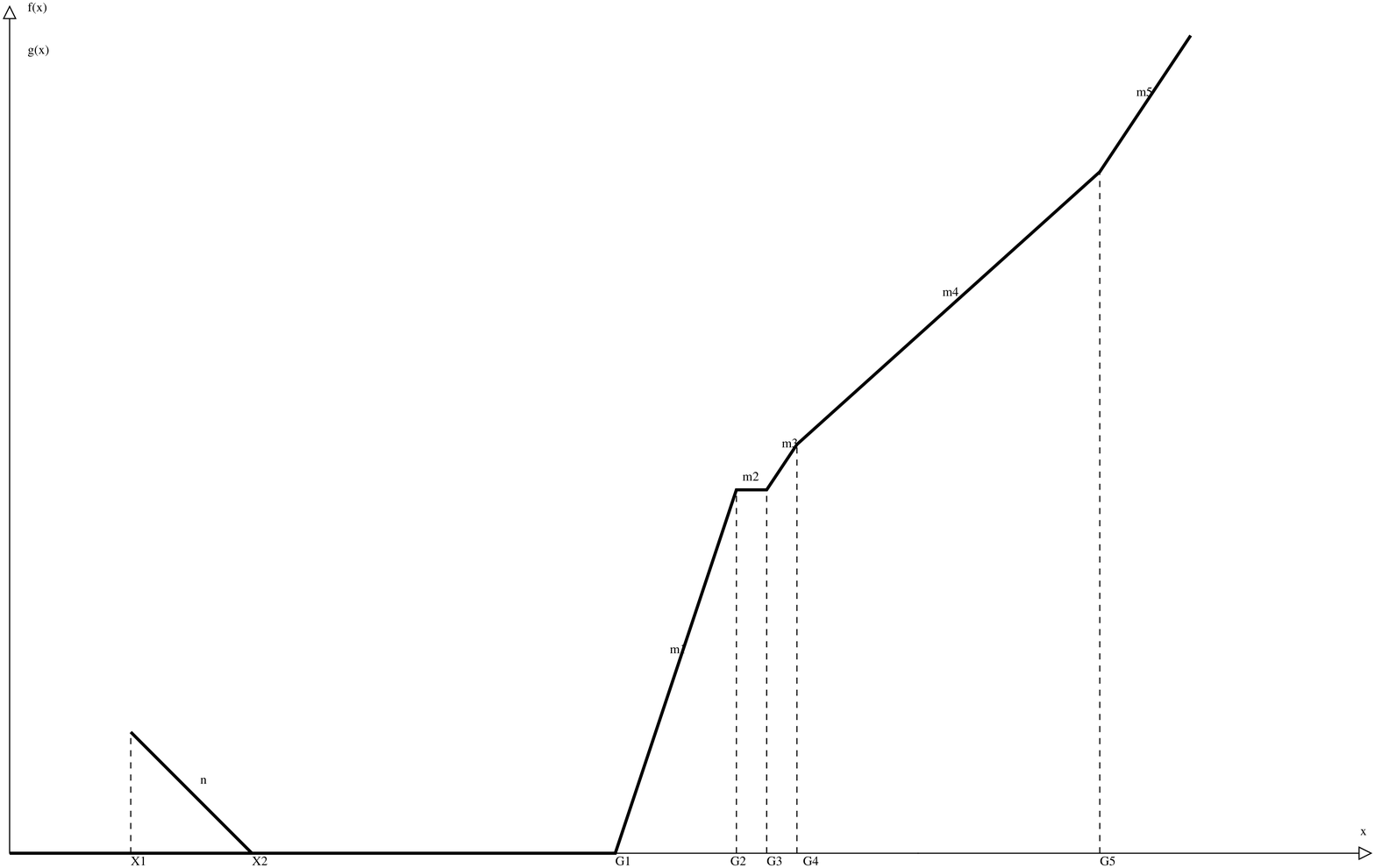}
\caption{Example 7 -- Functions $f(x)$ and $g(x)$.}
\label{fig:es7_fg}
\end{figure}

Algorithm~\ref{alg:tstar} provides $\omega_{1} = 19.5$ and $\omega_{2} = 21.\overline{3}$. Then, by applying lemma~\ref{lem:xopt} (taking into account $f(x+t)$, instead of $f(x)$, and $g(x)$) the following function $x^{\circ}(t)$ is obtained.
\begin{equation*}
x^{\circ}(t) = \left\{ \begin{array}{ll}
x_{\mathrm{s}}(t) & t < 19.5\\
x_{1}(t) & 19.5 \leq t < 21.\overline{3}\\
x_{\mathrm{e}}(t) & t \geq 21.\overline{3}
\end{array} \right.
\end{equation*}
with
\begin{equation*}
x_{\mathrm{s}}(t) = \left\{ \begin{array}{ll}
8 & t < 12\\
-t + 20 & 12 \leq t < 16\\
4 & 16 \leq t < 19.5
\end{array} \right. \qquad x_{1}(t) = \left\{ \begin{array}{ll}
-t + 25 & 19.5 \leq t < 21\\
4 & 21 \leq t < 21.\overline{3}
\end{array} \right.
\end{equation*}
\begin{equation*}
x_{\mathrm{e}}(t) = \left\{ \begin{array}{ll}
8 & 21.\overline{3} \leq t < 28\\
-t + 36 & 28 \leq t < 32\\
4 & t \geq 32
\end{array} \right.
\end{equation*}

Note that, $T = \{ 19.5 , 21.\overline{3} \}$, that is, $t^{\star}_{1} = 19.5$ and $t^{\star}_{2} = 21.\overline{3}$, and $Q = 2$. Since $T = \Omega$, the mapping function is basically $l(1) = 1$ and $l(2) = 2$. Moreover, $t^{\star}_{1} > \gamma_{a_{1}} - x_{1} = 16$ (then, $x_{\mathrm{s}}(t)$ has the structure of~\eqref{equ:xs_1}), $t^{\star}_{1} \geq \gamma_{a_{2}} - x_{2} = 17$ and $t^{\star}_{2} > \gamma_{a_{2}} - x_{1} = 21$ (then, $x_{1}(t)$ has the structure of~\eqref{equ:xj_2}), and $t^{\star}_{2} < \gamma_{a_{4}} - x_{2} = 28$ (then, $x_{\mathrm{e}}(t)$ has the structure of~\eqref{equ:xe_1}). The graphical representation of $x^{\circ}(t)$ is the following.

\begin{figure}[h]
\centering
\psfrag{xopt(t)}[cl][Bl][.8][0]{$x^{\circ}(t)$}
\psfrag{x}[bc][Bl][.8][0]{$t$}
\psfrag{XS}[bc][Bl][.7][0]{$x_{\mathrm{s}}(t)$}
\psfrag{X1}[cr][Bl][.7][0]{$x_{1}(t)$}
\psfrag{XE}[bc][Bl][.7][0]{$x_{\mathrm{e}}(t)$}
\psfrag{Y1}[cr][Bl][.7][0]{$4$}
\psfrag{Y2}[cr][Bl][.7][0]{$8$}
\psfrag{T1}[bc][Bl][.7][0]{$12$}
\psfrag{T2}[bc][Bl][.7][0]{$16$}
\psfrag{T3}[br][Bl][.7][0]{$19.5$}
\psfrag{T4}[br][Bl][.7][0]{$21$}
\psfrag{T5}[bl][Bl][.7][0]{$21.\overline{3}$}
\psfrag{T6}[bc][Bl][.7][0]{$28$}
\psfrag{T7}[bc][Bl][.7][0]{$32$}
\includegraphics[scale=.25]{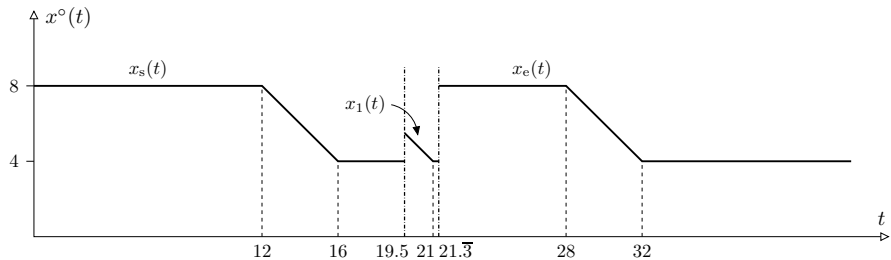}
\caption{Example 7 -- Functions $x^{\circ}(t)$.}
\label{fig:es7_xopt}
\end{figure}

By applying lemma~\ref{lem:h(t)} the following function $h(t) = f \big( x^{\circ}(t) + t \big) + g \big( x^{\circ}(t) \big)$ is obtained.

\begin{figure}[h]
\centering
\psfrag{h(t)}[cl][Bl][.8][0]{$h(t)$}
\psfrag{x}[bc][Bl][.8][0]{$t$}
\psfrag{m1}[cr][Bl][.6][0]{$1$}
\psfrag{m2}[cr][Bl][.6][0]{$3$}
\psfrag{m3}[cr][Bl][.6][0]{$1$}
\psfrag{m4}[cr][Bl][.6][0]{$1.5$}
\psfrag{m5}[cr][Bl][.6][0]{$0.9$}
\psfrag{m6}[cr][Bl][.6][0]{$1$}
\psfrag{m7}[cr][Bl][.6][0]{$1.5$}
\psfrag{G1}[bc][Bl][.7][0]{$12$}
\psfrag{G2}[bc][Bl][.7][0]{$16$}
\psfrag{G3}[br][Bl][.7][0]{$19.5$}
\psfrag{G4}[br][Bl][.7][0]{$21$}
\psfrag{G5}[bl][Bl][.7][0]{$21.\overline{3}$}
\psfrag{G6}[bc][Bl][.7][0]{$28$}
\psfrag{G7}[bc][Bl][.7][0]{$32$}
\includegraphics[scale=.25]{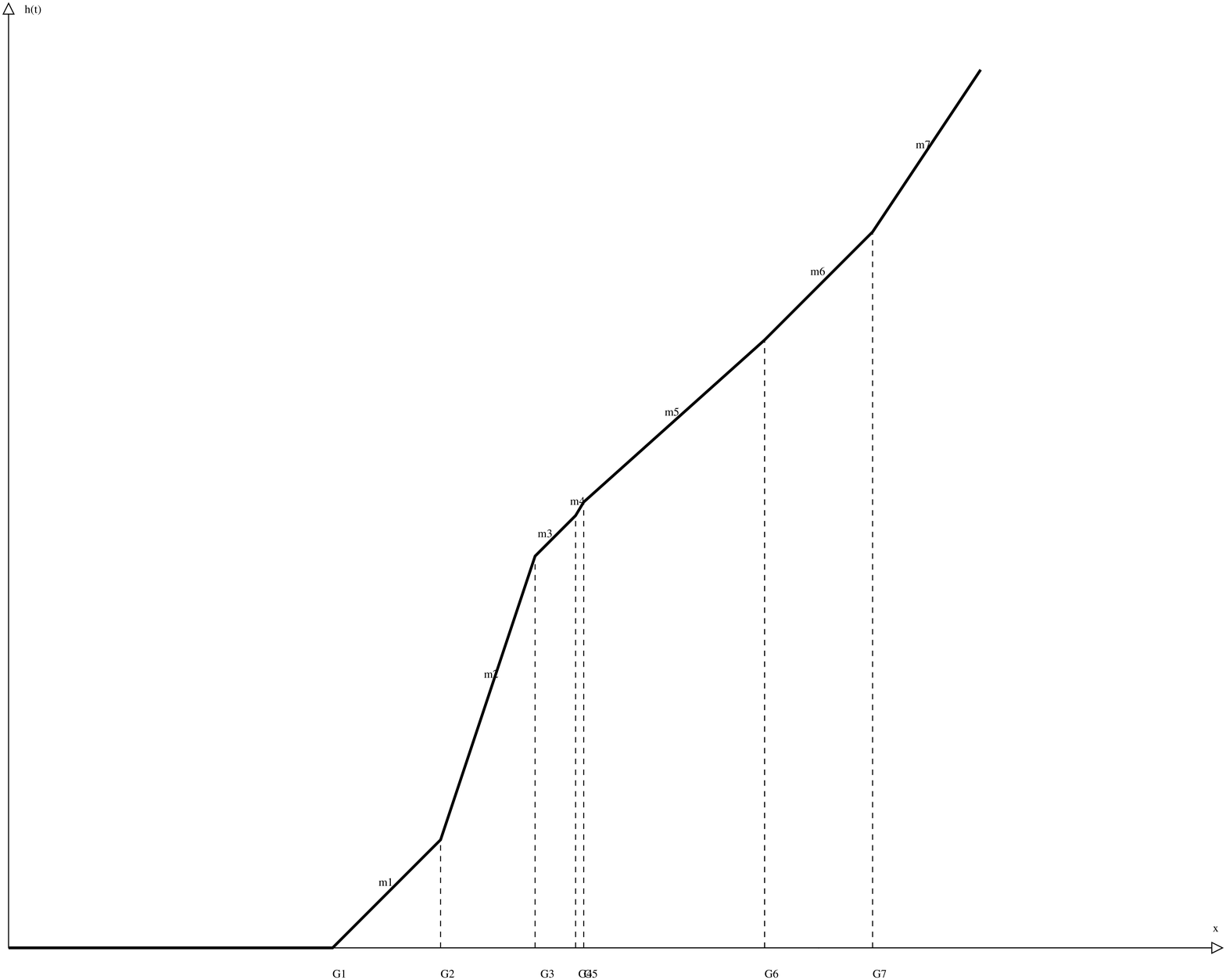}
\caption{Example 7 -- Function $h(t) = f \big( x^{\circ}(t) + t \big) + g \big( x^{\circ}(t) \big)$.}
\label{fig:es7_h}
\end{figure}

In accordance with lemma~\ref{lem:h(t)} and, in particular, with~\eqref{equ:h_3}, function $h(t)$ is
\begin{equation*}
h(t) = \left\{ \begin{array}{ll}
\vspace{3pt} f(8+t) & \forall \, t \, : \, x^{\circ}(t) = 8\\
\vspace{3pt} t-12 & \forall \, t < 19.5 : x^{\circ}(t) \neq \{ 4,8 \} \quad (\Rightarrow 12 < t < 16)\\
\vspace{3pt} t-5 & \forall \, t \in [ 19.5 , 21.\overline{3} ) : x^{\circ}(t) \neq \{ 4,8 \} \quad (\Rightarrow 19.5 < t < 21)\\
\vspace{3pt} t - 5.5 & \forall \, t \geq 21.\overline{3} : x^{\circ}(t) \neq \{ 4,8 \} \quad (\Rightarrow 28 < t < 32)\\
f(4+t) + 4 & \forall \, t \, : \, x^{\circ}(t) = 4\\
\end{array} \right.
\end{equation*}

\clearpage

\subsection{Example 8}

Consider the following functions $f(x)$ and $g(x)$ (depicted in the same graphic).

\begin{figure}[h]
\centering
\psfrag{f(x)}[cl][Bl][.8][0]{$f(x)$}
\psfrag{g(x)}[cl][Bl][.8][0]{$g(x)$}
\psfrag{x}[bc][Bl][.8][0]{$x$}
\psfrag{X1}[tc][Bl][.7][0]{$4$}
\psfrag{X2}[tc][Bl][.7][0]{$12$}
\psfrag{G1}[tc][Bl][.7][0]{$20$}
\psfrag{G2}[tc][Bl][.7][0]{$24$}
\psfrag{G3}[tc][Bl][.7][0]{$25$}
\psfrag{G4}[tc][Bl][.7][0]{$26$}
\psfrag{G5}[tc][Bl][.7][0]{$27$}
\psfrag{G6}[tc][Bl][.7][0]{$28$}
\psfrag{G7}[tc][Bl][.7][0]{$38$}
\psfrag{n}[cl][Bl][.6][0]{$-1$}
\psfrag{m1}[cr][Bl][.6][0]{$3$}
\psfrag{m2}[Bc][Bl][.6][0]{$0$}
\psfrag{m3}[cr][Bl][.6][0]{$3$}
\psfrag{m4}[Bc][Bl][.6][0]{$0$}
\psfrag{m5}[cr][Bl][.6][0]{$1.5$}
\psfrag{m6}[Bc][Bl][.6][0]{$0$}
\psfrag{m7}[cr][Bl][.6][0]{$1.5$}
\includegraphics[scale=.25]{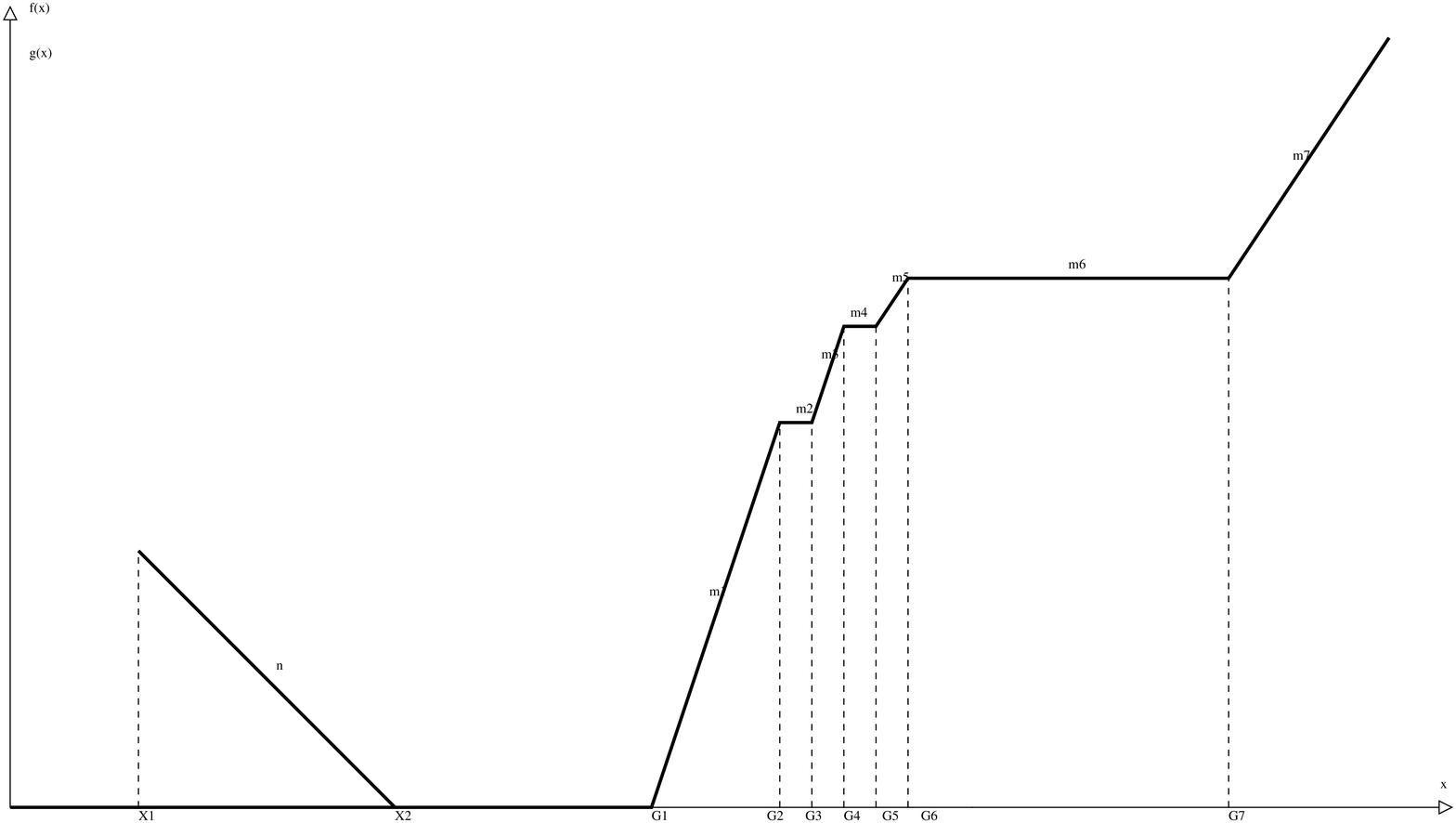}
\caption{Example 8 -- Functions $f(x)$ and $g(x)$.}
\label{fig:es8_fg}
\end{figure}

Algorithm~\ref{alg:tstar} provides $\omega_{1} = +\infty$, $\omega_{2} = +\infty$,  and $\omega_{3} = 18.8\overline{3}$. Then, by applying lemma~\ref{lem:xopt} (taking into account $f(x+t)$, instead of $f(x)$, and $g(x)$) the following function $x^{\circ}(t)$ is obtained.
\begin{equation*}
x^{\circ}(t) = \left\{ \begin{array}{ll}
x_{\mathrm{s}}(t) & t < 18.8\overline{3}\\
x_{\mathrm{e}}(t) & t \geq 18.8\overline{3}
\end{array} \right.
\end{equation*}
with
\begin{equation*}
x_{\mathrm{s}}(t) = \left\{ \begin{array}{ll}
12 & t < 8\\
-t + 20 & 8 \leq t < 16\\
4 & 16 \leq t < 18.8\overline{3}
\end{array} \right. \qquad x_{\mathrm{e}}(t) = \left\{ \begin{array}{ll}
12 & 18.8\overline{3} \leq t < 26\\
-t + 38 & 26 \leq t < 34\\
4 & t \geq 34
\end{array} \right.
\end{equation*}

Note that, $T = \{ 18.8\overline{3} \}$, that is, $t^{\star}_{1} = 18.8\overline{3}$, and $Q = 1$. The mapping function provides $l(1) = 3$. Moreover, $t^{\star}_{1} > \gamma_{a_{1}} - x_{1} = 16$ (then, $x_{\mathrm{s}}(t)$ has the structure of~\eqref{equ:xs_1}) and $t^{\star}_{1} < \gamma_{a_{4}} - x_{2} = 26$ (then, $x_{\mathrm{e}}(t)$ has the structure of~\eqref{equ:xe_1}). The graphical representation of $x^{\circ}(t)$ is the following.

\begin{figure}[h]
\centering
\psfrag{xopt(t)}[cl][Bl][.8][0]{$x^{\circ}(t)$}
\psfrag{x}[bc][Bl][.8][0]{$t$}
\psfrag{XS}[bc][Bl][.7][0]{$x_{\mathrm{s}}(t)$}
\psfrag{XE}[bc][Bl][.7][0]{$x_{\mathrm{e}}(t)$}
\psfrag{Y1}[cr][Bl][.7][0]{$4$}
\psfrag{Y2}[cr][Bl][.7][0]{$12$}
\psfrag{T1}[bc][Bl][.7][0]{$8$}
\psfrag{T2}[bc][Bl][.7][0]{$16$}
\psfrag{T3}[bc][Bl][.7][0]{$18.8\overline{3}$}
\psfrag{T4}[bc][Bl][.7][0]{$26$}
\psfrag{T5}[bc][Bl][.7][0]{$34$}
\includegraphics[scale=.25]{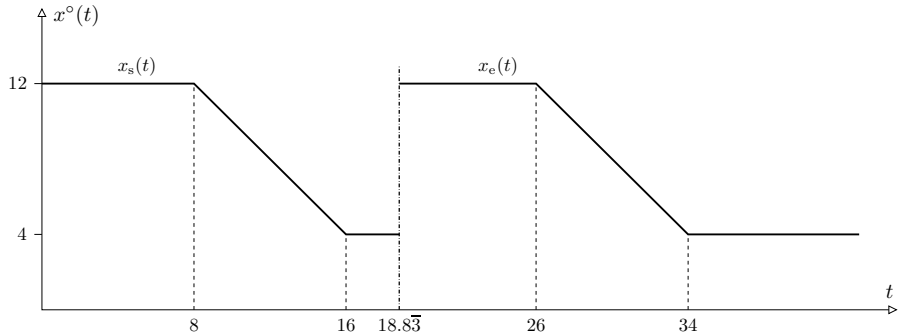}
\caption{Example 8 -- Functions $x^{\circ}(t)$.}
\label{fig:es8_xopt}
\end{figure}

By applying lemma~\ref{lem:h(t)} the following function $h(t) = f \big( x^{\circ}(t) + t \big) + g \big( x^{\circ}(t) \big)$ is obtained.

\begin{figure}[h]
\centering
\psfrag{h(t)}[cl][Bl][.8][0]{$h(t)$}
\psfrag{x}[bc][Bl][.8][0]{$t$}
\psfrag{m1}[cr][Bl][.6][0]{$1$}
\psfrag{m2}[cr][Bl][.6][0]{$3$}
\psfrag{m3}[Bc][Bl][.6][0]{$0$}
\psfrag{m4}[cr][Bl][.6][0]{$1$}
\psfrag{m5}[cr][Bl][.6][0]{$1.5$}
\psfrag{G1}[bc][Bl][.7][0]{$8$}
\psfrag{G2}[bc][Bl][.7][0]{$16$}
\psfrag{G3}[bc][Bl][.7][0]{$18.8\overline{3}$}
\psfrag{G4}[bc][Bl][.7][0]{$26$}
\psfrag{G5}[bc][Bl][.7][0]{$34$}
\includegraphics[scale=.25]{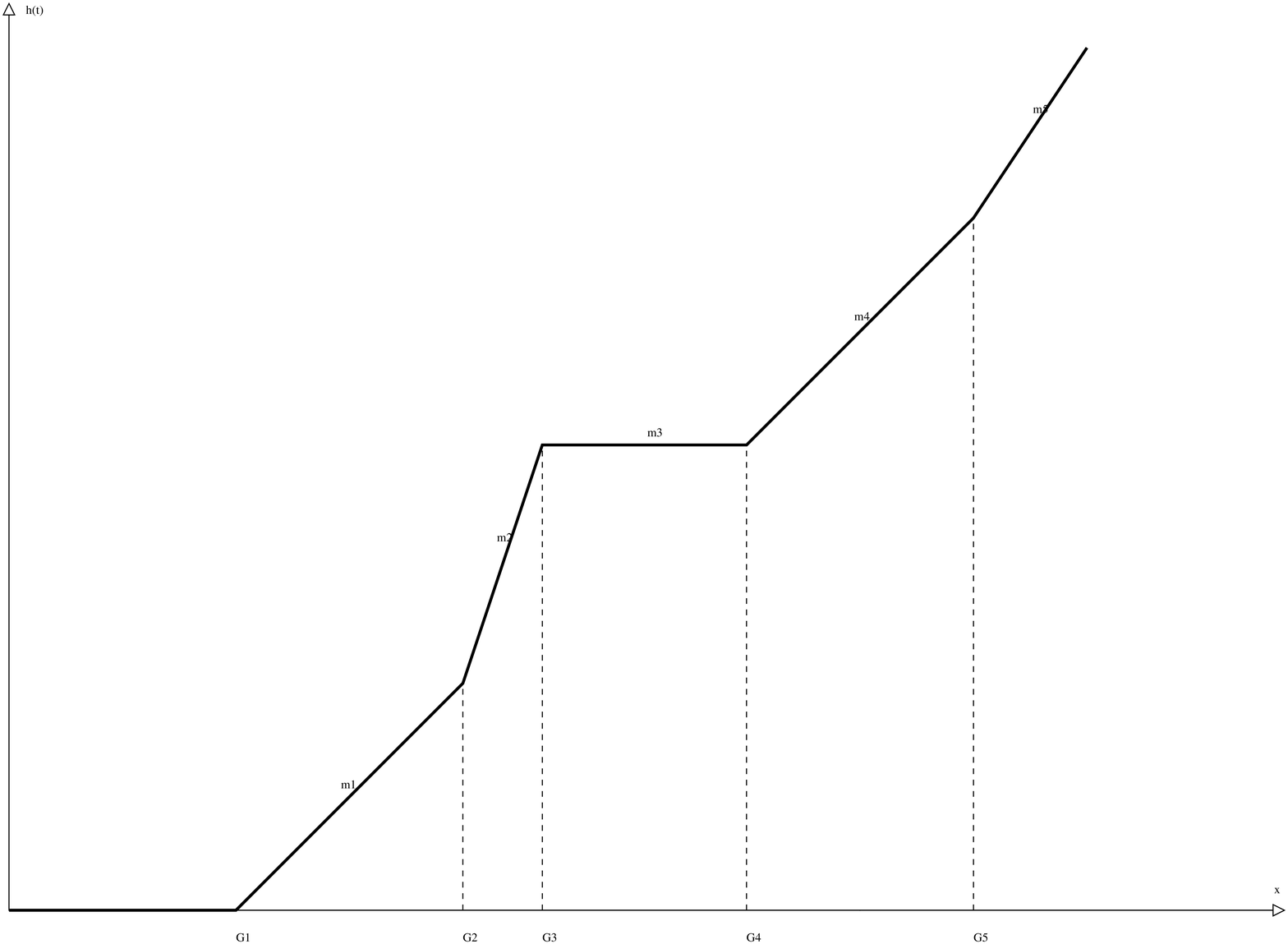}
\caption{Example 8 -- Function $h(t) = f \big( x^{\circ}(t) + t \big) + g \big( x^{\circ}(t) \big)$.}
\label{fig:es8_h}
\end{figure}

\vspace{10pt}

In accordance with lemma~\ref{lem:h(t)} and, in particular, with~\eqref{equ:h_2}, function $h(t)$ is
\begin{equation*}
h(t) = \left\{ \begin{array}{ll}
\vspace{3pt} f(12+t) & \forall \, t \, : \, x^{\circ}(t) = 12\\
\vspace{3pt} t-8 & \forall \, t < 18.8\overline{3} : x^{\circ}(t) \neq \{ 4,12 \} \quad (\Rightarrow 8 < t < 16)\\
\vspace{3pt} t - 9.5 & \forall \, t \geq 18.8\overline{3} : x^{\circ}(t) \neq \{ 4,12 \} \quad (\Rightarrow 26 < t < 34)\\
f(4+t) + 8 & \forall \, t \, : \, x^{\circ}(t) = 4\\
\end{array} \right.
\end{equation*}

\clearpage

\subsection{Example 9}

Consider the following functions $f(x)$ and $g(x)$ (depicted in the same graphic).

\begin{figure}[h]
\centering
\psfrag{f(x)}[cl][Bl][.8][0]{$f(x)$}
\psfrag{g(x)}[cl][Bl][.8][0]{$g(x)$}
\psfrag{x}[bc][Bl][.8][0]{$x$}
\psfrag{X1}[tc][Bl][.7][0]{$4$}
\psfrag{X2}[tc][Bl][.7][0]{$12$}
\psfrag{G1}[tc][Bl][.7][0]{$16$}
\psfrag{G2}[tc][Bl][.7][0]{$24$}
\psfrag{G3}[tc][Bl][.7][0]{$25$}
\psfrag{G4}[tc][Bl][.7][0]{$26$}
\psfrag{G5}[tc][Bl][.7][0]{$27$}
\psfrag{G6}[tc][Bl][.7][0]{$28$}
\psfrag{G7}[tc][Bl][.7][0]{$38$}
\psfrag{n}[cl][Bl][.6][0]{$-1$}
\psfrag{m1}[cr][Bl][.6][0]{$1.5$}
\psfrag{m2}[Bc][Bl][.6][0]{$0$}
\psfrag{m3}[cr][Bl][.6][0]{$3$}
\psfrag{m4}[Bc][Bl][.6][0]{$0$}
\psfrag{m5}[cr][Bl][.6][0]{$1.5$}
\psfrag{m6}[cr][Bl][.6][0]{$0.375$}
\psfrag{m7}[cr][Bl][.6][0]{$1.5$}
\includegraphics[scale=.25]{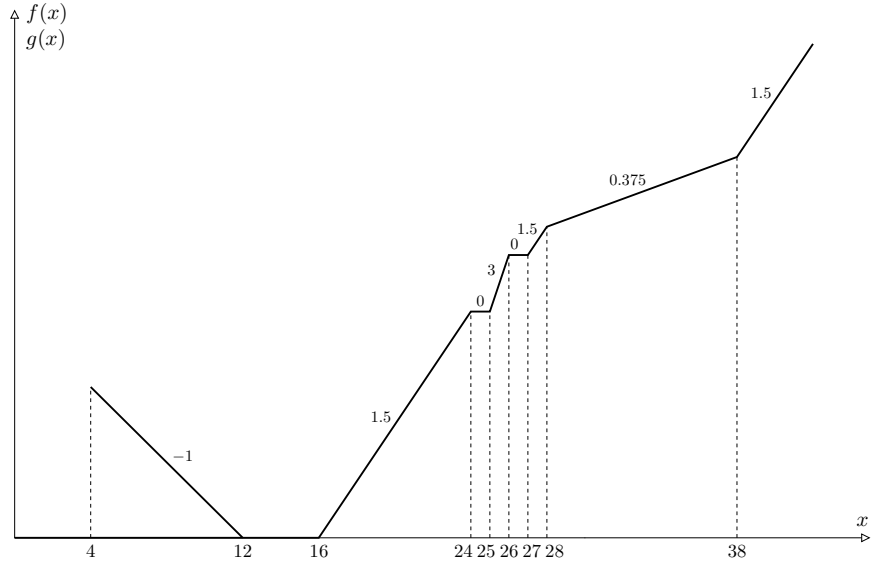}
\caption{Example 9 -- Functions $f(x)$ and $g(x)$.}
\label{fig:es9_fg}
\end{figure}

Algorithm~\ref{alg:tstar} provides $\omega_{1} = 18$, $\omega_{2} = +\infty$,  and $\omega_{3} = 18.4$ (the application of algorithm~\ref{alg:tstar} is reported in the following, for each value of $j$). Then, by applying lemma~\ref{lem:xopt} (taking into account $f(x+t)$, instead of $f(x)$, and $g(x)$) the following function $x^{\circ}(t)$ is obtained.
\begin{equation*}
x^{\circ}(t) = \left\{ \begin{array}{ll}
x_{\mathrm{s}}(t) & t < 18\\
x_{1}(t) & 18 \leq t < 18.4\\
x_{\mathrm{e}}(t) & t \geq 18.4
\end{array} \right.
\end{equation*}
with
\begin{equation*}
x_{\mathrm{s}}(t) = \left\{ \begin{array}{ll}
12 & t < 4\\
-t + 16 & 4 \leq t < 12\\
4 & 12 \leq t < 18
\end{array} \right. \qquad x_{1}(t) = -t + 25 \qquad x_{\mathrm{e}}(t) = \left\{ \begin{array}{ll}
12 & 18.4 \leq t < 26\\
-t + 38 & 26 \leq t < 34\\
4 & t \geq 34
\end{array} \right.
\end{equation*}

Note that, $T = \{ 18 , 18.4 \}$, that is, $t^{\star}_{1} = 18$ and $t^{\star}_{2} = 18.4$, and $Q = 2$. The mapping function provides $l(1) = 1$ and $l(2) = 3$. Moreover, $t^{\star}_{1} > \gamma_{a_{1}} - x_{1} = 12$ (then, $x_{\mathrm{s}}(t)$ has the structure of~\eqref{equ:xs_1}), $t^{\star}_{1} \geq \gamma_{a_{2}} - x_{2} = 13$ and $t^{\star}_{2} \leq \gamma_{a_{2}} - x_{1} = 21$ (then, $x_{1}(t)$ has the structure of~\eqref{equ:xj_4}), and $t^{\star}_{2} < \gamma_{a_{4}} - x_{2} = 26$ (then, $x_{\mathrm{e}}(t)$ has the structure of~\eqref{equ:xe_1}). The graphical representation of $x^{\circ}(t)$ is the following.

\begin{figure}[h]
\centering
\psfrag{xopt(t)}[cl][Bl][.8][0]{$x^{\circ}(t)$}
\psfrag{x}[bc][Bl][.8][0]{$t$}
\psfrag{XS}[bc][Bl][.7][0]{$x_{\mathrm{s}}(t)$}
\psfrag{X1}[cr][Bl][.7][0]{$x_{1}(t)$}
\psfrag{XE}[bc][Bl][.7][0]{$x_{\mathrm{e}}(t)$}
\psfrag{Y1}[cr][Bl][.7][0]{$4$}
\psfrag{Y2}[cr][Bl][.7][0]{$12$}
\psfrag{T1}[tc][Bl][.7][0]{$4$}
\psfrag{T2}[tc][Bl][.7][0]{$12$}
\psfrag{T3}[tr][Bl][.7][0]{$18$}
\psfrag{T4}[tl][Bl][.7][0]{$18.4$}
\psfrag{T5}[tc][Bl][.7][0]{$26$}
\psfrag{T6}[tc][Bl][.7][0]{$34$}
\includegraphics[scale=.25]{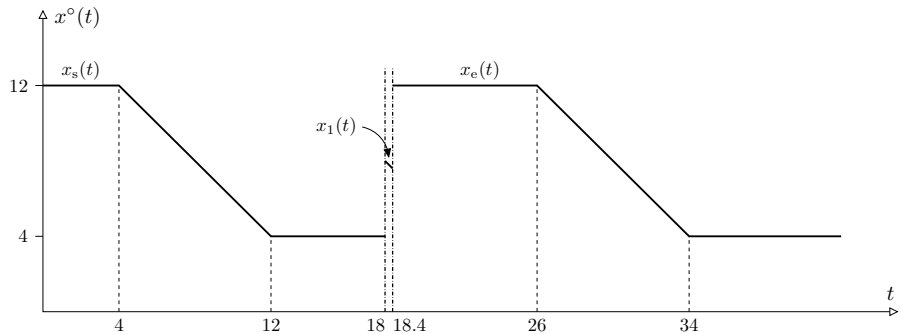}
\caption{Example 9 -- Functions $x^{\circ}(t)$.}
\label{fig:es9_xopt}
\end{figure}

By applying lemma~\ref{lem:h(t)} the following function $h(t) = f \big( x^{\circ}(t) + t \big) + g \big( x^{\circ}(t) \big)$ is obtained.

\begin{figure}[h]
\centering
\psfrag{h(t)}[cl][Bl][.8][0]{$h(t)$}
\psfrag{x}[bc][Bl][.8][0]{$t$}
\psfrag{m1}[cr][Bl][.6][0]{$1$}
\psfrag{m2}[cr][Bl][.6][0]{$1.5$}
\psfrag{m3}[cr][Bl][.6][0]{$1$}
\psfrag{m4}[cr][Bl][.6][0]{$0.375$}
\psfrag{m5}[cr][Bl][.6][0]{$1$}
\psfrag{m6}[cr][Bl][.6][0]{$1.5$}
\psfrag{G1}[tc][Bl][.7][0]{$4$}
\psfrag{G2}[tc][Bl][.7][0]{$12$}
\psfrag{G3}[tr][Bl][.7][0]{$18$}
\psfrag{G4}[tl][Bl][.7][0]{$18.4$}
\psfrag{G5}[tc][Bl][.7][0]{$26$}
\psfrag{G6}[tc][Bl][.7][0]{$34$}
\includegraphics[scale=.25]{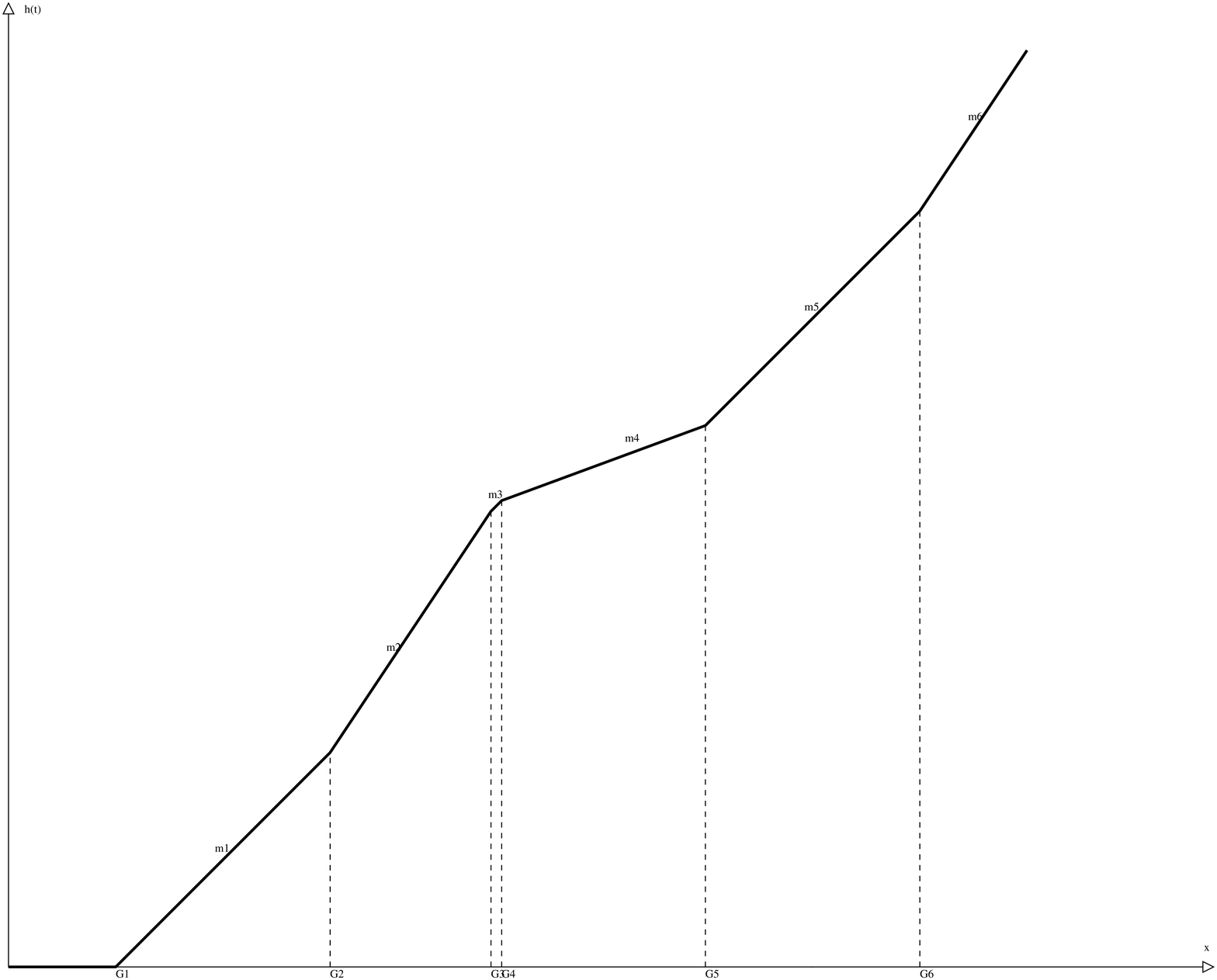}
\caption{Example 9 -- Function $h(t) = f \big( x^{\circ}(t) + t \big) + g \big( x^{\circ}(t) \big)$.}
\label{fig:es9_h}
\end{figure}

In accordance with lemma~\ref{lem:h(t)} and, in particular, with~\eqref{equ:h_3}, function $h(t)$ is
\begin{equation*}
h(t) = \left\{ \begin{array}{ll}
\vspace{3pt} f(12+t) & \forall \, t \, : \, x^{\circ}(t) = 12\\
\vspace{3pt} t-4 & \forall \, t < 18 : x^{\circ}(t) \neq \{ 4,12 \} \quad (\Rightarrow 4 < t < 12)\\
\vspace{3pt} t-1 & \forall \, t \in [ 18 , 18.4 ) : x^{\circ}(t) \neq \{ 4,12 \} \quad (\Rightarrow 18 < t < 18.4)\\
\vspace{3pt} t - 5.75 & \forall \, t \geq 18.4 : x^{\circ}(t) \neq \{ 4,12 \} \quad (\Rightarrow 26 < t < 34)\\
f(4+t) + 8 & \forall \, t \, : \, x^{\circ}(t) = 4\\
\end{array} \right.
\end{equation*}

%
%%========================================
%%========================================
%%          ESEMPIO SCHEDULING
%%========================================
%%========================================
%
\newpage

\section{Application to the single machine scheduling} \label{sec:app1}

Consider a single machine scheduling problem in which 1 job of class $P_{1}$ and 2 jobs of class $P_{2}$ must be executed. The due dates, the marginal tardiness costs of jobs, the processing time bounds and the marginal deviation costs of jobs are the:

\begin{center}
\begin{tabular}{|ll|}
\hline
$\alpha_{1,1} = 0.5$ & $dd_{1,1} = 10$\\
\hline
\multicolumn{2}{|c|}{$\beta_{1} = 1$}\\
$pt^{\mathrm{low}}_{1} = 1$ & $pt^{\mathrm{nom}}_{1} = 4$\\
\hline
\end{tabular}
\hspace{.5cm}
\begin{tabular}{|ll|}
\hline
$\alpha_{2,1} = 0.25$ & $dd_{2,1} = 12$\\
$\alpha_{2,2} = 0.75$ & $dd_{2,2} = 20$\\
\hline
\multicolumn{2}{|c|}{$\beta_{2} = 1$}\\
$pt^{\mathrm{low}}_{2} = 1$ & $pt^{\mathrm{nom}}_{2} = 2$\\
\hline
\end{tabular}
\end{center}
No setup is required between the execution of jobs of different classes. The evolution of the system state can be represented by the following diagram.
\begin{figure}[h]
\centering
\psfrag{S1}[bc][Bl][.8][0]{$S2 = \begin{bmatrix} 1 \\ 0 \\ t_{1}\end{bmatrix}$}
\psfrag{S2}[bc][Bl][.8][0]{$S4 = \begin{bmatrix} 1 \\ 1 \\ t_{2}\end{bmatrix}$}
\psfrag{S3}[bc][Bl][.8][0]{$S6 = \begin{bmatrix} 1 \\ 2 \\ t_{3}\end{bmatrix}$}
\psfrag{S4}[tc][Bl][.8][0]{$S1 = \begin{bmatrix} 0 \\ 0 \\ t_{0}\end{bmatrix}$}
\psfrag{S5}[tc][Bl][.8][0]{$S3 = \begin{bmatrix} 0 \\ 1 \\ t_{1}\end{bmatrix}$}
\psfrag{S6}[tc][Bl][.8][0]{$S5 = \begin{bmatrix} 0 \\ 2 \\ t_{2}\end{bmatrix}$}
\includegraphics[scale=.7]{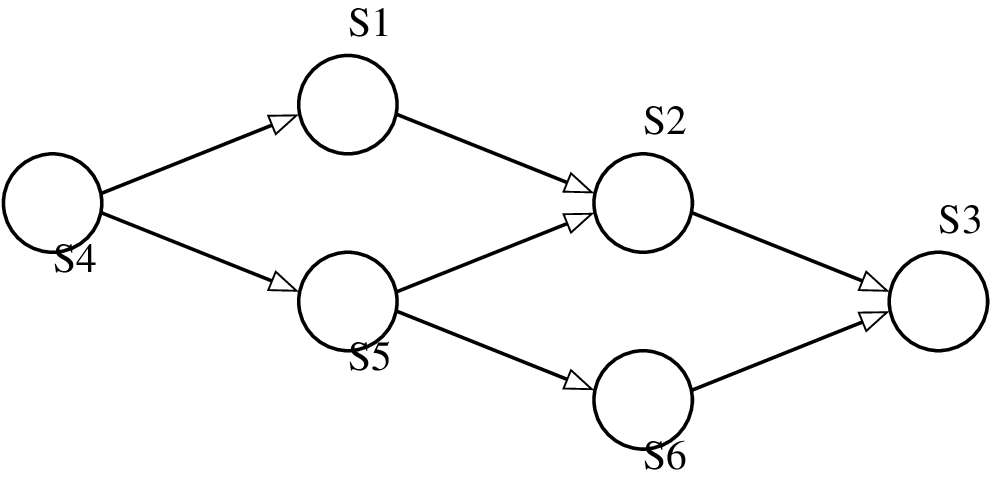}
\caption{State diagram in the case of two classes of jobs, where $N_{1} = 1$ and $N_{2} = 2$.}
\label{fig:esS_statediagram}
\end{figure}

The application of dynamic programming, in conjunction with the new lemmas, provides the following optimal control strategies.

\vspace{24pt}
{\bf Stage 2 -- State $\boldsymbol{[ 1 \  1 \  t_{2}]^{T}}$}

\vspace{6pt}
\hspace{1cm}\includegraphics[scale=.4]{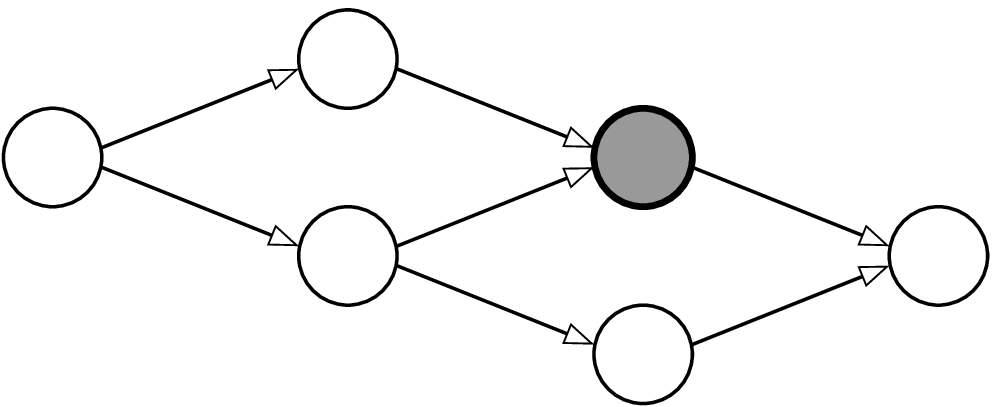}

\vspace{12pt}
In state $[1 \; 1 \; t_{2}]^{T}$ the unique job of class $P_{1}$ has been completed; then the decision about the class of the next job to be executed is mandatory. The cost function to be minimized in this state, with respect to the (continuos) decision variable $\tau$ only (which corresponds to the processing time $pt_{2,2}$), is
\begin{equation*}
\alpha_{2,2} \, \max \{ t_{2} + \tau - dd_{2,2} \, , \, 0 \} + \beta_{2} \, (pt^{\mathrm{nom}}_{2} - \tau) + J^{\circ}_{1,2}(t_{3})
\end{equation*}
that can be written as $f (pt_{2,2} + t_{2}) + g (pt_{2,2})$ being
\begin{equation*}
f (pt_{2,2} + t_{2}) = 0.75 \cdot \max \{ pt_{2,2} + t_{2} - 20 \, , \, 0 \}
\end{equation*}
\begin{equation*}
g (pt_{2,2}) = \left\{ \begin{array}{ll}
2 - pt_{2,2} & pt_{2,2} \in [ 1 , 2 )\\
0 & pt_{2,2} \notin [ 1 , 2 )
\end{array} \right.
\end{equation*}
the two functions illustrated in figure~\ref{fig:esS_salpha_1_1_2}.

\newpage

\begin{figure}[h]
\centering
\psfrag{f(x)}[cl][Bl][.8][0]{$f (pt_{2,2} + t_{2})$}
\psfrag{g(x)}[cl][Bl][.8][0]{$g (pt_{2,2})$}
\psfrag{x}[bc][Bl][.8][0]{$pt_{2,2}$}
\psfrag{X1}[tc][Bl][.7][0]{$1$}
\psfrag{X2}[tc][Bl][.7][0]{$2$}
\psfrag{G1}[tc][Bl][.7][0]{$20-t_{2}$}
\psfrag{n}[cl][Bl][.6][0]{$-1$}
\psfrag{m1}[br][Bl][.6][0]{$0.75$}
\includegraphics[scale=.25]{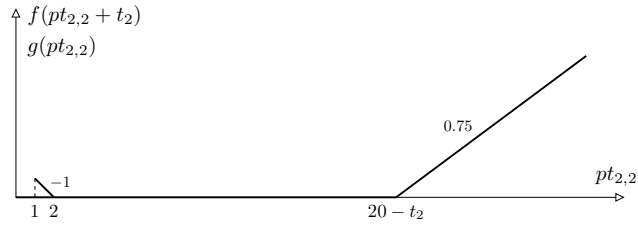}
\caption{Functions $f (pt_{2,2} + t_{2})$ and $g (pt_{2,2})$ in state $[ 1 \  1 \  t_{2}]^{T}$.}
\label{fig:esS_salpha_1_1_2}
\end{figure}

It is possible to apply lemma~\ref{lem:xopt} (note that $f (pt_{2,2} + t_{2})$ follows definition~\ref{def:f(x+t)} and $g (pt_{2,2})$ follows definition~\ref{def:g(x)}), which provides the optimal processing time
\begin{equation*}
pt^{\circ}_{2,2}(t_{2}) = \arg \min_{\substack{pt_{2,2}\\1 \leq pt_{2,2} \leq 2}} \big\{ f (pt_{2,2} + t_{2}) + g (pt_{2,2}) \big\} = x_{\mathrm{e}}(t_{2})
\end{equation*}
illustrated in figure~\ref{fig:esS_tau_1_1}, being $x_{\mathrm{e}}(t_{2})$ the function
\begin{equation*}
x_{\mathrm{e}}(t_{2}) = 2
\end{equation*}
$pt^{\circ}_{2,2}(t_{2})$ and $x_{\mathrm{e}}(t_{2})$ are in accordance with~\eqref{equ:xopt_1} and~\eqref{equ:xe_3}, respectively. Note that, in this case, $A = \emptyset$, $B = \emptyset$, $\lvert A \rvert = \lvert B \rvert = 0$, and then there is no need of executing algorithm~\ref{alg:tstar}. Taking into account the mandatory decision about the class of the next job to be executed; the optimal control strategies for this state are
\begin{equation*}
\delta_{1}^{\circ} (1,1, t_{2}) = 0 \quad \forall \, t_{2} \qquad \delta_{2}^{\circ} (1,1, t_{2}) = 1 \quad \forall \, t_{2}
\end{equation*}
\begin{equation*}
\tau^{\circ} (1,1, t_{2}) = pt^{\circ}_{2,2}(t_{2}) = 2 \quad \forall \, t_{2}
\end{equation*}

\begin{figure}[h]
\centering
\psfrag{f(x)}[cl][Bl][.8][0]{$pt^{\circ}_{2,2}(t_{2}) \equiv \tau^{\circ} (1,1, t_{2})$}
\psfrag{x}[bc][Bl][.8][0]{$t_{2}$}
\psfrag{Y1}[cr][Bl][.7][0]{$1$}
\psfrag{Y2}[cr][Bl][.7][0]{$2$}
\includegraphics[scale=.25]{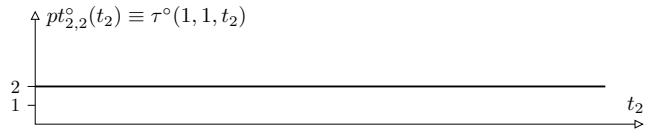}
\caption{Optimal control strategy $\tau^{\circ} (1,1, t_{2})$ (service time) in state $[ 1 \  1 \  t_{2}]^{T}$.}
\label{fig:esS_tau_1_1}
\end{figure}

The optimal cost-to-go
\begin{equation*}
J^{\circ}_{1,1} (t_{2}) = f \big( pt^{\circ}_{2,2}(t_{2}) + t_{2} \big) + g \big( pt^{\circ}_{2,2}(t_{2}) \big)
\end{equation*}
illustrated in figure~\ref{fig:esS_J_1_1}, is provided by lemma~\ref{lem:h(t)}.

\begin{figure}[h]
\centering
\psfrag{f(x)}[cl][Bl][.8][0]{$J^{\circ}_{1,1} (t_{2})$}
\psfrag{x}[bc][Bl][.8][0]{$t_{2}$}
\psfrag{G1}[tc][Bl][.7][0]{$18$}
\psfrag{m1}[br][Bl][.6][0]{$0.75$}
\includegraphics[scale=.25]{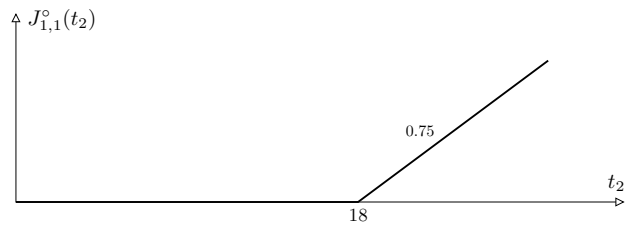}
\caption{Optimal cost-to-go $J^{\circ}_{1,1} (t_{2})$ in state $[ 1 \  1 \  t_{2}]^{T}$.}
\label{fig:esS_J_1_1}
\end{figure}

{\bf Stage 2 -- State $\boldsymbol{[ 0 \  2 \  t_{2}]}^{T}$}

\hspace{1cm}\includegraphics[scale=.4]{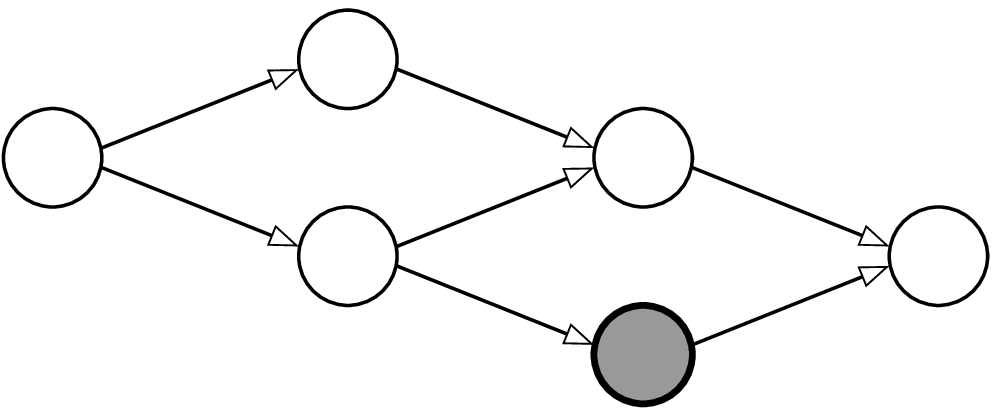}

In state $[0 \; 2 \; t_{2}]^{T}$ all jobs of class $P_{2}$ have been completed; then the decision about the class of the next job to be executed is mandatory. The cost function to be minimized in this state, with respect to the (continuos) decision variable $\tau$ only (which corresponds to the processing time $pt_{1,1}$), is
\begin{equation*}
\alpha_{1,1} \, \max \{ t_{2} + \tau - dd_{1,1} \, , \, 0 \} + \beta_{1} \, (pt^{\mathrm{nom}}_{1} - \tau) + J^{\circ}_{1,2}(t_{3})
\end{equation*}
that can be written as $f (pt_{1,1} + t_{2}) + g (pt_{1,1})$ being
\begin{equation*}
f (pt_{1,1} + t_{2}) = 0.5 \cdot \max \{ pt_{1,1} - 10 \, , \, 0 \}
\end{equation*}
\begin{equation*}
g (pt_{1,1}) = \left\{ \begin{array}{ll}
4 - pt_{1,1} & pt_{1,1} \in [ 1 , 4 )\\
0 & pt_{1,1} \notin [ 1 , 4 )
\end{array} \right.
\end{equation*}
the two functions illustrated in figure~\ref{fig:esS_salpha_0_2_1}.

\begin{figure}[h]
\centering
\psfrag{f(x)}[cl][Bl][.8][0]{$f (pt_{1,1} + t_{2})$}
\psfrag{g(x)}[cl][Bl][.8][0]{$g (pt_{1,1})$}
\psfrag{x}[bc][Bl][.8][0]{$pt_{1,1}$}
\psfrag{X1}[tc][Bl][.7][0]{$1$}
\psfrag{X2}[tc][Bl][.7][0]{$4$}
\psfrag{G1}[tc][Bl][.7][0]{$10-t_{2}$}
\psfrag{n}[cl][Bl][.6][0]{$-1$}
\psfrag{m1}[br][Bl][.6][0]{$0.5$}
\includegraphics[scale=.25]{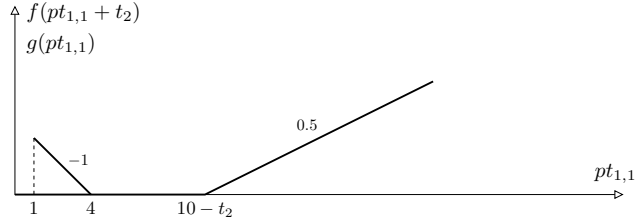}
\caption{Functions $f (pt_{1,1} + t_{2})$ and $g (pt_{1,1})$ in state $[ 0 \  2 \  t_{2}]^{T}$.}
\label{fig:esS_salpha_0_2_1}
\end{figure}

It is possible to apply lemma~\ref{lem:xopt} (note that $f (pt_{1,1}, t_{2})$ follows definition~\ref{def:f(x+t)} and $g (pt_{1,1})$ follows definition~\ref{def:g(x)}), which provides the optimal processing time
\begin{equation*}
pt^{\circ}_{1,1}(t_{2}) = \arg \min_{\substack{pt_{1,1}\\1 \leq pt_{1,1} \leq 4}} \big\{ f (pt_{1,1} + t_{2}) + g (pt_{1,1}) \big\} = x_{\mathrm{e}}(t_{2})
\end{equation*}
illustrated in figure~\ref{fig:esS_tau_0_2}, being $x_{\mathrm{e}}(t_{2})$ the function
\begin{equation*}
x_{\mathrm{e}}(t_{2}) = 4
\end{equation*}
$pt^{\circ}_{1,1}(t_{2})$ and $x_{\mathrm{e}}(t_{2})$ are in accordance with~\eqref{equ:xopt_1} and~\eqref{equ:xe_3}, respectively. Note that, in this case, $A = \emptyset$, $B = \emptyset$, $\lvert A \rvert = \lvert B \rvert = 0$, and then there is no need of executing algorithm~\ref{alg:tstar}. Taking into account the mandatory decision about the class of the next job to be executed; the optimal control strategies for this state are
\begin{equation*}
\delta_{1}^{\circ} (0,2, t_{2}) = 1 \quad \forall \, t_{2} \qquad \delta_{2}^{\circ} (0,2, t_{2}) = 0 \quad \forall \, t_{2}
\end{equation*}
\begin{equation*}
\tau^{\circ} (0,2, t_{2}) = pt^{\circ}_{1,1}(t_{2}) = 4 \quad \forall \, t_{2}
\end{equation*}

\begin{figure}[h]
\centering
\psfrag{f(x)}[cl][Bl][.8][0]{$pt^{\circ}_{1,1}(t_{2}) \equiv \tau^{\circ} (0,2, t_{2})$}
\psfrag{x}[bc][Bl][.8][0]{$t_{2}$}
\psfrag{Y1}[cr][Bl][.7][0]{$1$}
\psfrag{Y2}[cr][Bl][.7][0]{$4$}
\includegraphics[scale=.25]{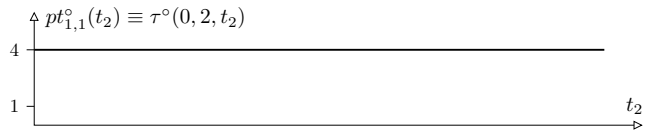}
\caption{Optimal control strategy $\tau^{\circ} (0,2, t_{2})$ (service time) in state $[ 0 \  2 \  t_{2}]^{T}$.}
\label{fig:esS_tau_0_2}
\end{figure}

The optimal cost-to-go
\begin{equation*}
J^{\circ}_{0,2} (t_{2}) = f \big( pt^{\circ}_{1,1}(t_{2}) + t_{2} \big) + g \big( pt^{\circ}_{1,1}(t_{2}) \big)
\end{equation*}
illustrated in figure~\ref{fig:esS_J_0_2}, is provided by lemma~\ref{lem:h(t)}.

\begin{figure}[h]
\centering
\psfrag{f(x)}[cl][Bl][.8][0]{$J^{\circ}_{0,2} (t_{2})$}
\psfrag{x}[bc][Bl][.8][0]{$t_{2}$}
\psfrag{G1}[tc][Bl][.7][0]{$6$}
\psfrag{m1}[br][Bl][.6][0]{$0.5$}
\includegraphics[scale=.25]{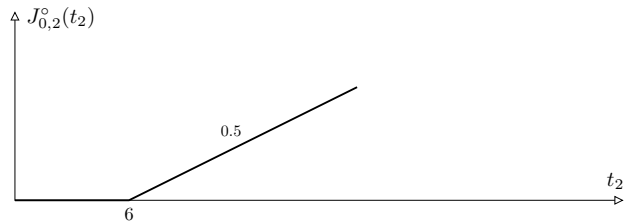}
\caption{Optimal cost-to-go $J^{\circ}_{0,2} (t_{2})$ in state $[ 0 \  2 \  t_{2}]^{T}$.}
\label{fig:esS_J_0_2}
\end{figure}

{\bf Stage 1 -- State $\boldsymbol{[ 1 \  0 \  t_{1}]}^{T}$}

\hspace{1cm}\includegraphics[scale=.4]{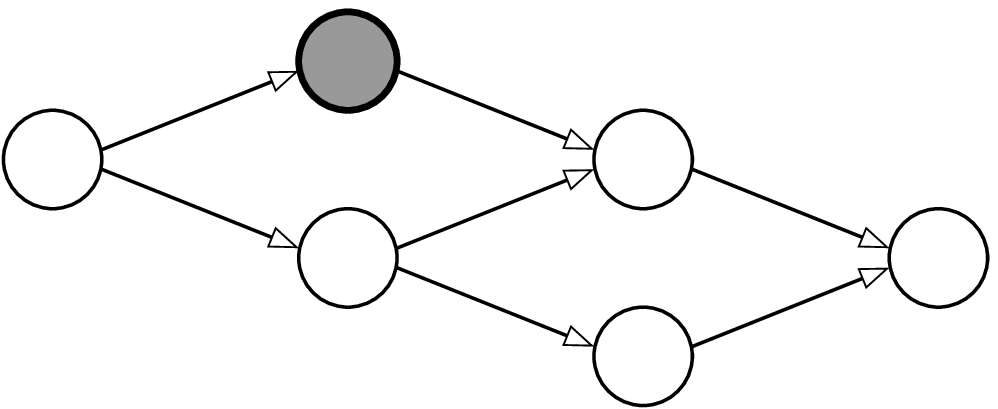}

In state $[1 \; 0 \; t_{1}]^{T}$ the unique job of class $P_{1}$ has been completed; then the decision about the class of the next job to be executed is mandatory. The cost function to be minimized in this state, with respect to the (continuos) decision variable $\tau$ only (which corresponds to the processing time $pt_{2,1}$), is
\begin{equation*}
\alpha_{2,1} \, \max \{ t_{1} + \tau - dd_{2,1} \, , \, 0 \} + (pt^{\mathrm{nom}}_{2} - \tau) +  J^{\circ}_{1,1} (t_{2})
\end{equation*}
that can be written as $f (pt_{2,1} + t_{1}) + g (pt_{2,1})$ being
\begin{equation*}
f (pt_{2,1} + t_{1}) = 0.25 \cdot \max \{ pt_{2,1} + t_{1} - 12 \, , \, 0 \} + J^{\circ}_{1,1} (pt_{2,1} + t_{1})
\end{equation*}
\begin{equation*}
g (pt_{2,1}) = \left\{ \begin{array}{ll}
2 - pt_{2,1} & pt_{2,1} \in [ 1 , 2 )\\
0 & pt_{2,1} \notin [ 1 , 2 )
\end{array} \right.
\end{equation*}
the two functions illustrated in figure~\ref{fig:esS_salpha_1_0_2}.

\begin{figure}[h]
\centering
\psfrag{f(x)}[cl][Bl][.8][0]{$f (pt_{2,1} + t_{1})$}
\psfrag{g(x)}[cl][Bl][.8][0]{$g (pt_{2,1})$}
\psfrag{x}[bc][Bl][.8][0]{$pt_{2,1}$}
\psfrag{X1}[tc][Bl][.7][0]{$1$}
\psfrag{X2}[tc][Bl][.7][0]{$2$}
\psfrag{G1}[tc][Bl][.7][0]{$12-t_{1}$}
\psfrag{G2}[tc][Bl][.7][0]{$18-t_{1}$}
\psfrag{n}[cl][Bl][.6][0]{$-1$}
\psfrag{m1}[bc][Bl][.6][0]{$0.25$}
\psfrag{m2}[cr][Bl][.6][0]{$1$}
\includegraphics[scale=.25]{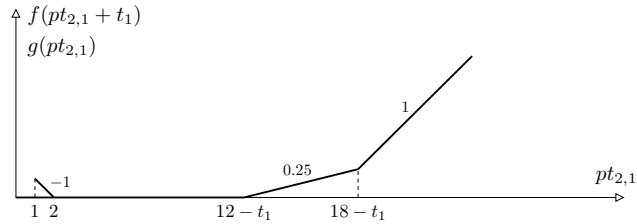}
\caption{Functions $f (pt_{2,1} + t_{1})$ and $g_{2,1} (pt_{2,1})$ in state $[ 1 \  0 \  t_{1}]^{T}$.}
\label{fig:esS_salpha_1_0_2}
\end{figure}

\begin{figure}[h]
\centering
\psfrag{f(x)}[cl][Bl][.8][0]{$pt^{\circ}_{2,1}(t_{1}) \equiv \tau^{\circ} (1,0, t_{1})$}
\psfrag{x}[bc][Bl][.8][0]{$t_{1}$}
\psfrag{Y1}[cr][Bl][.7][0]{$1$}
\psfrag{Y2}[cr][Bl][.7][0]{$2$}
\psfrag{G1}[tc][Bl][.7][0]{$16$}
\psfrag{G2}[tc][Bl][.7][0]{$17$}
\includegraphics[scale=.25]{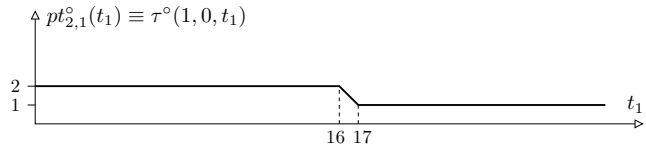}
\caption{Optimal control strategy $\tau^{\circ} (1,0, t_{1})$ (service time) in state $[ 1 \  0 \  t_{1}]^{T}$.}
\label{fig:esS_tau_1_0}
\end{figure}

It is possible to apply lemma~\ref{lem:xopt} (note that $f (pt_{2,1} + t_{1})$ follows definition~\ref{def:f(x+t)} and $g (pt_{2,1})$ follows definition~\ref{def:g(x)}), which provides the optimal processing time
\begin{equation*}
pt^{\circ}_{2,1}(t_{1}) = \arg \min_{\substack{pt_{2,1}\\1 \leq pt_{2,1} \leq 2}} \big\{ f (pt_{2,1} + t_{1}) + g (pt_{2,1}) \big\} = x_{\mathrm{e}}(t_{1})
\end{equation*}
illustrated in figure~\ref{fig:esS_tau_1_0}, being $x_{\mathrm{e}}(t_{1})$ the function
\begin{equation*}
x_{\mathrm{e}}(t_{1}) = \left\{ \begin{array}{ll}
2 &  t_{1} < 16\\
-t_{1} + 18 & 16 \leq t_{1} < 17\\
1 & t_{1} \geq 17
\end{array} \right.
\end{equation*}
$pt^{\circ}_{2,1}(t_{1})$ and $x_{\mathrm{e}}(t_{1})$ are in accordance with~\eqref{equ:xopt_1} and~\eqref{equ:xe_1}, respectively. Note that, in this case, $A = \{ 2 \}$, $\lvert A \rvert = 1$, $\gamma_{a_{1}} = 18$; moreover, since $B = \emptyset$ and $\lvert B \rvert = 0$, there is no need of executing algorithm~\ref{alg:tstar}. It is worth again remarking that, in the current state, the decision about the class of the next job to be executed is mandatory, since the unique job of class $P_{1}$ has been completed. Then,
\begin{equation*}
\delta_{1}^{\circ} (1,0, t_{1}) = 0 \quad \forall \, t_{1} \qquad \delta_{2}^{\circ} (1,0, t_{1}) = 1 \quad \forall \, t_{1}
\end{equation*}
\begin{equation*}
\tau^{\circ} (1,0, t_{1}) = pt^{\circ}_{2,1}(t_{1}) = \left\{ \begin{array}{ll}
2 &  t_{1} < 16\\
-t_{1} + 18 & 16 \leq t_{1} < 17\\
1 & t_{1} \geq 17
\end{array} \right.
\end{equation*}

The optimal cost-to-go
\begin{equation*}
J^{\circ}_{1,0} (t_{1}) = f \big( pt^{\circ}_{2,1}(t_{1}) + t_{1} \big) + g \big( pt^{\circ}_{2,1}(t_{1}) \big)
\end{equation*}
illustrated in figure~\ref{fig:esS_J_1_0}, is provided by lemma~\ref{lem:h(t)}.

\begin{figure}[h]
\centering
\psfrag{f(x)}[cl][Bl][.8][0]{$J^{\circ}_{1,0} (t_{1})$}
\psfrag{x}[bc][Bl][.8][0]{$t_{1}$}
\psfrag{G1}[tc][Bl][.7][0]{$10$}
\psfrag{G2}[tc][Bl][.7][0]{$16$}
\psfrag{G3}[tc][Bl][.7][0]{$17$}
\psfrag{m1}[bc][Bl][.6][0]{$0.25$}
\psfrag{m2}[cr][Bl][.6][0]{$1$}
\includegraphics[scale=.25]{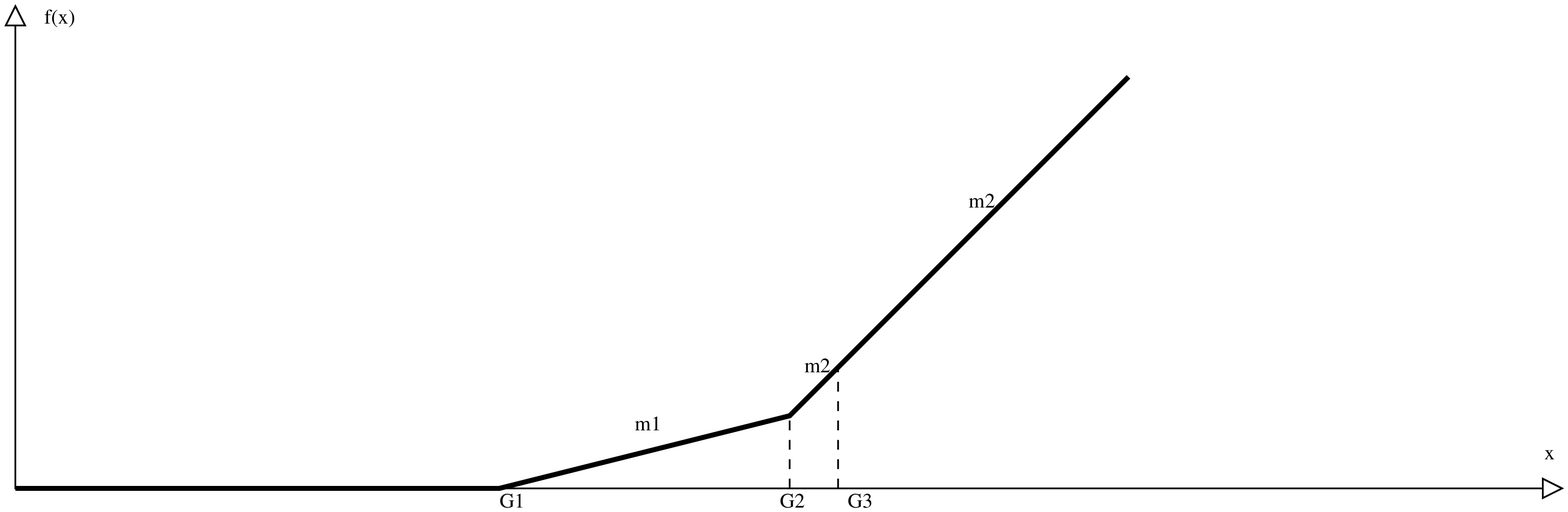}
\caption{Optimal cost-to-go $J^{\circ}_{1,0} (t_{1})$ in state $[ 1 \  0 \  t_{1}]^{T}$.}
\label{fig:esS_J_1_0}
\end{figure}

{\bf Stage 1 -- State $\boldsymbol{[ 0 \  1 \  t_{1}]}^{T}$}

\hspace{1cm}\includegraphics[scale=.4]{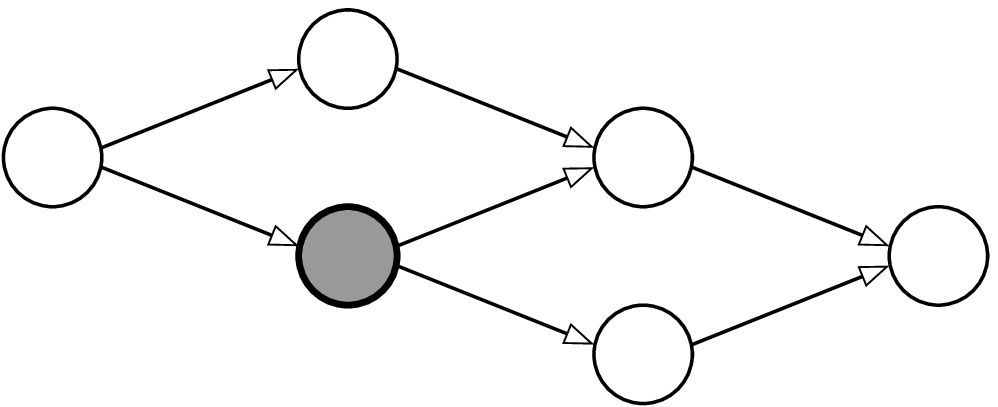}

In state $[0 \; 1 \; t_{1}]^{T}$, the cost function to be minimized, with respect to the (continuos) decision variable $\tau$ and to the (binary) decision variables $\delta_{1}$ and $\delta_{2}$ is
\begin{equation*}
\begin{split}
&\delta_{1} \big[ \alpha_{1,1} \, \max \{ t_{1} + \tau - dd_{1,1} \, , \, 0 \} + \beta_{1} \, ( pt^{\mathrm{nom}}_{1} - \tau ) + J^{\circ}_{1,1} (t_{2}) \big] +\\
&+ \delta_{2} \big[ \alpha_{2,2} \, \max \{ t_{1} + \tau - dd_{2,2} \, , \, 0 \} + \beta_{2} \, ( pt^{\mathrm{nom}}_{2} - \tau ) + J^{\circ}_{0,2} (t_{2}) \big]
\end{split}
\end{equation*}

{\it Case i)} in which it is assumed $\delta_{1} = 1$ (and $\delta_{2} = 0$).

\hspace{1cm}\includegraphics[scale=.4]{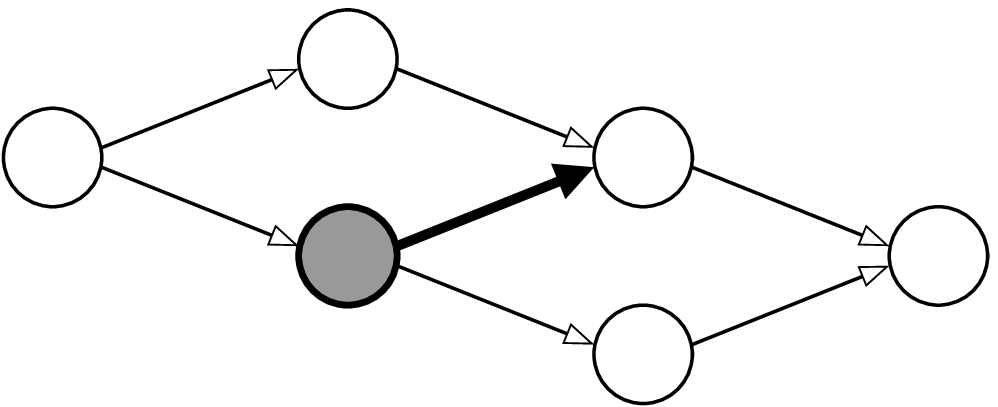}

In this case, it is necessary to minimize, with respect to the (continuos) decision variable $\tau$ which corresponds to the processing time $pt_{1,1}$, the following function
\begin{equation*}
\alpha_{1,1} \, \max \{ t_{1} + \tau - dd_{1,1} \, , \, 0 \} + \beta_{1} \, ( pt^{\mathrm{nom}}_{1} - \tau ) + J^{\circ}_{1,1} (t_{2})
\end{equation*}
that can be written as $f (pt_{1,1} + t_{1}) + g (pt_{1,1})$ being
\begin{equation*}
f (pt_{1,1} + t_{1}) = 0.5 \cdot \max \{ pt_{1,1} + t_{1} - 10 \, , \, 0 \} + J^{\circ}_{1,1} (pt_{1,1} + t_{1})
\end{equation*}
\begin{equation*}
g (pt_{1,1}) = \left\{ \begin{array}{ll}
4 - pt_{1,1} & pt_{1,1} \in [ 1 , 4 )\\
0 & pt_{1,1} \notin [ 1 , 4 )
\end{array} \right.
\end{equation*}
the two functions illustrated in figure~\ref{fig:esS_salpha_0_1_1}.

\begin{figure}[h]
\centering
\psfrag{f(x)}[cl][Bl][.8][0]{$f (pt_{1,1} + t_{1})$}
\psfrag{g(x)}[cl][Bl][.8][0]{$g (pt_{1,1})$}
\psfrag{x}[bc][Bl][.8][0]{$pt_{1,1}$}
\psfrag{X1}[tc][Bl][.7][0]{$1$}
\psfrag{X2}[tc][Bl][.7][0]{$4$}
\psfrag{G1}[tc][Bl][.7][0]{$10-t_{1}$}
\psfrag{G2}[tc][Bl][.7][0]{$18-t_{1}$}
\psfrag{n}[cl][Bl][.6][0]{$-1$}
\psfrag{m1}[bc][Bl][.6][0]{$0.5$}
\psfrag{m2}[cr][Bl][.6][0]{$1.25$}
\includegraphics[scale=.25]{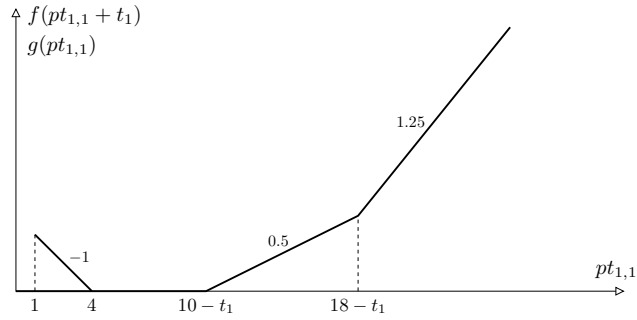}
\caption{Functions $f (pt_{1,1} + t_{1})$ and $g (pt_{1,1})$ in state $[ 0 \  1 \  t_{1}]^{T}$.}
\label{fig:esS_salpha_0_1_1}
\end{figure}

It is possible to apply lemma~\ref{lem:xopt} (note that $f (pt_{1,1} + t_{1})$ follows definition~\ref{def:f(x+t)} and $g (pt_{1,1})$ follows definition~\ref{def:g(x)}), which provides the function
\begin{equation*}
pt^{\circ}_{1,1}(t_{1}) = \arg \min_{\substack{pt_{1,1}\\1 \leq pt_{1,1} \leq 4}} \big\{ f (pt_{1,1} + t_{1}) + g (pt_{1,1}) \big\} = x_{\mathrm{e}}(t_{1})
\end{equation*}
illustrated in figure~\ref{fig:esS_tau_0_1_1}, being $x_{\mathrm{e}}(t_{1})$ the function
\begin{equation*}
x_{\mathrm{e}}(t_{1}) = \left\{ \begin{array}{ll}
4 &  t_{1} < 14\\
-t_{1} + 18 & 14 \leq t_{1} < 17\\
1 & t_{1} \geq 17
\end{array} \right.
\end{equation*}
$pt^{\circ}_{1,1}(t_{1})$ and $x_{\mathrm{e}}(t_{1})$ are in accordance with~\eqref{equ:xopt_1} and~\eqref{equ:xe_1}, respectively. Note that, in this case, $A = \{ 2 \}$, $\lvert A \rvert = 1$, $\gamma_{a_{1}} = 18$; moreover, since $B = \emptyset$ and $\lvert B \rvert = 0$, there is no need of executing algorithm~\ref{alg:tstar}.

\begin{figure}[h]
\centering
\psfrag{f(x)}[cl][Bl][.8][0]{$pt^{\circ}_{1,1}(t_{1})$}
\psfrag{x}[bc][Bl][.8][0]{$t_{1}$}
\psfrag{Y1}[cr][Bl][.7][0]{$1$}
\psfrag{Y2}[cr][Bl][.7][0]{$4$}
\psfrag{G1}[tc][Bl][.7][0]{$14$}
\psfrag{G2}[tc][Bl][.7][0]{$17$}
\includegraphics[scale=.25]{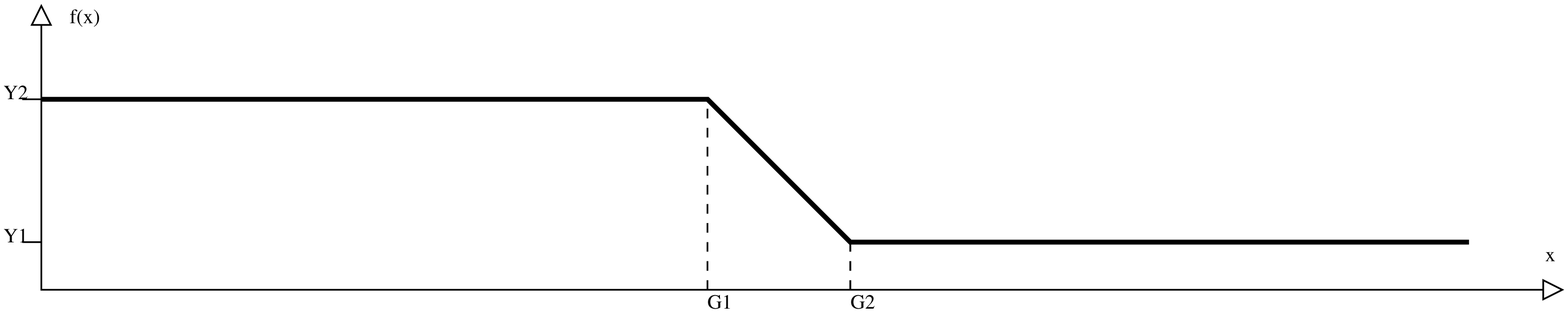}
\caption{Function $pt^{\circ}_{1,1}(t_{1})$.}
\label{fig:esS_tau_0_1_1}
\end{figure}

The conditioned cost-to-go
\begin{equation*}
J^{\circ}_{0,1} (t_{1} \mid \delta_{1} = 1) = f \big( pt^{\circ}_{1,1}(t_{1}) + t_{1} \big) + g \big( pt^{\circ}_{1,1}(t_{1}) \big)
\end{equation*}
illustrated in figure~\ref{fig:esS_J_0_1_1}, is provided by lemma~\ref{lem:h(t)}.

\begin{figure}[h]
\centering
\psfrag{f(x)}[cl][Bl][.8][0]{$J^{\circ}_{0,1} (t_{1} \mid \delta_{1} = 1)$}
\psfrag{x}[bc][Bl][.8][0]{$t_{1}$}
\psfrag{G1}[tc][Bl][.7][0]{$6$}
\psfrag{G2}[tc][Bl][.7][0]{$14$}
\psfrag{G3}[tc][Bl][.7][0]{$17$}
\psfrag{m1}[bc][Bl][.6][0]{$0.5$}
\psfrag{m2}[cr][Bl][.6][0]{$1$}
\psfrag{m3}[cr][Bl][.6][0]{$1.25$}
\includegraphics[scale=.25]{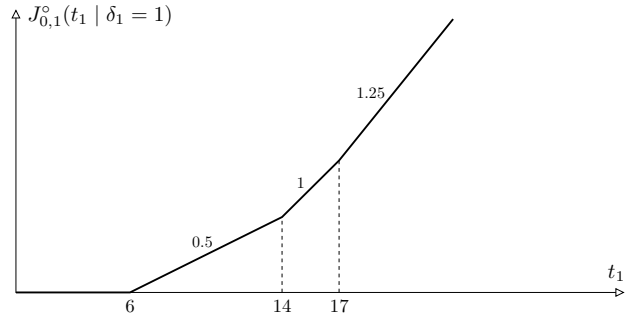}
\caption{Conditioned cost-to-go $J^{\circ}_{0,1} (t_{1} \mid \delta_{1} = 1)$.}
\label{fig:esS_J_0_1_1}
\end{figure}

{\it Case ii)} in which it is assumed $\delta_{2} = 1$ (and $\delta_{1} = 0$).

\hspace{1cm}\includegraphics[scale=.4]{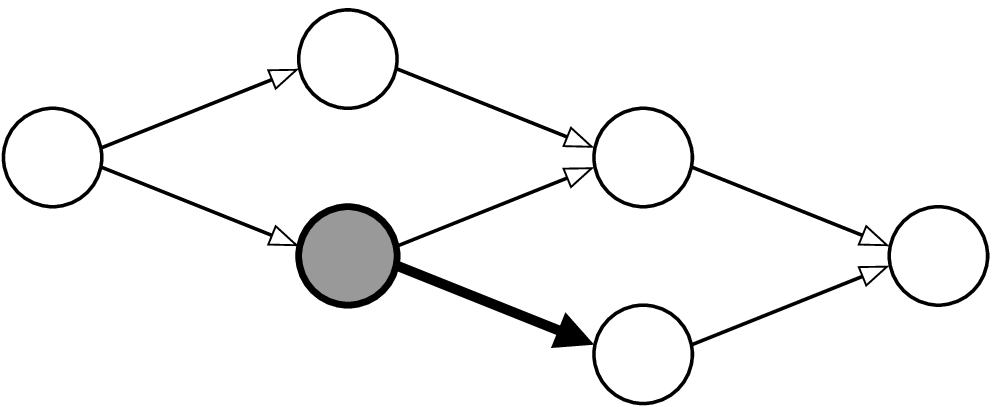}

In this case, it is necessary to minimize, with respect to the (continuos) decision variable $\tau$ which corresponds to the processing time $pt_{2,2}$, the following function
\begin{equation*}
\alpha_{2,2} \, \max \{ t_{1} + \tau - dd_{2,2} \, , \, 0 \} + \beta_{2} \, ( pt^{\mathrm{nom}}_{2} - \tau ) + J^{\circ}_{0,2} (t_{2})
\end{equation*}
that can be written as $f (pt_{2,2} + t_{1}) + g (pt_{2,2})$ being
\begin{equation*}
f (pt_{2,2} + t_{1}) = 0.75 \cdot \max \{ pt_{2,2} + t_{1} - 20 \, , \, 0 \} + J^{\circ}_{0,2} (pt_{2,2} + t_{1})
\end{equation*}
\begin{equation*}
g (pt_{2,2}) = \left\{ \begin{array}{ll}
2 - pt_{2,2} & pt_{2,2} \in [ 1 , 2 )\\
0 & pt_{2,2} \notin [ 1 , 2 )
\end{array} \right.
\end{equation*}

\newpage
the two functions illustrated in figure~\ref{fig:esS_salpha_0_1_2}.

\begin{figure}[h]
\centering
\psfrag{f(x)}[cl][Bl][.8][0]{$f (pt_{2,2} + t_{1})$}
\psfrag{g(x)}[cl][Bl][.8][0]{$g (pt_{2,2})$}
\psfrag{x}[bc][Bl][.8][0]{$pt_{2,2}$}
\psfrag{X1}[tc][Bl][.7][0]{$1$}
\psfrag{X2}[tc][Bl][.7][0]{$2$}
\psfrag{G1}[tc][Bl][.7][0]{$6-t_{1}$}
\psfrag{G2}[tc][Bl][.7][0]{$20-t_{1}$}
\psfrag{n}[cl][Bl][.6][0]{$-1$}
\psfrag{m1}[bc][Bl][.6][0]{$0.5$}
\psfrag{m2}[cr][Bl][.6][0]{$1.25$}
\includegraphics[scale=.25]{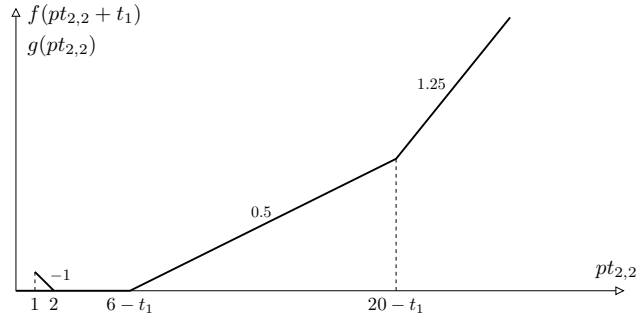}
\caption{Functions $f (pt_{2,2} + t_{1})$ and $g (pt_{2,2})$ in state $[ 0 \  1 \  t_{1}]^{T}$.}
\label{fig:esS_salpha_0_1_2}
\end{figure}

It is possible to apply lemma~\ref{lem:xopt} (note that $f (pt_{2,2} + t_{1})$ follows definition~\ref{def:f(x+t)} and $g (pt_{2,2})$ follows definition~\ref{def:g(x)}), which provides the function
\begin{equation*}
pt^{\circ}_{2,2}(t_{1}) = \arg \min_{\substack{pt_{2,2}\\1 \leq pt_{2,2} \leq 2}} \big\{ f (pt_{2,2} + t_{1}) + g (pt_{2,2}) \big\} = x_{\mathrm{e}}(t_{1})
\end{equation*}
illustrated in figure~\ref{fig:esS_tau_0_1_2}, being $x_{\mathrm{e}}(t_{1})$ the function
\begin{equation*}
x_{\mathrm{e}}(t_{1}) = \left\{ \begin{array}{ll}
2 &  t_{1} < 18\\
-t_{1} + 20 & 18 \leq t_{1} < 19\\
1 & t_{1} \geq 19
\end{array} \right.
\end{equation*}
$pt^{\circ}_{2,2}(t_{1})$ and $x_{\mathrm{e}}(t_{1})$ are in accordance with~\eqref{equ:xopt_1} and~\eqref{equ:xe_1}, respectively. Note that, in this case, $A = \{ 2 \}$, $\lvert A \rvert = 1$, $\gamma_{a_{1}} = 20$; moreover, since $B = \emptyset$ and $\lvert B \rvert = 0$, there is no need of executing algorithm~\ref{alg:tstar}.

\begin{figure}[h]
\centering
\psfrag{f(x)}[cl][Bl][.8][0]{$pt^{\circ}_{2,2}(t_{1})$}
\psfrag{x}[bc][Bl][.8][0]{$t_{1}$}
\psfrag{Y1}[cr][Bl][.7][0]{$1$}
\psfrag{Y2}[cr][Bl][.7][0]{$2$}
\psfrag{G1}[tc][Bl][.7][0]{$18$}
\psfrag{G2}[tc][Bl][.7][0]{$19$}
\includegraphics[scale=.25]{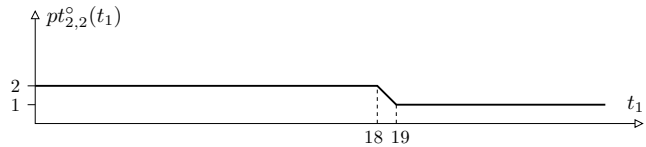}
\caption{Function $pt^{\circ}_{2,2}(t_{1})$.}
\label{fig:esS_tau_0_1_2}
\end{figure}

The conditioned cost-to-go
\begin{equation*}
J^{\circ}_{0,1} (t_{1} \mid \delta_{2} = 1) = f \big( pt^{\circ}_{2,2}(t_{1}) + t_{1} \big) + g \big( pt^{\circ}_{2,2}(t_{1}) \big)
\end{equation*}
illustrated in figure~\ref{fig:esS_J_0_1_2}, is provided by lemma~\ref{lem:h(t)}.

\begin{figure}[h]
\centering
\psfrag{f(x)}[cl][Bl][.8][0]{$J^{\circ}_{0,1} (t_{1} \mid \delta_{2} = 1)$}
\psfrag{x}[bc][Bl][.8][0]{$t_{1}$}
\psfrag{G1}[tc][Bl][.7][0]{$4$}
\psfrag{G2}[tc][Bl][.7][0]{$18$}
\psfrag{G3}[tc][Bl][.7][0]{$19$}
\psfrag{m1}[bc][Bl][.6][0]{$0.5$}
\psfrag{m2}[cr][Bl][.6][0]{$1$}
\psfrag{m3}[cr][Bl][.6][0]{$1.25$}
\includegraphics[scale=.25]{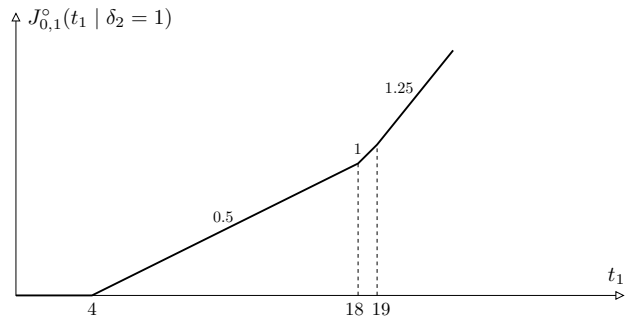}
\caption{Conditioned cost-to-go $J^{\circ}_{0,1} (t_{1} \mid \delta_{2} = 1)$.}
\label{fig:esS_J_0_1_2}
\end{figure}

In order to find the optimal cost-to-go $J^{\circ}_{0,1} (t_{1})$, it is necessary to carry out the following minimization
\begin{equation*}
J^{\circ}_{0,1} (t_{1}) = \min \big\{ J^{\circ}_{0,1} (t_{1} \mid \delta_{1} = 1) \, , \, J^{\circ}_{0,1} (t_{1} \mid \delta_{2} = 1) \big\}
\end{equation*}

\newpage
which provides, in accordance with lemma~\ref{lem:min}, the function illustrated in figure~\ref{fig:esS_J_0_1}.

\begin{figure}[h]
\centering
\psfrag{f(x)}[cl][Bl][.8][0]{$J^{\circ}_{0,1} (t_{1})$}
\psfrag{x}[bc][Bl][.8][0]{$t_{1}$}
\psfrag{G2}[tc][Bl][.7][0]{$6$}
\psfrag{G3}[tc][Bl][.7][0]{$14$}
\psfrag{G4}[tc][Bl][.7][0]{$16$}
\psfrag{G5}[tc][Bl][.7][0]{$18$}
\psfrag{G6}[tc][Bl][.7][0]{$19$}
\psfrag{m1}[bc][Bl][.6][0]{$0.5$}
\psfrag{m2}[cr][Bl][.6][0]{$1$}
\psfrag{m3}[cr][Bl][.6][0]{$1.25$}
\includegraphics[scale=.25]{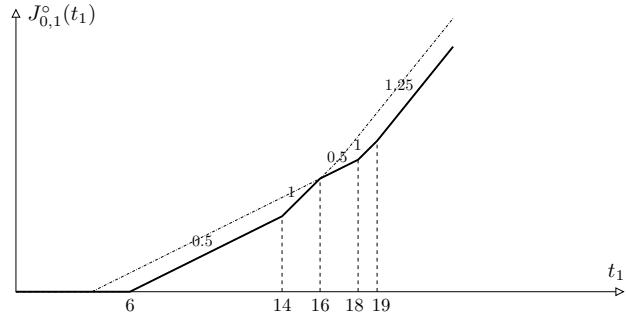}
\caption{Optimal cost-to-go $J^{\circ}_{0,1} (t_{1})$ in state $[ 0 \  1 \  t_{1}]^{T}$.}
\label{fig:esS_J_0_1}
\end{figure}

Since $J^{\circ}_{0,1} (t_{1} \mid \delta_{1} = 1)$ is the minimum in $(-\infty,16)$ and $J^{\circ}_{0,1} (t_{1} \mid \delta_{2} = 1)$ is the minimum in $[16,+\infty)$, the optimal control strategies for this state are
\begin{equation*}
\delta_{1}^{\circ} (0,1, t_{1}) = \left\{ \begin{array}{ll}
1 & t_{1} < 16\\
0 & t_{1} \geq 16
\end{array} \right. \qquad \delta_{2}^{\circ} (0,1, t_{1}) = \left\{ \begin{array}{ll}
0 & t_{1} < 16\\
1 & t_{1} \geq 16
\end{array} \right.
\end{equation*}
\begin{equation*}
\tau^{\circ} (0,1, t_{1}) = \delta_{1}^{\circ} (0,1, t_{1}) \, pt^{\circ}_{1,1}(t_{1}) + \delta_{2}^{\circ} (0,1, t_{1}) \, pt^{\circ}_{2,2}(t_{1}) = \left\{ \begin{array}{ll}
4 &  t_{1} < 14\\
-t_{1} + 18 & 14 \leq t_{1} < 16\\
2 & 16 \leq t_{1} < 18\\
-t_{1} + 20 & 18 \leq t_{1} < 19\\
1 & t_{1} \geq 19
\end{array} \right.
\end{equation*}
illustrated in figures~\ref{fig:esS_delta_0_1_1}, \ref{fig:esS_delta_0_1_2}, and~\ref{fig:esS_tau_0_1}, respectively.

\begin{figure}[h!]
\centering
\psfrag{f(x)}[cl][Bl][.8][0]{$\delta_{1}^{\circ} (0,1, t_{1})$}
\psfrag{x}[bc][Bl][.8][0]{$t_{1}$}
\psfrag{Y1}[cr][Bl][.7][0]{$1$}
\psfrag{G1}[tc][Bl][.7][0]{$16$}
\includegraphics[scale=.25]{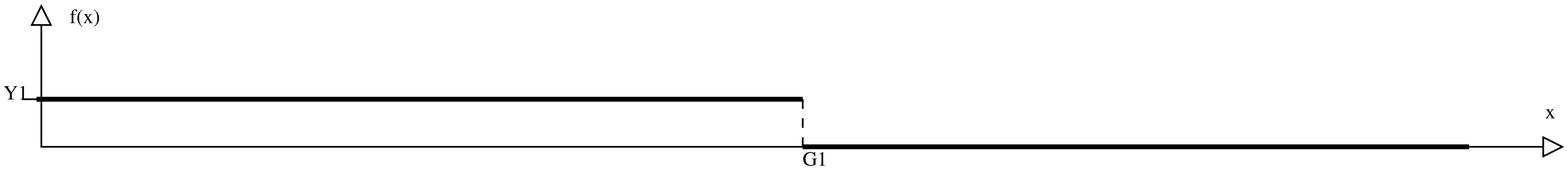}
\caption{Optimal control strategy $\delta_{1}^{\circ} (0,1, t_{1})$ in state $[ 0 \  1 \  t_{1}]^{T}$.}
\label{fig:esS_delta_0_1_1}
\end{figure}

\begin{figure}[h!]
\centering
\psfrag{f(x)}[cl][Bl][.8][0]{$\delta_{2}^{\circ} (0,1, t_{1})$}
\psfrag{x}[bc][Bl][.8][0]{$t_{1}$}
\psfrag{Y1}[cr][Bl][.7][0]{$1$}
\psfrag{G1}[tc][Bl][.7][0]{$16$}
\includegraphics[scale=.25]{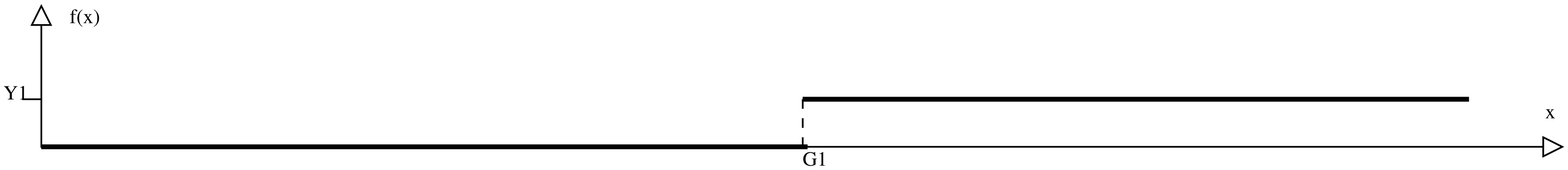}
\caption{Optimal control strategy $\delta_{2}^{\circ} (0,1, t_{1})$ in state $[ 0 \  1 \  t_{1}]^{T}$.}
\label{fig:esS_delta_0_1_2}
\end{figure}

\begin{figure}[h!]
\centering
\psfrag{f(x)}[cl][Bl][.8][0]{$\tau^{\circ} (0,1, t_{1})$}
\psfrag{x}[bc][Bl][.8][0]{$t_{1}$}
\psfrag{Y1}[cr][Bl][.7][0]{$1$}
\psfrag{Y2}[cr][Bl][.7][0]{$2$}
\psfrag{Y3}[cr][Bl][.7][0]{$4$}
\psfrag{G1}[tc][Bl][.7][0]{$14$}
\psfrag{G2}[tc][Bl][.7][0]{$16$}
\psfrag{G3}[tc][Bl][.7][0]{$18$}
\psfrag{G4}[tc][Bl][.7][0]{$19$}
\includegraphics[scale=.25]{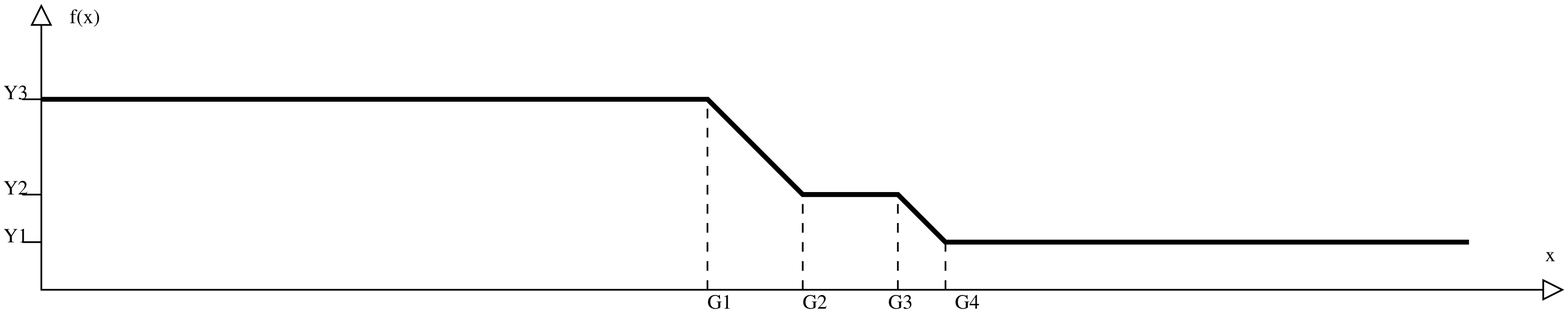}
\caption{Optimal control strategy $\tau^{\circ} (0,1, t_{1})$ (service time) in state $[ 0 \  1 \  t_{1}]^{T}$.}
\label{fig:esS_tau_0_1}
\end{figure}

%\vspace{12pt}
{\bf Stage 0 -- State $\boldsymbol{[ 0 \  0 \  t_{0}]}^{T}$} (initial state)

\hspace{1cm}\includegraphics[scale=.4]{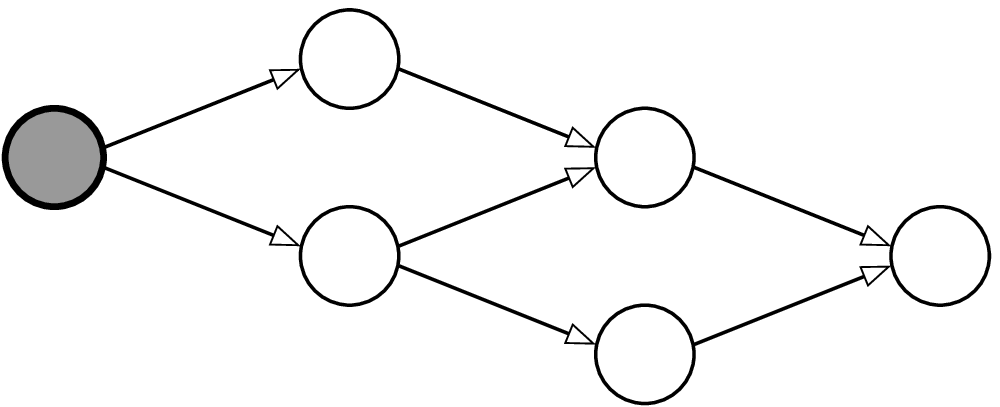}

In state $[0 \; 0 \; t_{0}]^{T}$, the cost function to be minimized, with respect to the (continuos) decision variable $\tau$ and to the (binary) decision variables $\delta_{1}$ and $\delta_{2}$ is
\begin{equation*}
\begin{split}
&\delta_{1} \big[ \alpha_{1,1} \, \max \{ t_{0} + \tau - dd_{1,1} \, , \, 0 \} + \beta_{1} \, ( pt^{\mathrm{nom}}_{1} - \tau ) + J^{\circ}_{1,0} (t_{1}) \big] +\\
&+ \delta_{2} \big[ \alpha_{2,1} \, \max \{ t_{0} + \tau - dd_{2,1} \, , \, 0 \} + \beta_{2} \, ( pt^{\mathrm{nom}}_{2} - \tau ) + J^{\circ}_{0,1} (t_{1}) \big]
\end{split}
\end{equation*}

{\it Case i)} in which it is assumed $\delta_{1} = 1$ (and $\delta_{2} = 0$).

\hspace{1cm}\includegraphics[scale=.4]{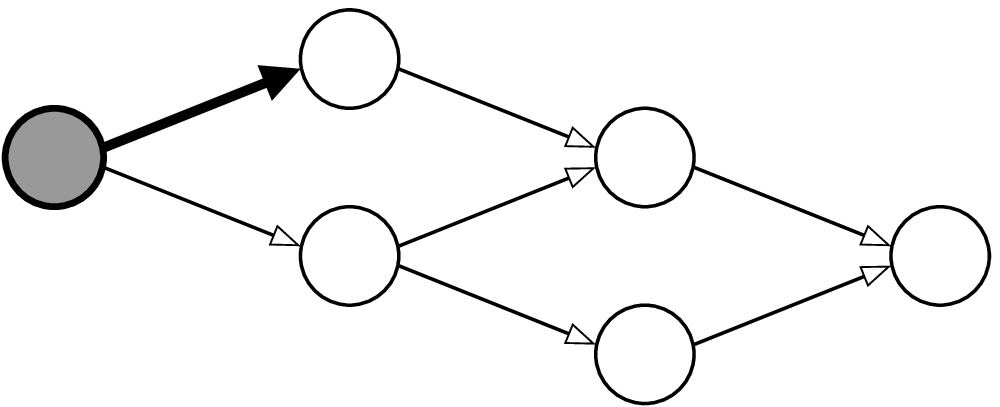}

In this case, it is necessary to minimize, with respect to the (continuos) decision variable $\tau$ which corresponds to the processing time $pt_{1,1}$, the following function
\begin{equation*}
\alpha_{1,1} \, \max \{ t_{0} + \tau - dd_{1,1} \, , \, 0 \} + \beta_{1} \, ( pt^{\mathrm{nom}}_{1} - \tau ) + J^{\circ}_{1,0} (t_{1})
\end{equation*}
that can be written as $f (pt_{1,1} + t_{0}) + g (pt_{1,1})$ being
\begin{equation*}
f (pt_{1,1} + t_{0}) = 0.5 \cdot \max \{ pt_{1,1} + t_{0} - 10 \, , \, 0 \} + J^{\circ}_{1,1} (pt_{1,1} + t_{0})
\end{equation*}
\begin{equation*}
g (pt_{1,1}) = \left\{ \begin{array}{ll}
4 - pt_{1,1} & pt_{1,1} \in [ 1 , 4 )\\
0 & pt_{1,1} \notin [ 1 , 4 )
\end{array} \right.
\end{equation*}
the two functions illustrated in figure~\ref{fig:esS_salpha_0_0_1}.

\begin{figure}[h]
\centering
\psfrag{f(x)}[cl][Bl][.8][0]{$f (pt_{1,1} + t_{0})$}
\psfrag{g(x)}[cl][Bl][.8][0]{$g (pt_{1,1})$}
\psfrag{x}[bc][Bl][.8][0]{$pt_{1,1}$}
\psfrag{X1}[tc][Bl][.7][0]{$1$}
\psfrag{X2}[tc][Bl][.7][0]{$4$}
\psfrag{G1}[tc][Bl][.7][0]{$10-t_{1}$}
\psfrag{G2}[tc][Bl][.7][0]{$16-t_{1}$}
\psfrag{n}[cl][Bl][.6][0]{$-1$}
\psfrag{m1}[cr][Bl][.6][0]{$0.75$}
\psfrag{m2}[cr][Bl][.6][0]{$1.5$}
\includegraphics[scale=.25]{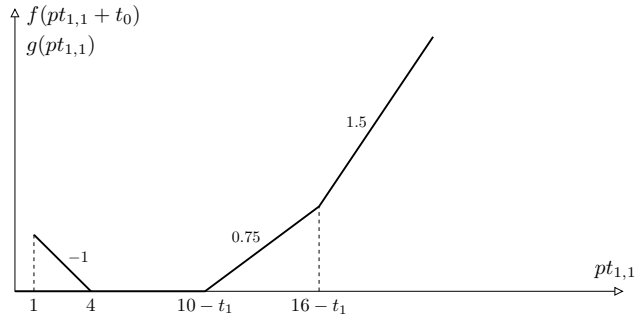}
\caption{Functions $f (pt_{1,1} + t_{0})$ and $g (pt_{1,1})$ in state $[ 0 \  0 \  t_{0}]^{T}$.}
\label{fig:esS_salpha_0_0_1}
\end{figure}

It is possible to apply lemma~\ref{lem:xopt} (note that $f (pt_{1,1} + t_{0})$ follows definition~\ref{def:f(x+t)} and $g (pt_{1,1})$ follows definition~\ref{def:g(x)}), which provides the function
\begin{equation*}
pt^{\circ}_{1,1}(t_{0}) = \arg \min_{\substack{pt_{1,1}\\1 \leq pt_{1,1} \leq 4}} \big\{ f (pt_{1,1} + t_{0}) + g (pt_{1,1}) \big\} = x_{\mathrm{e}}(t_{0})
\end{equation*}
illustrated in figure~\ref{fig:esS_tau_0_0_1}, being $x_{\mathrm{e}}(t_{0})$ the function
\begin{equation*}
x_{\mathrm{e}}(t_{0}) = \left\{ \begin{array}{ll}
4 &  t_{1} < 12\\
-t_{1} + 16 & 12 \leq t_{1} < 15\\
1 & t_{1} \geq 15
\end{array} \right.
\end{equation*}
$pt^{\circ}_{1,1}(t_{0})$ and $x_{\mathrm{e}}(t_{0})$ are in accordance with~\eqref{equ:xopt_1} and~\eqref{equ:xe_1}, respectively. Note that, in this case, $A = \{ 2 \}$, $\lvert A \rvert = 1$, $\gamma_{a_{1}} = 16$; moreover, since $B = \emptyset$ and $\lvert B \rvert = 0$, there is no need of executing algorithm~\ref{alg:tstar}.

\begin{figure}[h]
\centering
\psfrag{f(x)}[cl][Bl][.8][0]{$pt^{\circ}_{1,1}(t_{0})$}
\psfrag{x}[bc][Bl][.8][0]{$t_{0}$}
\psfrag{Y1}[cr][Bl][.7][0]{$1$}
\psfrag{Y2}[cr][Bl][.7][0]{$4$}
\psfrag{G1}[tc][Bl][.7][0]{$12$}
\psfrag{G2}[tc][Bl][.7][0]{$15$}
\includegraphics[scale=.25]{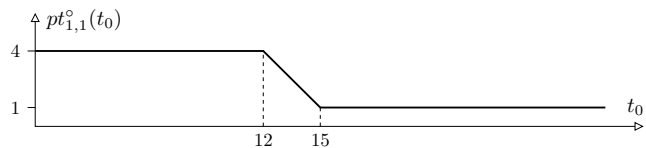}
\caption{Function $pt^{\circ}_{1,1}(t_{0})$.}
\label{fig:esS_tau_0_0_1}
\end{figure}

The conditioned cost-to-go
\begin{equation*}
J^{\circ}_{0,0} (t_{0} \mid \delta_{1} = 1) = f \big( pt^{\circ}_{1,1}(t_{0}) + t_{0} \big) + g \big( pt^{\circ}_{1,1}(t_{0}) \big)
\end{equation*}
illustrated in figure~\ref{fig:esS_J_0_0_1}, is provided by lemma~\ref{lem:h(t)}.

\begin{figure}[h]
\centering
\psfrag{f(x)}[cl][Bl][.8][0]{$J^{\circ}_{0,0} (t_{0} \mid \delta_{1} = 1)$}
\psfrag{x}[bc][Bl][.8][0]{$t_{0}$}
\psfrag{G1}[tc][Bl][.7][0]{$6$}
\psfrag{G2}[tc][Bl][.7][0]{$12$}
\psfrag{G3}[tc][Bl][.7][0]{$15$}
\psfrag{m1}[cr][Bl][.6][0]{$0.75$}
\psfrag{m2}[cr][Bl][.6][0]{$1$}
\psfrag{m3}[cr][Bl][.6][0]{$1.5$}
\includegraphics[scale=.25]{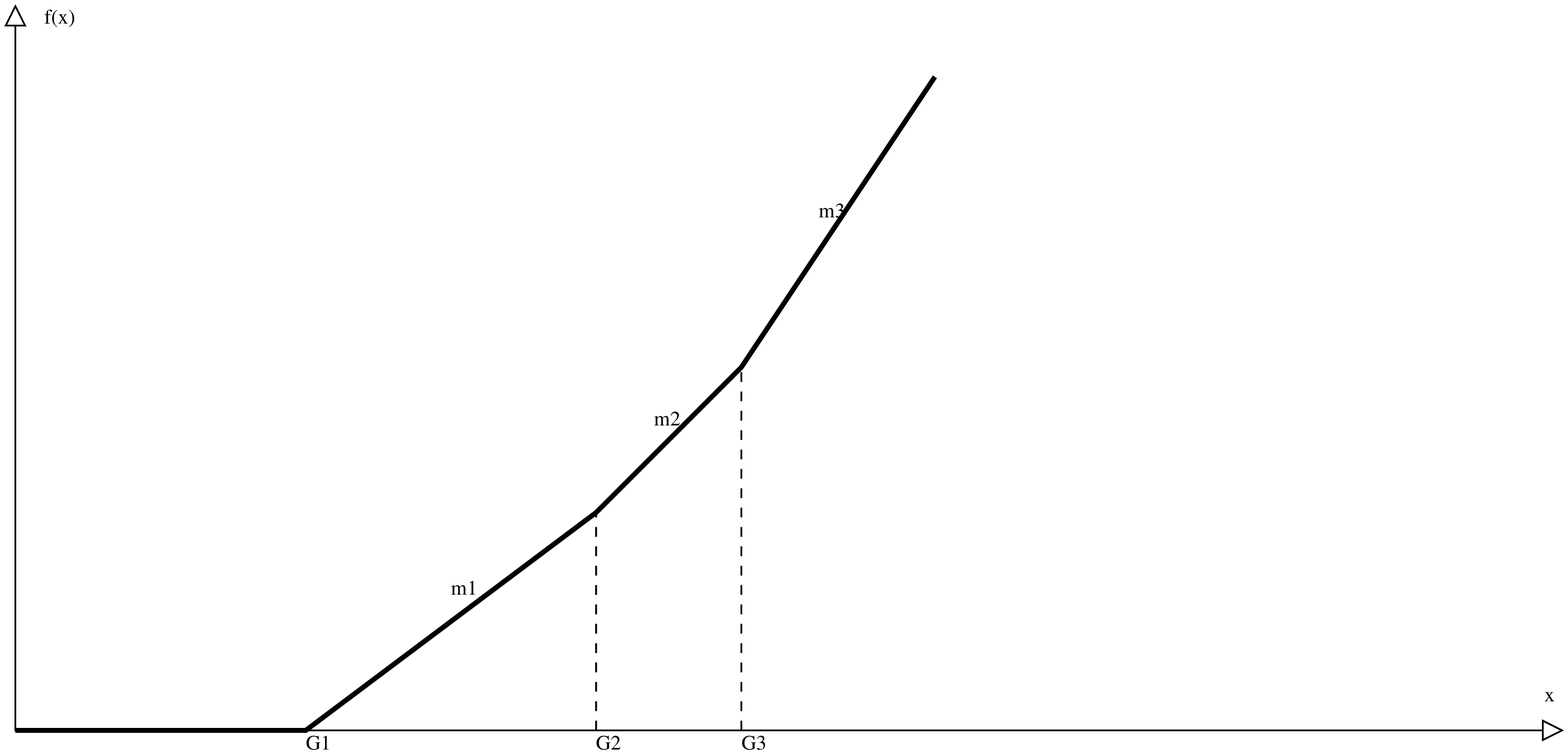}
\caption{Conditioned cost-to-go $J^{\circ}_{0,0} (t_{0} \mid \delta_{1} = 1)$.}
\label{fig:esS_J_0_0_1}
\end{figure}

{\it Case ii)} in which it is assumed $\delta_{2} = 1$ (and $\delta_{1} = 0$).

\hspace{1cm}\includegraphics[scale=.4]{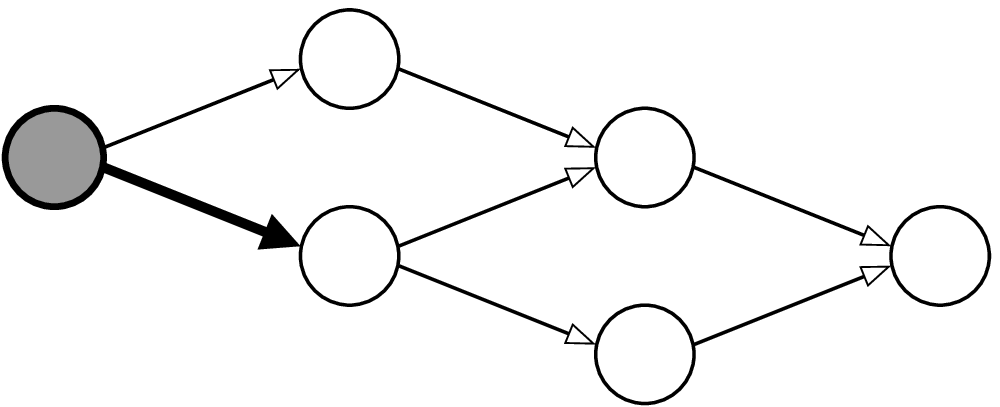}

In this case, it is necessary to minimize, with respect to the (continuos) decision variable $\tau$ which corresponds to the processing time $pt_{2,1}$, the following function
\begin{equation*}
\alpha_{2,1} \, \max \{ t_{0} + \tau - dd_{2,1} \, , \, 0 \} + \beta_{2} \, ( pt^{\mathrm{nom}}_{2} - \tau ) + J^{\circ}_{0,1} (t_{1})
\end{equation*}
that can be written as $f (pt_{2,1} + t_{0}) + g (pt_{2,1})$ being
\begin{equation*}
f (pt_{2,1} + t_{0}) = 0.25 \cdot \max \{ pt_{2,1} + t_{0} - 12 \, , \, 0 \} + J^{\circ}_{0,1} (pt_{2,1} + t_{0})
\end{equation*}
\begin{equation*}
g (pt_{2,1}) = \left\{ \begin{array}{ll}
2 - pt_{2,1} & pt_{2,1} \in [ 1 , 2 )\\
0 & pt_{2,1} \notin [ 1 , 2 )
\end{array} \right.
\end{equation*}
the two functions illustrated in figure~\ref{fig:esS_salpha_0_0_2}.

\begin{figure}[h]
\centering
\psfrag{f(x)}[cl][Bl][.8][0]{$f (pt_{2,1} + t_{0})$}
\psfrag{g(x)}[cl][Bl][.8][0]{$g (pt_{2,1})$}
\psfrag{x}[bc][Bl][.8][0]{$pt_{2,1}$}
\psfrag{X1}[tc][Bl][.7][0]{$1$}
\psfrag{X2}[tc][Bl][.7][0]{$2$}
\psfrag{G1}[Bl][Bl][.7][-45]{$6-t_{0}$}
\psfrag{G2}[Bl][Bl][.7][-45]{$12-t_{0}$}
\psfrag{G3}[Bl][Bl][.7][-45]{$14-t_{0}$}
\psfrag{G4}[Bl][Bl][.7][-45]{$16-t_{0}$}
\psfrag{G5}[Bl][Bl][.7][-45]{$18-t_{0}$}
\psfrag{G6}[Bl][Bl][.7][-45]{$19-t_{0}$}
\psfrag{n}[cl][Bl][.6][0]{$-1$}
\psfrag{m1}[bc][Bl][.6][0]{$0.5$}
\psfrag{m2}[cr][Bl][.6][0]{$0.75$}
\psfrag{m3}[cr][Bl][.6][0]{$1.25$}
\psfrag{m4}[cr][Bl][.6][0]{$1.5$}
\includegraphics[scale=.25]{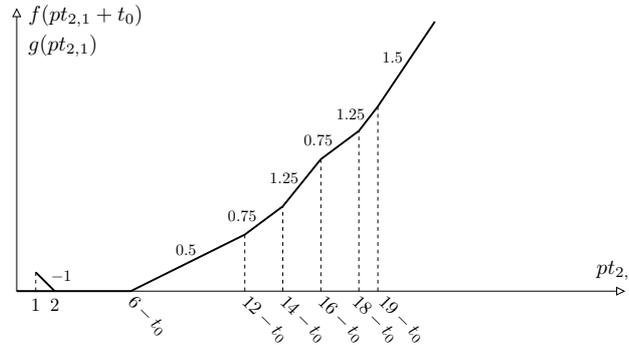}
\vspace{10pt}
\caption{Functions $f (pt_{2,1} + t_{0})$ and $g (pt_{2,1})$ in state $[ 0 \  0 \  t_{0}]^{T}$.}
\label{fig:esS_salpha_0_0_2}
\end{figure}

It is possible to apply lemma~\ref{lem:xopt} (note that $f (pt_{2,1} + t_{0})$ follows definition~\ref{def:f(x+t)} and $g (pt_{2,1})$ follows definition~\ref{def:g(x)}), which provides the function
\begin{equation*}
pt^{\circ}_{2,1}(t_{0}) = \arg \min_{\substack{pt_{2,1}\\1 \leq pt_{2,1} \leq 2}} \big\{ f (pt_{2,1} + t_{0}) + g (pt_{2,1}) \big\} = \left\{ \begin{array}{ll}
x_{\mathrm{s}}(t_{0}) & t_{0} < 14.5\\
x_{\mathrm{e}}(t_{0}) & t_{0} \geq 14.5 \\
\end{array} \right.
\end{equation*}
illustrated in figure~\ref{fig:esS_tau_0_0_2}, in which $14.5$ is the value $\omega_{1}$ determined by applying algorithm~\ref{alg:tstar}, and being $x_{\mathrm{s}}(t_{0})$ and $x_{\mathrm{e}}(t_{0})$ the functions
\begin{equation*}
x_{\mathrm{s}}(t_{0}) = \left\{ \begin{array}{ll}
2 &  t_{0} < 12\\
-t_{0} + 14 & 12 \leq t_{0} < 13\\
1 & 13 \leq t_{0} < 14.5
\end{array} \right.
\end{equation*}
\begin{equation*}
x_{\mathrm{e}}(t_{0}) = \left\{ \begin{array}{ll}
2 &  14.5 \leq t_{0} < 16\\
-t_{0} + 18 & 16 \leq t_{0} < 17\\
1 & t_{0} \geq 17
\end{array} \right.
\end{equation*}
$pt^{\circ}_{2,1}(t_{0})$, $x_{\mathrm{s}}(t_{0})$, and $x_{\mathrm{e}}(t_{0})$ are in accordance with~\eqref{equ:xopt_2}, \eqref{equ:xs_1}, and~\eqref{equ:xe_1}, respectively.

\begin{figure}[h]
\centering
\psfrag{f(x)}[cl][Bl][.8][0]{$pt^{\circ}_{2,1}(t_{0})$}
\psfrag{x}[bc][Bl][.8][0]{$t_{0}$}
\psfrag{Y1}[cr][Bl][.7][0]{$1$}
\psfrag{Y2}[cr][Bl][.7][0]{$2$}
\psfrag{G1}[tc][Bl][.7][0]{$12$}
\psfrag{G2}[tc][Bl][.7][0]{$13$}
\psfrag{G3}[tc][Bl][.7][0]{$14.5$}
\psfrag{G4}[tc][Bl][.7][0]{$16$}
\psfrag{G5}[tc][Bl][.7][0]{$17$}
\includegraphics[scale=.25]{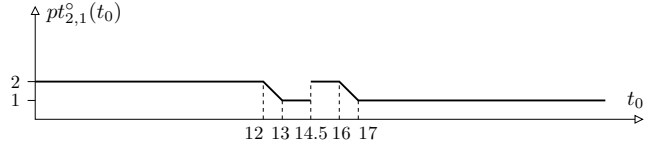}
\caption{Function $pt^{\circ}_{2,1}(t_{0})$.}
\label{fig:esS_tau_0_0_2}
\end{figure}

$\omega_{1}$ is determined by applying algorithm~\ref{alg:tstar} as follows.
\begin{align*}
& A = \{ 3 ,  5 \} & & \lvert A \rvert = 2 & & \gamma_{a_{1}} = 14 & & \gamma_{a_{2}} = 18 & & x_{1} = 1 & & x_{2} = 2\\
& B = \{ 4 \} & & \lvert B \rvert = 1 & & \gamma_{b_{1}} = 16
\end{align*}

With $j = 1$, the ``Section A -- Initialization'' part of the algorithm provides:
\begin{equation*}
\text{\small [row 1]:} \qquad \gamma_{0} = -\infty 
\end{equation*}
\begin{equation*}
\text{\small [row 2]:} \qquad h \geq 0 \; : \; \gamma_{h} \leq 16 - (2 - 1) < \gamma_{h+1} \quad \Rightarrow \quad h = 3
\end{equation*}
\begin{equation*}
\text{\small [row 3]:} \qquad i = 4
\end{equation*}
\begin{equation*}
\text{\small [row 4]:} \qquad \gamma_{7} = +\infty 
\end{equation*}
\begin{equation*}
\text{\small [row 5]:} \qquad k \leq 6 \; : \; \gamma_{k} < 16 + (2 - 1) \leq \gamma_{k+1} \quad \Rightarrow \quad k = 4 
\end{equation*}
\begin{equation*}
\text{\small [row 6]:} \qquad \text{condition: $j = \lvert B \rvert$ and $\lvert A \rvert = \lvert B \rvert$ ($1 = 1$ and $2 = 1$) is false}
\end{equation*}
\begin{equation*}
\text{\small [row 10]:} \qquad \left\{ \begin{array}{l}
\tilde{\mu}_{3} = 1.25 - 1 = 0.25\\
\tilde{\mu}_{4} = 0.75 - 1 = -0.25
\end{array} \right. 
\end{equation*}
\begin{equation*}
\text{\small [row 12]:} \qquad \tau = 16-(2-1) = 15
\end{equation*}
\begin{equation*}
\text{\small [row 13]:} \qquad \theta = 16
\end{equation*}
\begin{equation*}
\text{\small [row 14]:} \qquad d = \max \{ 0, 0.25 \cdot (16-15) \}Ê= 0.25
\end{equation*}
\begin{equation*}
\text{\small [row 15]:} \qquad \text{condition: $h < b_{j}-1$ ($3 < 4 - 1$) is false}
\end{equation*}
\begin{equation*}
\text{\small [row 20]:} \qquad \lambda = 3
\end{equation*}
\begin{equation*}
\text{\small [row 21]:} \qquad \xi = 4
\end{equation*}

Since condition: $h < b_{1}$ and $i < a_{2}$ ($4 < 5$ and $3 < 4$) {\small [row 22]} is true, the ``Section B -- First Loop'' part of the algorithm is executed:
\begin{equation*}
\text{\small [row 23]:} \qquad \psi = \min \{ 16-15 , 18-16 \} = 1
\end{equation*}
\begin{equation*}
\text{\small [row 24]:} \qquad \text{condition: $\gamma_{h+1} - \tau \leq \gamma_{i+1} - \theta$ ($16-15 \leq 18-16$) is true}
\end{equation*}
\begin{equation*}
\text{\small [row 25]:} \qquad \lambda = 3 + 1 = 4
\end{equation*}
\begin{equation*}
\text{\small [row 27]:} \qquad \text{condition: $\gamma_{h+1} - \tau \geq \gamma_{i+1} - \theta$ ($16-15 \geq 18-16$) is false}
\end{equation*}
\begin{equation*}
\text{\small [row 30]:} \qquad \delta = \max \{Ê0, -0.25 \cdot \big[Ê18 - (15+1) \big] = 0
\end{equation*}
\begin{equation*}
\text{\small [row 31]:} \qquad \text{condition: $\lambda < b_{j}-1$ ($3 < 4 - 1$) is false}
\end{equation*}
\begin{equation*}
\text{\small [row 36]:} \qquad \text{condition: $\xi = b_{j}$ ($4 = 4$) is true}
\end{equation*}
\begin{equation*}
\text{\small [row 37]:} \qquad \delta = 0 - 0.25 \cdot \big[Ê(16+1) - 16 \big] = -0.25
\end{equation*}
\begin{equation*}
\text{\small [row 43]:} \qquad \text{condition: $\delta \leq 0$ ($-0.25 \leq 0$) is true}
\end{equation*}
\begin{equation*}
\text{\small [row 44]:} \qquad a_{0} = 0 
\end{equation*}
\begin{equation*}
\text{\small [row 45]:} \qquad r \geq 1 \; : \; a_{r-1} \leq 3 < a_{r} \quad \Rightarrow \quad r = 2
\end{equation*}
\begin{equation*}
\text{\small [row 46]:} \qquad \text{condition: $r \leq j$ ($2 \leq 1$) is false}
\end{equation*}
\begin{equation*}
\text{\small [row 68]:} \qquad \omega_{1} = 15 - 1 + \dfrac{0.25}{0.25+0.25} = 14.5
\end{equation*}
\begin{equation*}
\text{\small [row 69]:} \qquad \text{exit algorithm}
\end{equation*}

The conditioned cost-to-go
\begin{equation*}
J^{\circ}_{0,0} (t_{0} \mid \delta_{2} = 1) = f \big( pt^{\circ}_{2,1}(t_{0}) + t_{0} \big) + g \big( pt^{\circ}_{2,1}(t_{0}) \big)
\end{equation*}
illustrated in figure~\ref{fig:esS_J_0_0_2}, is provided by lemma~\ref{lem:h(t)}.

\begin{figure}[h]
\centering
\psfrag{f(x)}[cl][Bl][.8][0]{$J^{\circ}_{0,0} (t_{0} \mid \delta_{2} = 1)$}
\psfrag{x}[bc][Bl][.8][0]{$t_{0}$}
\psfrag{G1}[tc][Bl][.7][0]{$4$}
\psfrag{G2}[tc][Bl][.7][0]{$10$}
\psfrag{G3}[tc][Bl][.7][0]{$12$}
\psfrag{G4}[tc][Bl][.7][0]{$13$}
\psfrag{G5}[tc][Bl][.7][0]{$14.5$}
\psfrag{G6}[tc][Bl][.7][0]{$16$}
\psfrag{G7}[tc][Bl][.7][0]{$17$}
\psfrag{G8}[tc][Bl][.7][0]{$18$}
\psfrag{m1}[bc][Bl][.6][0]{$0.5$}
\psfrag{m2}[cr][Bl][.6][0]{$1$}
\psfrag{m3}[cr][Bl][.6][0]{$0.75$}
\psfrag{m4}[cr][Bl][.6][0]{$1.25$}
\psfrag{m5}[cr][Bl][.6][0]{$1.5$}
\includegraphics[scale=.25]{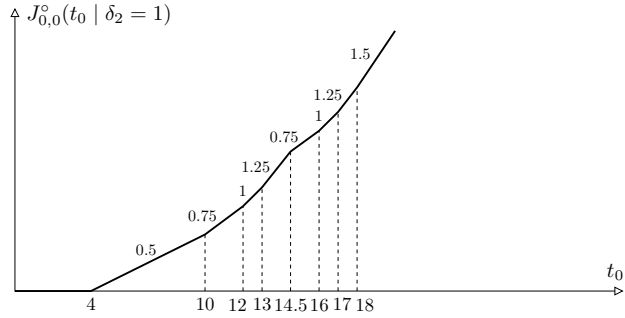}
\caption{Conditioned cost-to-go $J^{\circ}_{0,0} (t_{0} \mid \delta_{2} = 1)$.}
\label{fig:esS_J_0_0_2}
\end{figure}

In order to find the optimal cost-to-go $J^{\circ}_{0,0} (t_{0})$, it is necessary to carry out the following minimization
\begin{equation*}
J^{\circ}_{0,0} (t_{0}) = \min \big\{ J^{\circ}_{0,0} (t_{0} \mid \delta_{1} = 1) \, , \, J^{\circ}_{0,0} (t_{0} \mid \delta_{2} = 1) \big\}
\end{equation*}
which provides, in accordance with lemma~\ref{lem:min}, the function illustrated in figure~\ref{fig:esS_J_0_0}.

\begin{figure}[h]
\centering
\psfrag{f(x)}[cl][Bl][.8][0]{$J^{\circ}_{0,0} (t_{0})$}
\psfrag{x}[bc][Bl][.8][0]{$t_{1}$}
\psfrag{G1}[tc][Bl][.7][0]{$6$}
\psfrag{G2}[tc][Bl][.7][0]{$12$}
\psfrag{G3}[tc][Bl][.7][0]{$15$}
\psfrag{G4}[tc][Bl][.7][0]{$15.\overline{3}$}
\psfrag{G5}[tc][Bl][.7][0]{$16$}
\psfrag{G6}[tc][Bl][.7][0]{$17$}
\psfrag{G7}[tc][Bl][.7][0]{$18$}
\psfrag{m1}[cr][Bl][.6][0]{$0.75$}
\psfrag{m2}[cr][Bl][.6][0]{$1$}
\psfrag{m3}[cr][Bl][.6][0]{$1.5$}
\psfrag{m4}[cr][Bl][.6][0]{$1.25$}
\psfrag{m5}[cr][Bl][.6][0]{$1.5$}
\psfrag{m6}[cr][Bl][.6][0]{$0.75$}
\includegraphics[scale=.25]{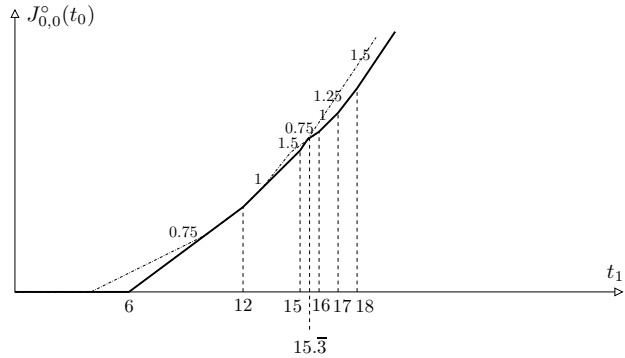}
\caption{Optimal cost-to-go $J^{\circ}_{0,0} (t_{0})$ in state $[ 0 \  0 \  t_{0}]^{T}$.}
\label{fig:esS_J_0_0}
\end{figure}

Since $J^{\circ}_{0,0} (t_{0} \mid \delta_{1} = 1)$ is the minimum in $(-\infty,15.\overline{3})$ and $J^{\circ}_{0,0} (t_{0} \mid \delta_{2} = 1)$ is the minimum in $[15.\overline{3},+\infty)$, the optimal control strategies for this state are
\begin{equation*}
\delta_{1}^{\circ} (0,0, t_{0}) = \left\{ \begin{array}{ll}
1 & t_{0} < 15.\overline{3}\\
0 & t_{0} \geq 15.\overline{3}
\end{array} \right. \qquad \delta_{2}^{\circ} (0,0, t_{0}) = \left\{ \begin{array}{ll}
0 & t_{0} < 15.\overline{3}\\
1 & t_{0} \geq 15.\overline{3}
\end{array} \right.
\end{equation*}
\begin{equation*}
\tau^{\circ} (0,0, t_{0}) = \delta_{1}^{\circ} (0,0, t_{0}) \, pt^{\circ}_{1,1}(t_{0}) + \delta_{2}^{\circ} (0,0, t_{0}) \, pt^{\circ}_{2,1}(t_{0}) = \left\{ \begin{array}{ll}
4 &  t_{0} < 12\\
-t_{0} + 16 & 12 \leq t_{0} < 15\\
2 & 15 \leq t_{0} < 15.\overline{3}\\
2 & 15.\overline{3} \leq t_{0} < 16\\
-t_{0} + 18 & 16 \leq t_{0} < 17\\
1 & t_{0} \geq 17
\end{array} \right.
\end{equation*}

\newpage
illustrated in figures~\ref{fig:esS_delta_0_0_1}, \ref{fig:esS_delta_0_0_2}, and~\ref{fig:esS_tau_0_0}, respectively.

\begin{figure}[h!]
\centering
\psfrag{f(x)}[cl][Bl][.8][0]{$\delta_{1}^{\circ} (0,0, t_{0})$}
\psfrag{x}[bc][Bl][.8][0]{$t_{0}$}
\psfrag{Y1}[cr][Bl][.7][0]{$1$}
\psfrag{G1}[tc][Bl][.7][0]{$15.\overline{3}$}
\includegraphics[scale=.25]{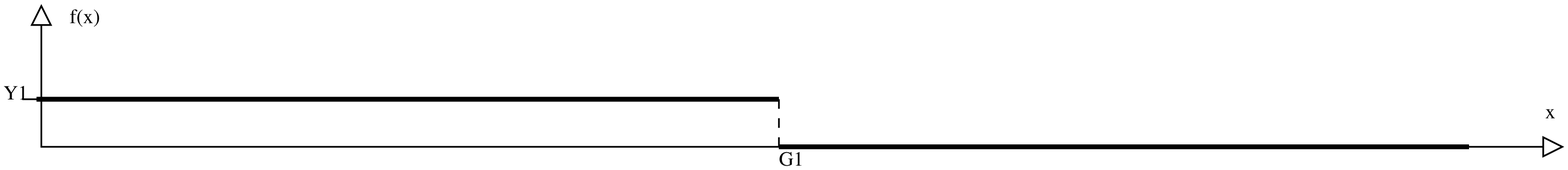}\vspace{-4pt}
\caption{Optimal control strategy $\delta_{1}^{\circ} (0,0, t_{0})$ in state $[ 0 \  0 \  t_{0}]^{T}$.}
\label{fig:esS_delta_0_0_1}
\end{figure}

\begin{figure}[h!]
\centering
\psfrag{f(x)}[cl][Bl][.8][0]{$\delta_{2}^{\circ} (0,0, t_{0})$}
\psfrag{x}[bc][Bl][.8][0]{$t_{0}$}
\psfrag{Y1}[cr][Bl][.7][0]{$1$}
\psfrag{G1}[tc][Bl][.7][0]{$15.\overline{3}$}
\includegraphics[scale=.25]{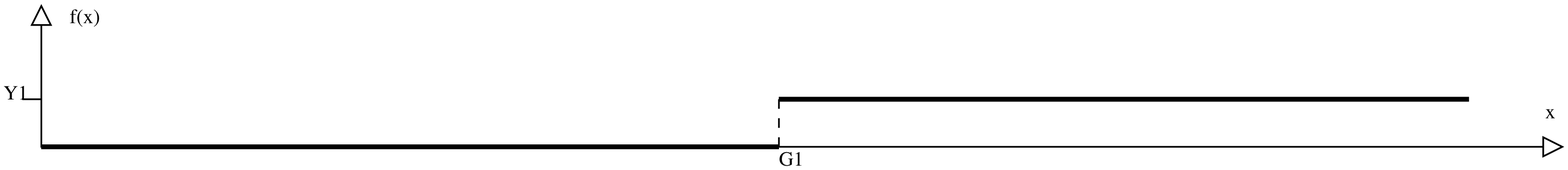}\vspace{-4pt}
\caption{Optimal control strategy $\delta_{2}^{\circ} (0,0, t_{0})$ in state $[ 0 \  0 \  t_{0}]^{T}$.}
\label{fig:esS_delta_0_0_2}
\end{figure}

\begin{figure}[h!]
\centering
\psfrag{f(x)}[cl][Bl][.8][0]{$\tau^{\circ} (0,0, t_{0})$}
\psfrag{x}[bc][Bl][.8][0]{$t_{0}$}
\psfrag{Y1}[cr][Bl][.7][0]{$1$}
\psfrag{Y2}[cr][Bl][.7][0]{$2$}
\psfrag{Y3}[cr][Bl][.7][0]{$4$}
\psfrag{G1}[tc][Bl][.7][0]{$12$}
\psfrag{G2}[tc][Bl][.7][0]{$15$}
\psfrag{G3}[tc][Bl][.7][0]{$16$}
\psfrag{G4}[tc][Bl][.7][0]{$17$}
\psfrag{G5}[tc][Bl][.7][0]{$15.\overline{3}$}
\includegraphics[scale=.25]{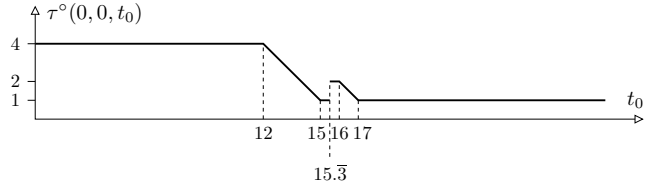}\vspace{-2pt}
\caption{Optimal control strategy $\tau^{\circ} (0,0, t_{0})$ (processing time) in state $[ 0 \  0 \  t_{0}]^{T}$.}
\label{fig:esS_tau_0_0}
\end{figure}

Since the two conditional costs-to-go have the same value in the interval $[10,13]$, the following functions represent alternative optimal control strategies for the considered state
\begin{equation*}
\delta_{1}^{\circ} (0,0, t_{0}) = \left\{ \begin{array}{ll}
1 & t_{0} \leq 10\\
0 & 10 \leq t_{0} < 13\\
1 & 13 \leq t_{0} < 15.\overline{3}\\
0 & t_{0} > 15.\overline{3}
\end{array} \right. \qquad \delta_{2}^{\circ} (0,0, t_{0}) = \left\{ \begin{array}{ll}
0 & t_{0} \leq 10\\
1 & 10 \leq t_{0} < 13\\
0 & 13 \leq t_{0} < 15.\overline{3}\\
1 & t_{0} > 15.\overline{3}
\end{array} \right.
\end{equation*}
\begin{equation*}
\tau^{\circ} (0,0, t_{1}) = \delta_{1}^{\circ} (0,0, t_{0}) \, pt^{\circ}_{1,1}(t_{0}) + \delta_{2}^{\circ} (0,0, t_{0}) \, pt^{\circ}_{2,1}(t_{0}) =  \left\{ \begin{array}{ll}
4 &  t_{0} < 10\\
2 & 10 \leq t_{0} < 12\\
-t_{0} + 14 & 12 \leq t_{0} < 13\\
-t_{0} + 16 & 13 \leq t_{0} < 15\\
1 & 15 \leq t_{0} < 15.\overline{3}\\
2 & 15.\overline{3} \leq t_{0} < 16\\
-t_{0} + 18 & 16 \leq t_{0} < 17\\
1 & t_{0} \geq 17
\end{array} \right.
\end{equation*}
Such functions are illustrated in figures~\ref{fig:esS_delta_0_0_1alt}, \ref{fig:esS_delta_0_0_2alt}, and~\ref{fig:esS_tau_0_0alt}, respectively.

\begin{figure}[h!]
\centering
\psfrag{f(x)}[cl][Bl][.8][0]{$\delta_{1}^{\circ} (0,0, t_{0})$}
\psfrag{x}[bc][Bl][.8][0]{$t_{0}$}
\psfrag{Y1}[cr][Bl][.7][0]{$1$}
\psfrag{G1}[Bc][Bl][.7][0]{$10$}
\psfrag{G2}[Bc][Bl][.7][0]{$13$}
\psfrag{G3}[Bc][Bl][.7][0]{$15.\overline{3}$}
\includegraphics[scale=.25]{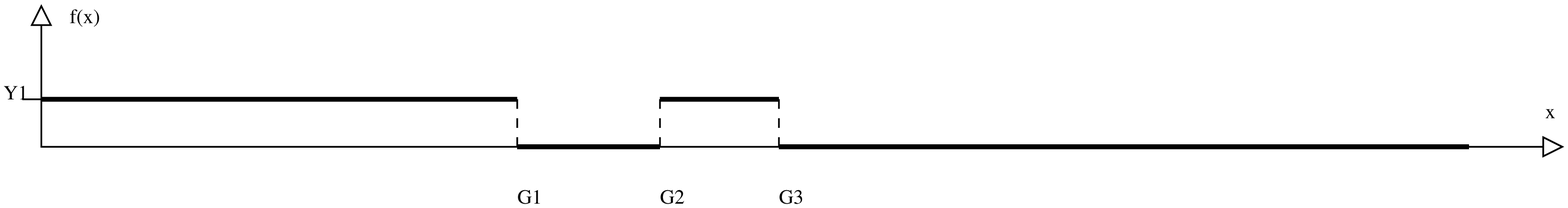}\vspace{-8pt}
\caption{Alternative optimal control strategy $\delta_{1}^{\circ} (0,0, t_{0})$ in state $[ 0 \  0 \  t_{0}]^{T}$.}
\label{fig:esS_delta_0_0_1alt}
\end{figure}

\begin{figure}[h!]
\centering
\psfrag{f(x)}[cl][Bl][.8][0]{$\delta_{2}^{\circ} (0,0, t_{0})$}
\psfrag{x}[bc][Bl][.8][0]{$t_{0}$}
\psfrag{Y1}[cr][Bl][.7][0]{$1$}
\psfrag{G1}[Bc][Bl][.7][0]{$10$}
\psfrag{G2}[Bc][Bl][.7][0]{$13$}
\psfrag{G3}[Bc][Bl][.7][0]{$15.\overline{3}$}
\includegraphics[scale=.25]{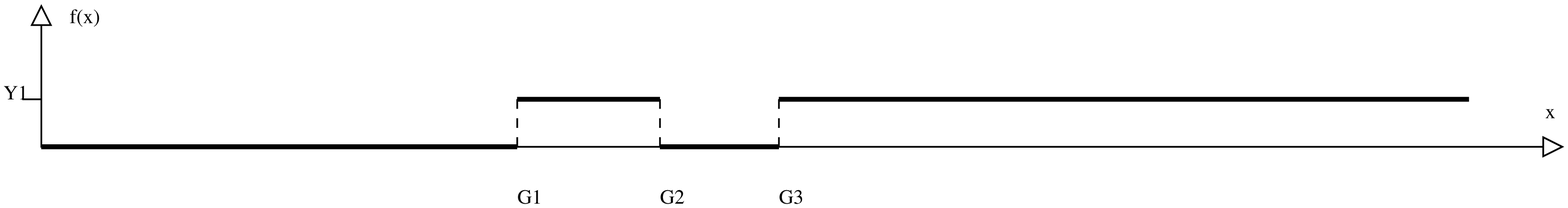}\vspace{-8pt}
\caption{Alternative optimal control strategy $\delta_{2}^{\circ} (0,0, t_{0})$ in state $[ 0 \  0 \  t_{0}]^{T}$.}
\label{fig:esS_delta_0_0_2alt}
\end{figure}

\begin{figure}[h!]
\centering
\psfrag{f(x)}[cl][Bl][.8][0]{$\tau^{\circ} (0,0, t_{0})$}
\psfrag{x}[bc][Bl][.8][0]{$t_{0}$}
\psfrag{Y1}[cr][Bl][.7][0]{$1$}
\psfrag{Y2}[cr][Bl][.7][0]{$2$}
\psfrag{Y3}[cr][Bl][.7][0]{$4$}
\psfrag{G1}[tc][Bl][.7][0]{$10$}
\psfrag{G2}[tc][Bl][.7][0]{$12$}
\psfrag{G3}[tc][Bl][.7][0]{$13$}
\psfrag{G4}[tc][Bl][.7][0]{$15$}
\psfrag{G5}[tc][Bl][.7][0]{$16$}
\psfrag{G6}[tc][Bl][.7][0]{$17$}
\psfrag{G7}[tc][Bl][.7][0]{$15.\overline{3}$}
\includegraphics[scale=.25]{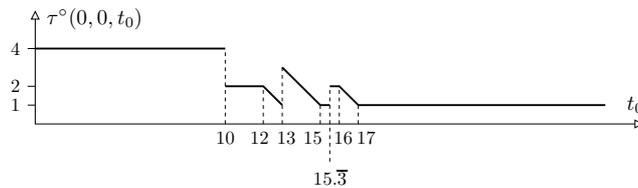}\vspace{-4pt}
\caption{Alternative optimal control strategy $\tau^{\circ} (0,0, t_{0})$ (service time) in state $[ 0 \  0 \  t_{0}]^{T}$.}
\label{fig:esS_tau_0_0alt}
\end{figure}

%%========================================
%%========================================
%%          ESEMPIO SCHEDULING 2
%%========================================
%%========================================
%
\clearpage

\section{Application to the single machine scheduling -- Example with setup} \label{sec:app2}

Consider a single machine scheduling problem in which 4 jobs of class $P_{1}$ and 3 jobs of class $P_{2}$ must be executed. The due dates, the marginal tardiness costs of jobs, the processing time bounds and the marginal deviation costs of jobs are:

\begin{center}
\begin{tabular}{|ll|}
\hline
$\alpha_{1,1} = 0.75$ & $dd_{1,1} = 19$\\
$\alpha_{1,2} = 0.5$ & $dd_{1,2} = 24$\\
$\alpha_{1,3} = 1.5$ & $dd_{1,3} = 29$\\
$\alpha_{1,4} = 0.5$ & $dd_{1,4} = 41$\\
\hline
\multicolumn{2}{|c|}{$\beta_{1} = 1$}\\
$pt^{\mathrm{low}}_{1} = 4$ & $pt^{\mathrm{nom}}_{1} = 8$\\
\hline
\end{tabular}
\hspace{.5cm}
\begin{tabular}{|ll|}
\hline
$\alpha_{2,1} = 2$ & $dd_{2,1} = 21$\\
$\alpha_{2,2} = 1$ & $dd_{2,2} = 24$\\
$\alpha_{2,3} = 1$ & $dd_{2,3} = 38$\\
\hline
\multicolumn{2}{|c|}{$\beta_{2} = 1.5$}\\
$pt^{\mathrm{low}}_{2} = 4$ & $pt^{\mathrm{nom}}_{2} = 6$\\
\hline
\end{tabular}
\end{center}
A setup is required between the execution of jobs of different classes. Setup times and costs are:
\begin{center}
\begin{tabular}{|ll|}
\hline
$st_{1,1} = 0$ & $st_{1,2} = 1$\\
$st_{2,1} = 0.5$ & $st_{2,2} = 0$\\
\hline
\end{tabular}
\hspace{.5cm}
\begin{tabular}{|ll|}
\hline
$sc_{1,1} = 0$ & $sc_{1,2} = 0.5$\\
$sc_{2,1} = 1$ & $sc_{2,2} = 0$\\
\hline
\end{tabular}
\end{center}

The evolution of the system state can be represented by the following diagram.
\begin{figure}[h]
\centering
\psfrag{S0}[bc][Bl][.7][0]{$S0$}
\psfrag{S1}[bc][Bl][.7][0]{$S1$}
\psfrag{S2}[tc][Bl][.7][0]{$S2$}
\psfrag{S3}[bc][Bl][.7][0]{$S3$}
\psfrag{S4}[bc][Bl][.7][0]{$S4$}
\psfrag{S5}[tc][Bl][.7][0]{$S5$}
\psfrag{S6}[tc][Bl][.7][0]{$S6$}
\psfrag{S7}[bc][Bl][.7][0]{$S7$}
\psfrag{S8}[bc][Bl][.7][0]{$S8$}
\psfrag{S9}[tc][Bl][.7][0]{$S9$}
\psfrag{S10}[bc][Bl][.7][0]{$S10$}
\psfrag{S11}[tc][Bl][.7][0]{$S11$}
\psfrag{S12}[tc][Bl][.7][0]{$S12$}
\psfrag{S13}[bc][Bl][.7][0]{$S13$}
\psfrag{S14}[bc][Bl][.7][0]{$S14$}
\psfrag{S15}[tc][Bl][.7][0]{$S15$}
\psfrag{S16}[bc][Bl][.7][0]{$S16$}
\psfrag{S17}[tc][Bl][.7][0]{$S17$}
\psfrag{S18}[bc][Bl][.7][0]{$S18$}
\psfrag{S19}[tc][Bl][.7][0]{$S19$}
\psfrag{S20}[bc][Bl][.7][0]{$S20$}
\psfrag{S21}[tc][Bl][.7][0]{$S21$}
\psfrag{S22}[bc][Bl][.7][0]{$S22$}
\psfrag{S23}[tc][Bl][.7][0]{$S23$}
\psfrag{S24}[bc][Bl][.7][0]{$S24$}
\psfrag{S25}[tc][Bl][.7][0]{$S25$}
\psfrag{S26}[bc][Bl][.7][0]{$S26$}
\psfrag{S27}[tc][Bl][.7][0]{$S27$}
\psfrag{S28}[bc][Bl][.7][0]{$S28$}
\psfrag{S29}[tc][Bl][.7][0]{$S29$}
\psfrag{S30}[bc][Bl][.7][0]{$S30$}
\psfrag{S31}[tc][Bl][.7][0]{$S31$}
\psfrag{ST0}[bc][Bl][.7][0]{\it stage $0$}
\psfrag{ST1}[bc][Bl][.8][0]{\it stage $1$}
\psfrag{ST2}[bc][Bl][.8][0]{\it stage $2$}
\psfrag{ST3}[bc][Bl][.8][0]{\it stage $3$}
\psfrag{ST4}[bc][Bl][.8][0]{\it stage $4$}
\psfrag{ST5}[bc][Bl][.8][0]{\it stage $5$}
\psfrag{ST6}[bc][Bl][.8][0]{\it stage $6$}
\psfrag{ST7}[bc][Bl][.8][0]{\it stage $7$}
\includegraphics[scale=.5]{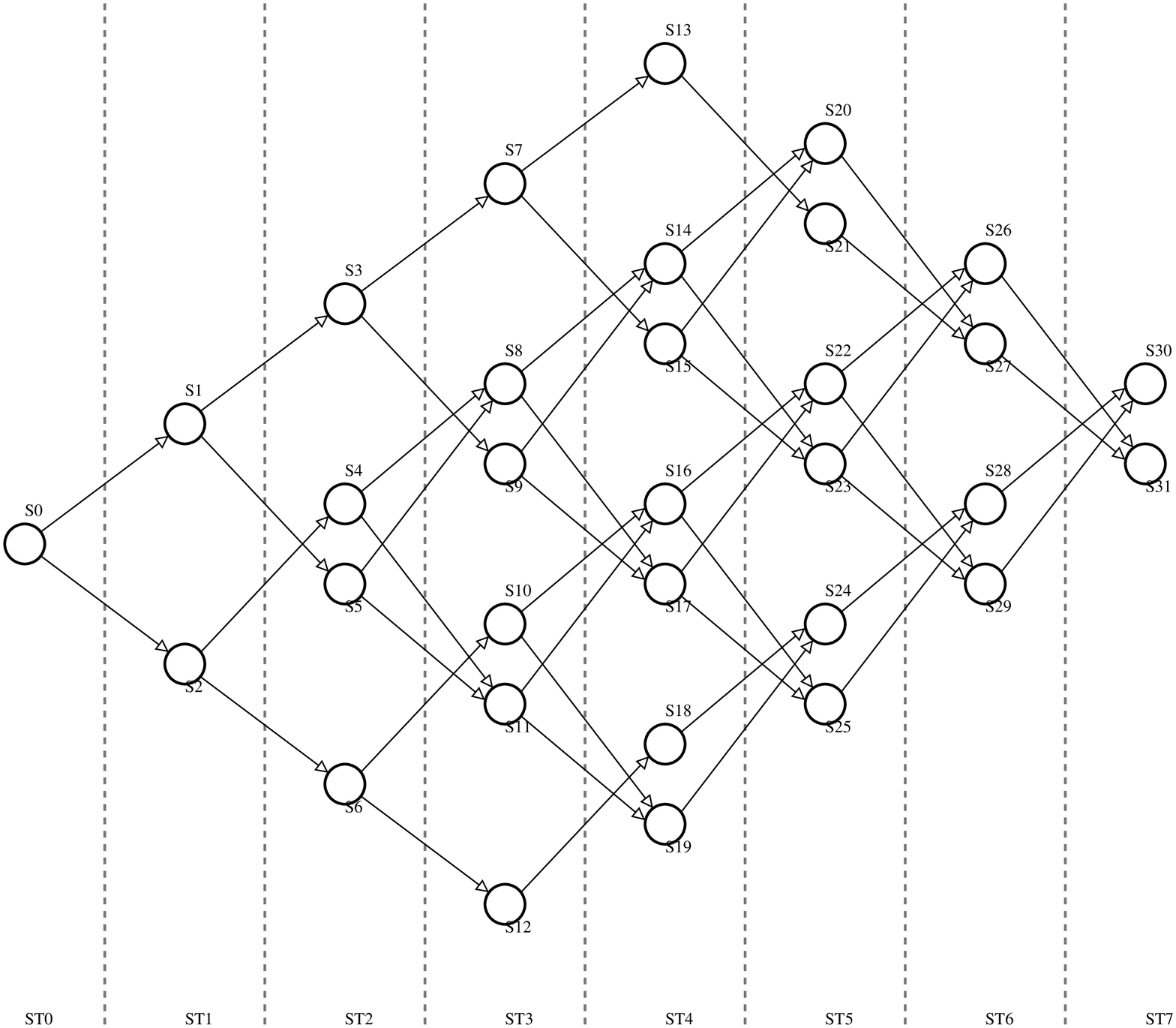}
\caption{State diagram in the case of two classes of jobs, where $N_{1} = 4$ and $N_{2} = 3$, with setup.}
\label{fig:esS2_statediagram}
\end{figure}

The 32 states (from $S0$ to $S31$) in the 7 stages are:

\begin{center}
\begin{tabular}{|ll|}
\hline
\multicolumn{2}{|c|}{\it stage $0$}\\
\hline
{\small $S0$} & $[0 \; 0 \; 0 \; t_{0}]^{T}$\\
\hline
\end{tabular}
\hspace{.5cm}
\begin{tabular}{|ll|}
\hline
\multicolumn{2}{|c|}{\it stage $1$}\\
\hline
{\small $S1$} & $[1 \; 0 \; 1 \; t_{1}]^{T}$\\
{\small $S2$} & $[0 \; 1 \; 2 \; t_{1}]^{T}$\\
\hline
\end{tabular}
\hspace{.5cm}
\begin{tabular}{|ll|}
\hline
\multicolumn{2}{|c|}{\it stage $2$}\\
\hline
{\small $S3$} & $[2 \; 0 \; 1 \; t_{2}]^{T}$\\
{\small $S4$} & $[1 \; 1 \; 1 \; t_{2}]^{T}$\\
{\small $S5$} & $[1 \; 1 \; 2 \; t_{2}]^{T}$\\
{\small $S6$} & $[0 \; 2 \; 2 \; t_{2}]^{T}$\\
\hline
\end{tabular}
\hspace{.5cm}
\begin{tabular}{|ll|}
\hline
\multicolumn{2}{|c|}{\it stage $3$}\\
\hline
{\small $S7$} & $[3 \; 0 \; 1 \; t_{3}]^{T}$\\
{\small $S8$} & $[2 \; 1 \; 1 \; t_{3}]^{T}$\\
{\small $S9$} & $[2 \; 1 \; 2 \; t_{3}]^{T}$\\
{\small $S10$} & $[1 \; 2 \; 1 \; t_{3}]^{T}$\\
{\small $S11$} & $[1 \; 2 \; 2 \; t_{3}]^{T}$\\
{\small $S12$} & $[0 \; 3 \; 2 \; t_{3}]^{T}$\\
\hline
\end{tabular}
\end{center}

\begin{center}
\begin{tabular}{|ll|}
\hline
\multicolumn{2}{|c|}{\it stage $4$}\\
\hline
{\small $S13$} & $[4 \; 0 \; 1 \; t_{4}]^{T}$\\
{\small $S14$} & $[3 \; 1 \; 1 \; t_{4}]^{T}$\\
{\small $S15$} & $[3 \; 1 \; 2 \; t_{4}]^{T}$\\
{\small $S16$} & $[2 \; 2 \; 1 \; t_{4}]^{T}$\\
{\small $S17$} & $[2 \; 2 \; 2 \; t_{4}]^{T}$\\
{\small $S18$} & $[1 \; 3 \; 1 \; t_{4}]^{T}$\\
{\small $S19$} & $[1 \; 3 \; 2 \; t_{4}]^{T}$\\
\hline
\end{tabular}
\hspace{.5cm}
\begin{tabular}{|ll|}
\hline
\multicolumn{2}{|c|}{\it stage $5$}\\
\hline
{\small $S20$} & $[4 \; 1 \; 1 \; t_{5}]^{T}$\\
{\small $S21$} & $[4 \; 1 \; 2 \; t_{5}]^{T}$\\
{\small $S22$} & $[3 \; 2 \; 1 \; t_{5}]^{T}$\\
{\small $S23$} & $[3 \; 2 \; 2 \; t_{5}]^{T}$\\
{\small $S24$} & $[2 \; 3 \; 1 \; t_{5}]^{T}$\\
{\small $S25$} & $[2 \; 3 \; 2 \; t_{5}]^{T}$\\
\hline
\end{tabular}
\hspace{.5cm}
\begin{tabular}{|ll|}
\hline
\multicolumn{2}{|c|}{\it stage $6$}\\
\hline
{\small $S26$} & $[4 \; 2 \; 1 \; t_{6}]^{T}$\\
{\small $S27$} & $[4 \; 2 \; 2 \; t_{6}]^{T}$\\
{\small $S28$} & $[3 \; 3 \; 1 \; t_{6}]^{T}$\\
{\small $S29$} & $[3 \; 3 \; 2 \; t_{6}]^{T}$\\
\hline
\end{tabular}
\hspace{.5cm}
\begin{tabular}{|ll|}
\hline
\multicolumn{2}{|c|}{\it stage $7$}\\
\hline
{\small $S30$} & $[4 \; 3 \; 1 \; t_{7}]^{T}$\\
{\small $S31$} & $[4 \; 3 \; 2 \; t_{7}]^{T}$\\
\hline
\end{tabular}
\end{center}

The application of dynamic programming, in conjunction with the new lemmas, provides the following optimal control strategies.

{\it Remark.} In the following, the time variables $t_{j}$, $j = 0, \ldots, 7$, will be considered $\in \mathbb{R}$, that is, also negative values are taken into account. Negative values of $t_{j}$ can be considered when the strategies are determined in advance with respect to the initial time instant $0$ at which the processing of the jobs starts. In this case, it is possible to exploit the optimal control strategies determined for the negative values of $t_{j}$ to start the execution of the jobs as soon as they become available, even before $0$.

\vspace{12pt}
{\bf Stage $7$ -- State $\boldsymbol{[4 \; 3 \; 2 \; t_{7}]^{T}}$ ($S31$)}

No decision has to be taken in state $[4 \; 3 \; 2 \; t_{7}]^{T}$. The optimal cost-to-go is obviously null, that is
\begin{equation*}
J^{\circ}_{4,3,2} (t_{7}) = 0
\end{equation*}

{\bf Stage $7$ -- State $\boldsymbol{[4 \; 3 \; 1 \; t_{7}]^{T}}$ ($S30$)}

No decision has to be taken in state $[4 \; 3 \; 1 \; t_{7}]^{T}$. The optimal cost-to-go is obviously null, that is
\begin{equation*}
J^{\circ}_{4,3,1} (t_{7}) = 0
\end{equation*}

{\bf Stage $6$ -- State $\boldsymbol{[3 \; 3 \; 2 \; t_{6}]^{T}}$ ($S29$)}

In state $[3 \; 3 \; 2 \; t_{6}]^{T}$ all jobs of class $P_{2}$ have been completed; then the decision about the class of the next job to be executed is mandatory. The cost function to be minimized in this state, with respect to the (continuos) decision variable $\tau$ only (which corresponds to the processing time $pt_{1,4}$), is
\begin{equation*}
\alpha_{1,4} \, \max \{ t_{6} + st_{2,1} + \tau - dd_{1,4} \, , \, 0 \} + \beta_{1} \, ( pt^{\mathrm{nom}}_{1} - \tau ) + sc_{2,1} + J^{\circ}_{4,3,1} (t_{7})
\end{equation*}
that can be written as $f (pt_{1,4} + t_{6}) + g (pt_{1,4})$ being
\begin{equation*}
f (pt_{1,4} + t_{6}) = 0.5 \cdot \max \{ pt_{1,4} + t_{6} - 40.5 \, , \, 0 \} + 1
\end{equation*}
\begin{equation*}
g (pt_{1,4}) = \left\{ \begin{array}{ll}
8 - pt_{1,4} & pt_{1,4} \in [ 4 , 8 )\\
0 & pt_{1,4} \notin [ 4 , 8 )
\end{array} \right.
\end{equation*}
The function $pt^{\circ}_{1,4}(t_{6}) = \arg \min_{pt_{1,4}} \{ f (pt_{1,4} + t_{6}) + g (pt_{1,4}) \} $, with $4 \leq pt_{1,4} \leq 8$, is determined by applying lemma~\ref{lem:xopt}. It is
\begin{equation*}
pt^{\circ}_{1,4}(t_{6}) = x_{\mathrm{e}}(t_{6}) \qquad \text{with} \quad x_{\mathrm{e}}(t_{6}) = 8 \quad \forall \, t_{6}
\end{equation*}

Taking into account the mandatory decision about the class of the next job to be executed, the optimal control strategies for this state are
\begin{equation*}
\delta_{1}^{\circ} (3,3,2, t_{6}) = 1 \quad \forall \, t_{6} \qquad \delta_{2}^{\circ} (3,3,2, t_{6}) = 0 \quad \forall \, t_{6}
\end{equation*}
\begin{equation*}
\tau^{\circ} (3,3,2, t_{6}) = 8 \quad \forall \, t_{6}
\end{equation*}
The optimal control strategy $\tau^{\circ} (3,3,2, t_{6})$ is illustrated in figure~\ref{fig:esS3_tau_3_3_2}.

\begin{figure}[h!]
\centering
\psfrag{f(x)}[Bl][Bl][.8][0]{$\tau^{\circ} (3,3,2, t_{6})$}
\psfrag{x}[bc][Bl][.8][0]{$t_{6}$}
\psfrag{0}[tc][Bl][.8][0]{$0$}
\includegraphics[scale=.2]{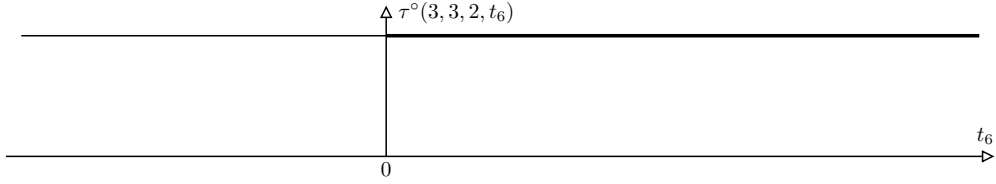}
\caption{Optimal control strategy $\tau^{\circ} (3,3,2, t_{6})$ in state $[ 3 \; 3 \; 2 \; t_{6}]^{T}$.}
\label{fig:esS3_tau_3_3_2}
\end{figure}

The optimal cost-to-go $J^{\circ}_{3,3,2} (t_{6}) = f ( pt^{\circ}_{1,4}(t_{6}) + t_{6} ) + g ( pt^{\circ}_{1,4}(t_{6}) )$, illustrated in figure~\ref{fig:esS2_J_3_3_2}, is provided by lemma~\ref{lem:h(t)}. It is specified by the initial value 1, by the abscissa $\gamma_{1} = 32.5$ at which the slope changes, and by the slope $\mu_{1} = 0.5$ in the interval $[32.5,+\infty)$.

\begin{figure}[h!]
\centering
\psfrag{0}[cc][tc][.8][0]{$0$}
\includegraphics[scale=.8]{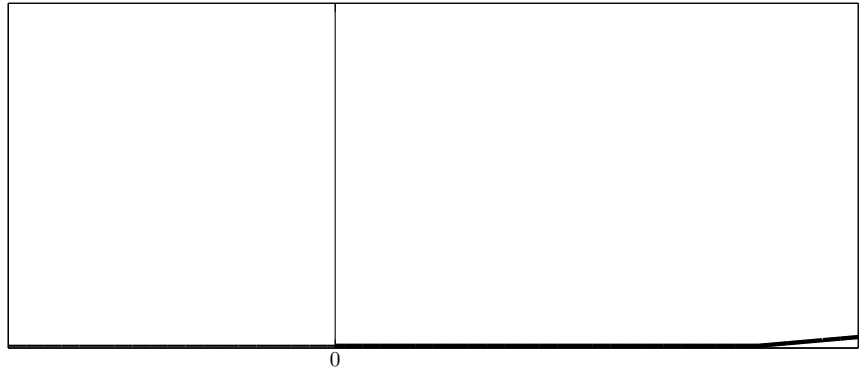}%
\vspace{-12pt}
\caption{Optimal cost-to-go $J^{\circ}_{3,3,2} (t_{6})$ in state $[3 \; 3 \; 2 \; t_{6}]^{T}$.}
\label{fig:esS2_J_3_3_2}
\end{figure}

\vspace{12pt}
{\bf Stage $6$ -- State $\boldsymbol{[3 \; 3 \; 1 \; t_{6}]^{T}}$ ($S28$)}

In state $[3 \; 3 \; 1 \; t_{6}]^{T}$ all jobs of class $P_{2}$ have been completed; then the decision about the class of the next job to be executed is mandatory. The cost function to be minimized in this state, with respect to the (continuos) decision variable $\tau$ only (which corresponds to the processing time $pt_{1,4}$), is
\begin{equation*}
\alpha_{1,4} \, \max \{ t_{6} + st_{1,1} + \tau - dd_{1,4} \, , \, 0 \} + \beta_{1} \, (pt^{\mathrm{nom}}_{1} - \tau) + sc_{1,1} + J^{\circ}_{4,3,1} (t_{7})
\end{equation*}
that can be written as $f (pt_{1,4} + t_{6}) + g (pt_{1,4})$ being
\begin{equation*}
f (pt_{1,4} + t_{6}) = 0.5 \cdot \max \{ pt_{1,4} + t_{6} - 41 \, , \, 0 \}
\end{equation*}
\begin{equation*}
g (pt_{1,4}) = \left\{ \begin{array}{ll}
8 - pt_{1,4} & pt_{1,4} \in [ 4 , 8 )\\
0 & pt_{1,4} \notin [ 4 , 8 )
\end{array} \right.
\end{equation*}
The function $pt^{\circ}_{1,4}(t_{6}) = \arg \min_{pt_{1,4}} \{ f (pt_{1,4} + t_{6}) + g (pt_{1,4}) \} $, with $4 \leq pt_{1,4} \leq 8$, is determined by applying lemma~\ref{lem:xopt}. It is
\begin{equation*}
pt^{\circ}_{1,4}(t_{6}) = x_{\mathrm{e}}(t_{6}) \qquad \text{with} \quad x_{\mathrm{e}}(t_{6}) = 8 \quad \forall \, t_{6}
\end{equation*}

Taking into account the mandatory decision about the class of the next job to be executed, the optimal control strategies for this state are
\begin{equation*}
\delta_{1}^{\circ} (3,3,1, t_{6}) = 1 \quad \forall \, t_{6} \qquad \delta_{2}^{\circ} (3,3,1, t_{6}) = 0 \quad \forall \, t_{6}
\end{equation*}
\begin{equation*}
\tau^{\circ} (3,3,1, t_{6}) = 8 \quad \forall \, t_{6}
\end{equation*}
The optimal control strategy $\tau^{\circ} (3,3,1, t_{6})$ is illustrated in figure~\ref{fig:esS3_tau_3_3_1}.

\begin{figure}[h!]
\centering
\psfrag{f(x)}[Bl][Bl][.8][0]{$\tau^{\circ} (3,3,1, t_{6})$}
\psfrag{x}[bc][Bl][.8][0]{$t_{6}$}
\psfrag{0}[tc][Bl][.8][0]{$0$}
\includegraphics[scale=.2]{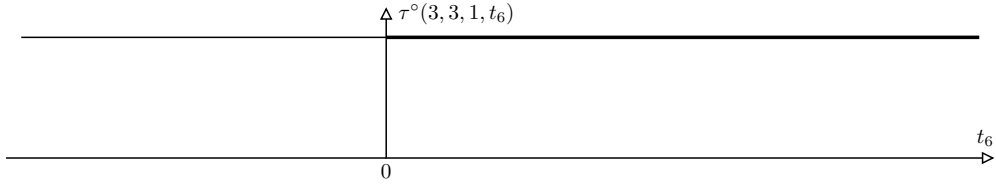}
\caption{Optimal control strategy $\tau^{\circ} (3,3,1, t_{6})$ in state $[ 3 \; 3 \; 1 \; t_{6}]^{T}$.}
\label{fig:esS3_tau_3_3_1}
\end{figure}

The optimal cost-to-go $J^{\circ}_{3,3,1} (t_{6}) = f ( pt^{\circ}_{1,4}(t_{6}) + t_{6} ) + g ( pt^{\circ}_{1,4}(t_{6}) )$, illustrated in figure~\ref{fig:esS2_J_3_3_1}, is provided by lemma~\ref{lem:h(t)}. It is specified by the initial value 0, by the abscissa $\gamma_{1} = 33$ at which the slope changes, and by the slope $\mu_{1} = 0.5$ in the interval $[33,+\infty)$.

\begin{figure}[h!]
\centering
\psfrag{0}[cc][tc][.8][0]{$0$}
\includegraphics[scale=.8]{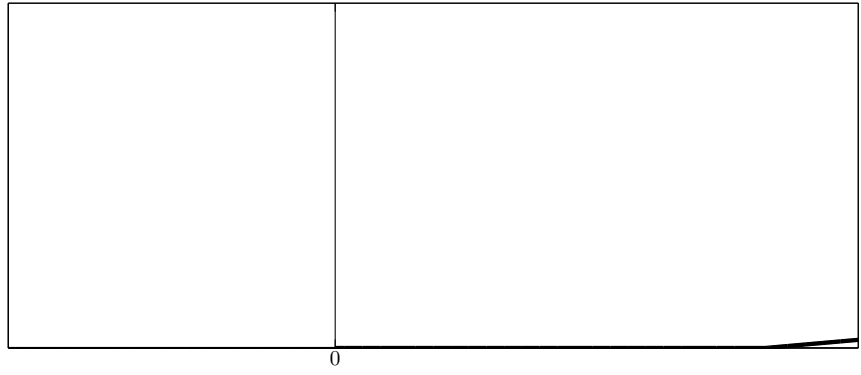}%
\vspace{-12pt}
\caption{Optimal cost-to-go $J^{\circ}_{3,3,1} (t_{6})$ in state $[3 \; 3 \; 1 \; t_{6}]^{T}$.}
\label{fig:esS2_J_3_3_1}
\end{figure}

{\bf Stage $6$ -- State $\boldsymbol{[4 \; 2 \; 2 \; t_{6}]^{T}}$ ($S27$)}

In state $[4 \; 2 \; 2 \; t_{6}]^{T}$ all jobs of class $P_{1}$ have been completed; then the decision about the class of the next job to be executed is mandatory. The cost function to be minimized in this state, with respect to the (continuos) decision variable $\tau$ only (which corresponds to the processing time $pt_{2,3}$), is
\begin{equation*}
\alpha_{2,3} \, \max \{ t_{6} + st_{2,2} + \tau - dd_{2,3} \, , \, 0 \} + \beta_{2} \, (pt^{\mathrm{nom}}_{2} - \tau) + sc_{2,2} + J^{\circ}_{4,3,2} (t_{7})
\end{equation*}
that can be written as $f (pt_{2,3} + t_{6}) + g (pt_{2,3})$ being
\begin{equation*}
f (pt_{2,3} + t_{6}) = \max \{ pt_{2,3} + t_{6} - 38 \, , \, 0 \}
\end{equation*}
\begin{equation*}
g (pt_{2,3}) = \left\{ \begin{array}{ll}
1.5 \cdot (6 - pt_{2,3}) & pt_{2,3} \in [ 4 , 6 )\\
0 & pt_{2,3} \notin [ 4 , 6 )
\end{array} \right.
\end{equation*}
The function $pt^{\circ}_{2,3}(t_{6}) = \arg \min_{pt_{2,3}} \{ f (pt_{2,3} + t_{6}) + g (pt_{2,3}) \} $, with $4 \leq pt_{2,3} \leq 6$, is determined by applying lemma~\ref{lem:xopt}. It is
\begin{equation*}
pt^{\circ}_{2,3}(t_{6}) = x_{\mathrm{e}}(t_{6}) \qquad \text{with} \quad x_{\mathrm{e}}(t_{6}) = 6 \quad \forall \, t_{6}
\end{equation*}

Taking into account the mandatory decision about the class of the next job to be executed, the optimal control strategies for this state are
\begin{equation*}
\delta_{1}^{\circ} (4,2,2, t_{6}) = 0 \quad \forall \, t_{6} \qquad \delta_{2}^{\circ} (4,2,2, t_{6}) = 1 \quad \forall \, t_{6}
\end{equation*}
\begin{equation*}
\tau^{\circ} (4,2,2, t_{6}) = 6 \quad \forall \, t_{6}
\end{equation*}
The optimal control strategy $\tau^{\circ} (4,2,2, t_{6})$ is illustrated in figure~\ref{fig:esS3_tau_4_2_2}.

The optimal cost-to-go $J^{\circ}_{4,2,1} (t_{6}) = f ( pt^{\circ}_{2,3}(t_{6}) + t_{6} ) + g ( pt^{\circ}_{2,3}(t_{6}) )$, illustrated in figure~\ref{fig:esS2_J_4_2_2}, is provided by lemma~\ref{lem:h(t)}. It is specified by the initial value 0, by the abscissa $\gamma_{1} = 32$ at which the slope changes, and by the slope $\mu_{1} = 1$ in the interval $[32,+\infty)$.

\newpage

\begin{figure}[h!]
\centering
\psfrag{f(x)}[Bl][Bl][.8][0]{$\tau^{\circ} (4,2,2, t_{6})$}
\psfrag{x}[bc][Bl][.8][0]{$t_{6}$}
\psfrag{0}[tc][Bl][.8][0]{$0$}
\includegraphics[scale=.2]{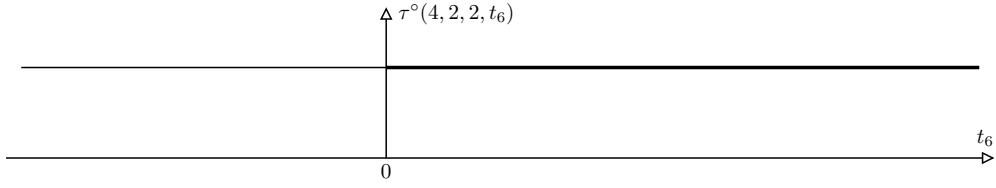}
\caption{Optimal control strategy $\tau^{\circ} (4,2,2, t_{6})$ in state $[ 4 \; 2 \; 2 \; t_{6}]^{T}$.}
\label{fig:esS3_tau_4_2_2}
\end{figure}

\begin{figure}[h!]
\centering
\psfrag{0}[cc][tc][.8][0]{$0$}
\includegraphics[scale=.8]{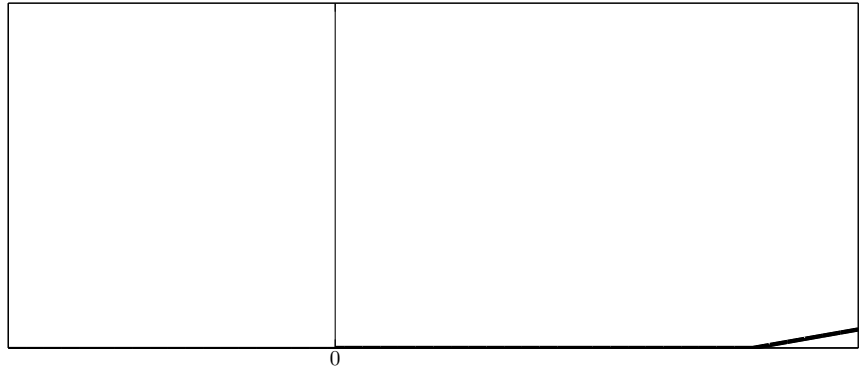}%
\vspace{-12pt}
\caption{Optimal cost-to-go $J^{\circ}_{4,2,2} (t_{6})$ in state $[4 \; 2 \; 2 \; t_{6}]^{T}$.}
\label{fig:esS2_J_4_2_2}
\end{figure}

{\bf Stage $6$ -- State $\boldsymbol{[4 \; 2 \; 1 \; t_{6}]^{T}}$ ($S26$)}

In state $[4 \; 2 \; 1 \; t_{6}]^{T}$ all jobs of class $P_{1}$ have been completed; then the decision about the class of the next job to be executed is mandatory. The cost function to be minimized in this state, with respect to the (continuos) decision variable $\tau$ only (which corresponds to the processing time $pt_{2,3}$), is
\begin{equation*}
\alpha_{2,3} \, \max \{ t_{6} + st_{1,2} + \tau - dd_{2,3} \, , \, 0 \} + \beta_{2} \, (pt^{\mathrm{nom}}_{2} - \tau) + sc_{1,2} + J^{\circ}_{4,3,2} (t_{7})
\end{equation*}
that can be written as $f (pt_{2,3} + t_{6}) + g (pt_{2,3})$ being
\begin{equation*}
f (pt_{2,3} + t_{6}) = \max \{ pt_{2,3} + t_{6} - 37 \, , \, 0 \} + 0.5
\end{equation*}
\begin{equation*}
g (pt_{2,3}) = \left\{ \begin{array}{ll}
1.5 \cdot (6 - pt_{2,3}) & pt_{2,3} \in [ 4 , 6 )\\
0 & pt_{2,3} \notin [ 4 , 6 )
\end{array} \right.
\end{equation*}
The function $pt^{\circ}_{2,3}(t_{6}) = \arg \min_{pt_{2,3}} \{ f (pt_{2,3} + t_{6}) + g (pt_{2,3}) \} $, with $4 \leq pt_{2,3} \leq 6$, is determined by applying lemma~\ref{lem:xopt}. It is
\begin{equation*}
pt^{\circ}_{2,3}(t_{6}) = x_{\mathrm{e}}(t_{6}) \qquad \text{with} \quad x_{\mathrm{e}}(t_{6}) = 6 \quad \forall \, t_{6}
\end{equation*}

Taking into account the mandatory decision about the class of the next job to be executed, the optimal control strategies for this state are
\begin{equation*}
\delta_{1}^{\circ} (4,2,1, t_{6}) = 0 \quad \forall \, t_{6} \qquad \delta_{2}^{\circ} (4,2,1, t_{6}) = 1 \quad \forall \, t_{6}
\end{equation*}
\begin{equation*}
\tau^{\circ} (4,2,1, t_{6}) = 6 \quad \forall \, t_{6}
\end{equation*}
The optimal control strategy $\tau^{\circ} (4,2,1, t_{6})$ is illustrated in figure~\ref{fig:esS3_tau_4_2_1}.

\begin{figure}[h!]
\centering
\psfrag{f(x)}[Bl][Bl][.8][0]{$\tau^{\circ} (4,2,1, t_{6})$}
\psfrag{x}[bc][Bl][.8][0]{$t_{6}$}
\psfrag{0}[tc][Bl][.8][0]{$0$}
\includegraphics[scale=.2]{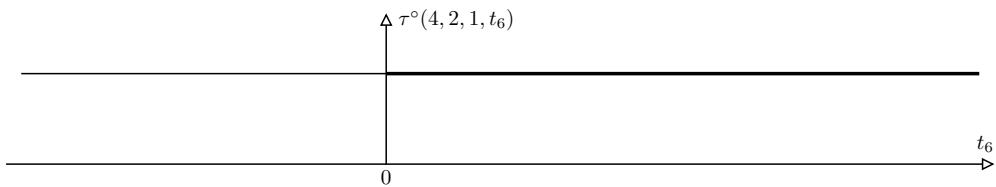}
\caption{Optimal control strategy $\tau^{\circ} (4,2,1, t_{6})$ in state $[ 4 \; 2 \; 1 \; t_{6}]^{T}$.}
\label{fig:esS3_tau_4_2_1}
\end{figure}

The optimal cost-to-go $J^{\circ}_{4,2,1} (t_{6}) = f ( pt^{\circ}_{2,3}(t_{6}) + t_{6} ) + g ( pt^{\circ}_{2,3}(t_{6}) )$, illustrated in figure~\ref{fig:esS2_J_4_2_1}, is provided by lemma~\ref{lem:h(t)}. It is specified by the initial value 0.5, by the abscissa $\gamma_{1} = 31$ at which the slope changes, and by the slope $\mu_{1} = 1$ in the interval $[31,+\infty)$.

\begin{figure}[h!]
\centering
\psfrag{0}[cc][tc][.8][0]{$0$}
\includegraphics[scale=.8]{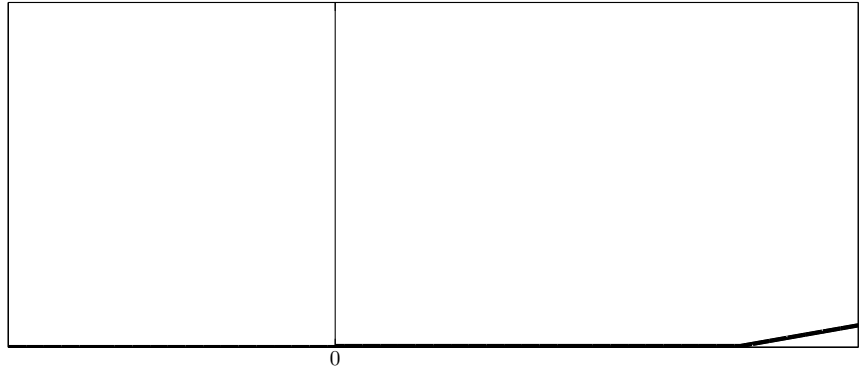}%
\vspace{-12pt}
\caption{Optimal cost-to-go $J^{\circ}_{4,2,1} (t_{6})$ in state $[4 \; 2 \; 1 \; t_{6}]^{T}$.}
\label{fig:esS2_J_4_2_1}
\end{figure}

{\bf Stage $5$ -- State $\boldsymbol{[2 \; 3 \; 2 \; t_{5}]^{T}}$ ($S25$)}

In state $[2 \; 3 \; 2 \; t_{5}]^{T}$ all jobs of class $P_{2}$ have been completed; then the decision about the class of the next job to be executed is mandatory. The cost function to be minimized in this state, with respect to the (continuos) decision variable $\tau$ only (which corresponds to the processing time $pt_{1,3}$), is
\begin{equation*}
\alpha_{1,3} \, \max \{ t_{5} + st_{2,1} + \tau - dd_{1,3} \, , \, 0 \} + \beta_{1} \, (pt^{\mathrm{nom}}_{1} - \tau) + sc_{2,1} + J^{\circ}_{3,3,1} (t_{6})
\end{equation*}
that can be written as $f (pt_{1,3} + t_{5}) + g (pt_{1,3})$ being
\begin{equation*}
f (pt_{1,3} + t_{5}) = 1.5 \cdot \max \{ pt_{1,3} + t_{5} - 28.5 \, , \, 0 \} + 1 + J^{\circ}_{3,3,1} (pt_{1,3} + t_{5} + 0.5)
\end{equation*}
\begin{equation*}
g (pt_{1,3}) = \left\{ \begin{array}{ll}
8 - pt_{1,3} & pt_{1,3} \in [ 4 , 8 )\\
0 & pt_{1,3} \notin [ 4 , 8 )
\end{array} \right.
\end{equation*}
The function $pt^{\circ}_{1,3}(t_{5}) = \arg \min_{pt_{1,3}} \{ f (pt_{1,3} + t_{5}) + g (pt_{1,3}) \} $, with $4 \leq pt_{1,3} \leq 8$, is determined by applying lemma~\ref{lem:xopt}. It is
\begin{equation*}
pt^{\circ}_{1,3}(t_{5}) = x_{\mathrm{e}}(t_{5}) \qquad \text{with} \quad x_{\mathrm{e}}(t_{5}) = \left\{ \begin{array}{ll}
8 &  t_{5} < 20.5\\
-t_{5} + 28.5 & 20.5 \leq t_{5} < 24.5\\
4 & t_{5} \geq 24.5
\end{array} \right.
\end{equation*}

Taking into account the mandatory decision about the class of the next job to be executed, the optimal control strategies for this state are
\begin{equation*}
\delta_{1}^{\circ} (2,3,2, t_{5}) = 1 \quad \forall \, t_{5} \qquad \delta_{2}^{\circ} (2,3,2, t_{5}) = 0 \quad \forall \, t_{5}
\end{equation*}
\begin{equation*}
\tau^{\circ} (2,3,2, t_{5}) = \left\{ \begin{array}{ll}
8 &  t_{5} < 20.5\\
-t_{5} + 28.5 & 20.5 \leq t_{5} < 24.5\\
4 & t_{5} \geq 24.5
\end{array} \right.
\end{equation*}
The optimal control strategy $\tau^{\circ} (2,3,2, t_{5})$ is illustrated in figure~\ref{fig:esS3_tau_2_3_2}.

\begin{figure}[h!]
\centering
\psfrag{f(x)}[Bl][Bl][.8][0]{$\tau^{\circ} (2,3,2, t_{5})$}
\psfrag{x}[bc][Bl][.8][0]{$t_{5}$}
\psfrag{0}[tc][Bl][.8][0]{$0$}
\includegraphics[scale=.2]{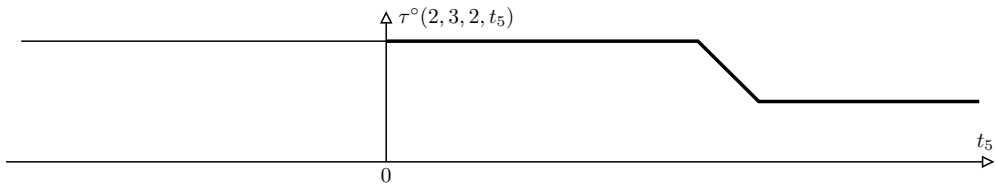}
\caption{Optimal control strategy $\tau^{\circ} (2,3,2, t_{5})$ in state $[ 2 \; 3 \; 2 \; t_{5}]^{T}$.}
\label{fig:esS3_tau_2_3_2}
\end{figure}

The optimal cost-to-go $J^{\circ}_{2,3,2} (t_{5}) = f ( pt^{\circ}_{1,3}(t_{5}) + t_{5} ) + g ( pt^{\circ}_{1,3}(t_{5}) )$, illustrated in figure~\ref{fig:esS2_J_2_3_2}, is provided by lemma~\ref{lem:h(t)}. It is specified by the initial value 1, by the set \{ 20.5, 24.5, 28.5 \} of abscissae $\gamma_{i}$, $i = 1, \ldots, 3$, at which the slope changes, and by the set \{ 1, 1.5, 2 \} of slopes $\mu_{i}$, $i = 1, \ldots, 3$, in the various intervals.

\begin{figure}[h!]
\centering
\psfrag{0}[cc][tc][.8][0]{$0$}
\includegraphics[scale=.8]{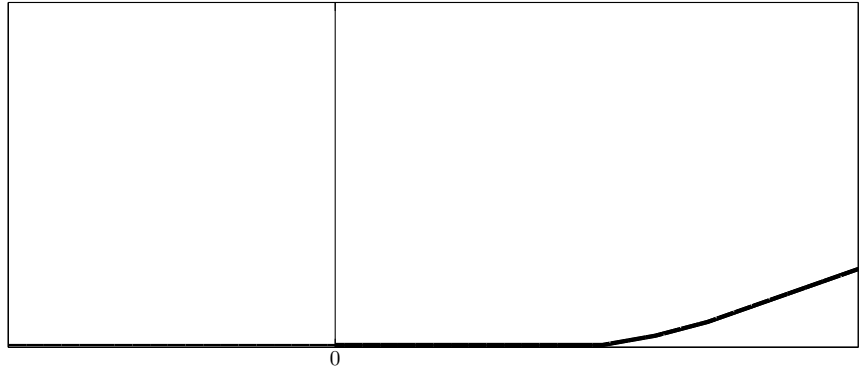}%
\vspace{-12pt}
\caption{Optimal cost-to-go $J^{\circ}_{2,3,2} (t_{5})$ in state $[2 \; 3 \; 2 \; t_{5}]^{T}$.}
\label{fig:esS2_J_2_3_2}
\end{figure}

{\bf Stage $5$ -- State $\boldsymbol{[2 \; 3 \; 1 \; t_{5}]^{T}}$ ($S24$)}

In state $[2 \; 3 \; 1 \; t_{5}]^{T}$ all jobs of class $P_{2}$ have been completed; then the decision about the class of the next job to be executed is mandatory. The cost function to be minimized in this state, with respect to the (continuos) decision variable $\tau$ only (which corresponds to the processing time $pt_{1,3}$), is
\begin{equation*}
\alpha_{1,3} \, \max \{ t_{5} + st_{1,1} + \tau - dd_{1,3} \, , \, 0 \} + \beta_{1} \, (pt^{\mathrm{nom}}_{1} - \tau) + sc_{1,1} + J^{\circ}_{3,3,1} (t_{6})
\end{equation*}
that can be written as $f (pt_{1,3} + t_{5}) + g (pt_{1,3})$ being
\begin{equation*}
f (pt_{1,3} + t_{5}) = 1.5 \cdot \max \{ pt_{1,3} + t_{5} - 29 \, , \, 0 \} + J^{\circ}_{3,3,1} (pt_{1,3} + t_{5})
\end{equation*}
\begin{equation*}
g (pt_{1,3}) = \left\{ \begin{array}{ll}
8 - pt_{1,3} & pt_{1,3} \in [ 4 , 8 )\\
0 & pt_{1,3} \notin [ 4 , 8 )
\end{array} \right.
\end{equation*}
The function $pt^{\circ}_{1,3}(t_{5}) = \arg \min_{pt_{1,3}} \{ f (pt_{1,3} + t_{5}) + g (pt_{1,3}) \} $, with $4 \leq pt_{1,3} \leq 8$, is determined by applying lemma~\ref{lem:xopt}. It is
\begin{equation*}
pt^{\circ}_{1,3}(t_{5}) = x_{\mathrm{e}}(t_{5}) \qquad \text{with} \quad x_{\mathrm{e}}(t_{5}) = \left\{ \begin{array}{ll}
8 &  t_{5} < 21\\
-t_{5} + 29 & 21 \leq t_{5} < 25\\
4 & t_{5} \geq 25
\end{array} \right.
\end{equation*}

Taking into account the mandatory decision about the class of the next job to be executed, the optimal control strategies for this state are
\begin{equation*}
\delta_{1}^{\circ} (2,3,1, t_{5}) = 1 \quad \forall \, t_{5} \qquad \delta_{2}^{\circ} (2,3,1, t_{5}) = 0 \quad \forall \, t_{5}
\end{equation*}
\begin{equation*}
\tau^{\circ} (2,3,1, t_{5}) = \left\{ \begin{array}{ll}
8 &  t_{5} < 21\\
-t_{5} + 29 & 21 \leq t_{5} < 25\\
4 & t_{5} \geq 25
\end{array} \right.
\end{equation*}
The optimal control strategy $\tau^{\circ} (2,3,1, t_{5})$ is illustrated in figure~\ref{fig:esS3_tau_2_3_1}.

\begin{figure}[h!]
\centering
\psfrag{f(x)}[Bl][Bl][.8][0]{$\tau^{\circ} (2,3,1, t_{5})$}
\psfrag{x}[bc][Bl][.8][0]{$t_{5}$}
\psfrag{0}[tc][Bl][.8][0]{$0$}
\includegraphics[scale=.2]{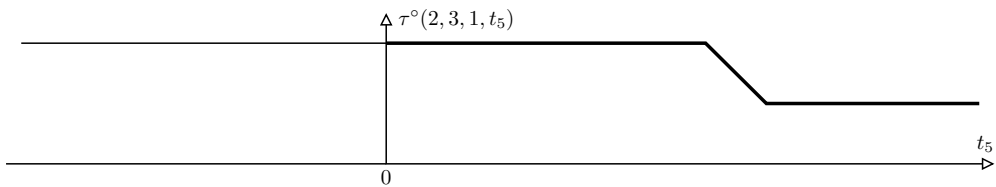}
\caption{Optimal control strategy $\tau^{\circ} (2,3,1, t_{5})$ in state $[ 2 \; 3 \; 1 \; t_{5}]^{T}$.}
\label{fig:esS3_tau_2_3_1}
\end{figure}

The optimal cost-to-go $J^{\circ}_{2,3,1} (t_{5}) = f ( pt^{\circ}_{1,3}(t_{5}) + t_{5} ) + g ( pt^{\circ}_{1,3}(t_{5}) )$, illustrated in figure~\ref{fig:esS2_J_2_3_1}, is provided by lemma~\ref{lem:h(t)}. It is specified by the initial value 0, by the set \{ 21, 25, 29 \} of abscissae $\gamma_{i}$, $i = 1, \ldots, 3$, at which the slope changes, and by the set \{ 1, 1.5, 2 \} of slopes $\mu_{i}$, $i = 1, \ldots, 3$, in the various intervals.

\begin{figure}[h!]
\centering
\psfrag{0}[cc][tc][.8][0]{$0$}
\includegraphics[scale=.8]{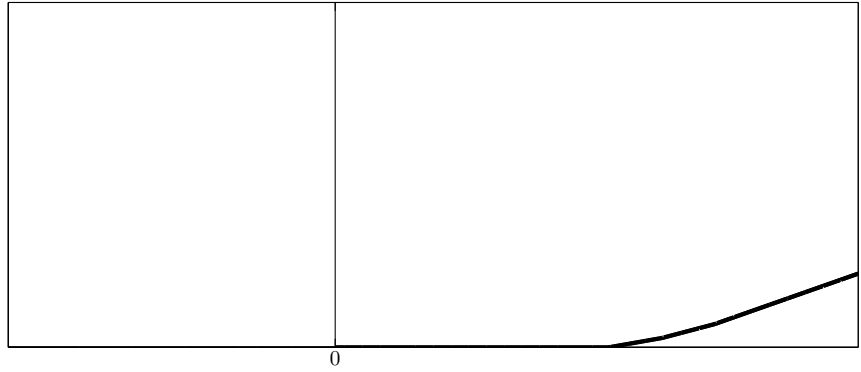}%
\vspace{-12pt}
\caption{Optimal cost-to-go $J^{\circ}_{2,3,1} (t_{5})$ in state $[2 \; 3 \; 1 \; t_{5}]^{T}$.}
\label{fig:esS2_J_2_3_1}
\end{figure}

{\bf Stage $5$ -- State $\boldsymbol{[3 \; 2 \; 2 \; t_{5}]^{T}}$ ($S23$)}

In state $[3 \; 2 \; 2 \; t_{5}]^{T}$, the cost function to be minimized, with respect to the (continuos) decision variable $\tau$ and to the (binary) decision variables $\delta_{1}$ and $\delta_{2}$ is
\begin{equation*}
\begin{split}
&\delta_{1} \big[ \alpha_{1,4} \, \max \{ t_{5} + st_{2,1} + \tau - dd_{1,4} \, , \, 0 \} + \beta_{1} \, ( pt^{\mathrm{nom}}_{1} - \tau ) + sc_{2,1} + J^{\circ}_{4,2,1} (t_{6}) \big] +\\
&+ \delta_{2} \big[ \alpha_{2,3} \, \max \{ t_{5} + st_{2,2} + \tau - dd_{2,3} \, , \, 0 \} + \beta_{2} \, ( pt^{\mathrm{nom}}_{2} - \tau ) + sc_{2,2} + J^{\circ}_{3,3,2} (t_{6}) \big]
\end{split}
\end{equation*}

{\it Case i)} in which it is assumed $\delta_{1} = 1$ (and $\delta_{2} = 0$).

In this case, it is necessary to minimize, with respect to the (continuos) decision variable $\tau$ which corresponds to the processing time $pt_{1,4}$, the following function
\begin{equation*}
\alpha_{1,4} \, \max \{ t_{5} + st_{2,1} + \tau - dd_{1,4} \, , \, 0 \} + \beta_{1} \, (pt^{\mathrm{nom}}_{1} - \tau) + sc_{2,1} + J^{\circ}_{4,2,1} (t_{6})
\end{equation*}
that can be written as $f (pt_{1,4} + t_{5}) + g (pt_{1,4})$ being
\begin{equation*}
f (pt_{1,4} + t_{5}) = 0.5 \cdot \max \{ pt_{1,4} + t_{5} - 40.5 \, , \, 0 \} + 1 + J^{\circ}_{4,2,1} (pt_{1,4} + t_{5} + 0.5)
\end{equation*}
\begin{equation*}
g (pt_{1,4}) = \left\{ \begin{array}{ll}
8 - pt_{1,4} & pt_{1,4} \in [ 4 , 8 )\\
0 & pt_{1,4} \notin [ 4 , 8 )
\end{array} \right.
\end{equation*}
The function $pt^{\circ}_{1,4}(t_{5}) = \arg \min_{pt_{1,4}} \{ f (pt_{1,4} + t_{5}) + g (pt_{1,4}) \} $, with $4 \leq pt_{1,4} \leq 8$, is determined by applying lemma~\ref{lem:xopt}. It is (see figure~\ref{fig:esS3_tau_3_2_2_pt_1_4})
\begin{equation*}
pt^{\circ}_{1,4}(t_{5}) = x_{\mathrm{e}}(t_{5}) \qquad \text{with} \quad x_{\mathrm{e}}(t_{5}) = \left\{ \begin{array}{ll}
8 &  t_{5} < 22.5\\
-t_{5} + 30.5 & 22.5 \leq t_{5} < 26.5\\
4 & t_{5} \geq 26.5
\end{array} \right.
\end{equation*}

\begin{figure}[h!]
\centering
\psfrag{f(x)}[Bl][Bl][.8][0]{$pt^{\circ}_{1,4}(t_{5})$}
\psfrag{x}[bc][Bl][.8][0]{$t_{5}$}
\psfrag{0}[tc][Bl][.8][0]{$0$}
\includegraphics[scale=.2]{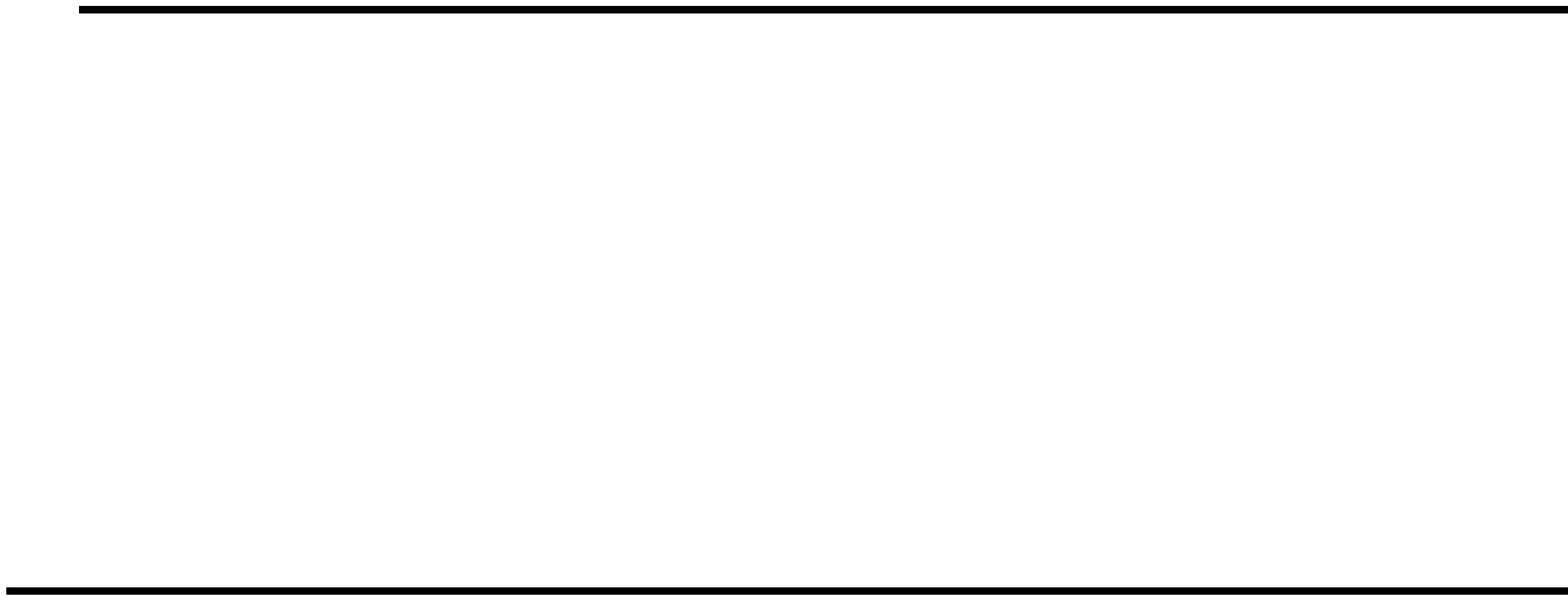}
\caption{Optimal processing time $pt^{\circ}_{1,4}(t_{5})$, under the assumption $\delta_{1} = 1$ in state $[ 3 \; 2 \; 2 \; t_{5}]^{T}$.}
\label{fig:esS3_tau_3_2_2_pt_1_4}
\end{figure}

The conditioned cost-to-go $J^{\circ}_{3,2,2} (t_{5} \mid \delta_{1}=1) = f ( pt^{\circ}_{1,4}(t_{5}) + t_{5} ) + g ( pt^{\circ}_{1,4}(t_{5}) )$, illustrated in figure~\ref{fig:esS3_J_3_2_2_min}, is provided by lemma~\ref{lem:h(t)}. It is specified by the initial value 1.5, by the set \{ 22.5, 36.5 \} of abscissae $\gamma_{i}$, $i = 1, \ldots, 2$, at which the slope changes, and by the set \{ 1, 1.5 \} of slopes $\mu_{i}$, $i = 1, \ldots, 2$, in the various intervals.

{\it Case ii)} in which it is assumed $\delta_{2} = 1$ (and $\delta_{1} = 0$).

In this case, it is necessary to minimize, with respect to the (continuos) decision variable $\tau$ which corresponds to the processing time $pt_{2,3}$, the following function
\begin{equation*}
\alpha_{2,3} \, \max \{ t_{5} + st_{2,2} + \tau - dd_{2,3} \, , \, 0 \} + \beta_{2} \, (pt^{\mathrm{nom}}_{2} - \tau) + sc_{2,2} + J^{\circ}_{3,3,2} (t_{6})
\end{equation*}
that can be written as $f (pt_{2,3} + t_{5}) + g (pt_{2,3})$ being
\begin{equation*}
f (pt_{2,3} + t_{5}) = \max \{ pt_{2,3} + t_{5} - 38 \, , \, 0 \} + J^{\circ}_{3,3,2} (pt_{2,3} + t_{5})
\end{equation*}
\begin{equation*}
g (pt_{2,3}) = \left\{ \begin{array}{ll}
1.5 \cdot (6 - pt_{2,3}) & pt_{2,3} \in [ 4 , 6 )\\
0 & pt_{2,3} \notin [ 4 , 6 )
\end{array} \right.
\end{equation*}
The function $pt^{\circ}_{2,3}(t_{5}) = \arg \min_{pt_{2,3}} \{ f (pt_{2,3} + t_{5}) + g (pt_{2,3}) \} $, with $4 \leq pt_{2,3} \leq 6$, is determined by applying lemma~\ref{lem:xopt}. It is (see figure~\ref{fig:esS3_tau_3_2_2_pt_2_3})
\begin{equation*}
pt^{\circ}_{2,3}(t_{5}) = x_{\mathrm{e}}(t_{5}) \qquad \text{with} \quad x_{\mathrm{e}}(t_{5}) = \left\{ \begin{array}{ll}
6 &  t_{5} < 32\\
-t_{5} + 38 & 32 \leq t_{5} < 34\\
4 & t_{5} \geq 34
\end{array} \right.
\end{equation*}

\begin{figure}[h!]
\centering
\psfrag{f(x)}[Bl][Bl][.8][0]{$pt^{\circ}_{2,3}(t_{5})$}
\psfrag{x}[bc][Bl][.8][0]{$t_{5}$}
\psfrag{0}[tc][Bl][.8][0]{$0$}
\includegraphics[scale=.2]{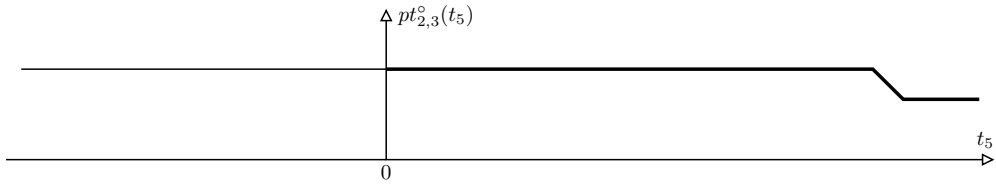}
\caption{Optimal processing time $pt^{\circ}_{2,3}(t_{5})$, under the assumption $\delta_{2} = 1$ in state $[ 3 \; 2 \; 2 \; t_{5}]^{T}$.}
\label{fig:esS3_tau_3_2_2_pt_2_3}
\end{figure}

The conditioned cost-to-go $J^{\circ}_{3,2,2} (t_{5} \mid \delta_{2}=1) = f ( pt^{\circ}_{2,3}(t_{5}) + t_{5} ) + g ( pt^{\circ}_{2,3}(t_{5}) )$, illustrated in figure~\ref{fig:esS3_J_3_2_2_min}, is provided by lemma~\ref{lem:h(t)}. It is specified by the initial value 1, by the set \{ 26.5, 32 \} of abscissae $\gamma_{i}$, $i = 1, \ldots, 2$, at which the slope changes, and by the set \{ 0.5, 1.5 \} of slopes $\mu_{i}$, $i = 1, \ldots, 2$, in the various intervals.

\begin{figure}[h!]
\centering
\psfrag{J1}[br][Bc][.8][0]{$J^{\circ}_{3,2,2} (t_{5} \mid \delta_{1} = 1)$}
\psfrag{J2}[Bl][Bc][.8][0]{$J^{\circ}_{3,2,2} (t_{5} \mid \delta_{2} = 1)$}
\psfrag{0}[cc][tc][.7][0]{$0$}\includegraphics[scale=.6]{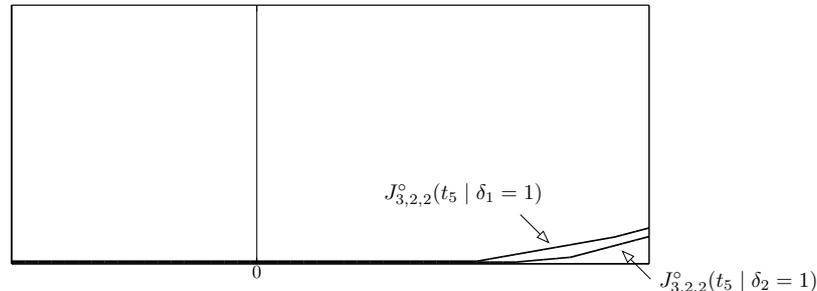}%
\vspace{-12pt}
\caption{Conditioned costs-to-go $J^{\circ}_{3,2,2} (t_{5} \mid \delta_{1} = 1)$ and $J^{\circ}_{3,2,2} (t_{5} \mid \delta_{2} = 1)$ in state $[3 \; 2 \; 2 \; t_{5}]^{T}$.}
\label{fig:esS3_J_3_2_2_min}
\end{figure}

In order to find the optimal cost-to-go $J^{\circ}_{3,2,2} (t_{5})$, it is necessary to carry out the following minimization
\begin{equation*}
J^{\circ}_{3,2,2} (t_{5}) = \min \big\{ J^{\circ}_{3,2,2} (t_{5} \mid \delta_{1} = 1) \, , \, J^{\circ}_{3,2,2} (t_{5} \mid \delta_{2} = 1) \big\}
\end{equation*}
which provides, in accordance with lemma~\ref{lem:min}, the continuous, nondecreasing, piecewise linear function illustrated in figure~\ref{fig:esS3_J_3_2_2}.

\begin{figure}[h!]
\centering
\psfrag{0}[cc][tc][.8][0]{$0$}
\includegraphics[scale=.8]{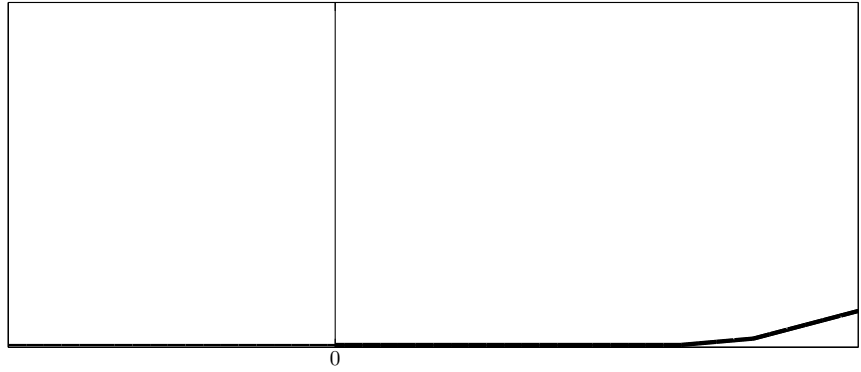}%
\vspace{-12pt}
\caption{Optimal cost-to-go $J^{\circ}_{3,2,2} (t_{5})$ in state $[3 \; 2 \; 2 \; t_{5}]^{T}$.}
\label{fig:esS3_J_3_2_2}
\end{figure}

The function $J^{\circ}_{3,2,2} (t_{5})$ is specified by the initial value 1, by the set \{ 26.5, 32 \} of abscissae $\gamma_{i}$, $i = 1, \ldots, 2$, at which the slope changes, and by the set \{ 0.5, 1.5 \} of slopes $\mu_{i}$, $i = 1, \ldots, 2$, in the various intervals.

Since $J^{\circ}_{3,2,2} (t_{5} \mid \delta_{2} = 1)$ is always the minimum (see again figure~\ref{fig:esS3_J_3_2_2_min}), the optimal control strategies for this state are
\begin{equation*}
\delta_{1}^{\circ} (3,2,2, t_{5}) = 0 \quad \forall \, t_{5} \qquad \delta_{2}^{\circ} (3,2,2, t_{5}) = 1 \quad \forall \, t_{5}
\end{equation*}
\begin{equation*}
\tau^{\circ} (3,2,2, t_{5}) = \left\{ \begin{array}{ll}
6 &  t_{5} < 32\\
-t_{5} + 38 & 32 \leq t_{5} < 34\\
4 & t_{5} \geq 34
\end{array} \right.
\end{equation*}

The optimal control strategy $\tau^{\circ} (3,2,2, t_{5})$ is illustrated in figure~\ref{fig:esS3_tau_3_2_2}.

\begin{figure}[h!]
\centering
\psfrag{f(x)}[Bl][Bl][.8][0]{$\tau^{\circ} (3,2,2, t_{5})$}
\psfrag{x}[bc][Bl][.8][0]{$t_{5}$}
\psfrag{0}[tc][Bl][.8][0]{$0$}
\includegraphics[scale=.2]{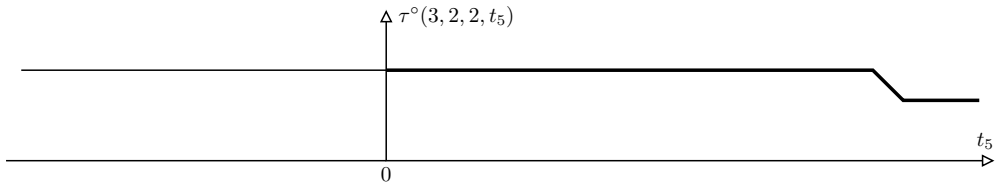}
\caption{Optimal control strategy $\tau^{\circ} (3,2,2, t_{5})$ in state $[ 3 \; 2 \; 2 \; t_{5}]^{T}$.}
\label{fig:esS3_tau_3_2_2}
\end{figure}

{\bf Stage $5$ -- State $\boldsymbol{[3 \; 2 \; 1 \; t_{5}]^{T}}$ ($S22$)}

In state $[3 \; 2 \; 1 \; t_{5}]^{T}$, the cost function to be minimized, with respect to the (continuos) decision variable $\tau$ and to the (binary) decision variables $\delta_{1}$ and $\delta_{2}$ is
\begin{equation*}
\begin{split}
&\delta_{1} \big[ \alpha_{1,4} \, \max \{ t_{5} + st_{1,1} + \tau - dd_{1,4} \, , \, 0 \} + \beta_{1} \, ( pt^{\mathrm{nom}}_{1} - \tau ) + sc_{1,1} + J^{\circ}_{4,2,1} (t_{6}) \big] +\\
&+ \delta_{2} \big[ \alpha_{2,3} \, \max \{ t_{5} + st_{1,2} + \tau - dd_{2,3} \, , \, 0 \} + \beta_{2} \, ( pt^{\mathrm{nom}}_{2} - \tau ) + sc_{1,2} + J^{\circ}_{3,3,2} (t_{6}) \big]
\end{split}
\end{equation*}

{\it Case i)} in which it is assumed $\delta_{1} = 1$ (and $\delta_{2} = 0$).

In this case, it is necessary to minimize, with respect to the (continuos) decision variable $\tau$ which corresponds to the processing time $pt_{1,4}$, the following function
\begin{equation*}
\alpha_{1,4} \, \max \{ t_{5} + st_{1,1} + \tau - dd_{1,4} \, , \, 0 \} + \beta_{1} \, (pt^{\mathrm{nom}}_{1} - \tau) + sc_{1,1} + J^{\circ}_{4,2,1} (t_{6})
\end{equation*}
that can be written as $f (pt_{1,4} + t_{5}) + g (pt_{1,4})$ being
\begin{equation*}
f (pt_{1,4} + t_{5}) = 0.5 \cdot \max \{ pt_{1,4} + t_{5} - 41 \, , \, 0 \} + J^{\circ}_{4,2,1} (pt_{1,4} + t_{5})
\end{equation*}
\begin{equation*}
g (pt_{1,4}) = \left\{ \begin{array}{ll}
8 - pt_{1,4} & pt_{1,4} \in [ 4 , 8 )\\
0 & pt_{1,4} \notin [ 4 , 8 )
\end{array} \right.
\end{equation*}
The function $pt^{\circ}_{1,4}(t_{5}) = \arg \min_{pt_{1,4}} \{ f (pt_{1,4} + t_{5}) + g (pt_{1,4}) \} $, with $4 \leq pt_{1,4} \leq 8$, is determined by applying lemma~\ref{lem:xopt}. It is (see figure~\ref{fig:esS3_tau_3_2_1_pt_1_4})
\begin{equation*}
pt^{\circ}_{1,4}(t_{5}) = x_{\mathrm{e}}(t_{5}) \qquad \text{with} \quad x_{\mathrm{e}}(t_{5}) = \left\{ \begin{array}{ll}
8 &  t_{5} < 23\\
-t_{5} + 31 & 23 \leq t_{5} < 27\\
4 & t_{5} \geq 27
\end{array} \right.
\end{equation*}

\begin{figure}[h!]
\centering
\psfrag{f(x)}[Bl][Bl][.8][0]{$pt^{\circ}_{1,4}(t_{5})$}
\psfrag{x}[bc][Bl][.8][0]{$t_{5}$}
\psfrag{0}[tc][Bl][.8][0]{$0$}
\includegraphics[scale=.2]{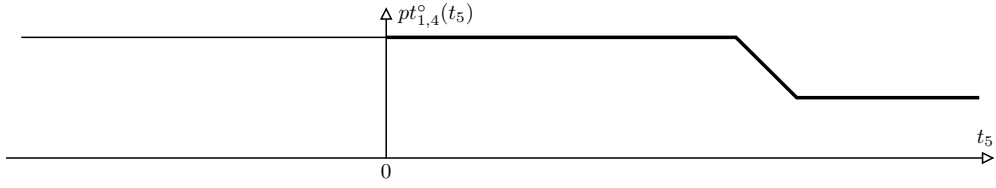}
\caption{Optimal processing time $pt^{\circ}_{1,4}(t_{5})$, under the assumption $\delta_{1} = 1$ in state $[ 3 \; 2 \; 1 \; t_{5}]^{T}$.}
\label{fig:esS3_tau_3_2_1_pt_1_4}
\end{figure}

The conditioned cost-to-go $J^{\circ}_{3,2,1} (t_{5} \mid \delta_{1}=1) = f ( pt^{\circ}_{1,4}(t_{5}) + t_{5} ) + g ( pt^{\circ}_{1,4}(t_{5}) )$, illustrated in figure~\ref{fig:esS3_J_3_2_1_min}, is provided by lemma~\ref{lem:h(t)}. It is specified by the initial value 0.5, by the set \{ 23, 37 \} of abscissae $\gamma_{i}$, $i = 1, \ldots, 2$, at which the slope changes, and by the set \{ 1, 1.5 \} of slopes $\mu_{i}$, $i = 1, \ldots, 2$, in the various intervals.

{\it Case ii)} in which it is assumed $\delta_{2} = 1$ (and $\delta_{1} = 0$).

In this case, it is necessary to minimize, with respect to the (continuos) decision variable $\tau$ which corresponds to the processing time $pt_{2,3}$, the following function
\begin{equation*}
\alpha_{2,3} \, \max \{ t_{5} + st_{1,2} + \tau - dd_{2,3} \, , \, 0 \} + \beta_{2} \, (pt^{\mathrm{nom}}_{2} - \tau) + sc_{1,2} + J^{\circ}_{3,3,2} (t_{6})
\end{equation*}
that can be written as $f (pt_{2,3} + t_{5}) + g (pt_{2,3})$ being
\begin{equation*}
f (pt_{2,3} + t_{5}) = \max \{ pt_{2,3} + t_{5} - 37 \, , \, 0 \} + 0.5 + J^{\circ}_{3,3,2} (pt_{2,3} + t_{5}+1)
\end{equation*}
\begin{equation*}
g (pt_{2,3}) = \left\{ \begin{array}{ll}
1.5 \cdot (6 - pt_{2,3}) & pt_{2,3} \in [ 4 , 6 )\\
0 & pt_{2,3} \notin [ 4 , 6 )
\end{array} \right.
\end{equation*}
The function $pt^{\circ}_{2,3}(t_{5}) = \arg \min_{pt_{2,3}} \{ f (pt_{2,3} + t_{5}) + g (pt_{2,3}) \} $, with $4 \leq pt_{2,3} \leq 6$, is determined by applying lemma~\ref{lem:xopt}. It is (see figure~\ref{fig:esS3_tau_3_2_1_pt_2_3})
\begin{equation*}
pt^{\circ}_{2,3}(t_{5}) = x_{\mathrm{e}}(t_{5}) \qquad \text{with} \quad x_{\mathrm{e}}(t_{5}) = \left\{ \begin{array}{ll}
6 &  t_{5} < 31\\
-t_{5} + 37 & 31 \leq t_{5} < 33\\
4 & t_{5} \geq 33
\end{array} \right.
\end{equation*}

\begin{figure}[h!]
\centering
\psfrag{f(x)}[Bl][Bl][.8][0]{$pt^{\circ}_{2,3}(t_{5})$}
\psfrag{x}[bc][Bl][.8][0]{$t_{5}$}
\psfrag{0}[tc][Bl][.8][0]{$0$}
\includegraphics[scale=.2]{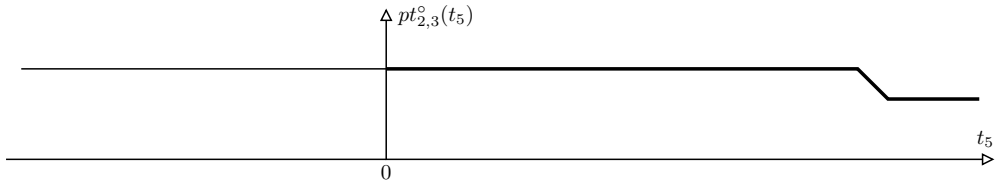}
\caption{Optimal processing time $pt^{\circ}_{2,3}(t_{5})$, under the assumption $\delta_{2} = 1$ in state $[ 3 \; 2 \; 1 \; t_{5}]^{T}$.}
\label{fig:esS3_tau_3_2_1_pt_2_3}
\end{figure}

The conditioned cost-to-go $J^{\circ}_{3,2,1} (t_{5} \mid \delta_{2}=1) = f ( pt^{\circ}_{2,3}(t_{5}) + t_{5} ) + g ( pt^{\circ}_{2,3}(t_{5}) )$, illustrated in figure~\ref{fig:esS3_J_3_2_1_min}, is provided by lemma~\ref{lem:h(t)}. It is specified by the initial value 1.5, by the set \{ 25.5, 31 \} of abscissae $\gamma_{i}$, $i = 1, \ldots, 2$, at which the slope changes, and by the set \{ 0.5, 1.5 \} of slopes $\mu_{i}$, $i = 1, \ldots, 2$, in the various intervals.

\begin{figure}[h!]
\centering
\psfrag{J1}[br][Bc][.8][0]{$J^{\circ}_{3,2,1} (t_{5} \mid \delta_{1} = 1)$}
\psfrag{J2}[Bl][Bc][.8][0]{$J^{\circ}_{3,2,1} (t_{5} \mid \delta_{2} = 1)$}
\psfrag{0}[cc][tc][.7][0]{$0$}\includegraphics[scale=.6]{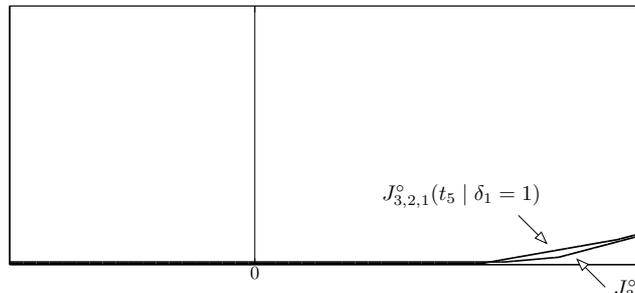}%
\vspace{-12pt}
\caption{Conditioned costs-to-go $J^{\circ}_{3,2,1} (t_{5} \mid \delta_{1} = 1)$ and $J^{\circ}_{3,2,1} (t_{5} \mid \delta_{2} = 1)$ in state $[3 \; 2 \; 1 \; t_{5}]^{T}$.}
\label{fig:esS3_J_3_2_1_min}
\end{figure}

In order to find the optimal cost-to-go $J^{\circ}_{3,2,1} (t_{5})$, it is necessary to carry out the following minimization
\begin{equation*}
J^{\circ}_{3,2,1} (t_{5}) = \min \big\{ J^{\circ}_{3,2,1} (t_{5} \mid \delta_{1} = 1) \, , \, J^{\circ}_{3,2,1} (t_{5} \mid \delta_{2} = 1) \big\}
\end{equation*}
which provides, in accordance with lemma~\ref{lem:min}, the continuous, nondecreasing, piecewise linear function illustrated in figure~\ref{fig:esS3_J_3_2_1}.

\begin{figure}[h!]
\centering
\psfrag{0}[cc][tc][.8][0]{$0$}
\includegraphics[scale=.8]{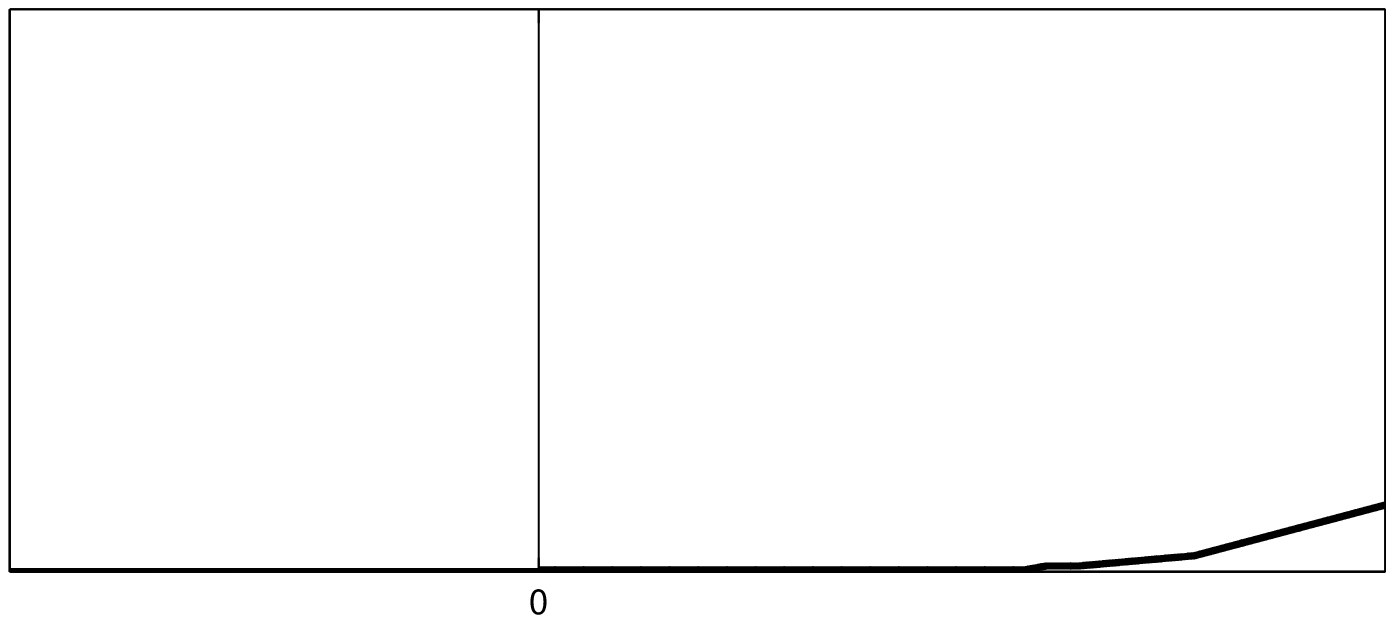}%
\vspace{-12pt}
\caption{Optimal cost-to-go $J^{\circ}_{3,2,1} (t_{5})$ in state $[3 \; 2 \; 1 \; t_{5}]^{T}$.}
\label{fig:esS3_J_3_2_1}
\end{figure}

The function $J^{\circ}_{3,2,1} (t_{5})$ is specified by the initial value 0.5, by the set \{ 23, 24, 25.5, 31 \} of abscissae $\gamma_{i}$, $i = 1, \ldots, 4$, at which the slope changes, and by the set \{ 1, 0, 0.5, 1.5 \} of slopes $\mu_{i}$, $i = 1, \ldots, 4$, in the various intervals.

Since $J^{\circ}_{3,2,1} (t_{5} \mid \delta_{1} = 1)$ is the minimum in $(-\infty,24)$, and $J^{\circ}_{3,2,1} (t_{5} \mid \delta_{2} = 1)$ is the minimum in $[24,+\infty)$, the optimal control strategies for this state are
\begin{equation*}
\delta_{1}^{\circ} (3,2,1, t_{5}) = \left\{ \begin{array}{ll}
1 &  t_{5} < 24\\
0 & t_{5} \geq 24
\end{array} \right. \qquad \delta_{2}^{\circ} (3,2,1, t_{5}) = \left\{ \begin{array}{ll}
0 &  t_{5} < 24\\
1 & t_{5} \geq 24
\end{array} \right.
\end{equation*}
\begin{equation*}
\tau^{\circ} (3,2,1, t_{5}) = \left\{ \begin{array}{ll}
8 &  t_{5} < 23\\
-t_{5} + 31 & 23 \leq t_{5} < 24\\
6 & 24 \leq t_{5} < 31\\
-t_{5} + 37 & 31 \leq t_{5} < 33\\
4 & t_{5} \geq 33
\end{array} \right.
\end{equation*}

The optimal control strategy $\tau^{\circ} (3,2,1, t_{5})$ is illustrated in figure~\ref{fig:esS3_tau_3_2_1}.

\begin{figure}[h!]
\centering
\psfrag{f(x)}[Bl][Bl][.8][0]{$\tau^{\circ} (3,2,1, t_{5})$}
\psfrag{x}[bc][Bl][.8][0]{$t_{5}$}
\psfrag{0}[tc][Bl][.8][0]{$0$}
\includegraphics[scale=.2]{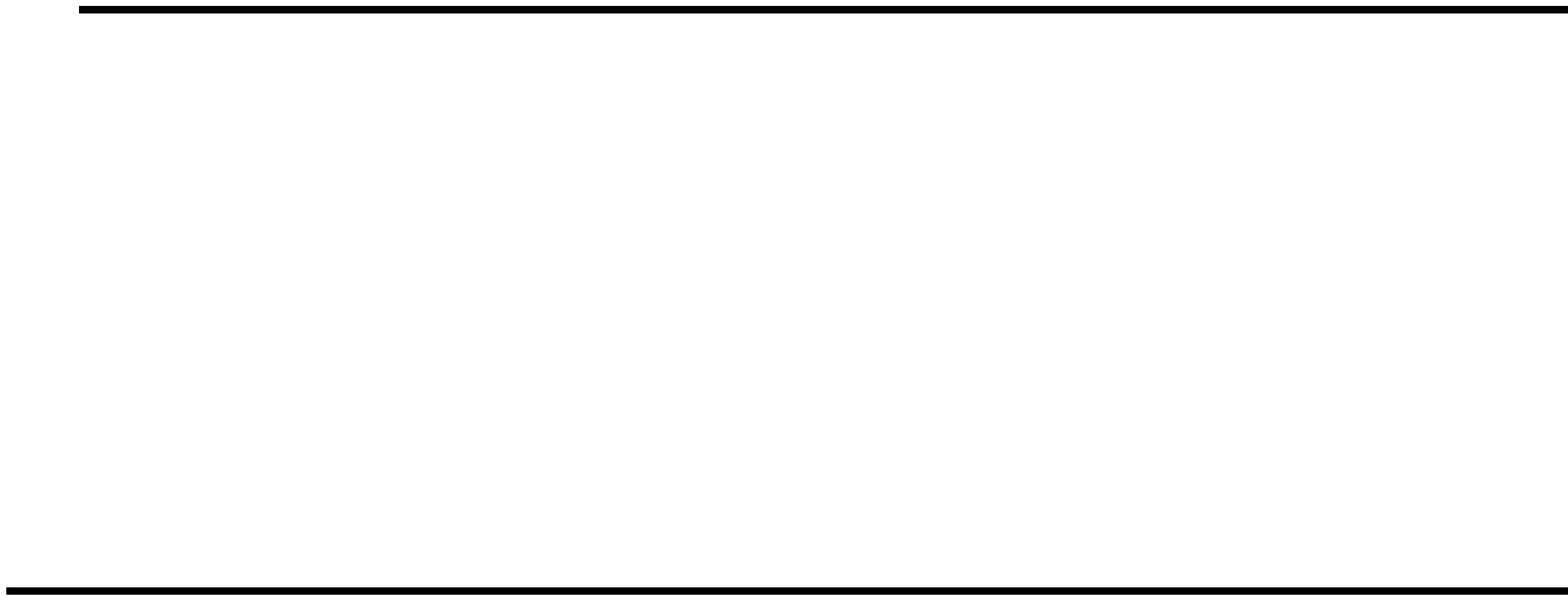}
\caption{Optimal control strategy $\tau^{\circ} (3,2,1, t_{5})$ in state $[ 3 \; 2 \; 1 \; t_{5}]^{T}$.}
\label{fig:esS3_tau_3_2_1}
\end{figure}

{\bf Stage $5$ -- State $\boldsymbol{[4 \; 1 \; 2 \; t_{5}]^{T}}$ ($S21$)}

In state $[4 \; 1 \; 2 \; t_{5}]^{T}$ all jobs of class $P_{1}$ have been completed; then the decision about the class of the next job to be executed is mandatory. The cost function to be minimized in this state, with respect to the (continuos) decision variable $\tau$ only (which corresponds to the processing time $pt_{2,2}$), is
\begin{equation*}
\alpha_{2,2} \, \max \{ t_{5} + st_{2,2} + \tau - dd_{2,2} \, , \, 0 \} + \beta_{2} \, (pt^{\mathrm{nom}}_{2} - \tau) + sc_{2,2} + J^{\circ}_{4,2,2} (t_{6})
\end{equation*}
that can be written as $f (pt_{2,2} + t_{5}) + g (pt_{2,2})$ being
\begin{equation*}
f (pt_{2,2} + t_{5}) = \max \{ pt_{2,2} + t_{5} - 24 \, , \, 0 \} + J^{\circ}_{4,2,2} (pt_{2,2} + t_{5})
\end{equation*}
\begin{equation*}
g (pt_{2,2}) = \left\{ \begin{array}{ll}
1.5 \cdot (6 - pt_{2,2}) & pt_{2,2} \in [ 4 , 6 )\\
0 & pt_{2,2} \notin [ 4 , 6 )
\end{array} \right.
\end{equation*}
The function $pt^{\circ}_{2,2}(t_{5}) = \arg \min_{pt_{2,2}} \{ f (pt_{2,2} + t_{5}) + g (pt_{2,2}) \} $, with $4 \leq pt_{2,2} \leq 6$, is determined by applying lemma~\ref{lem:xopt}. It is
\begin{equation*}
pt^{\circ}_{2,2}(t_{5}) = x_{\mathrm{e}}(t_{5}) \qquad \text{with} \quad x_{\mathrm{e}}(t_{5}) = \left\{ \begin{array}{ll}
6 &  t_{5} < 26\\
-t_{5} + 32 & 26 \leq t_{5} < 28\\
4 & t_{5} \geq 28
\end{array} \right.
\end{equation*}

Taking into account the mandatory decision about the class of the next job to be executed, the optimal control strategies for this state are
\begin{equation*}
\delta_{1}^{\circ} (4,1,2, t_{5}) = 0 \quad \forall \, t_{5} \qquad \delta_{2}^{\circ} (4,1,2, t_{5}) = 1 \quad \forall \, t_{5}
\end{equation*}
\begin{equation*}
\tau^{\circ} (4,1,2, t_{5}) = \left\{ \begin{array}{ll}
6 &  t_{5} < 26\\
-t_{5} + 32 & 26 \leq t_{5} < 28\\
4 & t_{5} \geq 28
\end{array} \right.
\end{equation*}
The optimal control strategy $\tau^{\circ} (4,1,2, t_{5})$ is illustrated in figure~\ref{fig:esS3_tau_4_1_2}.

\begin{figure}[h!]
\centering
\psfrag{f(x)}[Bl][Bl][.8][0]{$\tau^{\circ} (4,1,2, t_{5})$}
\psfrag{x}[bc][Bl][.8][0]{$t_{5}$}
\psfrag{0}[tc][Bl][.8][0]{$0$}
\includegraphics[scale=.2]{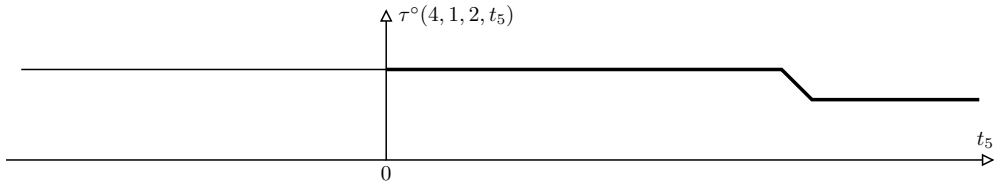}
\caption{Optimal control strategy $\tau^{\circ} (4,1,2, t_{5})$ in state $[ 4 \; 1 \; 2 \; t_{5}]^{T}$.}
\label{fig:esS3_tau_4_1_2}
\end{figure}

The optimal cost-to-go $J^{\circ}_{4,1,2} (t_{5}) = f ( pt^{\circ}_{2,2}(t_{5}) + t_{5} ) + g ( pt^{\circ}_{2,2}(t_{5}) )$, illustrated in figure~\ref{fig:esS2_J_4_1_2}, is provided by lemma~\ref{lem:h(t)}. It is specified by the initial value 0, by the set \{ 18, 26, 28 \} of abscissae $\gamma_{i}$, $i = 1, \ldots, 3$, at which the slope changes, and by the set \{ 1, 1.5, 2 \} of slopes $\mu_{i}$, $i = 1, \ldots, 3$, in the various intervals.

\begin{figure}[h!]
\centering
\psfrag{0}[cc][tc][.8][0]{$0$}
\includegraphics[scale=.8]{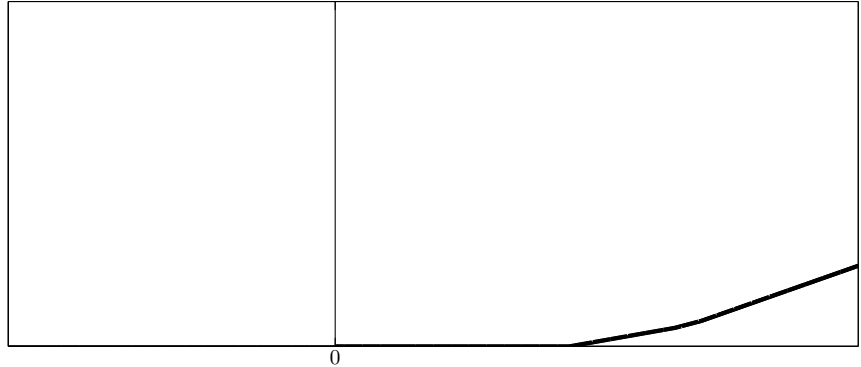}%
\vspace{-12pt}
\caption{Optimal cost-to-go $J^{\circ}_{4,1,2} (t_{5})$ in state $[4 \; 1 \; 2 \; t_{5}]^{T}$.}
\label{fig:esS2_J_4_1_2}
\end{figure}

{\bf Stage $5$ -- State $\boldsymbol{[4 \; 1 \; 1 \; t_{5}]^{T}}$ ($S20$)}

In state $[4 \; 1 \; 1 \; t_{5}]^{T}$ all jobs of class $P_{1}$ have been completed; then the decision about the class of the next job to be executed is mandatory. The cost function to be minimized in this state, with respect to the (continuos) decision variable $\tau$ only (which corresponds to the processing time $pt_{2,2}$), is
\begin{equation*}
\alpha_{2,2} \, \max \{ t_{5} + st_{1,2} + \tau - dd_{2,2} \, , \, 0 \} + \beta_{2} \, (pt^{\mathrm{nom}}_{2} - \tau) + sc_{1,2} + J^{\circ}_{4,2,2} (t_{6})
\end{equation*}
that can be written as $f (pt_{2,2} + t_{5}) + g (pt_{2,2})$ being
\begin{equation*}
f (pt_{2,2} + t_{5}) = \max \{ pt_{2,2} + t_{5} - 23 \, , \, 0 \} + 0.5 + J^{\circ}_{4,2,2} (pt_{2,2} + t_{5}+1)
\end{equation*}
\begin{equation*}
g (pt_{2,2}) = \left\{ \begin{array}{ll}
1.5 \cdot (6 - pt_{2,2}) & pt_{2,2} \in [ 4 , 6 )\\
0 & pt_{2,2} \notin [ 4 , 6 )
\end{array} \right.
\end{equation*}
The function $pt^{\circ}_{2,2}(t_{5}) = \arg \min_{pt_{2,2}} \{ f (pt_{2,2} + t_{5}) + g (pt_{2,2}) \} $, with $4 \leq pt_{2,2} \leq 6$, is determined by applying lemma~\ref{lem:xopt}. It is
\begin{equation*}
pt^{\circ}_{2,2}(t_{5}) = x_{\mathrm{e}}(t_{5}) \qquad \text{with} \quad x_{\mathrm{e}}(t_{5}) = \left\{ \begin{array}{ll}
6 &  t_{5} < 25\\
-t_{5} + 31 & 25 \leq t_{5} < 27\\
4 & t_{5} \geq 27
\end{array} \right.
\end{equation*}

Taking into account the mandatory decision about the class of the next job to be executed, the optimal control strategies for this state are
\begin{equation*}
\delta_{1}^{\circ} (4,1,1, t_{5}) = 0 \quad \forall \, t_{5} \qquad \delta_{2}^{\circ} (4,1,1, t_{5}) = 1 \quad \forall \, t_{5}
\end{equation*}
\begin{equation*}
\tau^{\circ} (4,1,1, t_{5}) = \left\{ \begin{array}{ll}
6 &  t_{5} < 25\\
-t_{5} + 31 & 25 \leq t_{5} < 27\\
4 & t_{5} \geq 27
\end{array} \right.
\end{equation*}
The optimal control strategy $\tau^{\circ} (4,1,1, t_{5})$ is illustrated in figure~\ref{fig:esS3_tau_4_1_1}.

\begin{figure}[h!]
\centering
\psfrag{f(x)}[Bl][Bl][.8][0]{$\tau^{\circ} (4,1,1, t_{5})$}
\psfrag{x}[bc][Bl][.8][0]{$t_{5}$}
\psfrag{0}[tc][Bl][.8][0]{$0$}
\includegraphics[scale=.2]{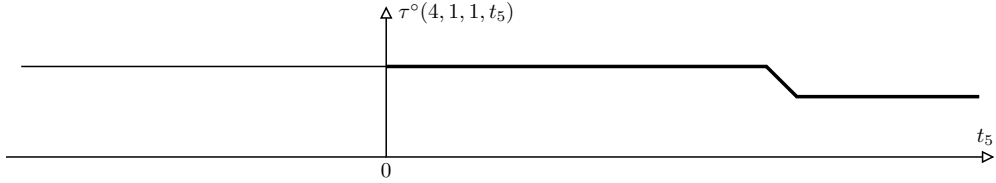}
\caption{Optimal control strategy $\tau^{\circ} (4,1,1, t_{5})$ in state $[ 4 \; 1 \; 1 \; t_{5}]^{T}$.}
\label{fig:esS3_tau_4_1_1}
\end{figure}

The optimal cost-to-go $J^{\circ}_{4,1,1} (t_{5}) = f ( pt^{\circ}_{2,2}(t_{5}) + t_{5} ) + g ( pt^{\circ}_{2,2}(t_{5}) )$, illustrated in figure~\ref{fig:esS2_J_4_1_1}, is provided by lemma~\ref{lem:h(t)}. It is specified by the initial value 0.5, by the set \{ 17, 25, 27 \} of abscissae $\gamma_{i}$, $i = 1, \ldots, 3$, at which the slope changes, and by the set \{ 1, 1.5, 2 \} of slopes $\mu_{i}$, $i = 1, \ldots, 3$, in the various intervals.

\begin{figure}[h!]
\centering
\psfrag{0}[cc][tc][.8][0]{$0$}
\includegraphics[scale=.8]{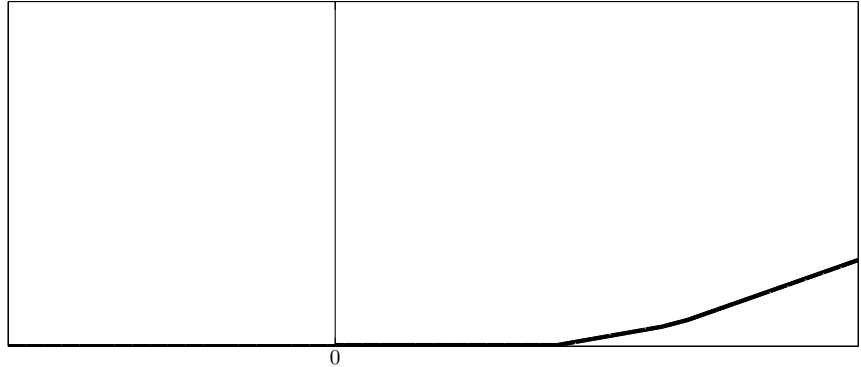}%
\vspace{-12pt}
\caption{Optimal cost-to-go $J^{\circ}_{4,1,1} (t_{5})$ in state $[4 \; 1 \; 1 \; t_{5}]^{T}$.}
\label{fig:esS2_J_4_1_1}
\end{figure}

%%
%% S19 - [1 3 2 t4]
%%

{\bf Stage $4$ -- State $\boldsymbol{[1 \; 3 \; 2 \; t_{4}]^{T}}$ ($S19$)}

In state $[1 \; 3 \; 2 \; t_{4}]^{T}$ all jobs of class $P_{2}$ have been completed; then the decision about the class of the next job to be executed is mandatory. The cost function to be minimized in this state, with respect to the (continuos) decision variable $\tau$ only (which corresponds to the processing time $pt_{1,2}$), is
\begin{equation*}
\alpha_{1,2} \, \max \{ t_{4} + st_{2,1} + \tau - dd_{1,2} \, , \, 0 \} + \beta_{1} \, (pt^{\mathrm{nom}}_{1} - \tau) + sc_{2,1} + J^{\circ}_{2,3,1} (t_{5})
\end{equation*}
that can be written as $f (pt_{1,2} + t_{4}) + g (pt_{1,2})$ being
\begin{equation*}
f (pt_{1,2} + t_{4}) = 0.5 \cdot \max \{ pt_{1,2} + t_{4} - 23.5 \, , \, 0 \} + 1 + J^{\circ}_{2,3,1} (pt_{1,2} + t_{4} + 0.5)
\end{equation*}
\begin{equation*}
g (pt_{1,2}) = \left\{ \begin{array}{ll}
8 - pt_{1,2} & pt_{1,2} \in [ 4 , 8 )\\
0 & pt_{1,2} \notin [ 4 , 8 )
\end{array} \right.
\end{equation*}
The function $pt^{\circ}_{1,2}(t_{4}) = \arg \min_{pt_{1,2}} \{ f (pt_{1,2} + t_{4}) + g (pt_{1,2}) \} $, with $4 \leq pt_{1,2} \leq 8$, is determined by applying lemma~\ref{lem:xopt}. It is
\begin{equation*}
pt^{\circ}_{1,2}(t_{4}) = x_{\mathrm{e}}(t_{4}) \qquad \text{with} \quad x_{\mathrm{e}}(t_{4}) = \left\{ \begin{array}{ll}
8 &  t_{4} < 12.5\\
-t_{4} + 20.5 & 12.5 \leq t_{4} < 16.5\\
4 & t_{4} \geq 16.5
\end{array} \right.
\end{equation*}

Taking into account the mandatory decision about the class of the next job to be executed, the optimal control strategies for this state are
\begin{equation*}
\delta_{1}^{\circ} (1,3,2, t_{4}) = 1 \quad \forall \, t_{4} \qquad \delta_{2}^{\circ} (1,3,2, t_{4}) = 0 \quad \forall \, t_{4}
\end{equation*}
\begin{equation*}
\tau^{\circ} (1,3,2, t_{4}) = \left\{ \begin{array}{ll}
8 &  t_{4} < 12.5\\
-t_{4} + 20.5 & 12.5 \leq t_{4} < 16.5\\
4 & t_{4} \geq 16.5
\end{array} \right.
\end{equation*}
The optimal control strategy $\tau^{\circ} (1,3,2, t_{4})$ is illustrated in figure~\ref{fig:esS3_tau_1_3_2}.

\begin{figure}[h!]
\centering
\psfrag{f(x)}[Bl][Bl][.8][0]{$\tau^{\circ} (1,3,2, t_{4})$}
\psfrag{x}[bc][Bl][.8][0]{$t_{4}$}
\psfrag{0}[tc][Bl][.8][0]{$0$}
\includegraphics[scale=.2]{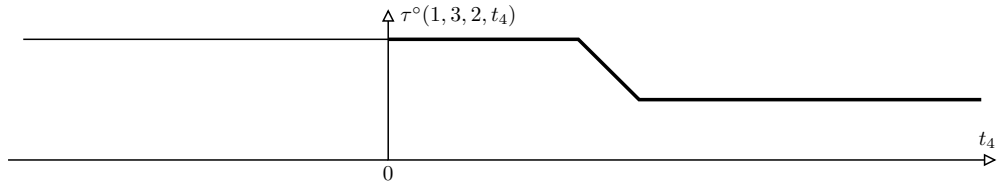}
\caption{Optimal control strategy $\tau^{\circ} (1,3,2, t_{4})$ in state $[ 1 \; 3 \; 2 \; t_{4}]^{T}$.}
\label{fig:esS3_tau_1_3_2}
\end{figure}

The optimal cost-to-go $J^{\circ}_{1,3,2} (t_{4}) = f ( pt^{\circ}_{1,2}(t_{4}) + t_{4} ) + g ( pt^{\circ}_{1,2}(t_{4}) )$, illustrated in figure~\ref{fig:esS2_J_1_3_2}, is provided by lemma~\ref{lem:h(t)}. It is specified by the initial value 1, by the set \{ 12.5, 19.5, 20.5, 24.5 \} of abscissae $\gamma_{i}$, $i = 1, \ldots, 4$, at which the slope changes, and by the set \{ 1, 1.5, 2, 2.5 \} of slopes $\mu_{i}$, $i = 1, \ldots, 4$, in the various intervals.

\begin{figure}[h!]
\centering
\psfrag{0}[cc][tc][.8][0]{$0$}
\includegraphics[scale=.8]{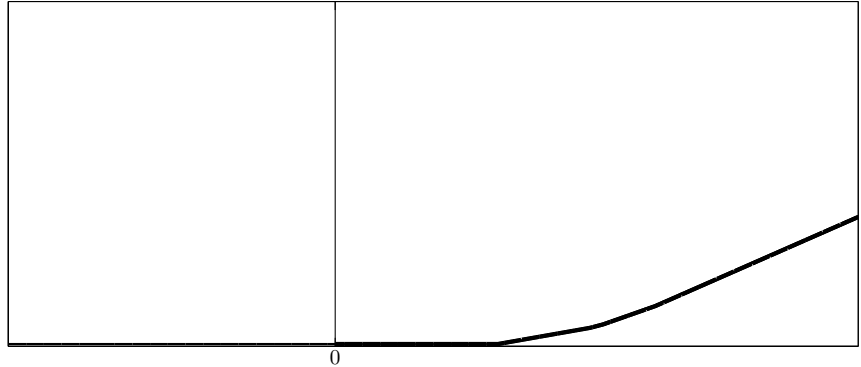}%
\vspace{-12pt}
\caption{Optimal cost-to-go $J^{\circ}_{1,3,2} (t_{4})$ in state $[1 \; 3 \; 2 \; t_{4}]^{T}$.}
\label{fig:esS2_J_1_3_2}
\end{figure}

%%
%% S18 - [1 3 1 t4]
%%

{\bf Stage $4$ -- State $\boldsymbol{[1 \; 3 \; 1 \; t_{4}]^{T}}$ ($S18$)}

In state $[1 \; 3 \; 1 \; t_{4}]^{T}$ all jobs of class $P_{2}$ have been completed; then the decision about the class of the next job to be executed is mandatory. The cost function to be minimized in this state, with respect to the (continuos) decision variable $\tau$ only (which corresponds to the processing time $pt_{1,2}$), is
\begin{equation*}
\alpha_{1,2} \, \max \{ t_{4} + st_{1,1} + \tau - dd_{1,2} \, , \, 0 \} + \beta_{1} \, (pt^{\mathrm{nom}}_{1} - \tau) + sc_{1,1} + J^{\circ}_{2,3,1} (t_{5})
\end{equation*}
that can be written as $f (pt_{1,2} + t_{4}) + g (pt_{1,2})$ being
\begin{equation*}
f (pt_{1,2} + t_{4}) = 0.5 \cdot \max \{ pt_{1,2} + t_{4} - 24 \, , \, 0 \} + J^{\circ}_{2,3,1} (pt_{1,2} + t_{4})
\end{equation*}
\begin{equation*}
g (pt_{1,2}) = \left\{ \begin{array}{ll}
8 - pt_{1,2} & pt_{1,2} \in [ 4 , 8 )\\
0 & pt_{1,2} \notin [ 4 , 8 )
\end{array} \right.
\end{equation*}
The function $pt^{\circ}_{1,2}(t_{4}) = \arg \min_{pt_{1,2}} \{ f (pt_{1,2} + t_{4}) + g (pt_{1,2}) \} $, with $4 \leq pt_{1,2} \leq 8$, is determined by applying lemma~\ref{lem:xopt}. It is
\begin{equation*}
pt^{\circ}_{1,2}(t_{4}) = x_{\mathrm{e}}(t_{4}) \qquad \text{with} \quad x_{\mathrm{e}}(t_{4}) = \left\{ \begin{array}{ll}
8 &  t_{4} < 13\\
-t_{4} + 21 & 13 \leq t_{4} < 17\\
4 & t_{4} \geq 17
\end{array} \right.
\end{equation*}

Taking into account the mandatory decision about the class of the next job to be executed, the optimal control strategies for this state are
\begin{equation*}
\delta_{1}^{\circ} (1,3,1, t_{4}) = 1 \quad \forall \, t_{4} \qquad \delta_{2}^{\circ} (1,3,1, t_{4}) = 0 \quad \forall \, t_{4}
\end{equation*}
\begin{equation*}
\tau^{\circ} (1,3,1, t_{4}) = \left\{ \begin{array}{ll}
8 &  t_{4} < 13\\
-t_{4} + 21 & 13 \leq t_{4} < 17\\
4 & t_{4} \geq 17
\end{array} \right.
\end{equation*}
The optimal control strategy $\tau^{\circ} (1,3,1, t_{4})$ is illustrated in figure~\ref{fig:esS3_tau_1_3_1}.

\begin{figure}[h!]
\centering
\psfrag{f(x)}[Bl][Bl][.8][0]{$\tau^{\circ} (1,3,1, t_{4})$}
\psfrag{x}[bc][Bl][.8][0]{$t_{4}$}
\psfrag{0}[tc][Bl][.8][0]{$0$}
\includegraphics[scale=.2]{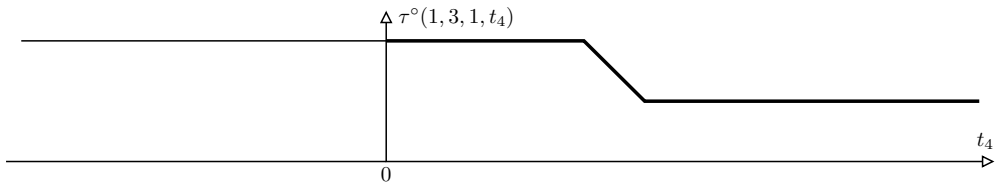}
\caption{Optimal control strategy $\tau^{\circ} (1,3,1, t_{4})$ in state $[ 1 \; 3 \; 1 \; t_{4}]^{T}$.}
\label{fig:esS3_tau_1_3_1}
\end{figure}

The optimal cost-to-go $J^{\circ}_{1,3,1} (t_{4}) = f ( pt^{\circ}_{1,2}(t_{4}) + t_{4} ) + g ( pt^{\circ}_{1,2}(t_{4}) )$, illustrated in figure~\ref{fig:esS2_J_1_3_1}, is provided by lemma~\ref{lem:h(t)}. It is specified by the initial value 0, by the set \{ 13, 20, 21, 25 \} of abscissae $\gamma_{i}$, $i = 1, \ldots, 4$, at which the slope changes, and by the set \{ 1, 1.5, 2, 2.5 \} of slopes $\mu_{i}$, $i = 1, \ldots, 4$, in the various intervals.

\begin{figure}[h!]
\centering
\psfrag{0}[cc][tc][.8][0]{$0$}
\includegraphics[scale=.8]{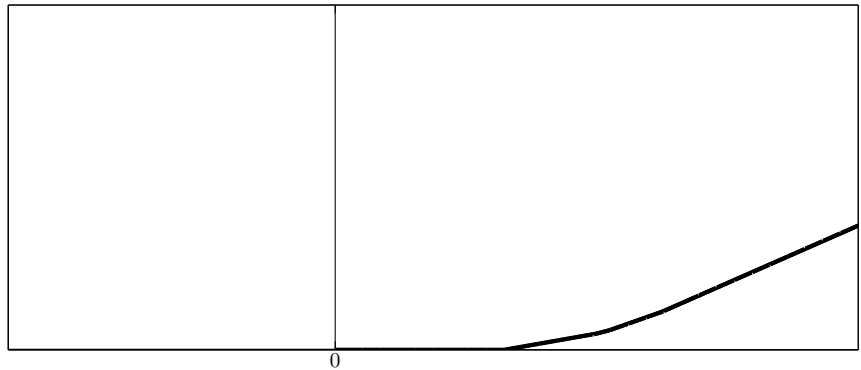}%
\vspace{-12pt}
\caption{Optimal cost-to-go $J^{\circ}_{1,3,1} (t_{4})$ in state $[1 \; 3 \; 1 \; t_{4}]^{T}$.}
\label{fig:esS2_J_1_3_1}
\end{figure}

%%
%% S17 - [2 2 2 t4]
%%

{\bf Stage $4$ -- State $\boldsymbol{[2 \; 2 \; 2 \; t_{4}]^{T}}$ ($S17$)}

In state $[2 \; 2 \; 2 \; t_{4}]^{T}$, the cost function to be minimized, with respect to the (continuos) decision variable $\tau$ and to the (binary) decision variables $\delta_{1}$ and $\delta_{2}$ is
\begin{equation*}
\begin{split}
&\delta_{1} \big[ \alpha_{1,3} \, \max \{ t_{4} + st_{2,1} + \tau - dd_{1,3} \, , \, 0 \} + \beta_{1} \, ( pt^{\mathrm{nom}}_{1} - \tau ) + sc_{2,1} + J^{\circ}_{3,2,1} (t_{5}) \big] +\\
&+ \delta_{2} \big[ \alpha_{2,3} \, \max \{ t_{4} + st_{2,2} + \tau - dd_{2,3} \, , \, 0 \} + \beta_{2} \, ( pt^{\mathrm{nom}}_{2} - \tau ) + sc_{2,2} + J^{\circ}_{2,3,2} (t_{5}) \big]
\end{split}
\end{equation*}

{\it Case i)} in which it is assumed $\delta_{1} = 1$ (and $\delta_{2} = 0$).

In this case, it is necessary to minimize, with respect to the (continuos) decision variable $\tau$ which corresponds to the processing time $pt_{1,3}$, the following function
\begin{equation*}
\alpha_{1,3} \, \max \{ t_{4} + st_{2,1} + \tau - dd_{1,3} \, , \, 0 \} + \beta_{1} \, (pt^{\mathrm{nom}}_{1} - \tau) + sc_{2,1} + J^{\circ}_{3,2,1} (t_{5})
\end{equation*}
that can be written as $f (pt_{1,3} + t_{4}) + g (pt_{1,3})$ being
\begin{equation*}
f (pt_{1,3} + t_{4}) = 1.5 \cdot \max \{ pt_{1,3} + t_{4} - 28.5 \, , \, 0 \} + 1 + J^{\circ}_{3,2,1} (pt_{1,3} + t_{4} + 0.5)
\end{equation*}
\begin{equation*}
g (pt_{1,3}) = \left\{ \begin{array}{ll}
8 - pt_{1,3} & pt_{1,3} \in [ 4 , 8 )\\
0 & pt_{1,3} \notin [ 4 , 8 )
\end{array} \right.
\end{equation*}
The function $pt^{\circ}_{1,3}(t_{4}) = \arg \min_{pt_{1,3}} \{ f (pt_{1,3} + t_{4}) + g (pt_{1,3}) \} $, with $4 \leq pt_{1,3} \leq 8$, is determined by applying lemma~\ref{lem:xopt}. It is (see figure~\ref{fig:esS3_tau_2_2_2_pt_1_3})
\begin{equation*}
pt^{\circ}_{1,3}(t_{4}) = \left\{ \begin{array}{ll}
x_{\mathrm{s}}(t_{4}) &  t_{4} < 15.5\\
x_{\mathrm{e}}(t_{4}) & t_{4} \geq 15.5
\end{array} \right. \qquad \text{with} \quad x_{\mathrm{s}}(t_{4}) = \left\{ \begin{array}{ll}
8 &  t_{4} < 14.5\\
-t_{4} + 22.5 & 14.5 \leq t_{4} < 15.5
\end{array} \right. \  , 
\end{equation*}
\begin{equation*}
\qquad \text{and} \quad x_{\mathrm{e}}(t_{4}) = \left\{ \begin{array}{ll}
8 &  15.5 \leq t_{4} < 20.5\\
-t_{4} + 28.5 & 20.5 \leq t_{4} < 24.5\\
4 & t_{4} \geq 24.5
\end{array} \right.
\end{equation*}

\begin{figure}[h!]
\centering
\psfrag{f(x)}[Bl][Bl][.8][0]{$pt^{\circ}_{1,3}(t_{4})$}
\psfrag{x}[bc][Bl][.8][0]{$t_{4}$}
\psfrag{0}[tc][Bl][.8][0]{$0$}
\includegraphics[scale=.2]{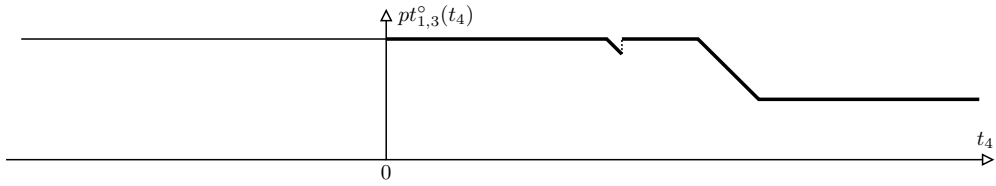}
\caption{Optimal processing time $pt^{\circ}_{1,3}(t_{4})$, under the assumption $\delta_{1} = 1$ in state $[ 2 \; 2 \; 2 \; t_{4}]^{T}$.}
\label{fig:esS3_tau_2_2_2_pt_1_3}
\end{figure}

The conditioned cost-to-go $J^{\circ}_{2,2,2} (t_{4} \mid \delta_{1}=1) = f ( pt^{\circ}_{1,3}(t_{4}) + t_{4} ) + g ( pt^{\circ}_{1,3}(t_{4}) )$, illustrated in figure~\ref{fig:esS3_J_2_2_2_min}, is provided by lemma~\ref{lem:h(t)}. It is specified by the initial value 1.5, by the set \{ 14.5, 15.5, 17, 20.5, 24.5, 26.5 \} of abscissae $\gamma_{i}$, $i = 1, \ldots, 6$, at which the slope changes, and by the set \{ 1, 0, 0.5, 1, 2, 3 \} of slopes $\mu_{i}$, $i = 1, \ldots, 6$, in the various intervals.

{\it Case ii)} in which it is assumed $\delta_{2} = 1$ (and $\delta_{1} = 0$).

In this case, it is necessary to minimize, with respect to the (continuos) decision variable $\tau$ which corresponds to the processing time $pt_{2,3}$, the following function
\begin{equation*}
\alpha_{2,3} \, \max \{ t_{4} + st_{2,2} + \tau - dd_{2,3} \, , \, 0 \} + \beta_{2} \, (pt^{\mathrm{nom}}_{2} - \tau) + sc_{2,2} + J^{\circ}_{2,3,2} (t_{5})
\end{equation*}
that can be written as $f (pt_{2,3} + t_{4}) + g (pt_{2,3})$ being
\begin{equation*}
f (pt_{2,3} + t_{4}) = \max \{ pt_{2,3} + t_{4} - 38 \, , \, 0 \} + J^{\circ}_{2,3,2} (pt_{2,3} + t_{4})
\end{equation*}
\begin{equation*}
g (pt_{2,3}) = \left\{ \begin{array}{ll}
1.5 \cdot (6 - pt_{2,3}) & pt_{2,3} \in [ 4 , 6 )\\
0 & pt_{2,3} \notin [ 4 , 6 )
\end{array} \right.
\end{equation*}
The function $pt^{\circ}_{2,3}(t_{4}) = \arg \min_{pt_{2,3}} \{ f (pt_{2,3} + t_{4}) + g (pt_{2,3}) \} $, with $4 \leq pt_{2,3} \leq 6$, is determined by applying lemma~\ref{lem:xopt}. It is (see figure~\ref{fig:esS3_tau_2_2_2_pt_2_3})
\begin{equation*}
pt^{\circ}_{2,3}(t_{4}) = x_{\mathrm{e}}(t_{4}) \qquad \text{with} \quad x_{\mathrm{e}}(t_{4}) = \left\{ \begin{array}{ll}
6 &  t_{4} < 18.5\\
-t_{4} + 24.5 & 18.5 \leq t_{4} < 20.5\\
4 & t_{4} \geq 20.5
\end{array} \right.
\end{equation*}

\begin{figure}[h!]
\centering
\psfrag{f(x)}[Bl][Bl][.8][0]{$pt^{\circ}_{2,3}(t_{4})$}
\psfrag{x}[bc][Bl][.8][0]{$t_{4}$}
\psfrag{0}[tc][Bl][.8][0]{$0$}
\includegraphics[scale=.2]{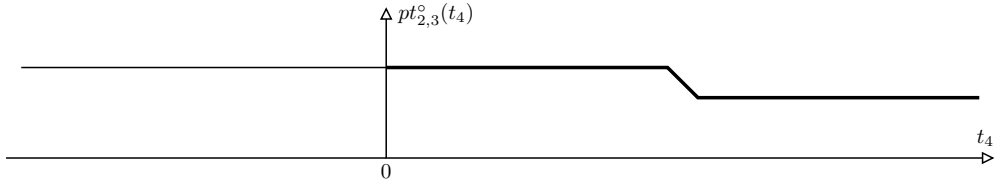}
\caption{Optimal processing time $pt^{\circ}_{2,3}(t_{4})$, under the assumption $\delta_{2} = 1$ in state $[ 2 \; 2 \; 2 \; t_{4}]^{T}$.}
\label{fig:esS3_tau_2_2_2_pt_2_3}
\end{figure}

The conditioned cost-to-go $J^{\circ}_{2,2,2} (t_{4} \mid \delta_{2}=1) = f ( pt^{\circ}_{2,3}(t_{4}) + t_{4} ) + g ( pt^{\circ}_{2,3}(t_{4}) )$, illustrated in figure~\ref{fig:esS3_J_2_2_2_min}, is provided by lemma~\ref{lem:h(t)}. It is specified by the initial value 1, by the set \{ 14.5, 18.5, 24.5, 34 \} of abscissae $\gamma_{i}$, $i = 1, \ldots, 4$, at which the slope changes, and by the set \{ 1, 1.5, 2, 3 \} of slopes $\mu_{i}$, $i = 1, \ldots, 4$, in the various intervals.

\begin{figure}[h!]
\centering
\psfrag{J1}[cl][Bc][.8][0]{$J^{\circ}_{2,2,2} (t_{4} \mid \delta_{1} = 1)$}
\psfrag{J2}[br][Bc][.8][0]{$J^{\circ}_{2,2,2} (t_{4} \mid \delta_{2} = 1)$}
\psfrag{0}[cc][tc][.7][0]{$0$}\includegraphics[scale=.6]{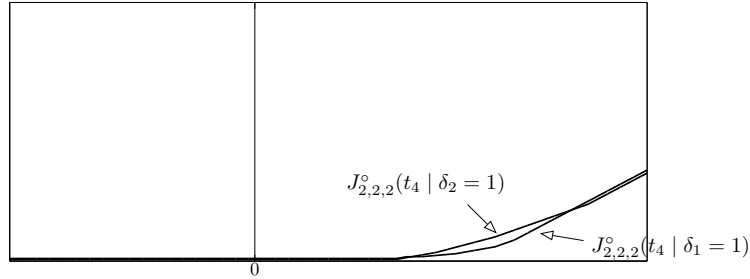}%
\vspace{-12pt}
\caption{Conditioned costs-to-go $J^{\circ}_{2,2,2} (t_{4} \mid \delta_{1} = 1)$ and $J^{\circ}_{2,2,2} (t_{4} \mid \delta_{2} = 1)$ in state $[2 \; 2 \; 2 \; t_{4}]^{T}$.}
\label{fig:esS3_J_2_2_2_min}
\end{figure}

In order to find the optimal cost-to-go $J^{\circ}_{2,2,2} (t_{4})$, it is necessary to carry out the following minimization
\begin{equation*}
J^{\circ}_{2,2,2} (t_{4}) = \min \big\{ J^{\circ}_{2,2,2} (t_{4} \mid \delta_{1} = 1) \, , \, J^{\circ}_{2,2,2} (t_{4} \mid \delta_{2} = 1) \big\}
\end{equation*}
which provides, in accordance with lemma~\ref{lem:min}, the continuous, nondecreasing, piecewise linear function illustrated in figure~\ref{fig:esS3_J_2_2_2}.

\begin{figure}[h!]
\centering
\psfrag{0}[cc][tc][.8][0]{$0$}
\includegraphics[scale=.8]{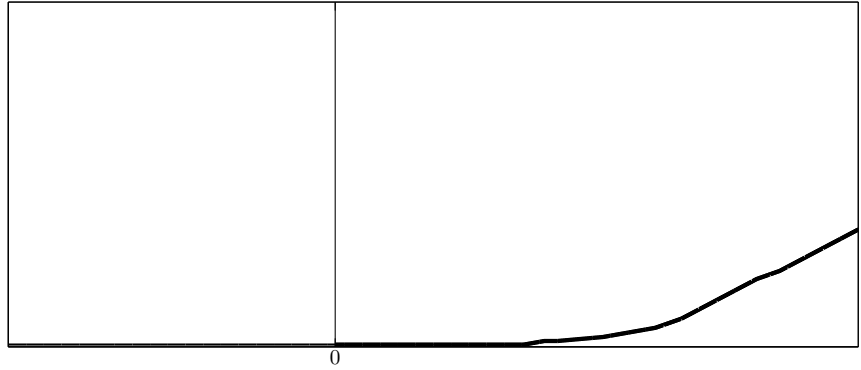}%
\vspace{-12pt}
\caption{Optimal cost-to-go $J^{\circ}_{2,2,2} (t_{4})$ in state $[2 \; 2 \; 2 \; t_{4}]^{T}$.}
\label{fig:esS3_J_2_2_2}
\end{figure}

The function $J^{\circ}_{2,2,2} (t_{4})$ is specified by the initial value 1, by the set \{ 14.5, 16, 17, 20.5, 24.5, 26.5, 32.25, 34 \} of abscissae $\gamma_{i}$, $i = 1, \ldots, 8$, at which the slope changes, and by the set \{ 1, 0, 0.5, 1, 2, 3, 2, 3 \} of slopes $\mu_{i}$, $i = 1, \ldots, 8$, in the various intervals.

Since $J^{\circ}_{2,2,2} (t_{4} \mid \delta_{1} = 1)$ is the minimum in $[16,32.25)$, and $J^{\circ}_{2,2,2} (t_{4} \mid \delta_{2} = 1)$ is the minimum in $(-\infty, 16)$ and in $[32.25,+\infty)$, the optimal control strategies for this state are
\begin{equation*}
\delta_{1}^{\circ} (2,2,2, t_{4}) = \left\{ \begin{array}{ll}
0 &  t_{4} < 16\\
1 &  16 \leq t_{4} < 32.25\\
0 & t_{4} \geq 32.25
\end{array} \right. \qquad \delta_{2}^{\circ} (2,2,2, t_{4}) = \left\{ \begin{array}{ll}
1 &  t_{4} < 16\\
0 &  16 \leq t_{4} < 32.25\\
1 & t_{4} \geq 32.25
\end{array} \right.
\end{equation*}
\begin{equation*}
\tau^{\circ} (2,2,2, t_{4}) = \left\{ \begin{array}{ll}
6 &  t_{4} < 16\\
8 & 16 \leq t_{4} < 20.5\\
-t_{4} + 28.5 & 20.5 \leq t_{4} < 24.5\\
4 & t_{4} \geq 24.5
\end{array} \right.
\end{equation*}
The optimal control strategy $\tau^{\circ} (2,2,2, t_{4})$ is illustrated in figure~\ref{fig:esS3_tau_2_2_2}.

\newpage

\begin{figure}[h!]
\centering
\psfrag{f(x)}[Bl][Bl][.8][0]{$\tau^{\circ} (2,2,2, t_{4})$}
\psfrag{x}[bc][Bl][.8][0]{$t_{4}$}
\psfrag{0}[tc][Bl][.8][0]{$0$}
\includegraphics[scale=.2]{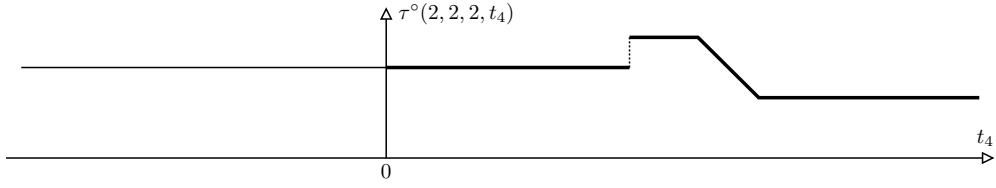}
\caption{Optimal control strategy $\tau^{\circ} (2,2,2, t_{4})$ in state $[ 2 \; 2 \; 2 \; t_{4}]^{T}$.}
\label{fig:esS3_tau_2_2_2}
\end{figure}

%%
%% S16 - [2 2 1 t4]
%%

{\bf Stage $4$ -- State $\boldsymbol{[2 \; 2 \; 1 \; t_{4}]^{T}}$ ($S16$)}

In state $[2 \; 2 \; 1 \; t_{4}]^{T}$, the cost function to be minimized, with respect to the (continuos) decision variable $\tau$ and to the (binary) decision variables $\delta_{1}$ and $\delta_{2}$ is
\begin{equation*}
\begin{split}
&\delta_{1} \big[ \alpha_{1,3} \, \max \{ t_{4} + st_{1,1} + \tau - dd_{1,3} \, , \, 0 \} + \beta_{1} \, ( pt^{\mathrm{nom}}_{1} - \tau ) + sc_{1,1} + J^{\circ}_{3,2,1} (t_{5}) \big] +\\
&+ \delta_{2} \big[ \alpha_{2,3} \, \max \{ t_{4} + st_{1,2} + \tau - dd_{2,3} \, , \, 0 \} + \beta_{2} \, ( pt^{\mathrm{nom}}_{2} - \tau ) + sc_{1,2} + J^{\circ}_{2,3,2} (t_{5}) \big]
\end{split}
\end{equation*}

{\it Case i)} in which it is assumed $\delta_{1} = 1$ (and $\delta_{2} = 0$).

In this case, it is necessary to minimize, with respect to the (continuos) decision variable $\tau$ which corresponds to the processing time $pt_{1,3}$, the following function
\begin{equation*}
\alpha_{1,3} \, \max \{ t_{4} + st_{1,1} + \tau - dd_{1,3} \, , \, 0 \} + \beta_{1} \, (pt^{\mathrm{nom}}_{1} - \tau) + sc_{1,1} + J^{\circ}_{3,2,1} (t_{5})
\end{equation*}
that can be written as $f (pt_{1,3} + t_{4}) + g (pt_{1,3})$ being
\begin{equation*}
f (pt_{1,3} + t_{4}) = 1.5 \cdot \max \{ pt_{1,3} + t_{4} - 29 \, , \, 0 \} + J^{\circ}_{3,2,1} (pt_{1,3} + t_{4})
\end{equation*}
\begin{equation*}
g (pt_{1,3}) = \left\{ \begin{array}{ll}
8 - pt_{1,3} & pt_{1,3} \in [ 4 , 8 )\\
0 & pt_{1,3} \notin [ 4 , 8 )
\end{array} \right.
\end{equation*}
The function $pt^{\circ}_{1,3}(t_{4}) = \arg \min_{pt_{1,3}} \{ f (pt_{1,3} + t_{4}) + g (pt_{1,3}) \} $, with $4 \leq pt_{1,3} \leq 8$, is determined by applying lemma~\ref{lem:xopt}. It is (see figure~\ref{fig:esS3_tau_2_2_1_pt_1_3})
\begin{equation*}
pt^{\circ}_{1,3}(t_{4}) = \left\{ \begin{array}{ll}
x_{\mathrm{s}}(t_{4}) &  t_{4} < 16\\
x_{\mathrm{e}}(t_{4}) & t_{4} \geq 16
\end{array} \right. \qquad \text{with} \quad x_{\mathrm{s}}(t_{4}) = \left\{ \begin{array}{ll}
8 &  t_{4} < 15\\
-t_{4} + 23 & 15 \leq t_{4} < 16
\end{array} \right. \  , 
\end{equation*}
\begin{equation*}
\qquad \text{and} \quad x_{\mathrm{e}}(t_{4}) = \left\{ \begin{array}{ll}
8 &  16 \leq t_{4} < 21\\
-t_{4} + 29 & 21 \leq t_{4} < 25\\
4 & t_{4} \geq 25
\end{array} \right.
\end{equation*}

\begin{figure}[h!]
\centering
\psfrag{f(x)}[Bl][Bl][.8][0]{$pt^{\circ}_{1,3}(t_{4})$}
\psfrag{x}[bc][Bl][.8][0]{$t_{4}$}
\psfrag{0}[tc][Bl][.8][0]{$0$}
\includegraphics[scale=.2]{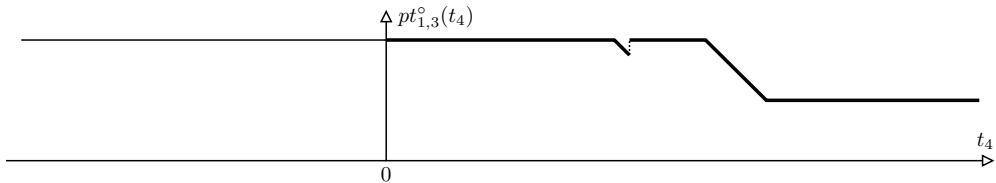}
\caption{Optimal processing time $pt^{\circ}_{1,3}(t_{4})$, under the assumption $\delta_{1} = 1$ in state $[ 2 \; 2 \; 1 \; t_{4}]^{T}$.}
\label{fig:esS3_tau_2_2_1_pt_1_3}
\end{figure}

The conditioned cost-to-go $J^{\circ}_{2,2,1} (t_{4} \mid \delta_{1}=1) = f ( pt^{\circ}_{1,3}(t_{4}) + t_{4} ) + g ( pt^{\circ}_{1,3}(t_{4}) )$, illustrated in figure~\ref{fig:esS3_J_2_2_1_min}, is provided by lemma~\ref{lem:h(t)}. It is specified by the initial value 1.5, by the set \{ 15, 16, 17.5, 21, 25, 27 \} of abscissae $\gamma_{i}$, $i = 1, \ldots, 6$, at which the slope changes, and by the set \{ 1, 0, 0.5, 1, 2, 3 \} of slopes $\mu_{i}$, $i = 1, \ldots, 6$, in the various intervals.

{\it Case ii)} in which it is assumed $\delta_{2} = 1$ (and $\delta_{1} = 0$).

In this case, it is necessary to minimize, with respect to the (continuos) decision variable $\tau$ which corresponds to the processing time $pt_{2,3}$, the following function
\begin{equation*}
\alpha_{2,3} \, \max \{ t_{4} + st_{1,2} + \tau - dd_{2,3} \, , \, 0 \} + \beta_{2} \, (pt^{\mathrm{nom}}_{2} - \tau) + sc_{1,2} + J^{\circ}_{2,3,2} (t_{5})
\end{equation*}
that can be written as $f (pt_{2,3} + t_{4}) + g (pt_{2,3})$ being
\begin{equation*}
f (pt_{2,3} + t_{4}) = \max \{ pt_{2,3} + t_{4} - 37 \, , \, 0 \} + 0.5 + J^{\circ}_{2,3,2} (pt_{2,3} + t_{4} + 1)
\end{equation*}
\begin{equation*}
g (pt_{2,3}) = \left\{ \begin{array}{ll}
1.5 \cdot (6 - pt_{2,3}) & pt_{2,3} \in [ 4 , 6 )\\
0 & pt_{2,3} \notin [ 4 , 6 )
\end{array} \right.
\end{equation*}
The function $pt^{\circ}_{2,3}(t_{4}) = \arg \min_{pt_{2,3}} \{ f (pt_{2,3} + t_{4}) + g (pt_{2,3}) \} $, with $4 \leq pt_{2,3} \leq 6$, is determined by applying lemma~\ref{lem:xopt}. It is (see figure~\ref{fig:esS3_tau_2_2_1_pt_2_3})
\begin{equation*}
pt^{\circ}_{2,3}(t_{4}) = x_{\mathrm{e}}(t_{4}) \qquad \text{with} \quad x_{\mathrm{e}}(t_{4}) = \left\{ \begin{array}{ll}
6 &  t_{4} < 17.5\\
-t_{4} + 23.5 & 17.5 \leq t_{4} < 19.5\\
4 & t_{4} \geq 19.5
\end{array} \right.
\end{equation*}

\begin{figure}[h!]
\centering
\psfrag{f(x)}[Bl][Bl][.8][0]{$pt^{\circ}_{2,3}(t_{4})$}
\psfrag{x}[bc][Bl][.8][0]{$t_{4}$}
\psfrag{0}[tc][Bl][.8][0]{$0$}
\includegraphics[scale=.2]{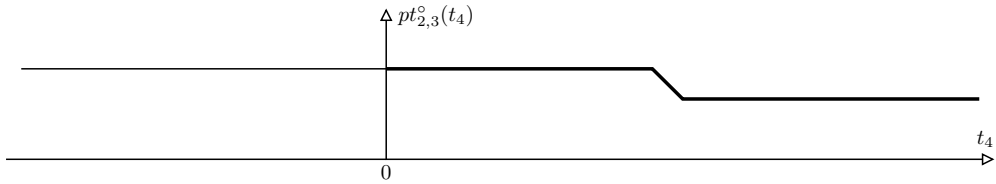}
\caption{Optimal processing time $pt^{\circ}_{2,3}(t_{4})$, under the assumption $\delta_{2} = 1$ in state $[ 2 \; 2 \; 1 \; t_{4}]^{T}$.}
\label{fig:esS3_tau_2_2_1_pt_2_3}
\end{figure}

The conditioned cost-to-go $J^{\circ}_{2,2,1} (t_{4} \mid \delta_{2}=1) = f ( pt^{\circ}_{2,3}(t_{4}) + t_{4} ) + g ( pt^{\circ}_{2,3}(t_{4}) )$, illustrated in figure~\ref{fig:esS3_J_2_2_1_min}, is provided by lemma~\ref{lem:h(t)}. It is specified by the initial value 1.5, by the set \{ 13.5, 17.5, 23.5, 33 \} of abscissae $\gamma_{i}$, $i = 1, \ldots, 4$, at which the slope changes, and by the set \{ 1, 1.5, 2, 3 \} of slopes $\mu_{i}$, $i = 1, \ldots, 4$, in the various intervals.

\begin{figure}[h!]
\centering
\psfrag{J1}[cl][Bc][.8][0]{$J^{\circ}_{2,2,1} (t_{4} \mid \delta_{1} = 1)$}
\psfrag{J2}[br][Bc][.8][0]{$J^{\circ}_{2,2,1} (t_{4} \mid \delta_{2} = 1)$}
\psfrag{0}[cc][tc][.7][0]{$0$}\includegraphics[scale=.6]{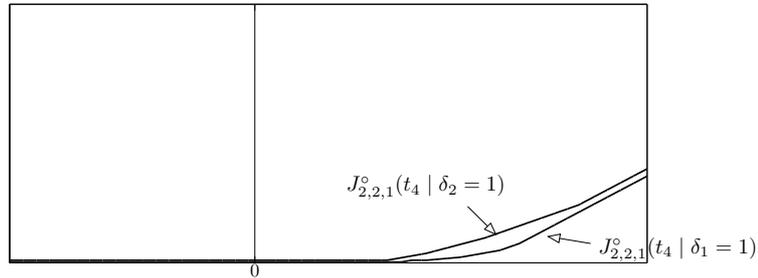}%
\vspace{-12pt}
\caption{Conditioned costs-to-go $J^{\circ}_{2,2,1} (t_{4} \mid \delta_{1} = 1)$ and $J^{\circ}_{2,2,1} (t_{4} \mid \delta_{2} = 1)$ in state $[2 \; 2 \; 1 \; t_{4}]^{T}$.}
\label{fig:esS3_J_2_2_1_min}
\end{figure}

In order to find the optimal cost-to-go $J^{\circ}_{2,2,1} (t_{4})$, it is necessary to carry out the following minimization
\begin{equation*}
J^{\circ}_{2,2,1} (t_{4}) = \min \big\{ J^{\circ}_{2,2,1} (t_{4} \mid \delta_{1} = 1) \, , \, J^{\circ}_{2,2,1} (t_{4} \mid \delta_{2} = 1) \big\}
\end{equation*}
which provides, in accordance with lemma~\ref{lem:min}, the continuous, nondecreasing, piecewise linear function illustrated in figure~\ref{fig:esS3_J_2_2_1}.

\begin{figure}[h!]
\centering
\psfrag{0}[cc][tc][.8][0]{$0$}
\includegraphics[scale=.8]{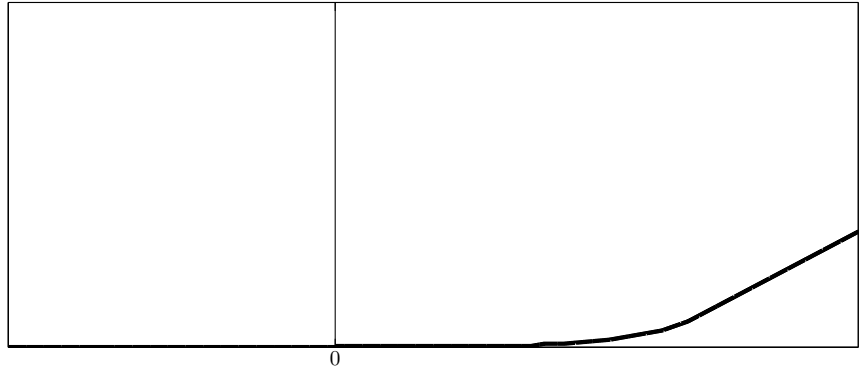}%
\vspace{-12pt}
\caption{Optimal cost-to-go $J^{\circ}_{2,2,1} (t_{4})$ in state $[2 \; 2 \; 1 \; t_{4}]^{T}$.}
\label{fig:esS3_J_2_2_1}
\end{figure}

The function $J^{\circ}_{2,2,1} (t_{4})$ is specified by the initial value 0.5, by the set \{ 15, 16, 17.5, 21, 25, 27 \} of abscissae $\gamma_{i}$, $i = 1, \ldots, 6$, at which the slope changes, and by the set \{ 1, 0, 0.5, 1, 2, 3 \} of slopes $\mu_{i}$, $i = 1, \ldots, 6$, in the various intervals.

Since $J^{\circ}_{2,2,1} (t_{4} \mid \delta_{1} = 1)$ is always the minimum (see again figure~\ref{fig:esS3_J_2_2_1_min}), the optimal control strategies for this state are
\begin{equation*}
\delta_{1}^{\circ} (2,2,1, t_{4}) = 1 \quad \forall \, t_{4} \qquad \delta_{2}^{\circ} (2,2,1, t_{4}) = 0 \quad \forall \, t_{4}
\end{equation*}
\begin{equation*}
\tau^{\circ} (2,2,1, t_{4}) = \left\{ \begin{array}{ll}
8 &  t_{4} < 15\\
-t_{4} + 23 & 15 \leq t_{4} < 16\\
8 &  16 \leq t_{4} < 21\\
-t_{4} + 29 & 21 \leq t_{4} < 25\\
4 & t_{4} \geq 25
\end{array} \right.
\end{equation*}

The optimal control strategy $\tau^{\circ} (2,2,1, t_{4})$ is illustrated in figure~\ref{fig:esS3_tau_2_2_1}.

\begin{figure}[h!]
\centering
\psfrag{f(x)}[Bl][Bl][.8][0]{$\tau^{\circ} (2,2,1, t_{4})$}
\psfrag{x}[bc][Bl][.8][0]{$t_{4}$}
\psfrag{0}[tc][Bl][.8][0]{$0$}
\includegraphics[scale=.2]{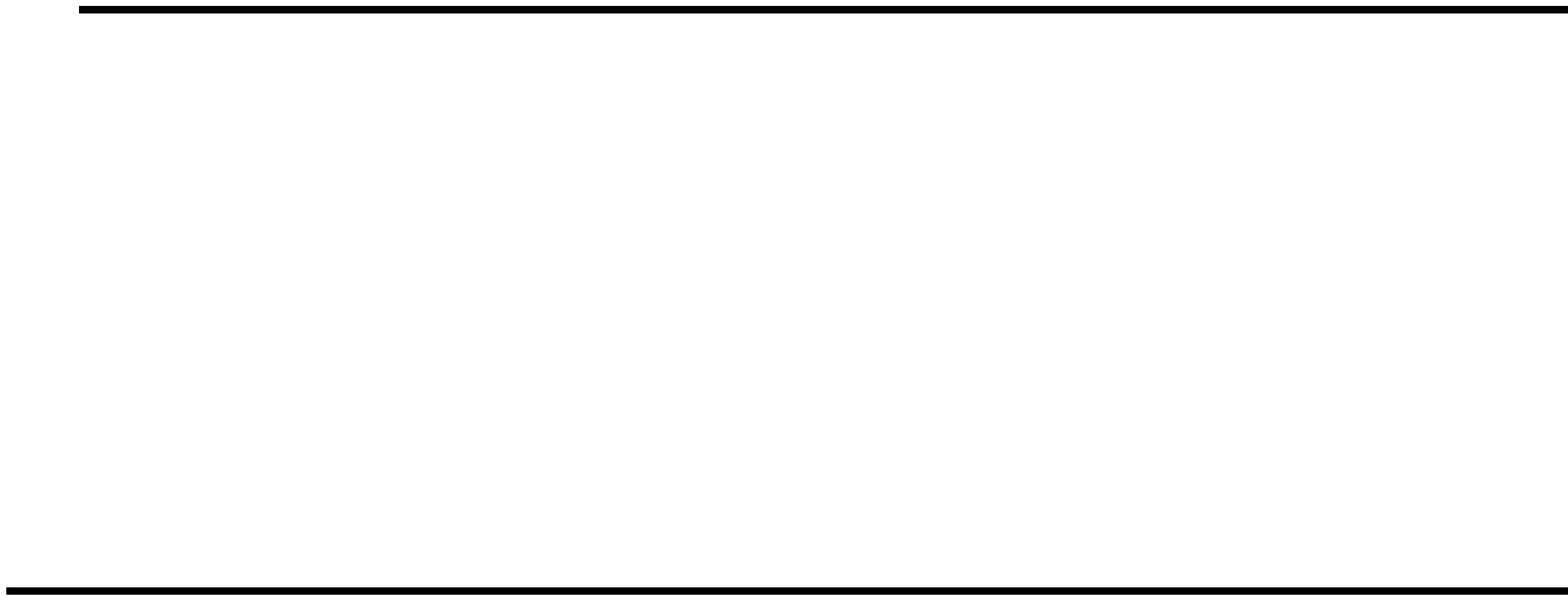}
\caption{Optimal control strategy $\tau^{\circ} (2,2,1, t_{4})$ in state $[ 2 \; 2 \; 1 \; t_{4}]^{T}$.}
\label{fig:esS3_tau_2_2_1}
\end{figure}

%%
%% S15 - [3 1 2 t4]
%%

{\bf Stage $4$ -- State $\boldsymbol{[3 \; 1 \; 2 \; t_{4}]^{T}}$ ($S15$)}

In state $[3 \; 1 \; 2 \; t_{4}]^{T}$, the cost function to be minimized, with respect to the (continuos) decision variable $\tau$ and to the (binary) decision variables $\delta_{1}$ and $\delta_{2}$ is
\begin{equation*}
\begin{split}
&\delta_{1} \big[ \alpha_{1,4} \, \max \{ t_{4} + st_{2,1} + \tau - dd_{1,4} \, , \, 0 \} + \beta_{1} \, ( pt^{\mathrm{nom}}_{1} - \tau ) + sc_{2,1} + J^{\circ}_{4,1,1} (t_{5}) \big] +\\
&+ \delta_{2} \big[ \alpha_{2,2} \, \max \{ t_{4} + st_{2,2} + \tau - dd_{2,2} \, , \, 0 \} + \beta_{2} \, ( pt^{\mathrm{nom}}_{2} - \tau ) + sc_{2,2} + J^{\circ}_{3,2,2} (t_{5}) \big]
\end{split}
\end{equation*}

{\it Case i)} in which it is assumed $\delta_{1} = 1$ (and $\delta_{2} = 0$).

In this case, it is necessary to minimize, with respect to the (continuos) decision variable $\tau$ which corresponds to the processing time $pt_{1,4}$, the following function
\begin{equation*}
\alpha_{1,4} \, \max \{ t_{4} + st_{2,1} + \tau - dd_{1,4} \, , \, 0 \} + \beta_{1} \, ( pt^{\mathrm{nom}}_{1} - \tau ) + sc_{2,1} + J^{\circ}_{4,1,1} (t_{5})
\end{equation*}
that can be written as $f (pt_{1,4} + t_{4}) + g (pt_{1,4})$ being
\begin{equation*}
f (pt_{1,4} + t_{4}) = 0.5 \cdot \max \{ pt_{1,4} + t_{4} - 40.5 \, , \, 0 \} + 1 + J^{\circ}_{4,1,1} (pt_{1,4} + t_{4} + 0.5)
\end{equation*}
\begin{equation*}
g (pt_{1,4}) = \left\{ \begin{array}{ll}
8 - pt_{1,4} & pt_{1,4} \in [ 4 , 8 )\\
0 & pt_{1,4} \notin [ 4 , 8 )
\end{array} \right.
\end{equation*}
The function $pt^{\circ}_{1,4}(t_{4}) = \arg \min_{pt_{1,4}} \{ f (pt_{1,4} + t_{4}) + g (pt_{1,4}) \} $, with $4 \leq pt_{1,4} \leq 8$, is determined by applying lemma~\ref{lem:xopt}. It is (see figure~\ref{fig:esS3_tau_3_1_2_pt_1_4})
\begin{equation*}
pt^{\circ}_{1,4}(t_{4}) = x_{\mathrm{e}}(t_{4}) \qquad \text{with} \quad x_{\mathrm{e}}(t_{4}) = \left\{ \begin{array}{ll}
8 &  t_{4} < 8.5\\
-t_{4} + 16.5 & 8.5 \leq t_{4} < 12.5\\
4 & t_{4} \geq 12.5
\end{array} \right.
\end{equation*}

The conditioned cost-to-go $J^{\circ}_{3,1,2} (t_{4} \mid \delta_{1}=1) = f ( pt^{\circ}_{1,4}(t_{4}) + t_{4} ) + g ( pt^{\circ}_{1,4}(t_{4}) )$, illustrated in figure~\ref{fig:esS3_J_3_1_2_min}, is provided by lemma~\ref{lem:h(t)}. It is specified by the initial value 1.5, by the set \{ 8.5, 20.5, 22.5, 36.5 \} of abscissae $\gamma_{i}$, $i = 1, \ldots, 4$, at which the slope changes, and by the set \{ 1, 1.5, 2, 2.5 \} of slopes $\mu_{i}$, $i = 1, \ldots, 4$, in the various intervals.

\newpage

\begin{figure}[h!]
\centering
\psfrag{f(x)}[Bl][Bl][.8][0]{$pt^{\circ}_{1,4}(t_{4})$}
\psfrag{x}[bc][Bl][.8][0]{$t_{4}$}
\psfrag{0}[tc][Bl][.8][0]{$0$}
\includegraphics[scale=.2]{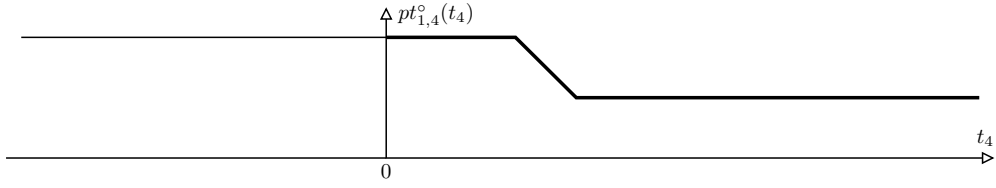}
\caption{Optimal processing time $pt^{\circ}_{1,4}(t_{4})$, under the assumption $\delta_{1} = 1$ in state $[ 3 \; 1 \; 2 \; t_{4}]^{T}$.}
\label{fig:esS3_tau_3_1_2_pt_1_4}
\end{figure}

{\it Case ii)} in which it is assumed $\delta_{2} = 1$ (and $\delta_{1} = 0$).

In this case, it is necessary to minimize, with respect to the (continuos) decision variable $\tau$ which corresponds to the processing time $pt_{2,2}$, the following function
\begin{equation*}
\alpha_{2,2} \, \max \{ t_{4} + st_{2,2} + \tau - dd_{2,2} \, , \, 0 \} + \beta_{2} \, ( pt^{\mathrm{nom}}_{2} - \tau ) + sc_{2,2} + J^{\circ}_{3,2,2} (t_{5})
\end{equation*}
that can be written as $f (pt_{2,2} + t_{4}) + g (pt_{2,2})$ being
\begin{equation*}
f (pt_{2,2} + t_{4}) = \max \{ pt_{2,2} + t_{4} - 24 \, , \, 0 \} + J^{\circ}_{3,2,2} (pt_{2,2} + t_{4})
\end{equation*}
\begin{equation*}
g (pt_{2,2}) = \left\{ \begin{array}{ll}
1.5 \cdot (6 - pt_{2,2}) & pt_{2,2} \in [ 4 , 6 )\\
0 & pt_{2,2} \notin [ 4 , 6 )
\end{array} \right.
\end{equation*}
The function $pt^{\circ}_{2,2}(t_{4}) = \arg \min_{pt_{2,2}} \{ f (pt_{2,2} + t_{4}) + g (pt_{2,2}) \} $, with $4 \leq pt_{2,2} \leq 6$, is determined by applying lemma~\ref{lem:xopt}. It is (see figure~\ref{fig:esS3_tau_3_1_2_pt_2_2})
\begin{equation*}
pt^{\circ}_{2,2}(t_{4}) = x_{\mathrm{e}}(t_{4}) \qquad \text{with} \quad x_{\mathrm{e}}(t_{4}) = \left\{ \begin{array}{ll}
6 &  t_{4} < 20.5\\
-t_{4} + 26.5 & 20.5 \leq t_{4} < 22.5\\
4 & t_{4} \geq 22.5
\end{array} \right.
\end{equation*}

\begin{figure}[h!]
\centering
\psfrag{f(x)}[Bl][Bl][.8][0]{$pt^{\circ}_{2,2}(t_{4})$}
\psfrag{x}[bc][Bl][.8][0]{$t_{4}$}
\psfrag{0}[tc][Bl][.8][0]{$0$}
\includegraphics[scale=.2]{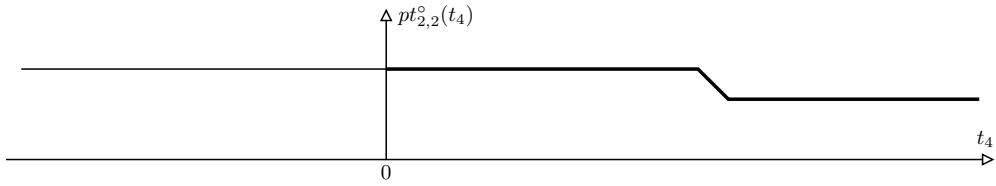}
\caption{Optimal processing time $pt^{\circ}_{2,2}(t_{4})$, under the assumption $\delta_{2} = 1$ in state $[ 3 \; 1 \; 2 \; t_{4}]^{T}$.}
\label{fig:esS3_tau_3_1_2_pt_2_2}
\end{figure}

The conditioned cost-to-go $J^{\circ}_{3,1,2} (t_{4} \mid \delta_{2}=1) = f ( pt^{\circ}_{2,2}(t_{4}) + t_{4} ) + g ( pt^{\circ}_{2,2}(t_{4}) )$, illustrated in figure~\ref{fig:esS3_J_3_1_2_min}, is provided by lemma~\ref{lem:h(t)}. It is specified by the initial value 1, by the set \{ 18, 20.5, 28 \} of abscissae $\gamma_{i}$, $i = 1, \ldots, 3$, at which the slope changes, and by the set \{ 1, 1.5, 2.5 \} of slopes $\mu_{i}$, $i = 1, \ldots, 3$, in the various intervals.

\begin{figure}[h!]
\centering
\psfrag{J1}[br][Bc][.8][0]{$J^{\circ}_{3,1,2} (t_{4} \mid \delta_{1} = 1)$}
\psfrag{J2}[cl][Bc][.8][0]{$J^{\circ}_{3,1,2} (t_{4} \mid \delta_{2} = 1)$}
\psfrag{0}[cc][tc][.7][0]{$0$}\includegraphics[scale=.6]{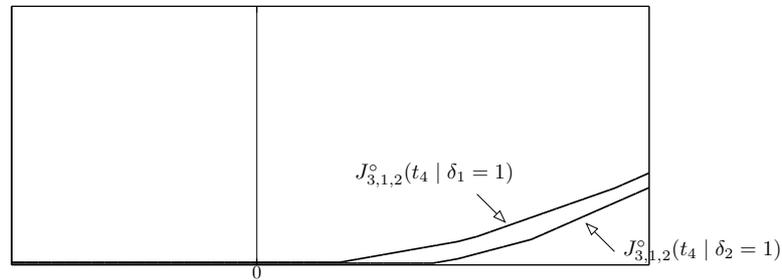}%
\vspace{-12pt}
\caption{Conditioned costs-to-go $J^{\circ}_{3,1,2} (t_{4} \mid \delta_{1} = 1)$ and $J^{\circ}_{3,1,2} (t_{4} \mid \delta_{2} = 1)$ in state $[3 \; 1 \; 2 \; t_{4}]^{T}$.}
\label{fig:esS3_J_3_1_2_min}
\end{figure}

In order to find the optimal cost-to-go $J^{\circ}_{3,1,2} (t_{4})$, it is necessary to carry out the following minimization
\begin{equation*}
J^{\circ}_{3,1,2} (t_{4}) = \min \big\{ J^{\circ}_{3,1,2} (t_{4} \mid \delta_{1} = 1) \, , \, J^{\circ}_{3,1,2} (t_{4} \mid \delta_{2} = 1) \big\}
\end{equation*}
which provides, in accordance with lemma~\ref{lem:min}, the continuous, nondecreasing, piecewise linear function illustrated in figure~\ref{fig:esS3_J_3_1_2}.

\begin{figure}[h!]
\centering
\psfrag{0}[cc][tc][.8][0]{$0$}
\includegraphics[scale=.8]{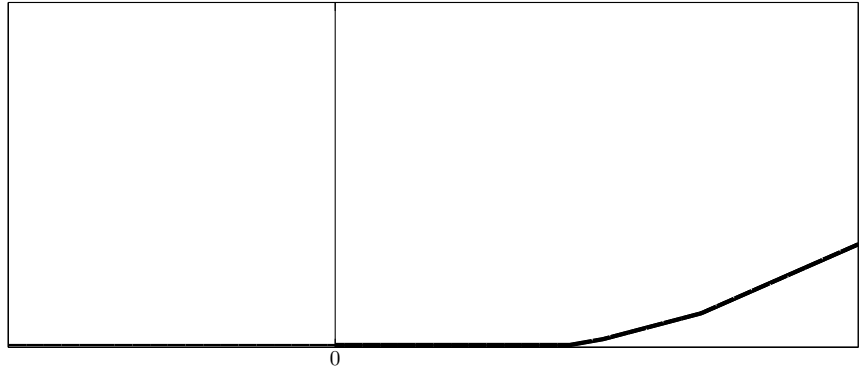}%
\vspace{-12pt}
\caption{Optimal cost-to-go $J^{\circ}_{3,1,2} (t_{4})$ in state $[3 \; 1 \; 2 \; t_{4}]^{T}$.}
\label{fig:esS3_J_3_1_2}
\end{figure}

The function $J^{\circ}_{3,1,2} (t_{4})$ is specified by the initial value 1, by the set \{ 18, 20.5, 28 \} of abscissae $\gamma_{i}$, $i = 1, \ldots, 3$, at which the slope changes, and by the set \{ 1, 1.5, 2.5 \} of slopes $\mu_{i}$, $i = 1, \ldots, 3$, in the various intervals.

Since $J^{\circ}_{3,1,2} (t_{4} \mid \delta_{2} = 1)$ is always the minimum (see again figure~\ref{fig:esS3_J_3_1_2_min}), the optimal control strategies for this state are
\begin{equation*}
\delta_{1}^{\circ} (3,1,2, t_{4}) = 0 \quad \forall \, t_{4} \qquad \delta_{2}^{\circ} (3,1,2, t_{4}) = 1 \quad \forall \, t_{4}
\end{equation*}
\begin{equation*}
\tau^{\circ} (3,1,2, t_{4}) = \left\{ \begin{array}{ll}
6 &  t_{4} < 20.5\\
-t_{4} + 26.5 & 20.5 \leq t_{4} < 22.5\\
4 & t_{4} \geq 22.5
\end{array} \right.
\end{equation*}

The optimal control strategy $\tau^{\circ} (3,1,2, t_{4})$ is illustrated in figure~\ref{fig:esS3_tau_3_1_2}.

\begin{figure}[h!]
\centering
\psfrag{f(x)}[Bl][Bl][.8][0]{$\tau^{\circ} (3,1,2, t_{4})$}
\psfrag{x}[bc][Bl][.8][0]{$t_{4}$}
\psfrag{0}[tc][Bl][.8][0]{$0$}
\includegraphics[scale=.2]{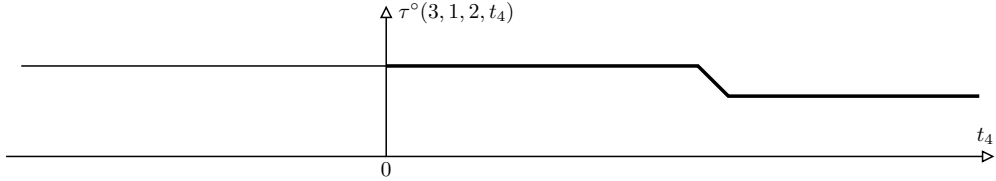}
\caption{Optimal control strategy $\tau^{\circ} (3,1,2, t_{4})$ in state $[ 3 \; 1 \; 2 \; t_{4}]^{T}$.}
\label{fig:esS3_tau_3_1_2}
\end{figure}

%%
%% S14 - [3 1 1 t4]
%%

{\bf Stage $4$ -- State $\boldsymbol{[3 \; 1 \; 1 \; t_{4}]^{T}}$ ($S14$)}

In state $[3 \; 1 \; 1 \; t_{4}]^{T}$, the cost function to be minimized, with respect to the (continuos) decision variable $\tau$ and to the (binary) decision variables $\delta_{1}$ and $\delta_{2}$ is
\begin{equation*}
\begin{split}
&\delta_{1} \big[ \alpha_{1,4} \, \max \{ t_{4} + st_{1,1} + \tau - dd_{1,4} \, , \, 0 \} + \beta_{1} \, ( pt^{\mathrm{nom}}_{1} - \tau ) + sc_{1,1} + J^{\circ}_{4,1,1} (t_{5}) \big] +\\
&+ \delta_{2} \big[ \alpha_{2,2} \, \max \{ t_{4} + st_{1,2} + \tau - dd_{2,2} \, , \, 0 \} + \beta_{2} \, ( pt^{\mathrm{nom}}_{2} - \tau ) + sc_{1,2} + J^{\circ}_{3,2,2} (t_{5}) \big]
\end{split}
\end{equation*}

{\it Case i)} in which it is assumed $\delta_{1} = 1$ (and $\delta_{2} = 0$).

In this case, it is necessary to minimize, with respect to the (continuos) decision variable $\tau$ which corresponds to the processing time $pt_{1,4}$, the following function
\begin{equation*}
\alpha_{1,4} \, \max \{ t_{4} + st_{1,1} + \tau - dd_{1,4} \, , \, 0 \} + \beta_{1} \, ( pt^{\mathrm{nom}}_{1} - \tau ) + sc_{1,1} + J^{\circ}_{4,1,1} (t_{5})
\end{equation*}
that can be written as $f (pt_{1,4} + t_{4}) + g (pt_{1,4})$ being
\begin{equation*}
f (pt_{1,4} + t_{4}) = 0.5 \cdot \max \{ pt_{1,4} + t_{4} - 41 \, , \, 0 \} + J^{\circ}_{4,1,1} (pt_{1,4} + t_{4})
\end{equation*}
\begin{equation*}
g (pt_{1,4}) = \left\{ \begin{array}{ll}
8 - pt_{1,4} & pt_{1,4} \in [ 4 , 8 )\\
0 & pt_{1,4} \notin [ 4 , 8 )
\end{array} \right.
\end{equation*}
The function $pt^{\circ}_{1,4}(t_{4}) = \arg \min_{pt_{1,4}} \{ f (pt_{1,4} + t_{4}) + g (pt_{1,4}) \} $, with $4 \leq pt_{1,4} \leq 8$, is determined by applying lemma~\ref{lem:xopt}. It is (see figure~\ref{fig:esS3_tau_3_1_1_pt_1_4})
\begin{equation*}
pt^{\circ}_{1,4}(t_{4}) = x_{\mathrm{e}}(t_{4}) \qquad \text{with} \quad x_{\mathrm{e}}(t_{4}) = \left\{ \begin{array}{ll}
8 &  t_{4} < 9\\
-t_{4} + 17 & 9 \leq t_{4} < 13\\
4 & t_{4} \geq 13
\end{array} \right.
\end{equation*}

\begin{figure}[h!]
\centering
\psfrag{f(x)}[Bl][Bl][.8][0]{$pt^{\circ}_{1,4}(t_{4})$}
\psfrag{x}[bc][Bl][.8][0]{$t_{4}$}
\psfrag{0}[tc][Bl][.8][0]{$0$}
\includegraphics[scale=.2]{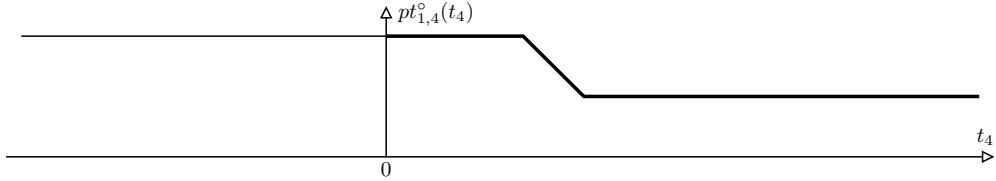}
\caption{Optimal processing time $pt^{\circ}_{1,4}(t_{4})$, under the assumption $\delta_{1} = 1$ in state $[ 3 \; 1 \; 1 \; t_{4}]^{T}$.}
\label{fig:esS3_tau_3_1_1_pt_1_4}
\end{figure}

The conditioned cost-to-go $J^{\circ}_{3,1,1} (t_{4} \mid \delta_{1}=1) = f ( pt^{\circ}_{1,4}(t_{4}) + t_{4} ) + g ( pt^{\circ}_{1,4}(t_{4}) )$, illustrated in figure~\ref{fig:esS3_J_3_1_1_min}, is provided by lemma~\ref{lem:h(t)}. It is specified by the initial value 0.5, by the set \{ 9, 21, 23, 37 \} of abscissae $\gamma_{i}$, $i = 1, \ldots, 4$, at which the slope changes, and by the set \{ 1, 1.5, 2, 2.5 \} of slopes $\mu_{i}$, $i = 1, \ldots, 4$, in the various intervals.

{\it Case ii)} in which it is assumed $\delta_{2} = 1$ (and $\delta_{1} = 0$).

In this case, it is necessary to minimize, with respect to the (continuos) decision variable $\tau$ which corresponds to the processing time $pt_{2,2}$, the following function
\begin{equation*}
\alpha_{2,2} \, \max \{ t_{4} + st_{1,2} + \tau - dd_{2,2} \, , \, 0 \} + \beta_{2} \, ( pt^{\mathrm{nom}}_{2} - \tau ) + sc_{1,2} + J^{\circ}_{3,2,2} (t_{5})
\end{equation*}
that can be written as $f (pt_{2,2} + t_{4}) + g (pt_{2,2})$ being
\begin{equation*}
f (pt_{2,2} + t_{4}) = \max \{ pt_{2,2} + t_{4} - 23 \, , \, 0 \} + 0.5 + J^{\circ}_{3,2,2} (pt_{2,2} + t_{4} + 1)
\end{equation*}
\begin{equation*}
g (pt_{2,2}) = \left\{ \begin{array}{ll}
1.5 \cdot (6 - pt_{2,2}) & pt_{2,2} \in [ 4 , 6 )\\
0 & pt_{2,2} \notin [ 4 , 6 )
\end{array} \right.
\end{equation*}
The function $pt^{\circ}_{2,2}(t_{4}) = \arg \min_{pt_{2,2}} \{ f (pt_{2,2} + t_{4}) + g (pt_{2,3}) \} $, with $4 \leq pt_{2,2} \leq 6$, is determined by applying lemma~\ref{lem:xopt}. It is (see figure~\ref{fig:esS3_tau_3_1_1_pt_2_2})
\begin{equation*}
pt^{\circ}_{2,2}(t_{4}) = x_{\mathrm{e}}(t_{4}) \qquad \text{with} \quad x_{\mathrm{e}}(t_{4}) = \left\{ \begin{array}{ll}
6 &  t_{4} < 19.5\\
-t_{4} + 25.5 & 19.5 \leq t_{4} < 21.5\\
4 & t_{4} \geq 21.5
\end{array} \right.
\end{equation*}

\begin{figure}[h!]
\centering
\psfrag{f(x)}[Bl][Bl][.8][0]{$pt^{\circ}_{2,2}(t_{4})$}
\psfrag{x}[bc][Bl][.8][0]{$t_{4}$}
\psfrag{0}[tc][Bl][.8][0]{$0$}
\includegraphics[scale=.2]{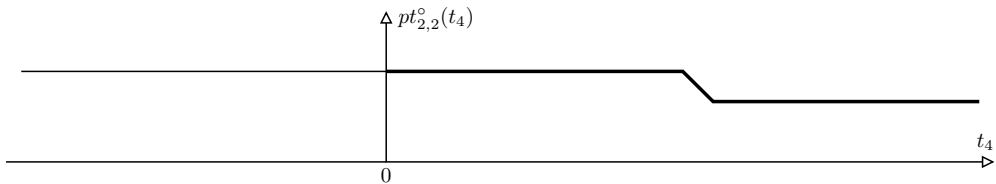}
\caption{Optimal processing time $pt^{\circ}_{2,2}(t_{4})$, under the assumption $\delta_{2} = 1$ in state $[ 3 \; 1 \; 1 \; t_{4}]^{T}$.}
\label{fig:esS3_tau_3_1_1_pt_2_2}
\end{figure}

The conditioned cost-to-go $J^{\circ}_{3,1,1} (t_{4} \mid \delta_{2}=1) = f ( pt^{\circ}_{2,2}(t_{4}) + t_{4} ) + g ( pt^{\circ}_{2,2}(t_{4}) )$, illustrated in figure~\ref{fig:esS3_J_3_1_1_min}, is provided by lemma~\ref{lem:h(t)}. It is specified by the initial value 1, by the set \{ 17, 19.5, 27 \} of abscissae $\gamma_{i}$, $i = 1, \ldots, 3$, at which the slope changes, and by the set \{ 1, 1.5, 2.5 \} of slopes $\mu_{i}$, $i = 1, \ldots, 3$, in the various intervals.

\begin{figure}[h!]
\centering
\psfrag{J1}[br][Bc][.8][0]{$J^{\circ}_{3,1,1} (t_{4} \mid \delta_{1} = 1)$}
\psfrag{J2}[cl][Bc][.8][0]{$J^{\circ}_{3,1,1} (t_{4} \mid \delta_{2} = 1)$}
\psfrag{0}[cc][tc][.7][0]{$0$}\includegraphics[scale=.6]{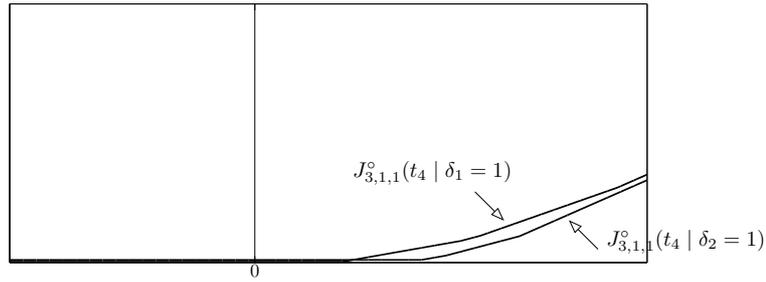}%
\vspace{-12pt}
\caption{Conditioned costs-to-go $J^{\circ}_{3,1,1} (t_{4} \mid \delta_{1} = 1)$ and $J^{\circ}_{3,1,1} (t_{4} \mid \delta_{2} = 1)$ in state $[3 \; 1 \; 1 \; t_{4}]^{T}$.}
\label{fig:esS3_J_3_1_1_min}
\end{figure}

In order to find the optimal cost-to-go $J^{\circ}_{3,1,1} (t_{4})$, it is necessary to carry out the following minimization
\begin{equation*}
J^{\circ}_{3,1,1} (t_{4}) = \min \big\{ J^{\circ}_{3,1,1} (t_{4} \mid \delta_{1} = 1) \, , \, J^{\circ}_{3,1,1} (t_{4} \mid \delta_{2} = 1) \big\}
\end{equation*}
which provides, in accordance with lemma~\ref{lem:min}, the continuous, nondecreasing, piecewise linear function illustrated in figure~\ref{fig:esS3_J_3_1_1}.

\begin{figure}[h!]
\centering
\psfrag{0}[cc][tc][.8][0]{$0$}
\includegraphics[scale=.8]{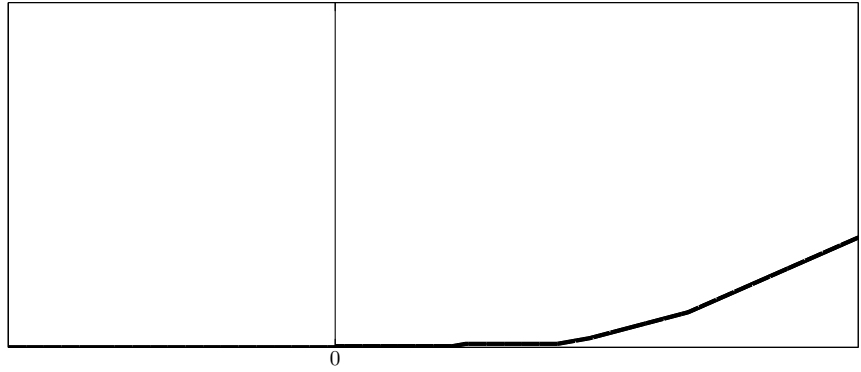}%
\vspace{-12pt}
\caption{Optimal cost-to-go $J^{\circ}_{3,1,1} (t_{4})$ in state $[3 \; 1 \; 1 \; t_{4}]^{T}$.}
\label{fig:esS3_J_3_1_1}
\end{figure}

The function $J^{\circ}_{3,1,1} (t_{4})$ is specified by the initial value 0.5, by the set \{ 9, 10, 17, 19.5, 27 \} of abscissae $\gamma_{i}$, $i = 1, \ldots, 5$, at which the slope changes, and by the set \{ 1, 0, 1, 1.5, 2.5 \} of slopes $\mu_{i}$, $i = 1, \ldots, 5$, in the various intervals.

Since $J^{\circ}_{3,1,1} (t_{4} \mid \delta_{1} = 1)$ is the minimum in $(-\infty,10)$, and $J^{\circ}_{3,1,1} (t_{4} \mid \delta_{2} = 1)$ is the minimum in $[10,+\infty)$, the optimal control strategies for this state are
\begin{equation*}
\delta_{1}^{\circ} (3,1,1, t_{4}) = \left\{ \begin{array}{ll}
1 &  t_{4} < 10\\
0 & t_{4} \geq 10
\end{array} \right. \qquad \delta_{2}^{\circ} (3,1,1, t_{4}) = \left\{ \begin{array}{ll}
0 &  t_{4} < 10\\
1 & t_{4} \geq 10
\end{array} \right.
\end{equation*}
\begin{equation*}
\tau^{\circ} (3,1,1, t_{4}) = \left\{ \begin{array}{ll}
8 &  t_{4} < 9\\
-t_{4} + 17 & 9 \leq t_{4} < 24.5\\
6 & 10 \leq t_{4} < 19.5\\
-t_{4} + 25.5 & 19.5 \leq t_{4} < 21.5\\
4 & t_{5} \geq 21.5
\end{array} \right.
\end{equation*}

The optimal control strategy $\tau^{\circ} (3,1,1, t_{4})$ is illustrated in figure~\ref{fig:esS3_tau_3_1_1}.

\begin{figure}[h!]
\centering
\psfrag{f(x)}[Bl][Bl][.8][0]{$\tau^{\circ} (3,1,1, t_{4})$}
\psfrag{x}[bc][Bl][.8][0]{$t_{4}$}
\psfrag{0}[tc][Bl][.8][0]{$0$}
\includegraphics[scale=.2]{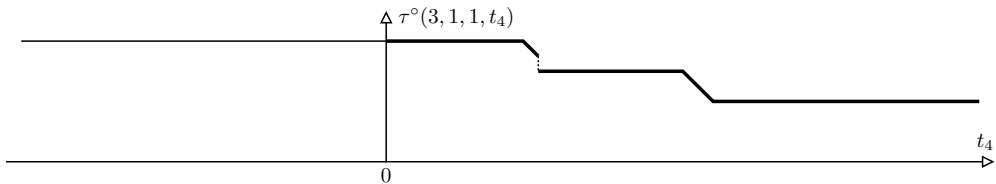}
\caption{Optimal control strategy $\tau^{\circ} (3,1,1, t_{4})$ in state $[ 3 \; 1 \; 1 \; t_{4}]^{T}$.}
\label{fig:esS3_tau_3_1_1}
\end{figure}

%%
%% S13 - [4 0 1 t4]
%%

{\bf Stage $4$ -- State $\boldsymbol{[4 \; 0 \; 1 \; t_{4}]^{T}}$ ($S13$)}

In state $[4 \; 0 \; 1 \; t_{4}]^{T}$ all jobs of class $P_{1}$ have been completed; then the decision about the class of the next job to be executed is mandatory. The cost function to be minimized in this state, with respect to the (continuos) decision variable $\tau$ only (which corresponds to the processing time $pt_{2,1}$), is
\begin{equation*}
\alpha_{2,1} \, \max \{ t_{4} + st_{1,2} + \tau - dd_{2,1} \, , \, 0 \} + \beta_{2} \, (pt^{\mathrm{nom}}_{2} - \tau) + sc_{1,2} + J^{\circ}_{4,1,2} (t_{5})
\end{equation*}
that can be written as $f (pt_{2,1} + t_{4}) + g (pt_{2,1})$ being
\begin{equation*}
f (pt_{2,1} + t_{4}) = 2 \cdot \max \{ pt_{2,1} + t_{4} - 20 \, , \, 0 \} + 0.5 + J^{\circ}_{4,1,2} (pt_{2,1} + t_{4} + 1)
\end{equation*}
\begin{equation*}
g (pt_{2,1}) = \left\{ \begin{array}{ll}
1.5 \cdot (6 - pt_{2,1}) & pt_{2,1} \in [ 4 , 6 )\\
0 & pt_{2,1} \notin [ 4 , 6 )
\end{array} \right.
\end{equation*}
The function $pt^{\circ}_{2,1}(t_{4}) = \arg \min_{pt_{2,1}} \{ f (pt_{2,1} + t_{4}) + g (pt_{2,1}) \} $, with $4 \leq pt_{2,1} \leq 6$, is determined by applying lemma~\ref{lem:xopt}. It is
\begin{equation*}
pt^{\circ}_{2,1}(t_{4}) = x_{\mathrm{e}}(t_{4}) \qquad \text{with} \quad x_{\mathrm{e}}(t_{4}) = \left\{ \begin{array}{ll}
6 &  t_{4} < 14\\
-t_{4} + 20 & 14 \leq t_{4} < 16\\
4 & t_{4} \geq 16
\end{array} \right.
\end{equation*}

Taking into account the mandatory decision about the class of the next job to be executed, the optimal control strategies for this state are
\begin{equation*}
\delta_{1}^{\circ} (4,0,1, t_{4}) = 0 \quad \forall \, t_{4} \qquad \delta_{2}^{\circ} (4,0,1, t_{4}) = 1 \quad \forall \, t_{4}
\end{equation*}
\begin{equation*}
\tau^{\circ} (4,0,1, t_{4}) = \left\{ \begin{array}{ll}
6 &  t_{4} < 14\\
-t_{4} + 20 & 14 \leq t_{4} < 16\\
4 & t_{4} \geq 16
\end{array} \right.
\end{equation*}
The optimal control strategy $\tau^{\circ} (4,0,1, t_{4})$ is illustrated in figure~\ref{fig:esS3_tau_4_0_1}.

\begin{figure}[h!]
\centering
\psfrag{f(x)}[Bl][Bl][.8][0]{$\tau^{\circ} (4,0,1, t_{4})$}
\psfrag{x}[bc][Bl][.8][0]{$t_{5}$}
\psfrag{0}[tc][Bl][.8][0]{$0$}
\includegraphics[scale=.2]{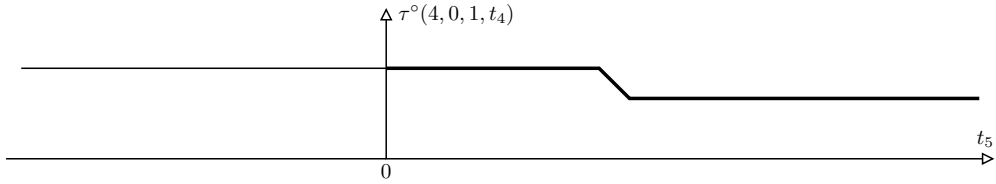}
\caption{Optimal control strategy $\tau^{\circ} (4,0,1, t_{4})$ in state $[ 4 \; 0 \; 1 \; t_{4}]^{T}$.}
\label{fig:esS3_tau_4_0_1}
\end{figure}

The optimal cost-to-go $J^{\circ}_{4,0,1} (t_{4}) = f ( pt^{\circ}_{2,1}(t_{4}) + t_{4} ) + g ( pt^{\circ}_{2,1}(t_{4}) )$, illustrated in figure~\ref{fig:esS2_J_4_0_1}, is provided by lemma~\ref{lem:h(t)}. It is specified by the initial value 0.5, by the set \{ 11, 14, 16, 21, 23 \} of abscissae $\gamma_{i}$, $i = 1, \ldots, 5$, at which the slope changes, and by the set \{ 1, 1.5, 3, 3.5, 4 \} of slopes $\mu_{i}$, $i = 1, \ldots, 5$, in the various intervals.

\begin{figure}[h!]
\centering
\psfrag{0}[cc][tc][.8][0]{$0$}
\includegraphics[scale=.8]{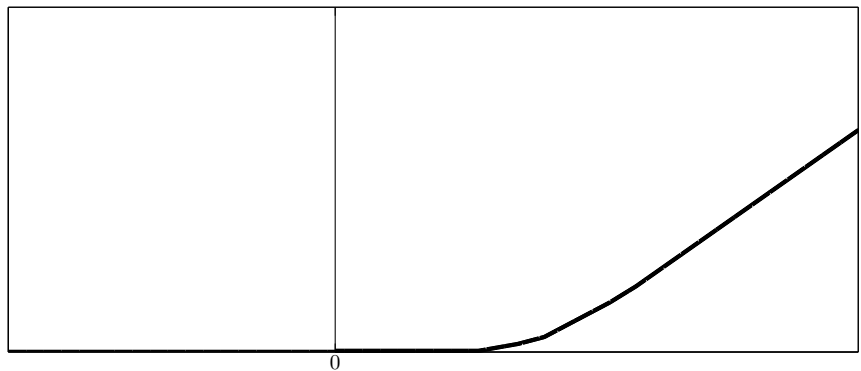}%
\vspace{-12pt}
\caption{Optimal cost-to-go $J^{\circ}_{4,0,1} (t_{4})$ in state $[4 \; 0 \; 1 \; t_{4}]^{T}$.}
\label{fig:esS2_J_4_0_1}
\end{figure}

%%
%% S12 - [0 3 2 t3]
%%

{\bf Stage $3$ -- State $\boldsymbol{[0 \; 3 \; 2 \; t_{3}]^{T}}$ ($S12$)}

In state $[0 \; 3 \; 2 \; t_{3}]^{T}$ all jobs of class $P_{2}$ have been completed; then the decision about the class of the next job to be executed is mandatory. The cost function to be minimized in this state, with respect to the (continuos) decision variable $\tau$ only (which corresponds to the processing time $pt_{1,1}$), is
\begin{equation*}
\alpha_{1,1} \, \max \{ t_{3} + st_{2,1} + \tau - dd_{1,1} \, , \, 0 \} + \beta_{1} \, (pt^{\mathrm{nom}}_{1} - \tau) + sc_{2,1} + J^{\circ}_{1,3,1} (t_{4})
\end{equation*}
that can be written as $f (pt_{1,1} + t_{3}) + g (pt_{1,1})$ being
\begin{equation*}
f (pt_{1,1} + t_{3}) = 0.75 \cdot \max \{ pt_{1,1} + t_{3} - 18.5 \, , \, 0 \} + 1 + J^{\circ}_{1,3,1} (pt_{1,1} + t_{3} + 0.5)
\end{equation*}
\begin{equation*}
g (pt_{1,1}) = \left\{ \begin{array}{ll}
8 - pt_{1,1} & pt_{1,1} \in [ 4 , 8 )\\
0 & pt_{1,1} \notin [ 4 , 8 )
\end{array} \right.
\end{equation*}
The function $pt^{\circ}_{1,1}(t_{3}) = \arg \min_{pt_{1,1}} \{ f (pt_{1,1} + t_{3}) + g (pt_{1,1}) \} $, with $4 \leq pt_{1,1} \leq 8$, is determined by applying lemma~\ref{lem:xopt}. It is
\begin{equation*}
pt^{\circ}_{1,1}(t_{3}) = x_{\mathrm{e}}(t_{3}) \qquad \text{with} \quad x_{\mathrm{e}}(t_{3}) = \left\{ \begin{array}{ll}
8 &  t_{3} < 4.5\\
-t_{3} + 12.5 & 4.5 \leq t_{3} < 8.5\\
4 & t_{3} \geq 8.5
\end{array} \right.
\end{equation*}

Taking into account the mandatory decision about the class of the next job to be executed, the optimal control strategies for this state are
\begin{equation*}
\delta_{1}^{\circ} (0,3,2, t_{3}) = 1 \quad \forall \, t_{3} \qquad \delta_{2}^{\circ} (0,3,2, t_{3}) = 0 \quad \forall \, t_{3}
\end{equation*}
\begin{equation*}
\tau^{\circ} (0,3,2, t_{3}) = \left\{ \begin{array}{ll}
8 &  t_{3} < 4.5\\
-t_{3} + 12.5 & 4.5 \leq t_{3} < 8.5\\
4 & t_{3} \geq 8.5
\end{array} \right.
\end{equation*}
The optimal control strategy $\tau^{\circ} (0,3,2, t_{3})$ is illustrated in figure~\ref{fig:esS3_tau_0_3_2}.

\begin{figure}[h!]
\centering
\psfrag{f(x)}[Bl][Bl][.8][0]{$\tau^{\circ} (0,3,2, t_{3})$}
\psfrag{x}[bc][Bl][.8][0]{$t_{3}$}
\psfrag{0}[tc][Bl][.8][0]{$0$}
\includegraphics[scale=.2]{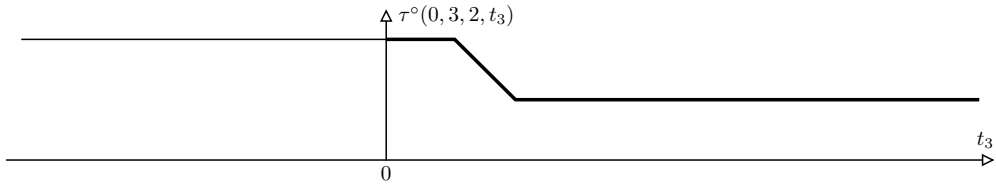}
\caption{Optimal control strategy $\tau^{\circ} (0,3,2, t_{3})$ in state $[ 0 \; 3 \; 2 \; t_{3}]^{T}$.}
\label{fig:esS3_tau_0_3_2}
\end{figure}

The optimal cost-to-go $J^{\circ}_{0,3,2} (t_{3}) = f ( pt^{\circ}_{1,1}(t_{3}) + t_{3} ) + g ( pt^{\circ}_{1,1}(t_{3}) )$, illustrated in figure~\ref{fig:esS2_J_0_3_2}, is provided by lemma~\ref{lem:h(t)}. It is specified by the initial value 1, by the set \{ 4.5, 14.5, 15.5, 16.5, 20.5 \} of abscissae $\gamma_{i}$, $i = 1, \ldots, 5$, at which the slope changes, and by the set \{ 1, 1.75, 2.25, 2.75, 3.25 \} of slopes $\mu_{i}$, $i = 1, \ldots, 5$, in the various intervals.

\begin{figure}[h!]
\centering
\psfrag{0}[cc][tc][.8][0]{$0$}
\includegraphics[scale=.8]{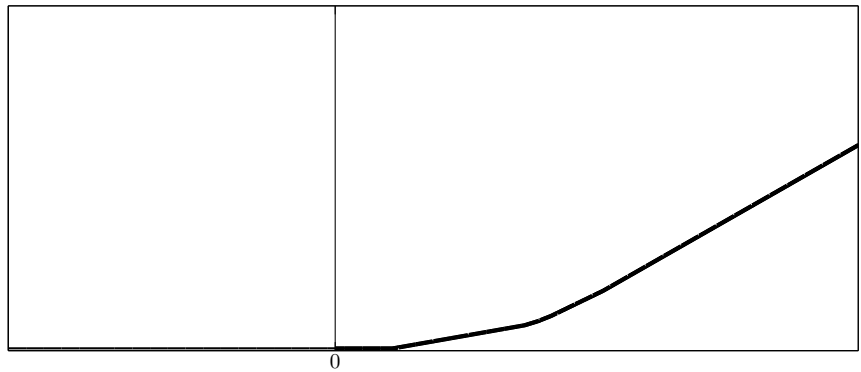}%
\vspace{-12pt}
\caption{Optimal cost-to-go $J^{\circ}_{0,3,2} (t_{3})$ in state $[0 \; 3 \; 2 \; t_{3}]^{T}$.}
\label{fig:esS2_J_0_3_2}
\end{figure}

%%
%% S11 - [1 2 2 t3]
%%

{\bf Stage $3$ -- State $\boldsymbol{[1 \; 2 \; 2 \; t_{3}]^{T}}$ ($S11$)}

In state $[1 \; 2 \; 2 \; t_{3}]^{T}$, the cost function to be minimized, with respect to the (continuos) decision variable $\tau$ and to the (binary) decision variables $\delta_{1}$ and $\delta_{2}$ is
\begin{equation*}
\begin{split}
&\delta_{1} \big[ \alpha_{1,2} \, \max \{ t_{3} + st_{2,1} + \tau - dd_{1,2} \, , \, 0 \} + \beta_{1} \, ( pt^{\mathrm{nom}}_{1} - \tau ) + sc_{2,1} + J^{\circ}_{2,2,1} (t_{4}) \big] +\\
&+ \delta_{2} \big[ \alpha_{2,3} \, \max \{ t_{3} + st_{2,2} + \tau - dd_{2,3} \, , \, 0 \} + \beta_{2} \, ( pt^{\mathrm{nom}}_{2} - \tau ) + sc_{2,2} + J^{\circ}_{1,3,2} (t_{4}) \big]
\end{split}
\end{equation*}

{\it Case i)} in which it is assumed $\delta_{1} = 1$ (and $\delta_{2} = 0$).

In this case, it is necessary to minimize, with respect to the (continuos) decision variable $\tau$ which corresponds to the processing time $pt_{1,2}$, the following function
\begin{equation*}
\alpha_{1,2} \, \max \{ t_{3} + st_{2,1} + \tau - dd_{1,2} \, , \, 0 \} + \beta_{1} \, ( pt^{\mathrm{nom}}_{1} - \tau ) + sc_{2,1} + J^{\circ}_{2,2,1} (t_{4})
\end{equation*}
that can be written as $f (pt_{1,2} + t_{3}) + g (pt_{1,2})$ being
\begin{equation*}
f (pt_{1,2} + t_{3}) = 0.5 \cdot \max \{ pt_{1,2} + t_{3} - 23.5 \, , \, 0 \} + 1 + J^{\circ}_{2,2,1} (pt_{1,2} + t_{3} + 0.5)
\end{equation*}
\begin{equation*}
g (pt_{1,2}) = \left\{ \begin{array}{ll}
8 - pt_{1,2} & pt_{1,2} \in [ 4 , 8 )\\
0 & pt_{1,2} \notin [ 4 , 8 )
\end{array} \right.
\end{equation*}
The function $pt^{\circ}_{1,2}(t_{3}) = \arg \min_{pt_{1,2}} \{ f (pt_{1,2} + t_{3}) + g (pt_{1,2}) \} $, with $4 \leq pt_{1,2} \leq 8$, is determined by applying lemma~\ref{lem:xopt}. It is (see figure~\ref{fig:esS3_tau_1_2_2_pt_1_2})
\begin{equation*}
pt^{\circ}_{1,2}(t_{3}) = \left\{ \begin{array}{ll}
x_{\mathrm{s}}(t_{3}) &  t_{3} < 7.5\\
x_{\mathrm{e}}(t_{3}) & t_{3} \geq 7.5
\end{array} \right. \qquad \text{with} \quad x_{\mathrm{s}}(t_{3}) = \left\{ \begin{array}{ll}
8 &  t_{3} < 6.5\\
-t_{3} + 14.5 & 6.5 \leq t_{3} < 7.5
\end{array} \right. \  , 
\end{equation*}
\begin{equation*}
\qquad \text{and} \quad x_{\mathrm{e}}(t_{3}) = \left\{ \begin{array}{ll}
8 &  7.5 \leq t_{3} < 12.5\\
-t_{3} + 20.5 & 12.5 \leq t_{3} < 16.5\\
4 & t_{3} \geq 16.5
\end{array} \right.
\end{equation*}

\begin{figure}[h!]
\centering
\psfrag{f(x)}[Bl][Bl][.8][0]{$pt^{\circ}_{1,2}(t_{3})$}
\psfrag{x}[bc][Bl][.8][0]{$t_{3}$}
\psfrag{0}[tc][Bl][.8][0]{$0$}
\includegraphics[scale=.2]{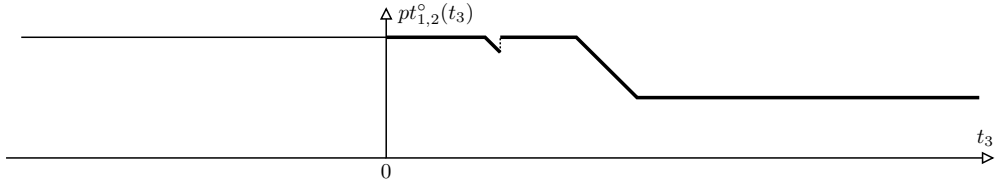}
\caption{Optimal processing time $pt^{\circ}_{1,2}(t_{3})$, under the assumption $\delta_{1} = 1$ in state $[ 1 \; 2 \; 2 \; t_{3}]^{T}$.}
\label{fig:esS3_tau_1_2_2_pt_1_2}
\end{figure}

The conditioned cost-to-go $J^{\circ}_{1,2,2} (t_{3} \mid \delta_{1}=1) = f ( pt^{\circ}_{1,2}(t_{3}) + t_{3} ) + g ( pt^{\circ}_{1,2}(t_{3}) )$, illustrated in figure~\ref{fig:esS3_J_1_2_2_min}, is provided by lemma~\ref{lem:h(t)}. It is specified by the initial value 1.5, by the set \{ 6.5, 7.5, 9, 12.5, 19.5, 20.5, 22.5 \} of abscissae $\gamma_{i}$, $i = 1, \ldots, 7$, at which the slope changes, and by the set \{ 1, 0, 0.5, 1, 1.5, 2.5, 3.5 \} of slopes $\mu_{i}$, $i = 1, \ldots, 7$, in the various intervals.

{\it Case ii)} in which it is assumed $\delta_{2} = 1$ (and $\delta_{1} = 0$).

In this case, it is necessary to minimize, with respect to the (continuos) decision variable $\tau$ which corresponds to the processing time $pt_{2,3}$, the following function
\begin{equation*}
\alpha_{2,3} \, \max \{ t_{3} + st_{2,2} + \tau - dd_{2,3} \, , \, 0 \} + \beta_{2} \, ( pt^{\mathrm{nom}}_{2} - \tau ) + sc_{2,2} + J^{\circ}_{1,3,2} (t_{4})
\end{equation*}
that can be written as $f (pt_{2,3} + t_{3}) + g (pt_{2,3})$ being
\begin{equation*}
f (pt_{2,3} + t_{3}) = \max \{ pt_{2,3} + t_{3} - 38 \, , \, 0 \} + J^{\circ}_{1,3,2} (pt_{2,3} + t_{3})
\end{equation*}
\begin{equation*}
g (pt_{2,3}) = \left\{ \begin{array}{ll}
1.5 \cdot (6 - pt_{2,3}) & pt_{2,3} \in [ 4 , 6 )\\
0 & pt_{2,3} \notin [ 4 , 6 )
\end{array} \right.
\end{equation*}
The function $pt^{\circ}_{2,3}(t_{3}) = \arg \min_{pt_{2,3}} \{ f (pt_{2,3} + t_{3}) + g (pt_{2,3}) \} $, with $4 \leq pt_{2,3} \leq 6$, is determined by applying lemma~\ref{lem:xopt}. It is (see figure~\ref{fig:esS3_tau_1_2_2_pt_2_3})
\begin{equation*}
pt^{\circ}_{2,3}(t_{3}) = x_{\mathrm{e}}(t_{3}) \qquad \text{with} \quad x_{\mathrm{e}}(t_{3}) = \left\{ \begin{array}{ll}
6 &  t_{3} < 13.5\\
-t_{3} + 19.5 & 13.5 \leq t_{3} < 15.5\\
4 & t_{3} \geq 15.5
\end{array} \right.
\end{equation*}

\begin{figure}[h!]
\centering
\psfrag{f(x)}[Bl][Bl][.8][0]{$pt^{\circ}_{2,3}(t_{3})$}
\psfrag{x}[bc][Bl][.8][0]{$t_{3}$}
\psfrag{0}[tc][Bl][.8][0]{$0$}
\includegraphics[scale=.2]{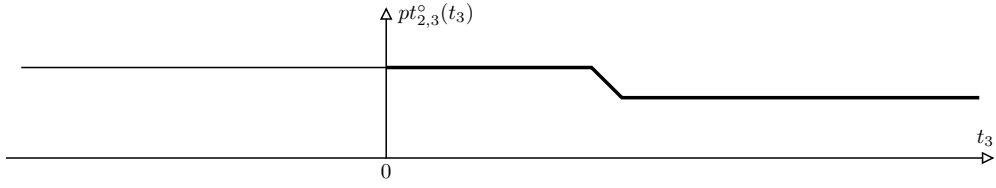}
\caption{Optimal processing time $pt^{\circ}_{2,3}(t_{3})$, under the assumption $\delta_{2} = 1$ in state $[ 1 \; 2 \; 2 \; t_{3}]^{T}$.}
\label{fig:esS3_tau_1_2_2_pt_2_3}
\end{figure}

The conditioned cost-to-go $J^{\circ}_{1,2,2} (t_{3} \mid \delta_{2}=1) = f ( pt^{\circ}_{2,3}(t_{3}) + t_{3} ) + g ( pt^{\circ}_{2,3}(t_{3}) )$, illustrated in figure~\ref{fig:esS3_J_1_2_2_min}, is provided by lemma~\ref{lem:h(t)}. It is specified by the initial value 1, by the set \{ 6.5, 13.5, 16.5, 20.5, 34 \} of abscissae $\gamma_{i}$, $i = 1, \ldots, 5$, at which the slope changes, and by the set \{ 1, 1.5, 2, 2.5, 3.5 \} of slopes $\mu_{i}$, $i = 1, \ldots, 5$, in the various intervals.

\begin{figure}[h!]
\centering
\psfrag{J1}[cl][Bc][.8][0]{$J^{\circ}_{1,2,2} (t_{3} \mid \delta_{1} = 1)$}
\psfrag{J2}[br][Bc][.8][0]{$J^{\circ}_{1,2,2} (t_{3} \mid \delta_{2} = 1)$}
\psfrag{0}[cc][tc][.7][0]{$0$}\includegraphics[scale=.6]{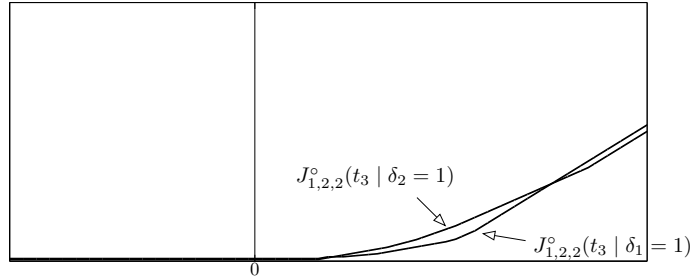}%
\vspace{-12pt}
\caption{Conditioned costs-to-go $J^{\circ}_{1,2,2} (t_{3} \mid \delta_{1} = 1)$ and $J^{\circ}_{1,2,2} (t_{3} \mid \delta_{2} = 1)$ in state $[1 \; 2 \; 2 \; t_{3}]^{T}$.}
\label{fig:esS3_J_1_2_2_min}
\end{figure}

In order to find the optimal cost-to-go $J^{\circ}_{1,2,2} (t_{3})$, it is necessary to carry out the following minimization
\begin{equation*}
J^{\circ}_{1,2,2} (t_{3}) = \min \big\{ J^{\circ}_{1,2,2} (t_{3} \mid \delta_{1} = 1) \, , \, J^{\circ}_{1,2,2} (t_{3} \mid \delta_{2} = 1) \big\}
\end{equation*}
which provides, in accordance with lemma~\ref{lem:min}, the continuous, nondecreasing, piecewise linear function illustrated in figure~\ref{fig:esS3_J_1_2_2}.

\begin{figure}[h!]
\centering
\psfrag{0}[cc][tc][.8][0]{$0$}
\includegraphics[scale=.8]{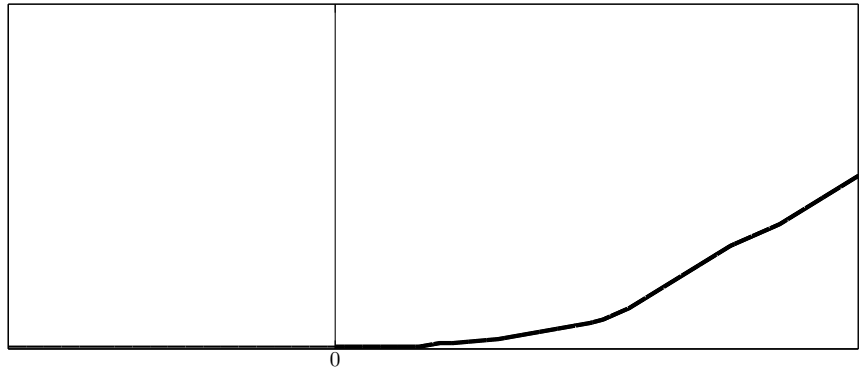}%
\vspace{-12pt}
\caption{Optimal cost-to-go $J^{\circ}_{1,2,2} (t_{3})$ in state $[1 \; 2 \; 2 \; t_{3}]^{T}$.}
\label{fig:esS3_J_1_2_2}
\end{figure}

The function $J^{\circ}_{1,2,2} (t_{3})$ is specified by the initial value 1, by the set \{ 6.5, 8, 9, 12.5, 19.5, 20.5, 22.5, 30.25, 34 \} of abscissae $\gamma_{i}$, $i = 1, \ldots, 9$, at which the slope changes, and by the set \{ 1, 0, 0.5, 1, 1.5, 2.5, 3.5, 2.5, 3.5 \} of slopes $\mu_{i}$, $i = 1, \ldots, 9$, in the various intervals.

Since $J^{\circ}_{1,2,2} (t_{3} \mid \delta_{1} = 1)$ is the minimum in $[8,30.25)$, and $J^{\circ}_{1,2,2} (t_{3} \mid \delta_{2} = 1)$ is the minimum in $(-\infty, 8)$ and in $[30.25,+\infty)$, the optimal control strategies for this state are
\begin{equation*}
\delta_{1}^{\circ} (1,2,2, t_{3}) = \left\{ \begin{array}{ll}
0 &  t_{3} < 8\\
1 &  8 \leq t_{3} < 30.25\\
0 & t_{3} \geq 30.25
\end{array} \right. \qquad \delta_{2}^{\circ} (1,2,2, t_{3}) = \left\{ \begin{array}{ll}
1 &  t_{3} < 8\\
0 &  8 \leq t_{3} < 30.25\\
1 & t_{3} \geq 30.25
\end{array} \right.
\end{equation*}
\begin{equation*}
\tau^{\circ} (1,2,2, t_{3}) = \left\{ \begin{array}{ll}
6 &  t_{3} < 8\\
8 & 8 \leq t_{34} < 12.5\\
-t_{3} + 20.5 & 12.5 \leq t_{3} < 16.5\\
4 & t_{3} \geq 16.5
\end{array} \right.
\end{equation*}
The optimal control strategy $\tau^{\circ} (1,2,2, t_{3})$ is illustrated in figure~\ref{fig:esS3_tau_1_2_2}.

\begin{figure}[h!]
\centering
\psfrag{f(x)}[Bl][Bl][.8][0]{$\tau^{\circ} (1,2,2, t_{3})$}
\psfrag{x}[bc][Bl][.8][0]{$t_{3}$}
\psfrag{0}[tc][Bl][.8][0]{$0$}
\includegraphics[scale=.2]{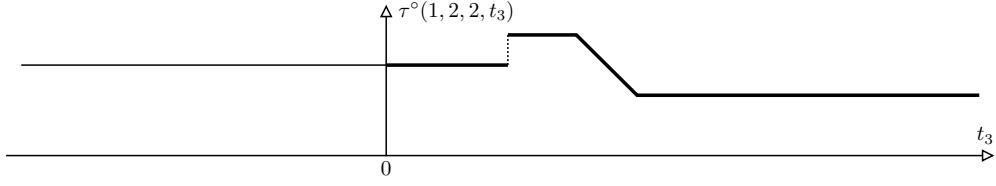}
\caption{Optimal control strategy $\tau^{\circ} (1,2,2, t_{3})$ in state $[ 1 \; 2 \; 2 \; t_{3}]^{T}$.}
\label{fig:esS3_tau_1_2_2}
\end{figure}

%%
%% S10 - [1 2 1 t3]
%%

{\bf Stage $3$ -- State $\boldsymbol{[1 \; 2 \; 1 \; t_{3}]^{T}}$ ($S10$)}

In state $[1 \; 2 \; 1 \; t_{3}]^{T}$, the cost function to be minimized, with respect to the (continuos) decision variable $\tau$ and to the (binary) decision variables $\delta_{1}$ and $\delta_{2}$ is
\begin{equation*}
\begin{split}
&\delta_{1} \big[ \alpha_{1,2} \, \max \{ t_{3} + st_{1,1} + \tau - dd_{1,2} \, , \, 0 \} + \beta_{1} \, ( pt^{\mathrm{nom}}_{1} - \tau ) + sc_{1,1} + J^{\circ}_{2,2,1} (t_{4}) \big] +\\
&+ \delta_{2} \big[ \alpha_{2,3} \, \max \{ t_{3} + st_{1,2} + \tau - dd_{2,3} \, , \, 0 \} + \beta_{2} \, ( pt^{\mathrm{nom}}_{2} - \tau ) + sc_{1,2} + J^{\circ}_{1,3,2} (t_{4}) \big]
\end{split}
\end{equation*}

{\it Case i)} in which it is assumed $\delta_{1} = 1$ (and $\delta_{2} = 0$).

In this case, it is necessary to minimize, with respect to the (continuos) decision variable $\tau$ which corresponds to the processing time $pt_{1,2}$, the following function
\begin{equation*}
\alpha_{1,2} \, \max \{ t_{3} + st_{1,1} + \tau - dd_{1,2} \, , \, 0 \} + \beta_{1} \, ( pt^{\mathrm{nom}}_{1} - \tau ) + sc_{1,1} + J^{\circ}_{2,2,1} (t_{4})
\end{equation*}
that can be written as $f (pt_{1,2} + t_{3}) + g (pt_{1,2})$ being
\begin{equation*}
f (pt_{1,2} + t_{3}) = 0.5 \cdot \max \{ pt_{1,2} + t_{3} - 24 \, , \, 0 \} + J^{\circ}_{2,2,1} (pt_{1,2} + t_{3})
\end{equation*}
\begin{equation*}
g (pt_{1,2}) = \left\{ \begin{array}{ll}
8 - pt_{1,2} & pt_{1,2} \in [ 4 , 8 )\\
0 & pt_{1,2} \notin [ 4 , 8 )
\end{array} \right.
\end{equation*}
The function $pt^{\circ}_{1,2}(t_{3}) = \arg \min_{pt_{1,2}} \{ f (pt_{1,2} + t_{3}) + g (pt_{1,2}) \} $, with $4 \leq pt_{1,2} \leq 8$, is determined by applying lemma~\ref{lem:xopt}. It is (see figure~\ref{fig:esS3_tau_1_2_1_pt_1_2})
\begin{equation*}
pt^{\circ}_{1,2}(t_{3}) = \left\{ \begin{array}{ll}
x_{\mathrm{s}}(t_{3}) &  t_{3} < 8\\
x_{\mathrm{e}}(t_{3}) & t_{3} \geq 8
\end{array} \right. \qquad \text{with} \quad x_{\mathrm{s}}(t_{3}) = \left\{ \begin{array}{ll}
8 &  t_{3} < 7\\
-t_{3} + 15 & 7 \leq t_{3} < 8
\end{array} \right. \  , 
\end{equation*}
\begin{equation*}
\qquad \text{and} \quad x_{\mathrm{e}}(t_{3}) = \left\{ \begin{array}{ll}
8 &  8 \leq t_{3} < 13\\
-t_{3} + 21 & 13 \leq t_{3} < 17\\
4 & t_{3} \geq 17
\end{array} \right.
\end{equation*}

\begin{figure}[h!]
\centering
\psfrag{f(x)}[Bl][Bl][.8][0]{$pt^{\circ}_{1,2}(t_{3})$}
\psfrag{x}[bc][Bl][.8][0]{$t_{3}$}
\psfrag{0}[tc][Bl][.8][0]{$0$}
\includegraphics[scale=.2]{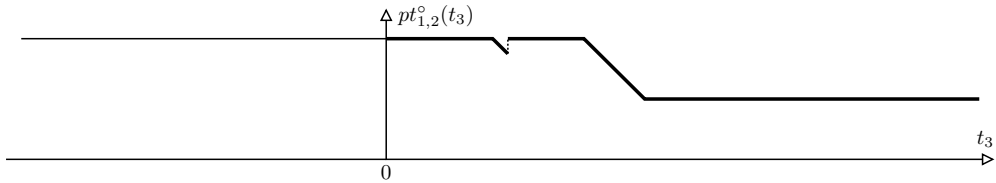}
\caption{Optimal processing time $pt^{\circ}_{1,2}(t_{3})$, under the assumption $\delta_{1} = 1$ in state $[ 1 \; 2 \; 1 \; t_{3}]^{T}$.}
\label{fig:esS3_tau_1_2_1_pt_1_2}
\end{figure}

The conditioned cost-to-go $J^{\circ}_{1,2,1} (t_{3} \mid \delta_{1}=1) = f ( pt^{\circ}_{1,2}(t_{3}) + t_{3} ) + g ( pt^{\circ}_{1,2}(t_{3}) )$, illustrated in figure~\ref{fig:esS3_J_1_2_1_min}, is provided by lemma~\ref{lem:h(t)}. It is specified by the initial value 0.5, by the set \{ 7, 8, 9.5, 13, 20, 21, 23 \} of abscissae $\gamma_{i}$, $i = 1, \ldots, 7$, at which the slope changes, and by the set \{ 1, 0, 0.5, 1, 1.5, 2.5, 3.5 \} of slopes $\mu_{i}$, $i = 1, \ldots, 7$, in the various intervals.

{\it Case ii)} in which it is assumed $\delta_{2} = 1$ (and $\delta_{1} = 0$).

In this case, it is necessary to minimize, with respect to the (continuos) decision variable $\tau$ which corresponds to the processing time $pt_{2,3}$, the following function
\begin{equation*}
\alpha_{2,3} \, \max \{ t_{3} + st_{1,2} + \tau - dd_{2,3} \, , \, 0 \} + \beta_{2} \, ( pt^{\mathrm{nom}}_{2} - \tau ) + sc_{1,2} + J^{\circ}_{1,3,2} (t_{4})
\end{equation*}
that can be written as $f (pt_{2,3} + t_{3}) + g (pt_{2,3})$ being
\begin{equation*}
f (pt_{2,3} + t_{3}) = \max \{ pt_{2,3} + t_{3} - 37 \, , \, 0 \} + 0.5 + J^{\circ}_{1,3,2} (pt_{2,3} + t_{3} + 1)
\end{equation*}
\begin{equation*}
g (pt_{2,3}) = \left\{ \begin{array}{ll}
1.5 \cdot (6 - pt_{2,3}) & pt_{2,3} \in [ 4 , 6 )\\
0 & pt_{2,3} \notin [ 4 , 6 )
\end{array} \right.
\end{equation*}
The function $pt^{\circ}_{2,3}(t_{3}) = \arg \min_{pt_{2,3}} \{ f (pt_{2,3} + t_{3}) + g (pt_{2,3}) \} $, with $4 \leq pt_{2,3} \leq 6$, is determined by applying lemma~\ref{lem:xopt}. It is (see figure~\ref{fig:esS3_tau_1_2_1_pt_2_3})
\begin{equation*}
pt^{\circ}_{2,3}(t_{3}) = x_{\mathrm{e}}(t_{3}) \qquad \text{with} \quad x_{\mathrm{e}}(t_{3}) = \left\{ \begin{array}{ll}
6 &  t_{3} < 12.5\\
-t_{3} + 18.5 & 12.5 \leq t_{3} < 14.5\\
4 & t_{3} \geq 14.5
\end{array} \right.
\end{equation*}

\begin{figure}[h!]
\centering
\psfrag{f(x)}[Bl][Bl][.8][0]{$pt^{\circ}_{2,3}(t_{3})$}
\psfrag{x}[bc][Bl][.8][0]{$t_{3}$}
\psfrag{0}[tc][Bl][.8][0]{$0$}
\includegraphics[scale=.2]{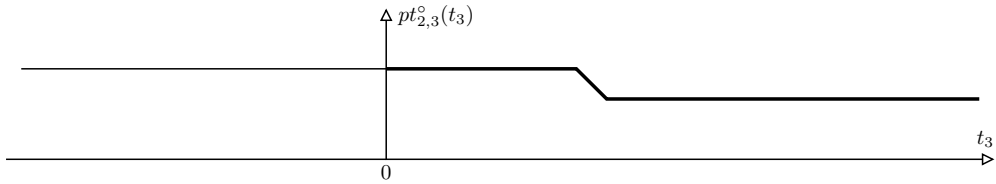}
\caption{Optimal processing time $pt^{\circ}_{2,3}(t_{3})$, under the assumption $\delta_{2} = 1$ in state $[ 1 \; 2 \; 1 \; t_{3}]^{T}$.}
\label{fig:esS3_tau_1_2_1_pt_2_3}
\end{figure}

The conditioned cost-to-go $J^{\circ}_{1,2,1} (t_{3} \mid \delta_{2}=1) = f ( pt^{\circ}_{2,3}(t_{3}) + t_{3} ) + g ( pt^{\circ}_{2,3}(t_{3}) )$, illustrated in figure~\ref{fig:esS3_J_1_2_1_min}, is provided by lemma~\ref{lem:h(t)}. It is specified by the initial value 1.5, by the set \{ 5.5, 12.5, 15.5, 19.5, 33 \} of abscissae $\gamma_{i}$, $i = 1, \ldots, 5$, at which the slope changes, and by the set \{ 1, 1.5, 2, 2.5, 3.5 \} of slopes $\mu_{i}$, $i = 1, \ldots, 5$, in the various intervals.

\begin{figure}[h!]
\centering
\psfrag{J1}[cl][Bc][.8][0]{$J^{\circ}_{1,2,1} (t_{3} \mid \delta_{1} = 1)$}
\psfrag{J2}[br][Bc][.8][0]{$J^{\circ}_{1,2,1} (t_{3} \mid \delta_{2} = 1)$}
\psfrag{0}[cc][tc][.7][0]{$0$}\includegraphics[scale=.6]{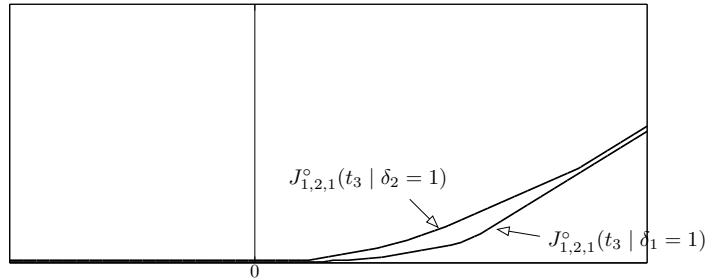}%
\vspace{-12pt}
\caption{Conditioned costs-to-go $J^{\circ}_{1,2,1} (t_{3} \mid \delta_{1} = 1)$ and $J^{\circ}_{1,2,1} (t_{3} \mid \delta_{2} = 1)$ in state $[1 \; 2 \; 1 \; t_{3}]^{T}$.}
\label{fig:esS3_J_1_2_1_min}
\end{figure}

In order to find the optimal cost-to-go $J^{\circ}_{1,2,1} (t_{3})$, it is necessary to carry out the following minimization
\begin{equation*}
J^{\circ}_{1,2,1} (t_{3}) = \min \big\{ J^{\circ}_{1,2,1} (t_{3} \mid \delta_{1} = 1) \, , \, J^{\circ}_{1,2,1} (t_{3} \mid \delta_{2} = 1) \big\}
\end{equation*}
which provides, in accordance with lemma~\ref{lem:min}, the continuous, nondecreasing, piecewise linear function illustrated in figure~\ref{fig:esS3_J_1_2_1}.

\begin{figure}[h!]
\centering
\psfrag{0}[cc][tc][.8][0]{$0$}
\includegraphics[scale=.8]{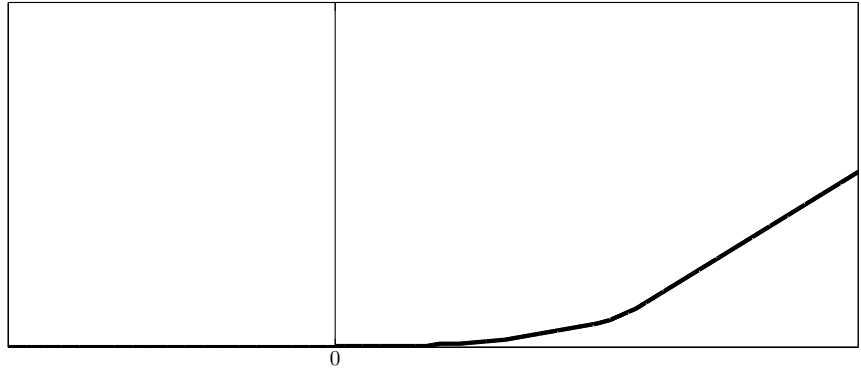}%
\vspace{-12pt}
\caption{Optimal cost-to-go $J^{\circ}_{1,2,1} (t_{3})$ in state $[1 \; 2 \; 1 \; t_{3}]^{T}$.}
\label{fig:esS3_J_1_2_1}
\end{figure}

The function $J^{\circ}_{1,2,1} (t_{3})$ is specified by the initial value 0.5, by the set \{ 7, 8, 9.5, 13, 20, 21, 23 \} of abscissae $\gamma_{i}$, $i = 1, \ldots, 7$, at which the slope changes, and by the set \{ 1, 0, 0.5, 1, 1.5, 2.5, 3.5 \} of slopes $\mu_{i}$, $i = 1, \ldots, 7$, in the various intervals.

Since $J^{\circ}_{1,2,1} (t_{3} \mid \delta_{1} = 1)$ is always the minimum (see again figure~\ref{fig:esS3_J_1_2_1_min}), the optimal control strategies for this state are
\begin{equation*}
\delta_{1}^{\circ} (1,2,1, t_{3}) = 1 \quad \forall \, t_{3} \qquad \delta_{2}^{\circ} (1,2,1, t_{3}) = 0 \quad \forall \, t_{3}
\end{equation*}
\begin{equation*}
\tau^{\circ} (1,2,1, t_{3}) = \left\{ \begin{array}{ll}
8 &  t_{3} < 7\\
-t_{3} + 15 & 7 \leq t_{3} < 8\\
8 &  8 \leq t_{3} < 13\\
-t_{3} + 21 & 13 \leq t_{3} < 17\\
4 & t_{3} \geq 17
\end{array} \right.
\end{equation*}
The optimal control strategy $\tau^{\circ} (1,2,1, t_{3})$ is illustrated in figure~\ref{fig:esS3_tau_1_2_1}.

\begin{figure}[h!]
\centering
\psfrag{f(x)}[Bl][Bl][.8][0]{$\tau^{\circ} (1,2,1, t_{3})$}
\psfrag{x}[bc][Bl][.8][0]{$t_{3}$}
\psfrag{0}[tc][Bl][.8][0]{$0$}
\includegraphics[scale=.2]{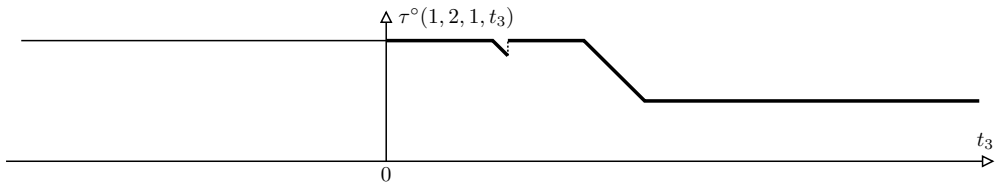}
\caption{Optimal control strategy $\tau^{\circ} (1,2,1, t_{3})$ in state $[ 1 \; 2 \; 1 \; t_{3}]^{T}$.}
\label{fig:esS3_tau_1_2_1}
\end{figure}

%%
%% S9 - [2 1 2 t3]
%%

{\bf Stage $3$ -- State $\boldsymbol{[2 \; 1 \; 2 \; t_{3}]^{T}}$ ($S9$)}

In state $[2 \; 1 \; 2 \; t_{3}]^{T}$, the cost function to be minimized, with respect to the (continuos) decision variable $\tau$ and to the (binary) decision variables $\delta_{1}$ and $\delta_{2}$ is
\begin{equation*}
\begin{split}
&\delta_{1} \big[ \alpha_{1,3} \, \max \{ t_{3} + st_{2,1} + \tau - dd_{1,3} \, , \, 0 \} + \beta_{1} \, ( pt^{\mathrm{nom}}_{1} - \tau ) + sc_{2,1} + J^{\circ}_{3,1,1} (t_{4}) \big] +\\
&+ \delta_{2} \big[ \alpha_{2,2} \, \max \{ t_{3} + st_{2,2} + \tau - dd_{2,2} \, , \, 0 \} + \beta_{2} \, ( pt^{\mathrm{nom}}_{2} - \tau ) + sc_{2,2} + J^{\circ}_{2,2,2} (t_{4}) \big]
\end{split}
\end{equation*}

{\it Case i)} in which it is assumed $\delta_{1} = 1$ (and $\delta_{2} = 0$).

In this case, it is necessary to minimize, with respect to the (continuos) decision variable $\tau$ which corresponds to the processing time $pt_{1,3}$, the following function
\begin{equation*}
\alpha_{1,3} \, \max \{ t_{3} + st_{2,1} + \tau - dd_{1,3} \, , \, 0 \} + \beta_{1} \, ( pt^{\mathrm{nom}}_{1} - \tau ) + sc_{2,1} + J^{\circ}_{3,1,1} (t_{4})
\end{equation*}
that can be written as $f (pt_{1,3} + t_{3}) + g (pt_{1,3})$ being
\begin{equation*}
f (pt_{1,3} + t_{3}) = 1.5 \cdot \max \{ pt_{1,3} + t_{3} - 28.5 \, , \, 0 \} + 1 + J^{\circ}_{3,1,1} (pt_{1,3} + t_{3} + 0.5)
\end{equation*}
\begin{equation*}
g (pt_{1,3}) = \left\{ \begin{array}{ll}
8 - pt_{1,3} & pt_{1,3} \in [ 4 , 8 )\\
0 & pt_{1,3} \notin [ 4 , 8 )
\end{array} \right.
\end{equation*}
The function $pt^{\circ}_{1,3}(t_{3}) = \arg \min_{pt_{1,3}} \{ f (pt_{1,3} + t_{3}) + g (pt_{1,3}) \} $, with $4 \leq pt_{1,3} \leq 8$, is determined by applying lemma~\ref{lem:xopt}. It is (see figure~\ref{fig:esS3_tau_2_1_2_pt_1_3})
\begin{equation*}
pt^{\circ}_{1,3}(t_{3}) = \left\{ \begin{array}{ll}
x_{\mathrm{s}}(t_{3}) &  t_{3} < 1.5\\
x_{\mathrm{e}}(t_{3}) & t_{3} \geq 1.5
\end{array} \right. \qquad \text{with} \quad x_{\mathrm{s}}(t_{3}) = \left\{ \begin{array}{ll}
8 &  t_{3} < 0.5\\
-t_{3} + 8.5 & 0.5 \leq t_{3} < 1.5
\end{array} \right. \  , 
\end{equation*}
\begin{equation*}
\qquad \text{and} \quad x_{\mathrm{e}}(t_{3}) = \left\{ \begin{array}{ll}
8 &  1.5 \leq t_{3} < 8.5\\
-t_{3} + 16.5 & 8.5 \leq t_{3} < 12.5\\
4 & t_{3} \geq 12.5
\end{array} \right.
\end{equation*}

\begin{figure}[h!]
\centering
\psfrag{f(x)}[Bl][Bl][.8][0]{$pt^{\circ}_{1,3}(t_{3})$}
\psfrag{x}[bc][Bl][.8][0]{$t_{3}$}
\psfrag{0}[tc][Bl][.8][0]{$0$}
\includegraphics[scale=.2]{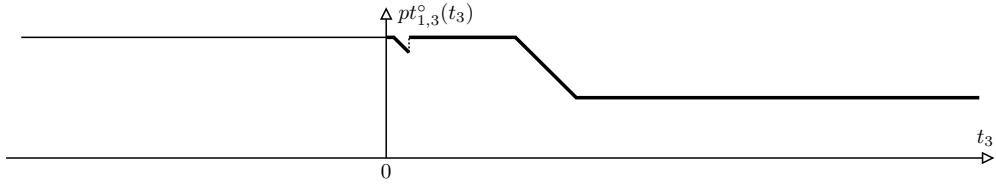}
\caption{Optimal processing time $pt^{\circ}_{1,3}(t_{3})$, under the assumption $\delta_{1} = 1$ in state $[ 2 \; 1 \; 2 \; t_{3}]^{T}$.}
\label{fig:esS3_tau_2_1_2_pt_1_3}
\end{figure}

The conditioned cost-to-go $J^{\circ}_{2,1,2} (t_{3} \mid \delta_{1}=1) = f ( pt^{\circ}_{1,3}(t_{3}) + t_{3} ) + g ( pt^{\circ}_{1,3}(t_{3}) )$, illustrated in figure~\ref{fig:esS3_J_2_1_2_min}, is provided by lemma~\ref{lem:h(t)}. It is specified by the initial value 1.5, by the set \{ 0.5, 1.5, 8.5, 15, 22.5, 24.5 \} of abscissae $\gamma_{i}$, $i = 1, \ldots, 6$, at which the slope changes, and by the set \{ 1, 0, 1, 1.5, 2.5, 4 \} of slopes $\mu_{i}$, $i = 1, \ldots, 6$, in the various intervals.

{\it Case ii)} in which it is assumed $\delta_{2} = 1$ (and $\delta_{1} = 0$).

In this case, it is necessary to minimize, with respect to the (continuos) decision variable $\tau$ which corresponds to the processing time $pt_{2,2}$, the following function
\begin{equation*}
\alpha_{2,2} \, \max \{ t_{3} + st_{2,2} + \tau - dd_{2,2} \, , \, 0 \} + \beta_{2} \, ( pt^{\mathrm{nom}}_{2} - \tau ) + sc_{2,2} + J^{\circ}_{2,2,2} (t_{4})
\end{equation*}
that can be written as $f (pt_{2,2} + t_{3}) + g (pt_{2,2})$ being
\begin{equation*}
f (pt_{2,2} + t_{3}) = \max \{ pt_{2,2} + t_{3} - 24 \, , \, 0 \} + J^{\circ}_{2,2,2} (pt_{2,2} + t_{3})
\end{equation*}
\begin{equation*}
g (pt_{2,2}) = \left\{ \begin{array}{ll}
1.5 \cdot (6 - pt_{2,2}) & pt_{2,2} \in [ 4 , 6 )\\
0 & pt_{2,2} \notin [ 4 , 6 )
\end{array} \right.
\end{equation*}
The function $pt^{\circ}_{2,2}(t_{3}) = \arg \min_{pt_{2,2}} \{ f (pt_{2,2} + t_{3}) + g (pt_{2,2}) \} $, with $4 \leq pt_{2,2} \leq 6$, is determined by applying lemma~\ref{lem:xopt}. It is (see figure~\ref{fig:esS3_tau_2_1_2_pt_2_2})
\begin{equation*}
pt^{\circ}_{2,2}(t_{3}) = x_{\mathrm{e}}(t_{3}) \qquad \text{with} \quad x_{\mathrm{e}}(t_{3}) = \left\{ \begin{array}{ll}
6 &  t_{3} < 18\\
-t_{3} + 24 & 18 \leq t_{3} < 20\\
4 & t_{3} \geq 20
\end{array} \right.
\end{equation*}

\begin{figure}[h!]
\centering
\psfrag{f(x)}[Bl][Bl][.8][0]{$pt^{\circ}_{2,2}(t_{3})$}
\psfrag{x}[bc][Bl][.8][0]{$t_{3}$}
\psfrag{0}[tc][Bl][.8][0]{$0$}
\includegraphics[scale=.2]{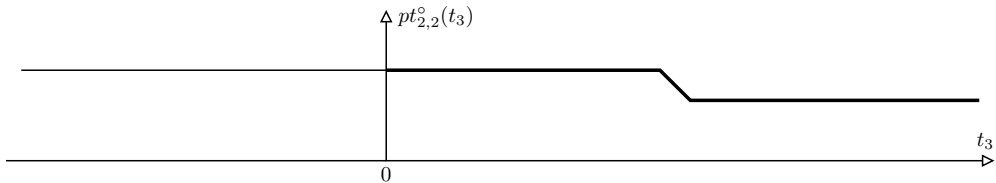}
\caption{Optimal processing time $pt^{\circ}_{2,2}(t_{3})$, under the assumption $\delta_{2} = 1$ in state $[ 2 \; 1 \; 2 \; t_{3}]^{T}$.}
\label{fig:esS3_tau_2_1_2_pt_2_2}
\end{figure}

The conditioned cost-to-go $J^{\circ}_{2,1,2} (t_{3} \mid \delta_{2}=1) = f ( pt^{\circ}_{2,2}(t_{3}) + t_{3} ) + g ( pt^{\circ}_{2,2}(t_{3}) )$, illustrated in figure~\ref{fig:esS3_J_2_1_2_min}, is provided by lemma~\ref{lem:h(t)}. It is specified by the initial value 1, by the set \{ 8.5, 10, 11, 14.5, 18, 20, 20.5, 22.5, 28.25, 30 \} of abscissae $\gamma_{i}$, $i = 1, \ldots, 10$, at which the slope changes, and by the set \{ 1, 0, 0.5, 1, 1.5, 2, 3, 4, 3, 4 \} of slopes $\mu_{i}$, $i = 1, \ldots, 10$, in the various intervals.

\begin{figure}[h!]
\centering
\psfrag{J1}[bl][Bc][.8][0]{$J^{\circ}_{2,1,2} (t_{3} \mid \delta_{2} = 1)$}
\psfrag{J2}[cr][Bc][.8][0]{$J^{\circ}_{2,1,2} (t_{3} \mid \delta_{1} = 1)$}
\psfrag{0}[cc][tc][.7][0]{$0$}\includegraphics[scale=.6]{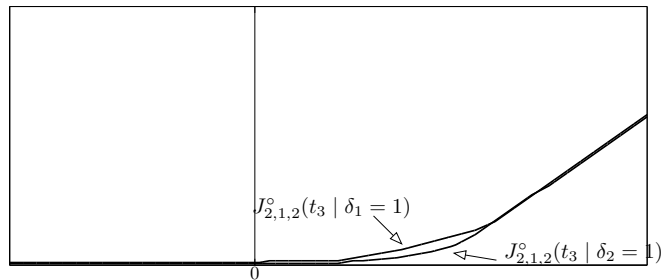}%
\vspace{-12pt}
\caption{Conditioned costs-to-go $J^{\circ}_{2,1,2} (t_{3} \mid \delta_{1} = 1)$ and $J^{\circ}_{2,1,2} (t_{3} \mid \delta_{2} = 1)$ in state $[2 \; 1 \; 2 \; t_{3}]^{T}$.}
\label{fig:esS3_J_2_1_2_min}
\end{figure}

In order to find the optimal cost-to-go $J^{\circ}_{2,1,2} (t_{3})$, it is necessary to carry out the following minimization
\begin{equation*}
J^{\circ}_{2,1,2} (t_{3}) = \min \big\{ J^{\circ}_{2,1,2} (t_{3} \mid \delta_{1} = 1) \, , \, J^{\circ}_{2,1,2} (t_{3} \mid \delta_{2} = 1) \big\}
\end{equation*}
which provides, in accordance with lemma~\ref{lem:min}, the continuous, nondecreasing, piecewise linear function illustrated in figure~\ref{fig:esS3_J_2_1_2}.

\begin{figure}[h!]
\centering
\psfrag{0}[cc][tc][.8][0]{$0$}
\includegraphics[scale=.8]{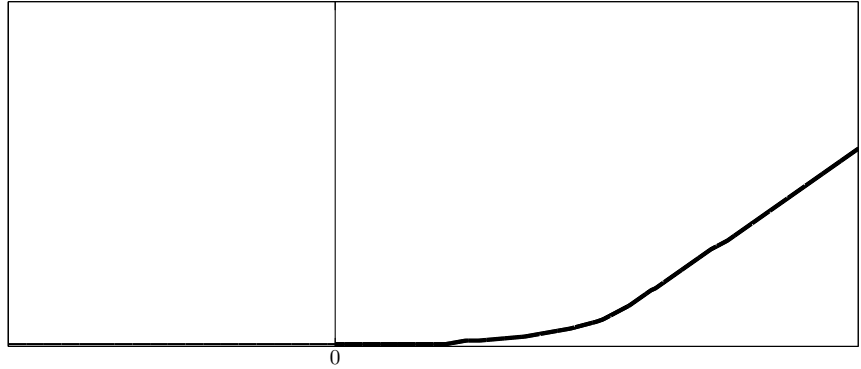}%
\vspace{-12pt}
\caption{Optimal cost-to-go $J^{\circ}_{2,1,2} (t_{3})$ in state $[2 \; 1 \; 2 \; t_{3}]^{T}$.}
\label{fig:esS3_J_2_1_2}
\end{figure}

The function $J^{\circ}_{2,1,2} (t_{3})$ is specified by the initial value 1, by the set \{ 8.5, 10, 11, 14.5, 18, 20, 20.5, 22.5, 24,1$\overline{6}$, 24.5, 28.75, 30 \} of abscissae $\gamma_{i}$, $i = 1, \ldots, 12$, at which the slope changes, and by the set \{ 1, 0, 0.5, 1, 1.5, 2, 3, 4, 2.5, 4, 3, 4 \} of slopes $\mu_{i}$, $i = 1, \ldots, 10$, in the various intervals.

Since $J^{\circ}_{2,1,2} (t_{3} \mid \delta_{1} = 1)$ is the minimum in $[24.1\overline{6},28.75)$, and $J^{\circ}_{2,1,2} (t_{3} \mid \delta_{2} = 1)$ is the minimum in $(-\infty, 24.1\overline{6})$ and in $[28.75,+\infty)$, the optimal control strategies for this state are
\begin{equation*}
\delta_{1}^{\circ} (2,1,2, t_{3}) = \left\{ \begin{array}{ll}
0 &  t_{3} < 24.1\overline{6}\\
1 &  24.1\overline{6} \leq t_{3} < 28.75\\
0 & t_{3} \geq 28.75
\end{array} \right. \qquad \delta_{2}^{\circ} (2,1,2, t_{3}) = \left\{ \begin{array}{ll}
1 &  t_{3} < 24.1\overline{6}\\
0 &  24.1\overline{6} \leq t_{3} < 28.75\\
1 & t_{3} \geq 28.75
\end{array} \right.
\end{equation*}
\begin{equation*}
\tau^{\circ} (2,1,2, t_{3}) = \left\{ \begin{array}{ll}
6 &  t_{3} < 18\\
-t_{3} + 24 & 18 \leq t_{3} < 20\\
4 & t_{3} \geq 20
\end{array} \right.
\end{equation*}
The optimal control strategy $\tau^{\circ} (2,1,2, t_{3})$ is illustrated in figure~\ref{fig:esS3_tau_2_1_2}.

\begin{figure}[h!]
\centering
\psfrag{f(x)}[Bl][Bl][.8][0]{$\tau^{\circ} (2,1,2, t_{3})$}
\psfrag{x}[bc][Bl][.8][0]{$t_{3}$}
\psfrag{0}[tc][Bl][.8][0]{$0$}
\includegraphics[scale=.2]{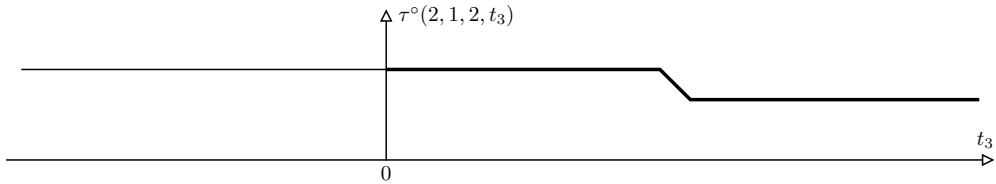}
\caption{Optimal control strategy $\tau^{\circ} (2,1,2, t_{3})$ in state $[ 2 \; 1 \; 2 \; t_{3}]^{T}$.}
\label{fig:esS3_tau_2_1_2}
\end{figure}

%%
%% S8 - [2 1 1 t3]
%%

{\bf Stage $3$ -- State $\boldsymbol{[2 \; 1 \; 1 \; t_{3}]^{T}}$ ($S8$)}

In state $[2 \; 1 \; 1 \; t_{3}]^{T}$, the cost function to be minimized, with respect to the (continuos) decision variable $\tau$ and to the (binary) decision variables $\delta_{1}$ and $\delta_{2}$ is
\begin{equation*}
\begin{split}
&\delta_{1} \big[ \alpha_{1,3} \, \max \{ t_{3} + st_{1,1} + \tau - dd_{1,3} \, , \, 0 \} + \beta_{1} \, ( pt^{\mathrm{nom}}_{1} - \tau ) + sc_{1,1} + J^{\circ}_{3,1,1} (t_{4}) \big] +\\
&+ \delta_{2} \big[ \alpha_{2,2} \, \max \{ t_{3} + st_{1,2} + \tau - dd_{2,2} \, , \, 0 \} + \beta_{2} \, ( pt^{\mathrm{nom}}_{2} - \tau ) + sc_{1,2} + J^{\circ}_{2,2,2} (t_{4}) \big]
\end{split}
\end{equation*}

{\it Case i)} in which it is assumed $\delta_{1} = 1$ (and $\delta_{2} = 0$).

In this case, it is necessary to minimize, with respect to the (continuos) decision variable $\tau$ which corresponds to the processing time $pt_{1,3}$, the following function
\begin{equation*}
\alpha_{1,3} \, \max \{ t_{3} + st_{1,1} + \tau - dd_{1,3} \, , \, 0 \} + \beta_{1} \, ( pt^{\mathrm{nom}}_{1} - \tau ) + sc_{1,1} + J^{\circ}_{3,1,1} (t_{4})
\end{equation*}
that can be written as $f (pt_{1,3} + t_{3}) + g (pt_{1,3})$ being
\begin{equation*}
f (pt_{1,3} + t_{3}) = 1.5 \cdot \max \{ pt_{1,3} + t_{3} - 29 \, , \, 0 \} + J^{\circ}_{3,1,1} (pt_{1,3} + t_{3})
\end{equation*}
\begin{equation*}
g (pt_{1,3}) = \left\{ \begin{array}{ll}
8 - pt_{1,3} & pt_{1,3} \in [ 4 , 8 )\\
0 & pt_{1,3} \notin [ 4 , 8 )
\end{array} \right.
\end{equation*}
The function $pt^{\circ}_{1,3}(t_{3}) = \arg \min_{pt_{1,3}} \{ f (pt_{1,3} + t_{3}) + g (pt_{1,3}) \} $, with $4 \leq pt_{1,3} \leq 8$, is determined by applying lemma~\ref{lem:xopt}. It is (see figure~\ref{fig:esS3_tau_2_1_1_pt_1_3})
\begin{equation*}
pt^{\circ}_{1,3}(t_{3}) = \left\{ \begin{array}{ll}
x_{\mathrm{s}}(t_{3}) &  t_{3} < 2\\
x_{\mathrm{e}}(t_{3}) & t_{3} \geq 2
\end{array} \right. \qquad \text{with} \quad x_{\mathrm{s}}(t_{3}) = \left\{ \begin{array}{ll}
8 &  t_{3} < 1\\
-t_{3} + 9 & 1 \leq t_{3} < 2
\end{array} \right. \  , 
\end{equation*}
\begin{equation*}
\qquad \text{and} \quad x_{\mathrm{e}}(t_{3}) = \left\{ \begin{array}{ll}
8 &  1.5 \leq t_{3} < 9\\
-t_{3} + 17 & 9 \leq t_{3} < 13\\
4 & t_{3} \geq 13
\end{array} \right.
\end{equation*}

\begin{figure}[h!]
\centering
\psfrag{f(x)}[Bl][Bl][.8][0]{$pt^{\circ}_{1,3}(t_{3})$}
\psfrag{x}[bc][Bl][.8][0]{$t_{3}$}
\psfrag{0}[tc][Bl][.8][0]{$0$}
\includegraphics[scale=.2]{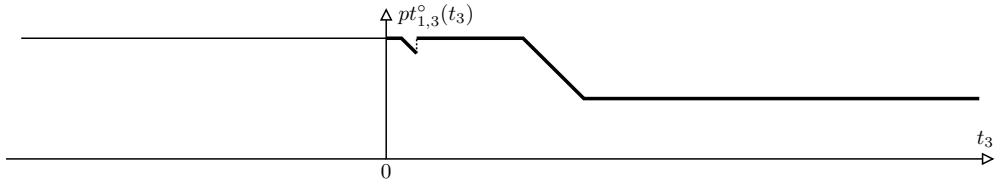}
\caption{Optimal processing time $pt^{\circ}_{1,3}(t_{3})$, under the assumption $\delta_{1} = 1$ in state $[ 2 \; 1 \; 1 \; t_{3}]^{T}$.}
\label{fig:esS3_tau_2_1_1_pt_1_3}
\end{figure}

The conditioned cost-to-go $J^{\circ}_{2,1,1} (t_{3} \mid \delta_{1}=1) = f ( pt^{\circ}_{1,3}(t_{3}) + t_{3} ) + g ( pt^{\circ}_{1,3}(t_{3}) )$, illustrated in figure~\ref{fig:esS3_J_2_1_1_min}, is provided by lemma~\ref{lem:h(t)}. It is specified by the initial value 0.5, by the set \{ 1, 2, 9, 15.5, 23, 25 \} of abscissae $\gamma_{i}$, $i = 1, \ldots, 6$, at which the slope changes, and by the set \{ 1, 0, 1, 1.5, 2.5, 4 \} of slopes $\mu_{i}$, $i = 1, \ldots, 6$, in the various intervals.

{\it Case ii)} in which it is assumed $\delta_{2} = 1$ (and $\delta_{1} = 0$).

In this case, it is necessary to minimize, with respect to the (continuos) decision variable $\tau$ which corresponds to the processing time $pt_{2,2}$, the following function
\begin{equation*}
\alpha_{2,2} \, \max \{ t_{3} + st_{1,2} + \tau - dd_{2,2} \, , \, 0 \} + \beta_{2} \, ( pt^{\mathrm{nom}}_{2} - \tau ) + sc_{1,2} + J^{\circ}_{2,2,2} (t_{4})
\end{equation*}
that can be written as $f (pt_{2,2} + t_{3}) + g (pt_{2,2})$ being
\begin{equation*}
f (pt_{2,2} + t_{3}) = \max \{ pt_{2,2} + t_{3} - 23 \, , \, 0 \} + 0.5 + J^{\circ}_{2,2,2} (pt_{2,2} + t_{3} + 1)
\end{equation*}
\begin{equation*}
g (pt_{2,2}) = \left\{ \begin{array}{ll}
1.5 \cdot (6 - pt_{2,2}) & pt_{2,2} \in [ 4 , 6 )\\
0 & pt_{2,2} \notin [ 4 , 6 )
\end{array} \right.
\end{equation*}
The function $pt^{\circ}_{2,2}(t_{3}) = \arg \min_{pt_{2,2}} \{ f (pt_{2,2} + t_{3}) + g (pt_{2,2}) \} $, with $4 \leq pt_{2,2} \leq 6$, is determined by applying lemma~\ref{lem:xopt}. It is (see figure~\ref{fig:esS3_tau_2_1_1_pt_2_2})
\begin{equation*}
pt^{\circ}_{2,2}(t_{3}) = x_{\mathrm{e}}(t_{3}) \qquad \text{with} \quad x_{\mathrm{e}}(t_{3}) = \left\{ \begin{array}{ll}
6 &  t_{3} < 17\\
-t_{3} + 23 & 17 \leq t_{3} < 19\\
4 & t_{3} \geq 19
\end{array} \right.
\end{equation*}

\begin{figure}[h!]
\centering
\psfrag{f(x)}[Bl][Bl][.8][0]{$pt^{\circ}_{2,2}(t_{3})$}
\psfrag{x}[bc][Bl][.8][0]{$t_{3}$}
\psfrag{0}[tc][Bl][.8][0]{$0$}
\includegraphics[scale=.2]{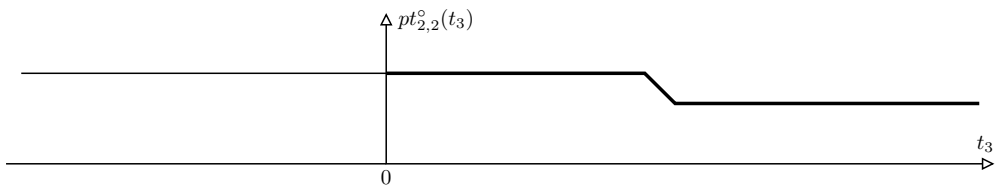}
\caption{Optimal processing time $pt^{\circ}_{2,2}(t_{3})$, under the assumption $\delta_{2} = 1$ in state $[ 2 \; 1 \; 1 \; t_{3}]^{T}$.}
\label{fig:esS3_tau_2_1_1_pt_2_2}
\end{figure}

The conditioned cost-to-go $J^{\circ}_{2,1,1} (t_{3} \mid \delta_{2}=1) = f ( pt^{\circ}_{2,2}(t_{3}) + t_{3} ) + g ( pt^{\circ}_{2,2}(t_{3}) )$, illustrated in figure~\ref{fig:esS3_J_2_1_1_min}, is provided by lemma~\ref{lem:h(t)}. It is specified by the initial value 1.5, by the set \{ 7.5, 9, 10, 13.5, 17, 19, 19.5, 21.5, 27.25, 29 \} of abscissae $\gamma_{i}$, $i = 1, \ldots, 10$, at which the slope changes, and by the set \{ 1, 0, 0.5, 1, 1.5, 2, 3, 4, 3, 4 \} of slopes $\mu_{i}$, $i = 1, \ldots, 10$, in the various intervals.

\begin{figure}[h!]
\centering
\psfrag{J1}[tl][Bc][.8][0]{$J^{\circ}_{2,1,1} (t_{3} \mid \delta_{1} = 1)$}
\psfrag{J2}[br][Bc][.8][0]{$J^{\circ}_{2,1,1} (t_{3} \mid \delta_{2} = 1)$}
\psfrag{0}[cc][tc][.7][0]{$0$}\includegraphics[scale=.6]{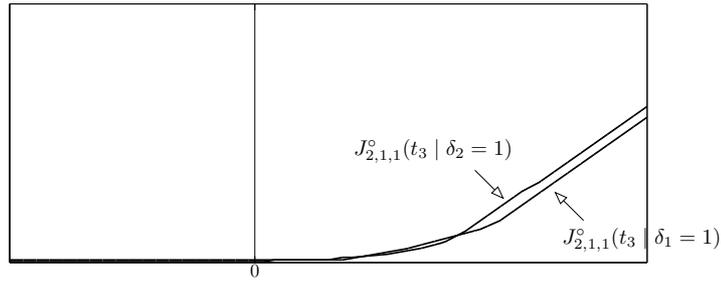}%
\vspace{-12pt}
\caption{Conditioned costs-to-go $J^{\circ}_{2,1,1} (t_{3} \mid \delta_{1} = 1)$ and $J^{\circ}_{2,1,1} (t_{3} \mid \delta_{2} = 1)$ in state $[2 \; 1 \; 1 \; t_{3}]^{T}$.}
\label{fig:esS3_J_2_1_1_min}
\end{figure}

In order to find the optimal cost-to-go $J^{\circ}_{2,1,1} (t_{3})$, it is necessary to carry out the following minimization
\begin{equation*}
J^{\circ}_{2,1,1} (t_{3}) = \min \big\{ J^{\circ}_{2,1,1} (t_{3} \mid \delta_{1} = 1) \, , \, J^{\circ}_{2,1,1} (t_{3} \mid \delta_{2} = 1) \big\}
\end{equation*}
which provides, in accordance with lemma~\ref{lem:min}, the continuous, nondecreasing, piecewise linear function illustrated in figure~\ref{fig:esS3_J_2_1_1}.

\begin{figure}[h!]
\centering
\psfrag{0}[cc][tc][.8][0]{$0$}
\includegraphics[scale=.8]{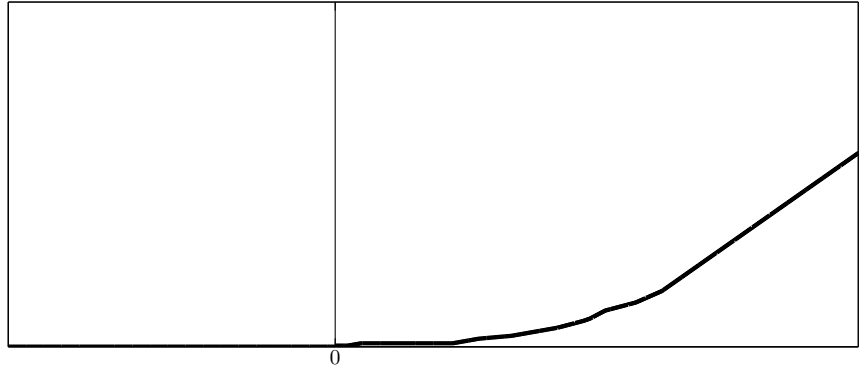}%
\vspace{-12pt}
\caption{Optimal cost-to-go $J^{\circ}_{2,1,1} (t_{3})$ in state $[2 \; 1 \; 1 \; t_{3}]^{T}$.}
\label{fig:esS3_J_2_1_1}
\end{figure}

The function $J^{\circ}_{2,1,1} (t_{3})$ is specified by the initial value 0.5, by the set \{ 1, 2, 9, 11, 13.5, 17, 19, 19.5, 20.$\overline{6}$, 23, 25 \} of abscissae $\gamma_{i}$, $i = 1, \ldots, 11$, at which the slope changes, and by the set \{ 1, 0, 1, 0.5, 1, 1.5, 2, 3, 1.5, 2.5, 4 \} of slopes $\mu_{i}$, $i = 1, \ldots, 11$, in the various intervals.

Since $J^{\circ}_{2,1,1} (t_{3} \mid \delta_{1} = 1)$ is the minimum in $(-\infty, 2)$, in $[9,11]$, and in $[20.\overline{6},+\infty)$, and $J^{\circ}_{2,1,1} (t_{3} \mid \delta_{2} = 1)$ is the minimum in $[2,9]$ and in $[11,20.\overline{6}]$, the optimal control strategies for this state are
\begin{equation*}
\delta_{1}^{\circ} (2,1,1, t_{3}) = \left\{ \begin{array}{ll}
1 &  t_{3} < 2\\
0 &  2 \leq t_{3} < 9\\
1 &  9 \leq t_{3} < 11\\
0 &  11 \leq t_{3} < 20.\overline{6}\\
1 & t_{3} \geq 20.\overline{6}
\end{array} \right. \qquad \delta_{2}^{\circ} (2,1,1, t_{3}) = \left\{ \begin{array}{ll}
0 &  t_{3} < 2\\
1 &  2 \leq t_{3} < 9\\
0 &  9 \leq t_{3} < 11\\
1 &  11 \leq t_{3} < 20.\overline{6}\\
0 & t_{3} \geq 20.\overline{6}
\end{array} \right.
\end{equation*}
\begin{equation*}
\tau^{\circ} (2,1,1, t_{3}) = \left\{ \begin{array}{ll}
8 &  t_{3} < 1\\
-t_{3} + 9 & 1 \leq t_{3} < 2\\
6 & 2 \leq t_{3} < 9\\
-t_{3} + 17 & 9 \leq t_{3} < 11\\
6 &  11 \leq t_{3} < 17\\
-t_{3} + 23 & 17 \leq t_{3} < 19\\
4 & t_{3} \geq 19
\end{array} \right.
\end{equation*}
The optimal control strategy $\tau^{\circ} (2,1,1, t_{3})$ is illustrated in figure~\ref{fig:esS3_tau_2_1_1}.

\newpage

\begin{figure}[h!]
\centering
\psfrag{f(x)}[Bl][Bl][.8][0]{$\tau^{\circ} (2,1,1, t_{3})$}
\psfrag{x}[bc][Bl][.8][0]{$t_{3}$}
\psfrag{0}[tc][Bl][.8][0]{$0$}
\includegraphics[scale=.2]{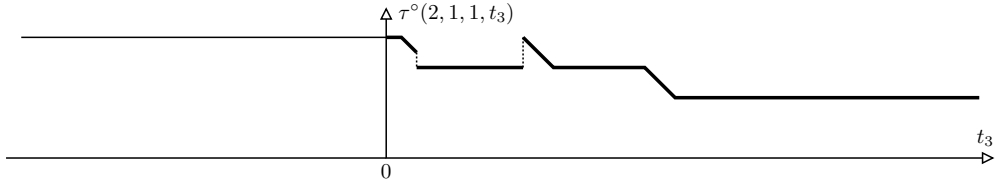}
\caption{Optimal control strategy $\tau^{\circ} (2,1,1, t_{3})$ in state $[ 2 \; 1 \; 1 \; t_{3}]^{T}$.}
\label{fig:esS3_tau_2_1_1}
\end{figure}

%%
%% S7 - [3 0 1 t3]
%%

{\bf Stage $3$ -- State $\boldsymbol{[3 \; 0 \; 1 \; t_{3}]^{T}}$ ($S7$)}

In state $[3 \; 0 \; 1 \; t_{3}]^{T}$, the cost function to be minimized, with respect to the (continuos) decision variable $\tau$ and to the (binary) decision variables $\delta_{1}$ and $\delta_{2}$ is
\begin{equation*}
\begin{split}
&\delta_{1} \big[ \alpha_{1,4} \, \max \{ t_{3} + st_{1,1} + \tau - dd_{1,4} \, , \, 0 \} + \beta_{1} \, ( pt^{\mathrm{nom}}_{1} - \tau ) + sc_{1,1} + J^{\circ}_{4,0,1} (t_{4}) \big] +\\
&+ \delta_{2} \big[ \alpha_{2,1} \, \max \{ t_{3} + st_{1,2} + \tau - dd_{2,1} \, , \, 0 \} + \beta_{2} \, ( pt^{\mathrm{nom}}_{2} - \tau ) + sc_{1,2} + J^{\circ}_{3,1,2} (t_{4}) \big]
\end{split}
\end{equation*}

{\it Case i)} in which it is assumed $\delta_{1} = 1$ (and $\delta_{2} = 0$).

In this case, it is necessary to minimize, with respect to the (continuos) decision variable $\tau$ which corresponds to the processing time $pt_{1,4}$, the following function
\begin{equation*}
\alpha_{1,4} \, \max \{ t_{3} + st_{1,1} + \tau - dd_{1,4} \, , \, 0 \} + \beta_{1} \, ( pt^{\mathrm{nom}}_{1} - \tau ) + sc_{1,1} + J^{\circ}_{4,0,1} (t_{4})
\end{equation*}
that can be written as $f (pt_{1,4} + t_{3}) + g (pt_{1,4})$ being
\begin{equation*}
f (pt_{1,4} + t_{3}) = 0.5 \cdot \max \{ pt_{1,4} + t_{3} - 41 \, , \, 0 \} + J^{\circ}_{4,0,1} (pt_{1,4} + t_{3})
\end{equation*}
\begin{equation*}
g (pt_{1,4}) = \left\{ \begin{array}{ll}
8 - pt_{1,4} & pt_{1,4} \in [ 4 , 8 )\\
0 & pt_{1,4} \notin [ 4 , 8 )
\end{array} \right.
\end{equation*}
The function $pt^{\circ}_{1,4}(t_{3}) = \arg \min_{pt_{1,4}} \{ f (pt_{1,4} + t_{3}) + g (pt_{1,4}) \} $, with $4 \leq pt_{1,4} \leq 8$, is determined by applying lemma~\ref{lem:xopt}. It is (see figure~\ref{fig:esS3_tau_3_0_1_pt_1_4})
\begin{equation*}
pt^{\circ}_{1,4}(t_{3}) = x_{\mathrm{e}}(t_{3}) \qquad \text{with} \quad x_{\mathrm{e}}(t_{3}) = \left\{ \begin{array}{ll}
8 &  t_{3} < 3\\
-t_{3} + 23 & 3 \leq t_{3} < 7\\
4 & t_{3} \geq 7
\end{array} \right.
\end{equation*}

\begin{figure}[h!]
\centering
\psfrag{f(x)}[Bl][Bl][.8][0]{$pt^{\circ}_{1,4}(t_{3})$}
\psfrag{x}[bc][Bl][.8][0]{$t_{3}$}
\psfrag{0}[tc][Bl][.8][0]{$0$}
\includegraphics[scale=.2]{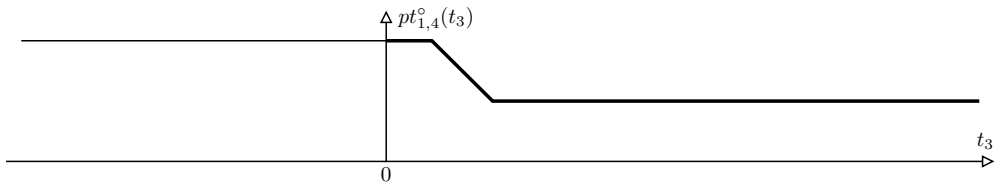}
\caption{Optimal processing time $pt^{\circ}_{1,4}(t_{3})$, under the assumption $\delta_{1} = 1$ in state $[ 3 \; 0 \; 1 \; t_{3}]^{T}$.}
\label{fig:esS3_tau_3_0_1_pt_1_4}
\end{figure}

The conditioned cost-to-go $J^{\circ}_{3,0,1} (t_{3} \mid \delta_{1}=1) = f ( pt^{\circ}_{1,4}(t_{3}) + t_{3} ) + g ( pt^{\circ}_{1,4}(t_{3}) )$, illustrated in figure~\ref{fig:esS3_J_3_0_1_min}, is provided by lemma~\ref{lem:h(t)}. It is specified by the initial value 0.5, by the set \{ 3, 10, 12, 17, 19, 37 \} of abscissae $\gamma_{i}$, $i = 1, \ldots, 6$, at which the slope changes, and by the set \{ 1, 1.5, 3, 3.5, 4, 4.5 \} of slopes $\mu_{i}$, $i = 1, \ldots, 6$, in the various intervals.

{\it Case ii)} in which it is assumed $\delta_{2} = 1$ (and $\delta_{1} = 0$).

In this case, it is necessary to minimize, with respect to the (continuos) decision variable $\tau$ which corresponds to the processing time $pt_{2,1}$, the following function
\begin{equation*}
\alpha_{2,1} \, \max \{ t_{3} + st_{1,2} + \tau - dd_{2,1} \, , \, 0 \} + \beta_{2} \, ( pt^{\mathrm{nom}}_{2} - \tau ) + sc_{1,2} + J^{\circ}_{3,1,2} (t_{4})
\end{equation*}
that can be written as $f (pt_{2,1} + t_{3}) + g (pt_{2,1})$ being
\begin{equation*}
f (pt_{2,1} + t_{3}) = 2 \cdot \max \{ pt_{2,1} + t_{3} - 20 \, , \, 0 \} + 0.5 + J^{\circ}_{3,1,2} (pt_{2,1} + t_{3} + 1)
\end{equation*}
\begin{equation*}
g (pt_{2,1}) = \left\{ \begin{array}{ll}
1.5 \cdot (6 - pt_{2,1}) & pt_{2,1} \in [ 4 , 6 )\\
0 & pt_{2,1} \notin [ 4 , 6 )
\end{array} \right.
\end{equation*}
The function $pt^{\circ}_{2,1}(t_{3}) = \arg \min_{pt_{2,1}} \{ f (pt_{2,1} + t_{3}) + g (pt_{2,1}) \} $, with $4 \leq pt_{2,1} \leq 6$, is determined by applying lemma~\ref{lem:xopt}. It is (see figure~\ref{fig:esS3_tau_3_0_1_pt_2_1})
\begin{equation*}
pt^{\circ}_{2,1}(t_{3}) = x_{\mathrm{e}}(t_{3}) \qquad \text{with} \quad x_{\mathrm{e}}(t_{3}) = \left\{ \begin{array}{ll}
6 &  t_{3} < 13.5\\
-t_{3} + 19.5 & 13.5 \leq t_{3} < 15.5\\
4 & t_{3} \geq 15.5
\end{array} \right.
\end{equation*}

\begin{figure}[h!]
\centering
\psfrag{f(x)}[Bl][Bl][.8][0]{$pt^{\circ}_{2,1}(t_{3})$}
\psfrag{x}[bc][Bl][.8][0]{$t_{3}$}
\psfrag{0}[tc][Bl][.8][0]{$0$}
\includegraphics[scale=.2]{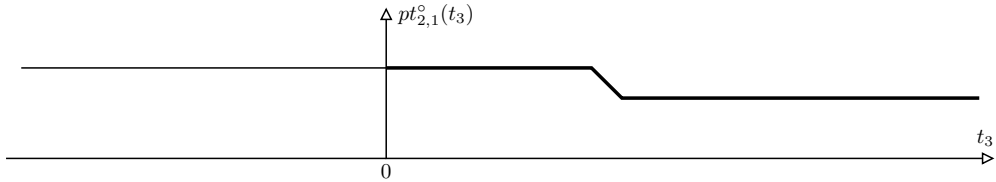}
\caption{Optimal processing time $pt^{\circ}_{2,1}(t_{3})$, under the assumption $\delta_{2} = 1$ in state $[ 3 \; 0 \; 1 \; t_{3}]^{T}$.}
\label{fig:esS3_tau_3_0_1_pt_2_1}
\end{figure}

The conditioned cost-to-go $J^{\circ}_{3,0,1} (t_{3} \mid \delta_{2}=1) = f ( pt^{\circ}_{2,1}(t_{3}) + t_{3} ) + g ( pt^{\circ}_{2,1}(t_{3}) )$, illustrated in figure~\ref{fig:esS3_J_3_0_1_min}, is provided by lemma~\ref{lem:h(t)}. It is specified by the initial value 1.5, by the set \{ 11, 13.5, 16, 23 \} of abscissae $\gamma_{i}$, $i = 1, \ldots, 4$, at which the slope changes, and by the set \{ 1, 1.5, 3.5, 4.5 \} of slopes $\mu_{i}$, $i = 1, \ldots, 4$, in the various intervals.

\begin{figure}[h!]
\centering
\psfrag{J1}[br][Bc][.8][0]{$J^{\circ}_{3,0,1} (t_{3} \mid \delta_{1} = 1)$}
\psfrag{J2}[tl][Bc][.8][0]{$J^{\circ}_{3,0,1} (t_{3} \mid \delta_{2} = 1)$}
\psfrag{0}[cc][tc][.7][0]{$0$}\includegraphics[scale=.6]{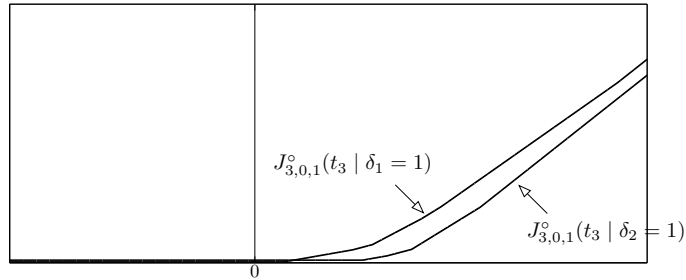}%
\vspace{-12pt}
\caption{Conditioned costs-to-go $J^{\circ}_{3,0,1} (t_{3} \mid \delta_{1} = 1)$ and $J^{\circ}_{3,0,1} (t_{3} \mid \delta_{2} = 1)$ in state $[3 \; 0 \; 1 \; t_{3}]^{T}$.}
\label{fig:esS3_J_3_0_1_min}
\end{figure}

In order to find the optimal cost-to-go $J^{\circ}_{3,0,1} (t_{3})$, it is necessary to carry out the following minimization
\begin{equation*}
J^{\circ}_{3,0,1} (t_{3}) = \min \big\{ J^{\circ}_{3,0,1} (t_{3} \mid \delta_{1} = 1) \, , \, J^{\circ}_{3,0,1} (t_{3} \mid \delta_{2} = 1) \big\}
\end{equation*}
which provides, in accordance with lemma~\ref{lem:min}, the continuous, nondecreasing, piecewise linear function illustrated in figure~\ref{fig:esS3_J_3_0_1}.

\begin{figure}[h!]
\centering
\psfrag{0}[cc][tc][.8][0]{$0$}
\includegraphics[scale=.8]{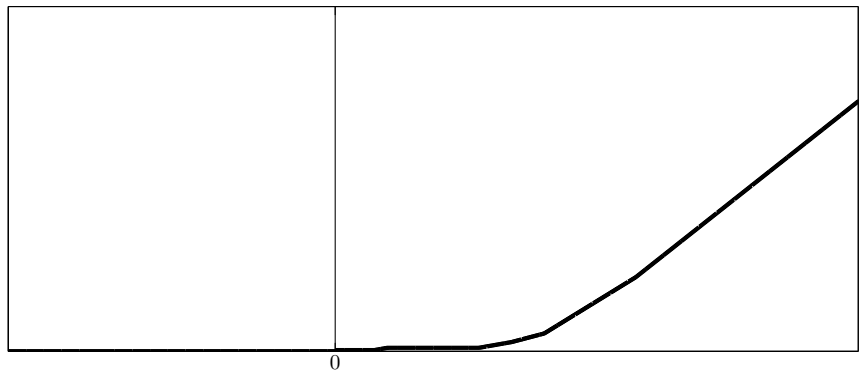}%
\vspace{-12pt}
\caption{Optimal cost-to-go $J^{\circ}_{3,0,1} (t_{3})$ in state $[3 \; 0 \; 1 \; t_{3}]^{T}$.}
\label{fig:esS3_J_3_0_1}
\end{figure}

The function $J^{\circ}_{3,0,1} (t_{3})$ is specified by the initial value 0.5, by the set \{ 3, 4, 11, 13.5, 16, 23 \} of abscissae $\gamma_{i}$, $i = 1, \ldots, 6$, at which the slope changes, and by the set \{ 1, 0, 1, 1.5, 3.5, 4.5 \} of slopes $\mu_{i}$, $i = 1, \ldots, 6$, in the various intervals.

Since $J^{\circ}_{3,0,1} (t_{3} \mid \delta_{1} = 1)$ is the minimum in $(-\infty, 4)$, and $J^{\circ}_{3,0,1} (t_{3} \mid \delta_{2} = 1)$ is the minimum in $[4,+\infty$, the optimal control strategies for this state are
\begin{equation*}
\delta_{1}^{\circ} (3,0,1, t_{3}) = \left\{ \begin{array}{ll}
1 &  t_{3} < 4\\
0 & t_{3} \geq 4
\end{array} \right. \qquad \delta_{2}^{\circ} (3,0,1, t_{3}) = \left\{ \begin{array}{ll}
0 &  t_{3} < 4\\
1 & t_{3} \geq 4
\end{array} \right.
\end{equation*}
\begin{equation*}
\tau^{\circ} (3,0,1, t_{3}) = \left\{ \begin{array}{ll}
8 &  t_{3} < 3\\
-t_{3} + 23 & 3 \leq t_{3} < 4\\
6 &  4 \leq t_{3} < 13.5\\
-t_{3} + 19.5 & 13.5 \leq t_{3} < 15.5\\
4 & t_{3} \geq 15.5
\end{array} \right.
\end{equation*}
The optimal control strategy $\tau^{\circ} (3,0,1, t_{3})$ is illustrated in figure~\ref{fig:esS3_tau_2_1_1}.

\begin{figure}[h!]
\centering
\psfrag{f(x)}[Bl][Bl][.8][0]{$\tau^{\circ} (3,0,1, t_{3})$}
\psfrag{x}[bc][Bl][.8][0]{$t_{3}$}
\psfrag{0}[tc][Bl][.8][0]{$0$}
\includegraphics[scale=.2]{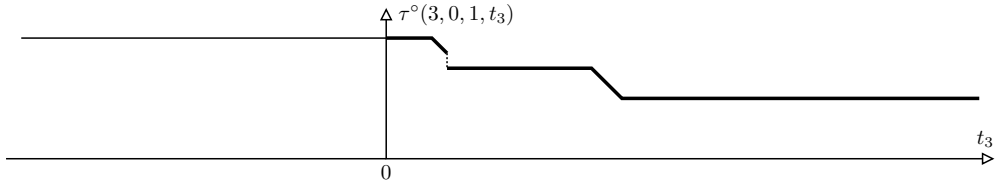}
\caption{Optimal control strategy $\tau^{\circ} (3,0,1, t_{3})$ in state $[ 3 \; 0 \; 1 \; t_{3}]^{T}$.}
\label{fig:esS3_tau_3_0_1}
\end{figure}

%%
%% S6 - [0 2 2 t2]
%%

{\bf Stage $2$ -- State $\boldsymbol{[0 \; 2 \; 2 \; t_{2}]^{T}}$ ($S6$)}

In state $[0 \; 2 \; 2 \; t_{2}]^{T}$, the cost function to be minimized, with respect to the (continuos) decision variable $\tau$ and to the (binary) decision variables $\delta_{1}$ and $\delta_{2}$ is
\begin{equation*}
\begin{split}
&\delta_{1} \big[ \alpha_{1,1} \, \max \{ t_{2} + st_{2,1} + \tau - dd_{1,1} \, , \, 0 \} + \beta_{1} \, ( pt^{\mathrm{nom}}_{1} - \tau ) + sc_{2,1} + J^{\circ}_{1,2,1} (t_{3}) \big] +\\
&+ \delta_{2} \big[ \alpha_{2,3} \, \max \{ t_{2} + st_{2,2} + \tau - dd_{2,3} \, , \, 0 \} + \beta_{2} \, ( pt^{\mathrm{nom}}_{2} - \tau ) + sc_{2,2} + J^{\circ}_{0,3,2} (t_{3}) \big]
\end{split}
\end{equation*}

{\it Case i)} in which it is assumed $\delta_{1} = 1$ (and $\delta_{2} = 0$).

In this case, it is necessary to minimize, with respect to the (continuos) decision variable $\tau$ which corresponds to the processing time $pt_{1,1}$, the following function
\begin{equation*}
\alpha_{1,1} \, \max \{ t_{2} + st_{2,1} + \tau - dd_{1,1} \, , \, 0 \} + \beta_{1} \, ( pt^{\mathrm{nom}}_{1} - \tau ) + sc_{2,1} + J^{\circ}_{1,2,1} (t_{3})
\end{equation*}
that can be written as $f (pt_{1,1} + t_{2}) + g (pt_{1,1})$ being
\begin{equation*}
f (pt_{1,1} + t_{2}) = 0.75 \cdot \max \{ pt_{1,1} + t_{2} - 18.5 \, , \, 0 \} + 1 + J^{\circ}_{1,2,1} (pt_{1,1} + t_{2} + 0.5)
\end{equation*}
\begin{equation*}
g (pt_{1,1}) = \left\{ \begin{array}{ll}
8 - pt_{1,1} & pt_{1,1} \in [ 4 , 8 )\\
0 & pt_{1,1} \notin [ 4 , 8 )
\end{array} \right.
\end{equation*}
The function $pt^{\circ}_{1,1}(t_{2}) = \arg \min_{pt_{1,1}} \{ f (pt_{1,1} + t_{2}) + g (pt_{1,1}) \} $, with $4 \leq pt_{1,1} \leq 8$, is determined by applying lemma~\ref{lem:xopt}. It is (see figure~\ref{fig:esS3_tau_0_2_2_pt_1_1})
\begin{equation*}
pt^{\circ}_{1,1}(t_{2}) = \left\{ \begin{array}{ll}
x_{\mathrm{s}}(t_{2}) &  t_{2} < -0.5\\
x_{\mathrm{e}}(t_{2}) & t_{2} \geq -0.5
\end{array} \right. \qquad \text{with} \quad x_{\mathrm{s}}(t_{2}) = \left\{ \begin{array}{ll}
8 &  t_{2} < -1.5\\
-t_{2} + 6.5 & -1.5 \leq t_{2} < -0.5
\end{array} \right. \  , 
\end{equation*}
\begin{equation*}
\qquad \text{and} \quad x_{\mathrm{e}}(t_{2}) = \left\{ \begin{array}{ll}
8 &  -0.5 \leq t_{2} < 4.5\\
-t_{2} + 12.5 & 4.5 \leq t_{2} < 8.5\\
4 & t_{2} \geq 8.5
\end{array} \right.
\end{equation*}

The conditioned cost-to-go $J^{\circ}_{0,2,2} (t_{2} \mid \delta_{1}=1) = f ( pt^{\circ}_{1,1}(t_{2}) + t_{2} ) + g ( pt^{\circ}_{1,1}(t_{2}) )$, illustrated in figure~\ref{fig:esS3_J_0_2_2_min}, is provided by lemma~\ref{lem:h(t)}. It is specified by the initial value 1.5, by the set \{ --1.5, --0.5, 1, 4.5, 14.5, 15.5, 16.5, 18.5 \} of abscissae $\gamma_{i}$, $i = 1, \ldots, 8$, at which the slope changes, and by the set \{ 1, 0, 0.5, 1, 1.75, 2.25, 3.25, 4.25 \} of slopes $\mu_{i}$, $i = 1, \ldots, 8$, in the various intervals.

\newpage

\begin{figure}[h!]
\centering
\psfrag{f(x)}[Bl][Bl][.8][0]{$pt^{\circ}_{1,1}(t_{2})$}
\psfrag{x}[bc][Bl][.8][0]{$t_{2}$}
\psfrag{0}[tc][Bl][.8][0]{$0$}
\includegraphics[scale=.2]{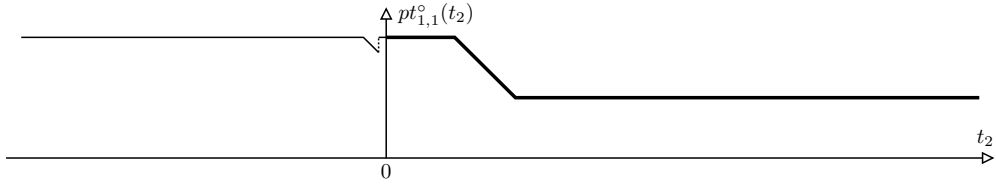}
\caption{Optimal processing time $pt^{\circ}_{1,1}(t_{2})$, under the assumption $\delta_{1} = 1$ in state $[ 0 \; 2 \; 2 \; t_{2}]^{T}$.}
\label{fig:esS3_tau_0_2_2_pt_1_1}
\end{figure}

{\it Case ii)} in which it is assumed $\delta_{2} = 1$ (and $\delta_{1} = 0$).

In this case, it is necessary to minimize, with respect to the (continuos) decision variable $\tau$ which corresponds to the processing time $pt_{2,3}$, the following function
\begin{equation*}
\alpha_{2,3} \, \max \{ t_{2} + st_{2,2} + \tau - dd_{2,3} \, , \, 0 \} + \beta_{2} \, ( pt^{\mathrm{nom}}_{2} - \tau ) + sc_{2,2} + J^{\circ}_{0,3,2} (t_{3})
\end{equation*}
that can be written as $f (pt_{2,3} + t_{2}) + g (pt_{2,3})$ being
\begin{equation*}
f (pt_{2,3} + t_{2}) = \max \{ pt_{2,3} + t_{2} - 38 \, , \, 0 \} + J^{\circ}_{0,3,2} (pt_{2,3} + t_{2})
\end{equation*}
\begin{equation*}
g (pt_{2,3}) = \left\{ \begin{array}{ll}
1.5 \cdot (6 - pt_{2,3}) & pt_{2,3} \in [ 4 , 6 )\\
0 & pt_{2,3} \notin [ 4 , 6 )
\end{array} \right.
\end{equation*}
The function $pt^{\circ}_{2,3}(t_{2}) = \arg \min_{pt_{2,3}} \{ f (pt_{2,3} + t_{2}) + g (pt_{2,3}) \} $, with $4 \leq pt_{2,3} \leq 6$, is determined by applying lemma~\ref{lem:xopt}. It is (see figure~\ref{fig:esS3_tau_0_2_2_pt_2_3})
\begin{equation*}
pt^{\circ}_{2,3}(t_{2}) = x_{\mathrm{e}}(t_{2}) \qquad \text{with} \quad x_{\mathrm{e}}(t_{2}) = \left\{ \begin{array}{ll}
6 &  t_{2} < 8.5\\
-t_{2} + 14.5 & 8.5 \leq t_{2} < 10.5\\
4 & t_{2} \geq 10.5
\end{array} \right.
\end{equation*}

\begin{figure}[h!]
\centering
\psfrag{f(x)}[Bl][Bl][.8][0]{$pt^{\circ}_{2,3}(t_{2})$}
\psfrag{x}[bc][Bl][.8][0]{$t_{2}$}
\psfrag{0}[tc][Bl][.8][0]{$0$}
\includegraphics[scale=.2]{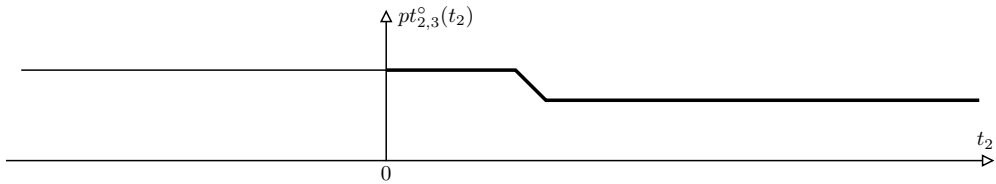}
\caption{Optimal processing time $pt^{\circ}_{2,3}(t_{2})$, under the assumption $\delta_{2} = 1$ in state $[ 0 \; 2 \; 2 \; t_{2}]^{T}$.}
\label{fig:esS3_tau_0_2_2_pt_2_3}
\end{figure}

The conditioned cost-to-go $J^{\circ}_{0,2,2} (t_{2} \mid \delta_{2}=1) = f ( pt^{\circ}_{2,3}(t_{2}) + t_{2} ) + g ( pt^{\circ}_{2,3}(t_{2}) )$, illustrated in figure~\ref{fig:esS3_J_0_2_2_min}, is provided by lemma~\ref{lem:h(t)}. It is specified by the initial value 1, by the set \{ --1.5, 8.5, 10.5, 11.5, 12.5, 16.5, 34 \} of abscissae $\gamma_{i}$, $i = 1, \ldots, 7$, at which the slope changes, and by the set \{ 1, 1.5, 1.75, 2.25, 2.75, 3.25, 4.25 \} of slopes $\mu_{i}$, $i = 1, \ldots, 7$, in the various intervals.

\begin{figure}[h!]
\centering
\psfrag{J1}[tl][Bc][.8][0]{$J^{\circ}_{0,2,2} (t_{2} \mid \delta_{1} = 1)$}
\psfrag{J2}[br][Bc][.8][0]{$J^{\circ}_{0,2,2} (t_{2} \mid \delta_{2} = 1)$}
\psfrag{0}[cc][tc][.7][0]{$0$}\includegraphics[scale=.6]{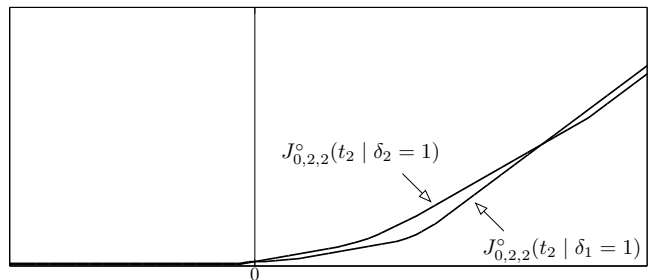}%
\vspace{-12pt}
\caption{Conditioned costs-to-go $J^{\circ}_{0,2,2} (t_{2} \mid \delta_{1} = 1)$ and $J^{\circ}_{0,2,2} (t_{2} \mid \delta_{2} = 1)$ in state $[0 \; 2 \; 2 \; t_{2}]^{T}$.}
\label{fig:esS3_J_0_2_2_min}
\end{figure}

In order to find the optimal cost-to-go $J^{\circ}_{0,2,2} (t_{2})$, it is necessary to carry out the following minimization
\begin{equation*}
J^{\circ}_{0,2,2} (t_{2}) = \min \big\{ J^{\circ}_{0,2,2} (t_{2} \mid \delta_{1} = 1) \, , \, J^{\circ}_{0,2,2} (t_{2} \mid \delta_{2} = 1) \big\}
\end{equation*}
which provides, in accordance with lemma~\ref{lem:min}, the continuous, nondecreasing, piecewise linear function illustrated in figure~\ref{fig:esS3_J_0_2_2}.

\begin{figure}[h!]
\centering
\psfrag{0}[cc][tc][.8][0]{$0$}
\includegraphics[scale=.8]{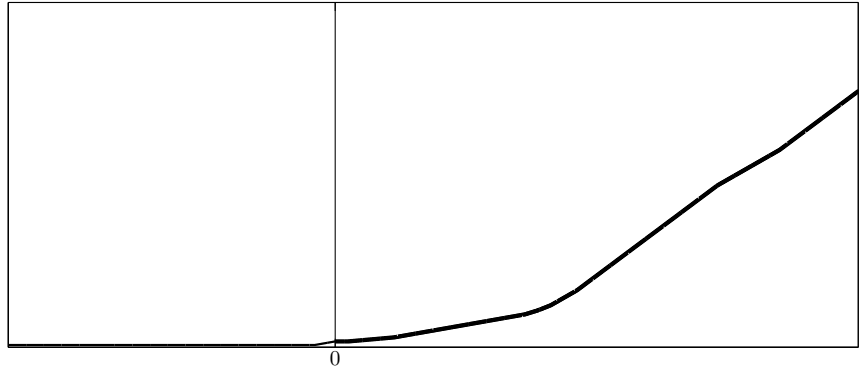}%
\vspace{-12pt}
\caption{Optimal cost-to-go $J^{\circ}_{0,2,2} (t_{2})$ in state $[0 \; 2 \; 2 \; t_{2}]^{T}$.}
\label{fig:esS3_J_0_2_2}
\end{figure}

The function $J^{\circ}_{0,2,2} (t_{2})$ is specified by the initial value 1, by the set \{ --1.5, 0, 1, 4.5, 14.5, 15.5, 16.5, 18.5, 29.25, 34 \} of abscissae $\gamma_{i}$, $i = 1, \ldots, 10$, at which the slope changes, and by the set \{ 1, 0, 0.5, 1, 1.75, 2.25, 3.25, 4.25, 3.25, 4.25 \} of slopes $\mu_{i}$, $i = 1, \ldots, 10$, in the various intervals.

Since $J^{\circ}_{0,2,2} (t_{2} \mid \delta_{1} = 1)$ is the minimum in $[0,29.25)$, and $J^{\circ}_{0,2,2} (t_{2} \mid \delta_{2} = 1)$ is the minimum in $(-\infty, 0)$ and in $[29.25,+\infty)$, the optimal control strategies for this state are
\begin{equation*}
\delta_{1}^{\circ} (0,2,2, t_{2}) = \left\{ \begin{array}{ll}
0 &  t_{2} < 0\\
1 &  0 \leq t_{2} < 29.25\\
0 & t_{2} \geq 29.25
\end{array} \right. \qquad \delta_{2}^{\circ} (0,2,2, t_{2}) = \left\{ \begin{array}{ll}
1 &  t_{2} < 0\\
0 &  0 \leq t_{2} < 29.25\\
1 & t_{2} \geq 29.25
\end{array} \right.
\end{equation*}
\begin{equation*}
\tau^{\circ} (0,2,2, t_{2}) = \left\{ \begin{array}{ll}
6 &  t_{2} < 0\\
8 &  0 \leq t_{2} < 4.5\\
-t_{2} + 12.5 & 4.5 \leq t_{2} < 8.5\\
4 & t_{2} \geq 8.5
\end{array} \right.
\end{equation*}
The optimal control strategy $\tau^{\circ} (0,2,2, t_{2})$ is illustrated in figure~\ref{fig:esS3_tau_0_2_2}.

\begin{figure}[h!]
\centering
\psfrag{f(x)}[Bl][Bl][.8][0]{$\tau^{\circ} (0,2,2, t_{2})$}
\psfrag{x}[bc][Bl][.8][0]{$t_{2}$}
\psfrag{0}[tc][Bl][.8][0]{$0$}
\includegraphics[scale=.2]{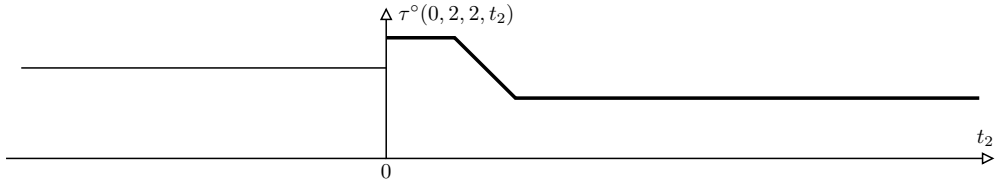}
\caption{Optimal control strategy $\tau^{\circ} (0,2,2, t_{2})$ in state $[ 0 \; 2 \; 2 \; t_{2}]^{T}$.}
\label{fig:esS3_tau_0_2_2}
\end{figure}

%%
%% S5 - [1 1 2 t2]
%%

{\bf Stage $2$ -- State $\boldsymbol{[1 \; 1 \; 2 \; t_{2}]^{T}}$ ($S5$)}

In state $[1 \; 1 \; 2 \; t_{2}]^{T}$, the cost function to be minimized, with respect to the (continuos) decision variable $\tau$ and to the (binary) decision variables $\delta_{1}$ and $\delta_{2}$ is
\begin{equation*}
\begin{split}
&\delta_{1} \big[ \alpha_{1,2} \, \max \{ t_{2} + st_{2,1} + \tau - dd_{1,2} \, , \, 0 \} + \beta_{1} \, ( pt^{\mathrm{nom}}_{1} - \tau ) + sc_{2,1} + J^{\circ}_{2,1,1} (t_{3}) \big] +\\
&+ \delta_{2} \big[ \alpha_{2,2} \, \max \{ t_{2} + st_{2,2} + \tau - dd_{2,2} \, , \, 0 \} + \beta_{2} \, ( pt^{\mathrm{nom}}_{2} - \tau ) + sc_{2,2} + J^{\circ}_{1,2,2} (t_{3}) \big]
\end{split}
\end{equation*}

{\it Case i)} in which it is assumed $\delta_{1} = 1$ (and $\delta_{2} = 0$).

In this case, it is necessary to minimize, with respect to the (continuos) decision variable $\tau$ which corresponds to the processing time $pt_{1,2}$, the following function
\begin{equation*}
\alpha_{1,2} \, \max \{ t_{2} + st_{2,1} + \tau - dd_{1,2} \, , \, 0 \} + \beta_{1} \, ( pt^{\mathrm{nom}}_{1} - \tau ) + sc_{2,1} + J^{\circ}_{2,1,1} (t_{3})
\end{equation*}
that can be written as $f (pt_{1,2} + t_{2}) + g (pt_{1,2})$ being
\begin{equation*}
f (pt_{1,2} + t_{2}) = 0.5 \cdot \max \{ pt_{1,2} + t_{2} - 23.5 \, , \, 0 \} + 1 + J^{\circ}_{2,1,1} (pt_{1,2} + t_{2} + 0.5)
\end{equation*}
\begin{equation*}
g (pt_{1,2}) = \left\{ \begin{array}{ll}
8 - pt_{1,2} & pt_{1,2} \in [ 4 , 8 )\\
0 & pt_{1,2} \notin [ 4 , 8 )
\end{array} \right.
\end{equation*}
The function $pt^{\circ}_{1,2}(t_{2}) = \arg \min_{pt_{1,2}} \{ f (pt_{1,2} + t_{2}) + g (pt_{1,2}) \} $, with $4 \leq pt_{1,2} \leq 8$, is determined by applying lemma~\ref{lem:xopt}. It is (see figure~\ref{fig:esS3_tau_1_1_2_pt_1_2})
\begin{equation*}
pt^{\circ}_{1,2}(t_{2}) = \left\{ \begin{array}{ll}
x_{\mathrm{s}}(t_{2}) &  t_{2} < -6.5\\
x_{1}(t_{2}) & -6.5 \leq t_{2} < 2.5\\
x_{\mathrm{e}}(t_{2}) & t_{2} \geq 2.5
\end{array} \right. \qquad \text{with} \quad x_{\mathrm{s}}(t_{2}) = \left\{ \begin{array}{ll}
8 &  t_{2} < -7.5\\
-t_{2} +0.5 & -7.5 \leq t_{2} < -6.5
\end{array} \right. \  , 
\end{equation*}
\begin{equation*}
\qquad x_{1}(t_{2}) = \left\{ \begin{array}{ll}
8 &  -6.5 \leq t_{2} < 0.5\\
-t_{2} + 8.5 & 0.5 \leq t_{2} < 2.5
\end{array} \right. \  , \  \text{and} \quad x_{\mathrm{e}}(t_{2}) = \left\{ \begin{array}{ll}
8 &  2.5 \leq t_{2} < 5\\
-t_{2} +13 & 5 \leq t_{2} < 9\\
4 & t_{2} \geq 9
\end{array} \right.
\end{equation*}

\begin{figure}[h!]
\centering
\psfrag{f(x)}[Bl][Bl][.8][0]{$pt^{\circ}_{1,2}(t_{2})$}
\psfrag{x}[bc][Bl][.8][0]{$t_{2}$}
\psfrag{0}[tc][Bl][.8][0]{$0$}
\includegraphics[scale=.2]{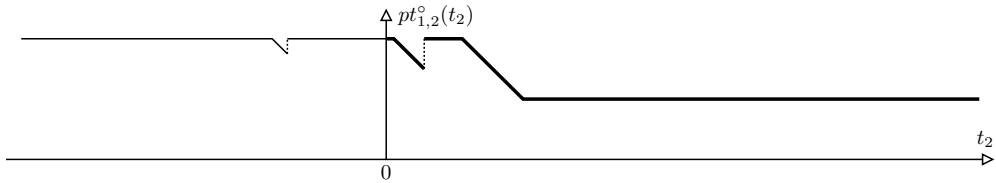}
\caption{Optimal processing time $pt^{\circ}_{1,2}(t_{2})$, under the assumption $\delta_{1} = 1$ in state $[ 1 \; 1 \; 2 \; t_{2}]^{T}$.}
\label{fig:esS3_tau_1_1_2_pt_1_2}
\end{figure}

The conditioned cost-to-go $J^{\circ}_{1,1,2} (t_{2} \mid \delta_{1}=1) = f ( pt^{\circ}_{1,2}(t_{2}) + t_{2} ) + g ( pt^{\circ}_{1,2}(t_{2}) )$, illustrated in figure~\ref{fig:esS3_J_1_1_2_min}, is provided by lemma~\ref{lem:h(t)}. It is specified by the initial value 1.5, by the set \{ --7.5, --6.5, 0.5, 2.5, 5, 12.5, 14.5, 15, 16.1$\overline{6}$, 18.5, 19.5, 20.5 \} of abscissae $\gamma_{i}$, $i = 1, \ldots, 12$, at which the slope changes, and by the set \{ 1, 0, 1, 0.5, 1, 1.5, 2, 3, 1.5, 2.5, 3, 4.5 \} of slopes $\mu_{i}$, $i = 1, \ldots, 12$, in the various intervals.

{\it Case ii)} in which it is assumed $\delta_{2} = 1$ (and $\delta_{1} = 0$).

In this case, it is necessary to minimize, with respect to the (continuos) decision variable $\tau$ which corresponds to the processing time $pt_{2,2}$, the following function
\begin{equation*}
\alpha_{2,2} \, \max \{ t_{2} + st_{2,2} + \tau - dd_{2,2} \, , \, 0 \} + \beta_{2} \, ( pt^{\mathrm{nom}}_{2} - \tau ) + sc_{2,2} + J^{\circ}_{1,2,2} (t_{3})
\end{equation*}
that can be written as $f (pt_{2,2} + t_{2}) + g (pt_{2,2})$ being
\begin{equation*}
f (pt_{2,2} + t_{2}) = \max \{ pt_{2,2} + t_{2} - 24 \, , \, 0 \} + J^{\circ}_{1,2,2} (pt_{2,2} + t_{2})
\end{equation*}
\begin{equation*}
g (pt_{2,2}) = \left\{ \begin{array}{ll}
1.5 \cdot (6 - pt_{2,2}) & pt_{2,2} \in [ 4 , 6 )\\
0 & pt_{2,2} \notin [ 4 , 6 )
\end{array} \right.
\end{equation*}
The function $pt^{\circ}_{2,2}(t_{2}) = \arg \min_{pt_{2,2}} \{ f (pt_{2,2} + t_{2}) + g (pt_{2,2}) \} $, with $4 \leq pt_{2,2} \leq 6$, is determined by applying lemma~\ref{lem:xopt}. It is (see figure~\ref{fig:esS3_tau_1_1_2_pt_2_2})
\begin{equation*}
pt^{\circ}_{2,2}(t_{2}) = x_{\mathrm{e}}(t_{2}) \qquad \text{with} \quad x_{\mathrm{e}}(t_{2}) = \left\{ \begin{array}{ll}
6 &  t_{2} < 13.5\\
-t_{2} + 19.5 & 13.5 \leq t_{2} < 15.5\\
4 & t_{2} \geq 15.5
\end{array} \right.
\end{equation*}

\begin{figure}[h!]
\centering
\psfrag{f(x)}[Bl][Bl][.8][0]{$pt^{\circ}_{2,2}(t_{2})$}
\psfrag{x}[bc][Bl][.8][0]{$t_{2}$}
\psfrag{0}[tc][Bl][.8][0]{$0$}
\includegraphics[scale=.2]{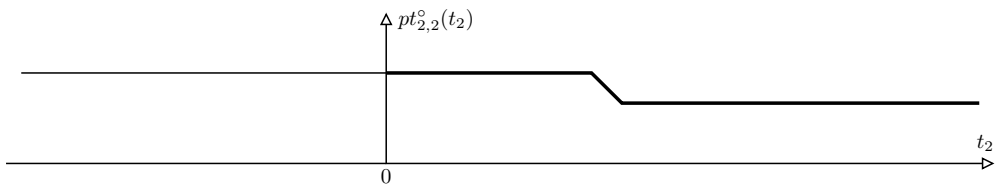}
\caption{Optimal processing time $pt^{\circ}_{2,2}(t_{2})$, under the assumption $\delta_{2} = 1$ in state $[ 1 \; 1 \; 2 \; t_{2}]^{T}$.}
\label{fig:esS3_tau_1_1_2_pt_2_2}
\end{figure}

The conditioned cost-to-go $J^{\circ}_{1,1,2} (t_{2} \mid \delta_{2}=1) = f ( pt^{\circ}_{2,2}(t_{2}) + t_{2} ) + g ( pt^{\circ}_{2,2}(t_{2}) )$, illustrated in figure~\ref{fig:esS3_J_1_1_2_min}, is provided by lemma~\ref{lem:h(t)}. It is specified by the initial value 1, by the set \{ 0.5, 2, 3, 6.5, 13.5, 16.5, 18.5, 20, 26.25, 30~\} of abscissae $\gamma_{i}$, $i = 1, \ldots, 10$, at which the slope changes, and by the set \{ 1, 0, 0.5, 1, 1.5, 2.5, 3.5, 4.5, 3.5, 4.5 \} of slopes $\mu_{i}$, $i = 1, \ldots, 10$, in the various intervals.

\begin{figure}[h!]
\centering
\psfrag{J1}[br][Bc][.8][0]{$J^{\circ}_{1,1,2} (t_{2} \mid \delta_{1} = 1)$}
\psfrag{J2}[tl][Bc][.8][0]{$J^{\circ}_{1,1,2} (t_{2} \mid \delta_{2} = 1)$}
\psfrag{0}[cc][tc][.7][0]{$0$}\includegraphics[scale=.6]{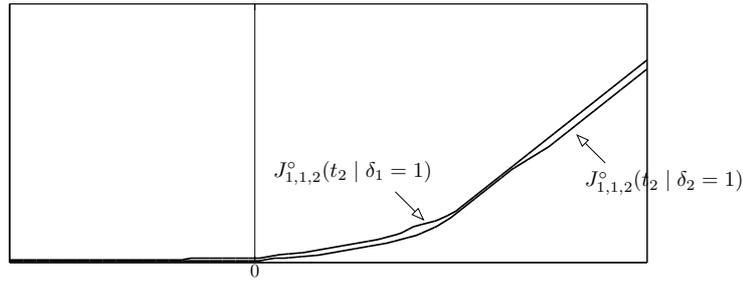}%
\vspace{-12pt}
\caption{Conditioned costs-to-go $J^{\circ}_{1,1,2} (t_{2} \mid \delta_{1} = 1)$ and $J^{\circ}_{1,1,2} (t_{2} \mid \delta_{2} = 1)$ in state $[1 \; 1 \; 2 \; t_{2}]^{T}$.}
\label{fig:esS3_J_1_1_2_min}
\end{figure}

In order to find the optimal cost-to-go $J^{\circ}_{1,1,2} (t_{2})$, it is necessary to carry out the following minimization
\begin{equation*}
J^{\circ}_{1,1,2} (t_{2}) = \min \big\{ J^{\circ}_{1,1,2} (t_{2} \mid \delta_{1} = 1) \, , \, J^{\circ}_{1,1,2} (t_{2} \mid \delta_{2} = 1) \big\}
\end{equation*}
which provides, in accordance with lemma~\ref{lem:min}, the continuous, nondecreasing, piecewise linear function illustrated in figure~\ref{fig:esS3_J_1_1_2}.

\begin{figure}[h!]
\centering
\psfrag{0}[cc][tc][.8][0]{$0$}
\includegraphics[scale=.8]{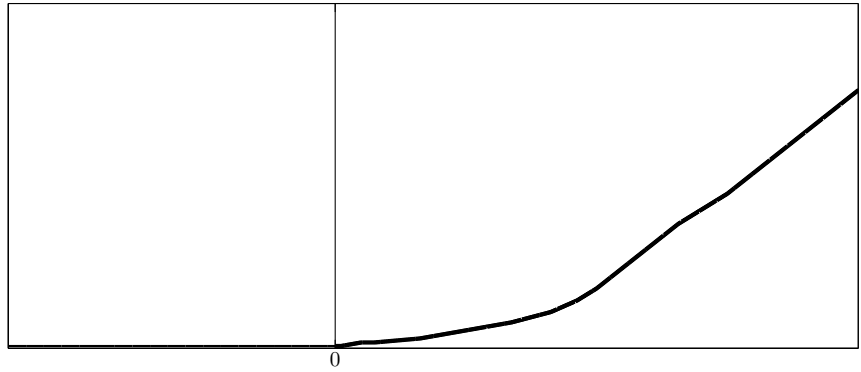}%
\vspace{-12pt}
\caption{Optimal cost-to-go $J^{\circ}_{1,1,2} (t_{2})$ in state $[1 \; 1 \; 2 \; t_{2}]^{T}$.}
\label{fig:esS3_J_1_1_2}
\end{figure}

The function $J^{\circ}_{1,1,2} (t_{2})$ is specified by the initial value 1, by the set \{ 0.5, 2, 3, 6.5, 13.5, 16.5, 18.5, 20, 26.25, 30~\} of abscissae $\gamma_{i}$, $i = 1, \ldots, 10$, at which the slope changes, and by the set \{ 1, 0, 0.5, 1, 1.5, 2.5, 3.5, 4.5, 3.5, 4.5 \} of slopes $\mu_{i}$, $i = 1, \ldots, 10$, in the various intervals.

Since $J^{\circ}_{1,1,2} (t_{2} \mid \delta_{2} = 1)$ is always the minimum (see again figure~\ref{fig:esS3_J_1_1_2_min}), the optimal control strategies for this state are
\begin{equation*}
\delta_{1}^{\circ} (1,1,2, t_{2}) = 0 \quad \forall \, t_{2} \qquad \delta_{2}^{\circ} (1,1,2, t_{2}) = 1 \quad \forall \, t_{2}
\end{equation*}
\begin{equation*}
\tau^{\circ} (1,1,2, t_{2}) = \left\{ \begin{array}{ll}
6 &  t_{2} < 13.5\\
-t_{2} + 19.5 & 13.5 \leq t_{2} < 15.5\\
4 & t_{2} \geq 15.5
\end{array} \right.
\end{equation*}
The optimal control strategy $\tau^{\circ} (1,1,2, t_{2})$ is illustrated in figure~\ref{fig:esS3_tau_1_1_2}.

\begin{figure}[h!]
\centering
\psfrag{f(x)}[Bl][Bl][.8][0]{$\tau^{\circ} (1,1,2, t_{2})$}
\psfrag{x}[bc][Bl][.8][0]{$t_{2}$}
\psfrag{0}[tc][Bl][.8][0]{$0$}
\includegraphics[scale=.2]{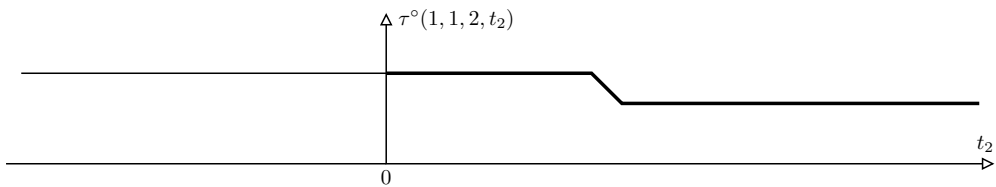}
\caption{Optimal control strategy $\tau^{\circ} (1,1,2, t_{2})$ in state $[ 1 \; 1 \; 2 \; t_{2}]^{T}$.}
\label{fig:esS3_tau_1_1_2}
\end{figure}

%%
%% S4 - [1 1 1 t2]
%%

{\bf Stage $2$ -- State $\boldsymbol{[1 \; 1 \; 1 \; t_{2}]^{T}}$ ($S4$)}

In state $[1 \; 1 \; 1 \; t_{2}]^{T}$, the cost function to be minimized, with respect to the (continuos) decision variable $\tau$ and to the (binary) decision variables $\delta_{1}$ and $\delta_{2}$ is
\begin{equation*}
\begin{split}
&\delta_{1} \big[ \alpha_{1,2} \, \max \{ t_{2} + st_{1,1} + \tau - dd_{1,2} \, , \, 0 \} + \beta_{1} \, ( pt^{\mathrm{nom}}_{1} - \tau ) + sc_{1,1} + J^{\circ}_{2,1,1} (t_{3}) \big] +\\
&+ \delta_{2} \big[ \alpha_{2,2} \, \max \{ t_{2} + st_{1,2} + \tau - dd_{2,2} \, , \, 0 \} + \beta_{2} \, ( pt^{\mathrm{nom}}_{2} - \tau ) + sc_{1,2} + J^{\circ}_{1,2,2} (t_{3}) \big]
\end{split}
\end{equation*}

{\it Case i)} in which it is assumed $\delta_{1} = 1$ (and $\delta_{2} = 0$).

In this case, it is necessary to minimize, with respect to the (continuos) decision variable $\tau$ which corresponds to the processing time $pt_{1,2}$, the following function
\begin{equation*}
\alpha_{1,2} \, \max \{ t_{2} + st_{1,1} + \tau - dd_{1,2} \, , \, 0 \} + \beta_{1} \, ( pt^{\mathrm{nom}}_{1} - \tau ) + sc_{1,1} + J^{\circ}_{2,1,1} (t_{3})
\end{equation*}
that can be written as $f (pt_{1,2} + t_{2}) + g (pt_{1,2})$ being
\begin{equation*}
f (pt_{1,2} + t_{2}) = 0.5 \cdot \max \{ pt_{1,2} + t_{2} - 24 \, , \, 0 \} + J^{\circ}_{2,1,1} (pt_{1,2} + t_{2})
\end{equation*}
\begin{equation*}
g (pt_{1,2}) = \left\{ \begin{array}{ll}
8 - pt_{1,2} & pt_{1,2} \in [ 4 , 8 )\\
0 & pt_{1,2} \notin [ 4 , 8 )
\end{array} \right.
\end{equation*}
The function $pt^{\circ}_{1,2}(t_{2}) = \arg \min_{pt_{1,2}} \{ f (pt_{1,2} + t_{2}) + g (pt_{1,2}) \} $, with $4 \leq pt_{1,2} \leq 8$, is determined by applying lemma~\ref{lem:xopt}. It is (see figure~\ref{fig:esS3_tau_1_1_1_pt_1_2})
\begin{equation*}
pt^{\circ}_{1,2}(t_{2}) = \left\{ \begin{array}{ll}
x_{\mathrm{s}}(t_{2}) &  t_{2} < -6\\
x_{1}(t_{2}) & -6 \leq t_{2} < 3\\
x_{\mathrm{e}}(t_{2}) & t_{2} \geq 3
\end{array} \right. \qquad \text{with} \quad x_{\mathrm{s}}(t_{2}) = \left\{ \begin{array}{ll}
8 &  t_{2} < -7\\
-t_{2} + 1 & -7 \leq t_{2} < -6
\end{array} \right. \  , 
\end{equation*}
\begin{equation*}
\qquad x_{1}(t_{2}) = \left\{ \begin{array}{ll}
8 &  -6 \leq t_{2} < 1\\
-t_{2} + 9 & 1 \leq t_{2} < 3
\end{array} \right. \  , \  \text{and} \quad x_{\mathrm{e}}(t_{2}) = \left\{ \begin{array}{ll}
8 &  3 \leq t_{2} < 5.5\\
-t_{2} +13.5 & 5.5 \leq t_{2} < 9.5\\
4 & t_{2} \geq 9.5
\end{array} \right.
\end{equation*}

\begin{figure}[h!]
\centering
\psfrag{f(x)}[Bl][Bl][.8][0]{$pt^{\circ}_{1,2}(t_{2})$}
\psfrag{x}[bc][Bl][.8][0]{$t_{2}$}
\psfrag{0}[tc][Bl][.8][0]{$0$}
\includegraphics[scale=.2]{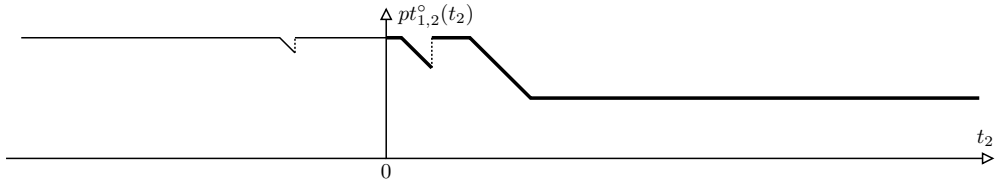}
\caption{Optimal processing time $pt^{\circ}_{1,2}(t_{2})$, under the assumption $\delta_{1} = 1$ in state $[ 1 \; 1 \; 1 \; t_{2}]^{T}$.}
\label{fig:esS3_tau_1_1_1_pt_1_2}
\end{figure}

The conditioned cost-to-go $J^{\circ}_{1,1,1} (t_{2} \mid \delta_{1}=1) = f ( pt^{\circ}_{1,2}(t_{2}) + t_{2} ) + g ( pt^{\circ}_{1,2}(t_{2}) )$, illustrated in figure~\ref{fig:esS3_J_1_1_1_min}, is provided by lemma~\ref{lem:h(t)}. It is specified by the initial value 0.5, by the set \{ --7, --6, 1, 3, 5.5, 13, 15, 15.5, 16.$\overline{6}$, 19, 20, 21 \} of abscissae $\gamma_{i}$, $i = 1, \ldots, 12$, at which the slope changes, and by the set \{ 1, 0, 1, 0.5, 1, 1.5, 2, 3, 1.5, 2.5, 3, 4.5 \} of slopes $\mu_{i}$, $i = 1, \ldots, 12$, in the various intervals.

{\it Case ii)} in which it is assumed $\delta_{2} = 1$ (and $\delta_{1} = 0$).

In this case, it is necessary to minimize, with respect to the (continuos) decision variable $\tau$ which corresponds to the processing time $pt_{2,2}$, the following function
\begin{equation*}
\alpha_{2,2} \, \max \{ t_{2} + st_{1,2} + \tau - dd_{2,2} \, , \, 0 \} + \beta_{2} \, ( pt^{\mathrm{nom}}_{2} - \tau ) + sc_{1,2} + J^{\circ}_{1,2,2} (t_{3})
\end{equation*}
that can be written as $f (pt_{2,2} + t_{2}) + g (pt_{2,2})$ being
\begin{equation*}
f (pt_{2,2} + t_{2}) = \max \{ pt_{2,2} + t_{2} - 23 \, , \, 0 \} + 0.5 + J^{\circ}_{1,2,2} (pt_{2,2} + t_{2} + 1)
\end{equation*}
\begin{equation*}
g (pt_{2,2}) = \left\{ \begin{array}{ll}
1.5 \cdot (6 - pt_{2,2}) & pt_{2,2} \in [ 4 , 6 )\\
0 & pt_{2,2} \notin [ 4 , 6 )
\end{array} \right.
\end{equation*}
The function $pt^{\circ}_{2,2}(t_{2}) = \arg \min_{pt_{2,2}} \{ f (pt_{2,2} + t_{2}) + g (pt_{2,2}) \} $, with $4 \leq pt_{2,2} \leq 6$, is determined by applying lemma~\ref{lem:xopt}. It is (see figure~\ref{fig:esS3_tau_1_1_1_pt_2_2})
\begin{equation*}
pt^{\circ}_{2,2}(t_{2}) = x_{\mathrm{e}}(t_{2}) \qquad \text{with} \quad x_{\mathrm{e}}(t_{2}) = \left\{ \begin{array}{ll}
6 &  t_{2} < 12.5\\
-t_{2} + 18.5 & 12.5 \leq t_{2} < 14.5\\
4 & t_{2} \geq 14.5
\end{array} \right.
\end{equation*}

\begin{figure}[h!]
\centering
\psfrag{f(x)}[Bl][Bl][.8][0]{$pt^{\circ}_{2,2}(t_{2})$}
\psfrag{x}[bc][Bl][.8][0]{$t_{2}$}
\psfrag{0}[tc][Bl][.8][0]{$0$}
\includegraphics[scale=.2]{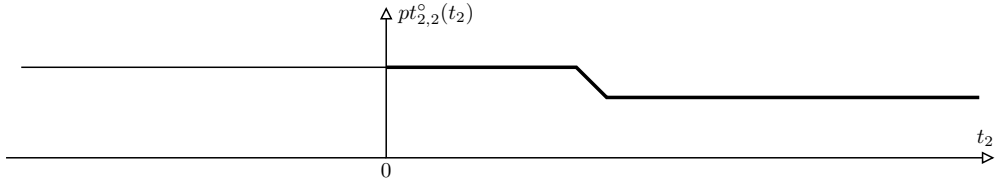}
\caption{Optimal processing time $pt^{\circ}_{2,2}(t_{2})$, under the assumption $\delta_{2} = 1$ in state $[ 1 \; 1 \; 1 \; t_{2}]^{T}$.}
\label{fig:esS3_tau_1_1_1_pt_2_2}
\end{figure}

The conditioned cost-to-go $J^{\circ}_{1,1,1} (t_{2} \mid \delta_{2}=1) = f ( pt^{\circ}_{2,2}(t_{2}) + t_{2} ) + g ( pt^{\circ}_{2,2}(t_{2}) )$, illustrated in figure~\ref{fig:esS3_J_1_1_1_min}, is provided by lemma~\ref{lem:h(t)}. It is specified by the initial value 1.5, by the set \{ --0.5, 1, 2, 5.5, 12.5, 15.5, 17.5, 19, 25.25, 29 \} of abscissae $\gamma_{i}$, $i = 1, \ldots, 10$, at which the slope changes, and by the set \{ 1, 0, 0.5, 1, 1.5, 2.5, 3.5, 4.5, 3.5, 4.5 \} of slopes $\mu_{i}$, $i = 1, \ldots, 10$, in the various intervals.

\begin{figure}[h!]
\centering
\psfrag{J1}[tl][Bc][.8][0]{$J^{\circ}_{1,1,1} (t_{2} \mid \delta_{1} = 1)$}
\psfrag{J2}[br][Bc][.8][0]{$J^{\circ}_{1,1,1} (t_{2} \mid \delta_{2} = 1)$}
\psfrag{0}[cc][tc][.7][0]{$0$}\includegraphics[scale=.6]{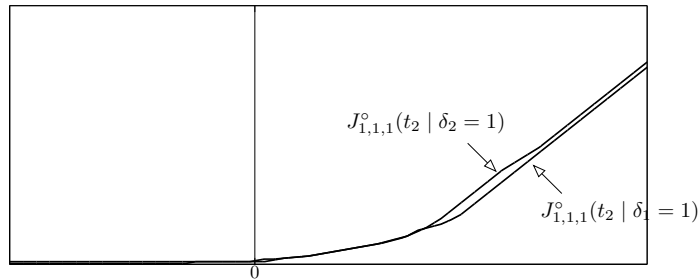}%
\vspace{-12pt}
\caption{Conditioned costs-to-go $J^{\circ}_{1,1,1} (t_{2} \mid \delta_{1} = 1)$ and $J^{\circ}_{1,1,1} (t_{2} \mid \delta_{2} = 1)$ in state $[1 \; 1 \; 1 \; t_{2}]^{T}$.}
\label{fig:esS3_J_1_1_1_min}
\end{figure}

In order to find the optimal cost-to-go $J^{\circ}_{1,1,1} (t_{2})$, it is necessary to carry out the following minimization
\begin{equation*}
J^{\circ}_{1,1,1} (t_{2}) = \min \big\{ J^{\circ}_{1,1,1} (t_{2} \mid \delta_{1} = 1) \, , \, J^{\circ}_{1,1,1} (t_{2} \mid \delta_{2} = 1) \big\}
\end{equation*}
which provides, in accordance with lemma~\ref{lem:min}, the continuous, nondecreasing, piecewise linear function illustrated in figure~\ref{fig:esS3_J_1_1_1}.

\begin{figure}[h!]
\centering
\psfrag{0}[cc][tc][.8][0]{$0$}
\includegraphics[scale=.8]{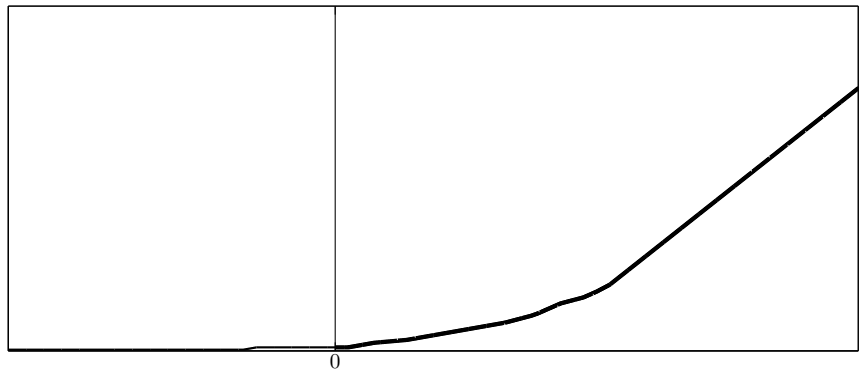}%
\vspace{-12pt}
\caption{Optimal cost-to-go $J^{\circ}_{1,1,1} (t_{2})$ in state $[1 \; 1 \; 1 \; t_{2}]^{T}$.}
\label{fig:esS3_J_1_1_1}
\end{figure}

The function $J^{\circ}_{1,1,1} (t_{2})$ is specified by the initial value 0.5, by the set \{ --7, --6, 1, 3, 5.5, 13, 15, 15.5, 17.25, 19, 20, 21 \} of abscissae $\gamma_{i}$, $i = 1, \ldots, 10$, at which the slope changes, and by the set \{ 1, 0, 1, 0.5, 1, 1.5, 2, 2.5, 1.5, 2.5, 3, 4.5 \} of slopes $\mu_{i}$, $i = 1, \ldots, 10$, in the various intervals.

Since $J^{\circ}_{1,1,1} (t_{2} \mid \delta_{1} = 1)$ is the minimum in $(-\infty, -6)$, in $[1,15.5)$, and in $[17.25,+\infty)$, and $J^{\circ}_{1,1,1} (t_{2} \mid \delta_{2} = 1)$ is the minimum in $[-6,-1)$ and in $[15.5,17.25)$, the optimal control strategies for this state are
\begin{equation*}
\delta_{1}^{\circ} (1,1,1, t_{2}) = \left\{ \begin{array}{ll}
1 &  t_{2} < -6\\
0 &  -6 \leq t_{2} < 1\\
1 &  1 \leq t_{2} < 15.5\\
0 &  15.5 \leq t_{2} < 17.25\\
1 & t_{2} \geq 17.25
\end{array} \right. \qquad \delta_{2}^{\circ} (1,1,1, t_{2}) = \left\{ \begin{array}{ll}
0 &  t_{2} < -6\\
1 &  -6 \leq t_{2} < 1\\
0 &  1 \leq t_{2} < 15.5\\
1 &  15.5 \leq t_{2} < 17.25\\
0 & t_{2} \geq 17.25
\end{array} \right.
\end{equation*}
\begin{equation*}
\tau^{\circ} (1,1,1, t_{2}) = \left\{ \begin{array}{ll}
8 &  t_{2} < -7\\
-t_{2} + 1 & -7 \leq t_{2} < -6\\
6 & -6 \leq t_{2} < 1\\
-t_{2} + 9 & 1 \leq t_{2} < 3\\
8 &  3 \leq t_{2} < 5.5\\
-t_{2} +13.5 & 5.5 \leq t_{2} < 9.5\\
4 & t_{2} \geq 9.5
\end{array} \right.
\end{equation*}
The optimal control strategy $\tau^{\circ} (1,1,1, t_{2})$ is illustrated in figure~\ref{fig:esS3_tau_1_1_1}.

\begin{figure}[h!]
\centering
\psfrag{f(x)}[Bl][Bl][.8][0]{$\tau^{\circ} (1,1,1, t_{2})$}
\psfrag{x}[bc][Bl][.8][0]{$t_{2}$}
\psfrag{0}[tc][Bl][.8][0]{$0$}
\includegraphics[scale=.2]{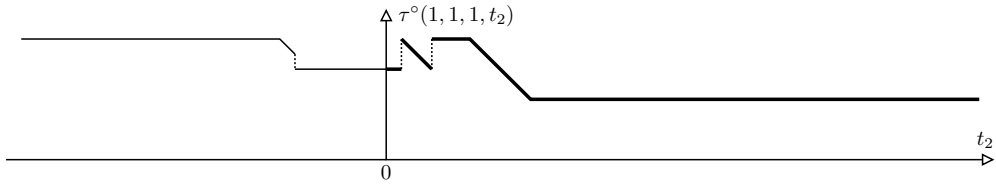}
\caption{Optimal control strategy $\tau^{\circ} (1,1,1, t_{2})$ in state $[ 1 \; 1 \; 1 \; t_{2}]^{T}$.}
\label{fig:esS3_tau_1_1_1}
\end{figure}

%%
%% S3 - [2 0 1 t2]
%%

{\bf Stage $2$ -- State $\boldsymbol{[2 \; 0 \; 1 \; t_{2}]^{T}}$ ($S3$)}

In state $[2 \; 0 \; 1 \; t_{2}]^{T}$, the cost function to be minimized, with respect to the (continuos) decision variable $\tau$ and to the (binary) decision variables $\delta_{1}$ and $\delta_{2}$ is
\begin{equation*}
\begin{split}
&\delta_{1} \big[ \alpha_{1,3} \, \max \{ t_{2} + st_{1,1} + \tau - dd_{1,3} \, , \, 0 \} + \beta_{1} \, ( pt^{\mathrm{nom}}_{1} - \tau ) + sc_{1,1} + J^{\circ}_{3,0,1} (t_{3}) \big] +\\
&+ \delta_{2} \big[ \alpha_{2,1} \, \max \{ t_{2} + st_{1,2} + \tau - dd_{2,1} \, , \, 0 \} + \beta_{2} \, ( pt^{\mathrm{nom}}_{2} - \tau ) + sc_{1,2} + J^{\circ}_{2,1,2} (t_{3}) \big]
\end{split}
\end{equation*}

{\it Case i)} in which it is assumed $\delta_{1} = 1$ (and $\delta_{2} = 0$).

In this case, it is necessary to minimize, with respect to the (continuos) decision variable $\tau$ which corresponds to the processing time $pt_{1,3}$, the following function
\begin{equation*}
\alpha_{1,3} \, \max \{ t_{2} + st_{1,1} + \tau - dd_{1,3} \, , \, 0 \} + \beta_{1} \, ( pt^{\mathrm{nom}}_{1} - \tau ) + sc_{1,1} + J^{\circ}_{3,0,1} (t_{3})
\end{equation*}
that can be written as $f (pt_{1,3} + t_{2}) + g (pt_{1,3})$ being
\begin{equation*}
f (pt_{1,3} + t_{2}) = 1.5 \cdot \max \{ pt_{1,3} + t_{2} - 29 \, , \, 0 \} + J^{\circ}_{3,0,1} (pt_{1,3} + t_{2})
\end{equation*}
\begin{equation*}
g (pt_{1,3}) = \left\{ \begin{array}{ll}
8 - pt_{1,3} & pt_{1,3} \in [ 4 , 8 )\\
0 & pt_{1,3} \notin [ 4 , 8 )
\end{array} \right.
\end{equation*}
The function $pt^{\circ}_{1,3}(t_{2}) = \arg \min_{pt_{1,3}} \{ f (pt_{1,3} + t_{2}) + g (pt_{1,3}) \} $, with $4 \leq pt_{1,3} \leq 8$, is determined by applying lemma~\ref{lem:xopt}. It is (see figure~\ref{fig:esS3_tau_2_0_1_pt_1_3})
\begin{equation*}
pt^{\circ}_{1,3}(t_{2}) = \left\{ \begin{array}{ll}
x_{\mathrm{s}}(t_{2}) &  t_{2} < -4\\
x_{\mathrm{e}}(t_{2}) & t_{2} \geq -4
\end{array} \right. \qquad \text{with} \quad x_{\mathrm{s}}(t_{2}) = \left\{ \begin{array}{ll}
8 &  t_{2} < -5\\
-t_{2} + 3 & -5 \leq t_{2} < -4
\end{array} \right. \  , 
\end{equation*}
\begin{equation*}
\qquad \text{and} \quad x_{\mathrm{e}}(t_{2}) = \left\{ \begin{array}{ll}
8 &  -4 \leq t_{2} < 3\\
-t_{2} + 11 & 3 \leq t_{2} < 7\\
4 & t_{2} \geq 7
\end{array} \right.
\end{equation*}

\begin{figure}[h!]
\centering
\psfrag{f(x)}[Bl][Bl][.8][0]{$pt^{\circ}_{1,3}(t_{2})$}
\psfrag{x}[bc][Bl][.8][0]{$t_{2}$}
\psfrag{0}[tc][Bl][.8][0]{$0$}
\includegraphics[scale=.2]{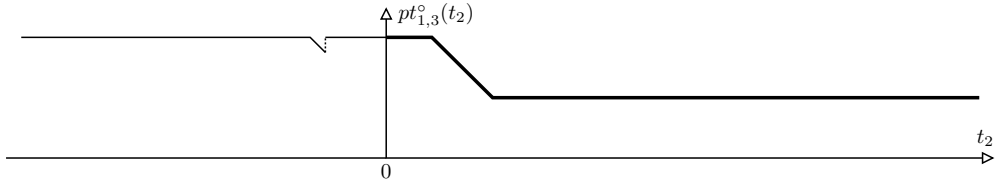}
\caption{Optimal processing time $pt^{\circ}_{1,3}(t_{2})$, under the assumption $\delta_{1} = 1$ in state $[ 2 \; 0 \; 1 \; t_{2}]^{T}$.}
\label{fig:esS3_tau_2_0_1_pt_1_3}
\end{figure}

The conditioned cost-to-go $J^{\circ}_{2,0,1} (t_{2} \mid \delta_{1}=1) = f ( pt^{\circ}_{1,3}(t_{2}) + t_{2} ) + g ( pt^{\circ}_{1,3}(t_{2}) )$, illustrated in figure~\ref{fig:esS3_J_2_0_1_min}, is provided by lemma~\ref{lem:h(t)}. It is specified by the initial value 0.5, by the set \{ --5, --4, 3, 9.5, 12, 19, 25 \} of abscissae $\gamma_{i}$, $i = 1, \ldots, 7$, at which the slope changes, and by the set \{ 1, 0, 1, 1.5, 3.5, 4.5, 6 \} of slopes $\mu_{i}$, $i = 1, \ldots, 7$, in the various intervals.

{\it Case ii)} in which it is assumed $\delta_{2} = 1$ (and $\delta_{1} = 0$).

In this case, it is necessary to minimize, with respect to the (continuos) decision variable $\tau$ which corresponds to the processing time $pt_{2,1}$, the following function
\begin{equation*}
\alpha_{2,1} \, \max \{ t_{2} + st_{1,2} + \tau - dd_{2,1} \, , \, 0 \} + \beta_{2} \, ( pt^{\mathrm{nom}}_{2} - \tau ) + sc_{1,2} + J^{\circ}_{2,1,2} (t_{3})
\end{equation*}
that can be written as $f (pt_{2,1} + t_{2}) + g (pt_{2,1})$ being
\begin{equation*}
f (pt_{2,1} + t_{2}) = 2 \cdot \max \{ pt_{2,1} + t_{2} - 20 \, , \, 0 \} + 0.5 + J^{\circ}_{2,1,2} (pt_{2,1} + t_{2} + 1)
\end{equation*}
\begin{equation*}
g (pt_{2,1}) = \left\{ \begin{array}{ll}
1.5 \cdot (6 - pt_{2,1}) & pt_{2,1} \in [ 4 , 6 )\\
0 & pt_{2,1} \notin [ 4 , 6 )
\end{array} \right.
\end{equation*}
The function $pt^{\circ}_{2,1}(t_{2}) = \arg \min_{pt_{2,1}} \{ f (pt_{2,1} + t_{2}) + g (pt_{2,1}) \} $, with $4 \leq pt_{2,1} \leq 6$, is determined by applying lemma~\ref{lem:xopt}. It is (see figure~\ref{fig:esS3_tau_2_0_1_pt_2_1})
\begin{equation*}
pt^{\circ}_{2,1}(t_{2}) = x_{\mathrm{e}}(t_{2}) \qquad \text{with} \quad x_{\mathrm{e}}(t_{2}) = \left\{ \begin{array}{ll}
6 &  t_{2} < 11\\
-t_{2} + 17 & 11 \leq t_{2} < 13\\
4 & t_{2} \geq 13
\end{array} \right.
\end{equation*}

\begin{figure}[h!]
\centering
\psfrag{f(x)}[Bl][Bl][.8][0]{$pt^{\circ}_{2,1}(t_{2})$}
\psfrag{x}[bc][Bl][.8][0]{$t_{2}$}
\psfrag{0}[tc][Bl][.8][0]{$0$}
\includegraphics[scale=.2]{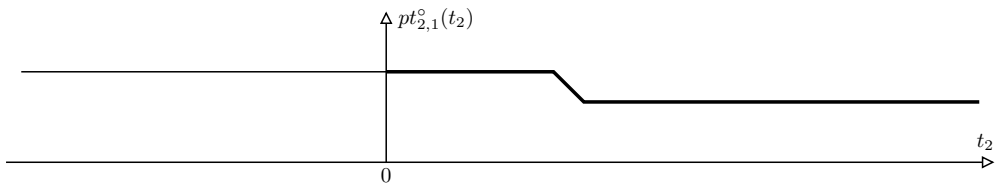}
\caption{Optimal processing time $pt^{\circ}_{2,1}(t_{2})$, under the assumption $\delta_{2} = 1$ in state $[ 2 \; 0 \; 1 \; t_{2}]^{T}$.}
\label{fig:esS3_tau_2_0_1_pt_2_1}
\end{figure}

The conditioned cost-to-go $J^{\circ}_{2,0,1} (t_{2} \mid \delta_{2}=1) = f ( pt^{\circ}_{2,1}(t_{2}) + t_{2} ) + g ( pt^{\circ}_{2,1}(t_{2}) )$, illustrated in figure~\ref{fig:esS3_J_2_0_1_min}, is provided by lemma~\ref{lem:h(t)}. It is specified by the initial value 1.5, by the set \{ 1.5, 3, 4, 7.5, 11, 15, 15.5, 16, 17.5, 19.1$\overline{6}$, 19.5, 23.75, 25 \} of abscissae $\gamma_{i}$, $i = 1, \ldots, 13$, at which the slope changes, and by the set \{ 1, 0, 0.5, 1, 1.5, 2, 3, 5, 6, 4.5, 6, 5, 6 \} of slopes $\mu_{i}$, $i = 1, \ldots, 13$, in the various intervals.

\begin{figure}[h!]
\centering
\psfrag{J1}[tl][Bc][.8][0]{$J^{\circ}_{2,0,1} (t_{2} \mid \delta_{1} = 1)$}
\psfrag{J2}[br][Bc][.8][0]{$J^{\circ}_{2,0,1} (t_{2} \mid \delta_{2} = 1)$}
\psfrag{0}[cc][tc][.7][0]{$0$}\includegraphics[scale=.6]{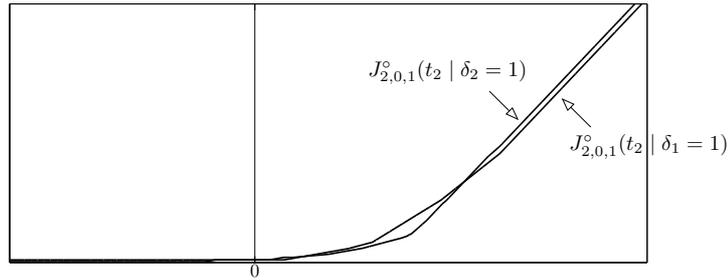}%
\vspace{-12pt}
\caption{Conditioned costs-to-go $J^{\circ}_{2,0,1} (t_{2} \mid \delta_{1} = 1)$ and $J^{\circ}_{2,0,1} (t_{2} \mid \delta_{2} = 1)$ in state $[2 \; 0 \; 1 \; t_{2}]^{T}$.}
\label{fig:esS3_J_2_0_1_min}
\end{figure}

In order to find the optimal cost-to-go $J^{\circ}_{2,0,1} (t_{2})$, it is necessary to carry out the following minimization
\begin{equation*}
J^{\circ}_{2,0,1} (t_{2}) = \min \big\{ J^{\circ}_{2,0,1} (t_{2} \mid \delta_{1} = 1) \, , \, J^{\circ}_{2,0,1} (t_{2} \mid \delta_{2} = 1) \big\}
\end{equation*}
which provides, in accordance with lemma~\ref{lem:min}, the continuous, nondecreasing, piecewise linear function illustrated in figure~\ref{fig:esS3_J_2_0_1}.

\begin{figure}[h!]
\centering
\psfrag{0}[cc][tc][.8][0]{$0$}
\includegraphics[scale=.8]{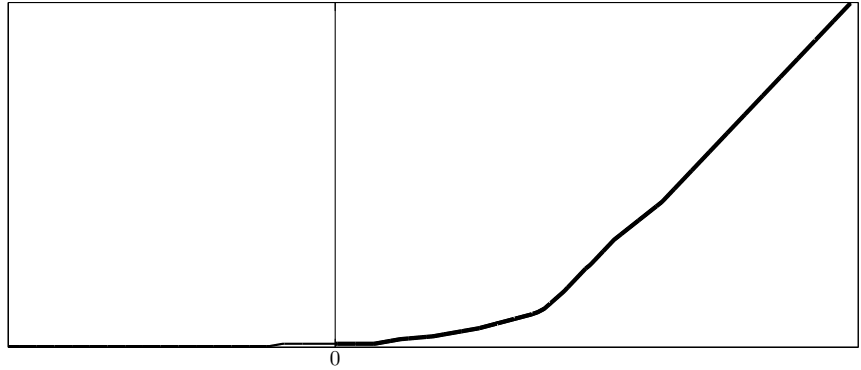}%
\vspace{-12pt}
\caption{Optimal cost-to-go $J^{\circ}_{2,0,1} (t_{2})$ in state $[2 \; 0 \; 1 \; t_{2}]^{T}$.}
\label{fig:esS3_J_2_0_1}
\end{figure}

The function $J^{\circ}_{2,0,1} (t_{2})$ is specified by the initial value 0.5, by the set \{ --5, --4, 3, 5, 7.5, 11, 15, 15.5, 16, 17.5, 19.1$\overline{6}$, 19.5, 21.$\overline{3}$, 25 \} of abscissae $\gamma_{i}$, $i = 1, \ldots, 14$, at which the slope changes, and by the set \{ 1, 0, 1, 0.5, 1, 1.5, 2, 3, 5, 6, 4.5, 6, 4.5, 6 \} of slopes $\mu_{i}$, $i = 1, \ldots, 14$, in the various intervals.

Since $J^{\circ}_{2,0,1} (t_{2} \mid \delta_{1} = 1)$ is the minimum in $(-\infty, -4)$, in $[3,5)$, and in $[21.\overline{3},+\infty)$, and $J^{\circ}_{2,0,1} (t_{2} \mid \delta_{2} = 1)$ is the minimum in $[-4,3)$ and in $[5,21.\overline{3})$, the optimal control strategies for this state are
\begin{equation*}
\delta_{1}^{\circ} (2,0,1, t_{2}) = \left\{ \begin{array}{ll}
1 &  t_{2} < -4\\
0 &  -4 \leq t_{2} < 3\\
1 &  3 \leq t_{2} < 5\\
0 & 5 \leq t_{2} < 21.\overline{3}\\
1 & t_{2} \geq 21.\overline{3}
\end{array} \right. \qquad \delta_{2}^{\circ} (2,0,1, t_{2}) = \left\{ \begin{array}{ll}
0 &  t_{2} < -4\\
1 &  -4 \leq t_{2} < 3\\
0 &  3 \leq t_{2} < 5\\
1 & 5 \leq t_{2} < 21.\overline{3}\\
0 & t_{2} \geq 21.\overline{3}
\end{array} \right.
\end{equation*}
\begin{equation*}
\tau^{\circ} (2,0,1, t_{2}) = \left\{ \begin{array}{ll}
8 &  t_{2} < -5\\
-t_{2} + 3 & -5 \leq t_{2} < -4\\
6 & -4 \leq t_{2} <3\\
-t_{2} + 11 & 3 \leq t_{2} < 5\\
6 &  5 \leq t_{2} < 11\\
-t_{2} + 17 & 11 \leq t_{2} < 13\\
4 & t_{2} \geq 13
\end{array} \right.
\end{equation*}
The optimal control strategy $\tau^{\circ} (2,0,1, t_{2})$ is illustrated in figure~\ref{fig:esS3_tau_2_0_1}.

\begin{figure}[h!]
\centering
\psfrag{f(x)}[Bl][Bl][.8][0]{$\tau^{\circ} (2,0,1, t_{2})$}
\psfrag{x}[bc][Bl][.8][0]{$t_{2}$}
\psfrag{0}[tc][Bl][.8][0]{$0$}
\includegraphics[scale=.2]{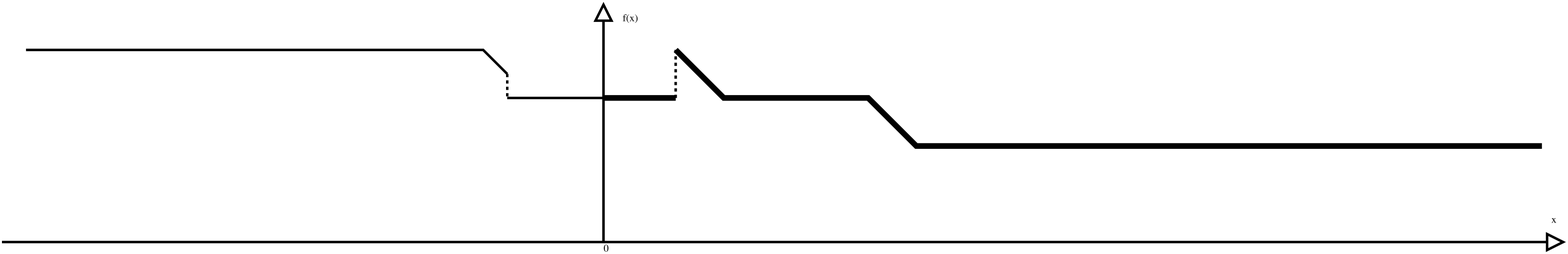}
\caption{Optimal control strategy $\tau^{\circ} (2,0,1, t_{2})$ in state $[ 2 \; 0 \; 1 \; t_{2}]^{T}$.}
\label{fig:esS3_tau_2_0_1}
\end{figure}

\newpage

%%
%% S2 - [0 1 2 t1]
%%

{\bf Stage $1$ -- State $\boldsymbol{[0 \; 1 \; 2 \; t_{1}]^{T}}$ ($S2$)}

In state $[0 \; 1 \; 2 \; t_{1}]^{T}$, the cost function to be minimized, with respect to the (continuos) decision variable $\tau$ and to the (binary) decision variables $\delta_{1}$ and $\delta_{2}$ is
\begin{equation*}
\begin{split}
&\delta_{1} \big[ \alpha_{1,1} \, \max \{ t_{1} + st_{2,1} + \tau - dd_{1,1} \, , \, 0 \} + \beta_{1} \, ( pt^{\mathrm{nom}}_{1} - \tau ) + sc_{2,1} + J^{\circ}_{1,1,1} (t_{2}) \big] +\\
&+ \delta_{2} \big[ \alpha_{2,2} \, \max \{ t_{1} + st_{2,2} + \tau - dd_{2,2} \, , \, 0 \} + \beta_{2} \, ( pt^{\mathrm{nom}}_{2} - \tau ) + sc_{2,2} + J^{\circ}_{0,2,2} (t_{2}) \big]
\end{split}
\end{equation*}

{\it Case i)} in which it is assumed $\delta_{1} = 1$ (and $\delta_{2} = 0$).

In this case, it is necessary to minimize, with respect to the (continuos) decision variable $\tau$ which corresponds to the processing time $pt_{1,1}$, the following function
\begin{equation*}
\alpha_{1,1} \, \max \{ t_{1} + st_{2,1} + \tau - dd_{1,1} \, , \, 0 \} + \beta_{1} \, ( pt^{\mathrm{nom}}_{1} - \tau ) + sc_{2,1} + J^{\circ}_{1,1,1} (t_{2})
\end{equation*}
that can be written as $f (pt_{1,1} + t_{1}) + g (pt_{1,1})$ being
\begin{equation*}
f (pt_{1,1} + t_{1}) = 0.75 \cdot \max \{ pt_{1,1} + t_{1} - 18.5 \, , \, 0 \} + 1 + J^{\circ}_{1,1,1} (pt_{1,1} + t_{1} + 0.5)
\end{equation*}
\begin{equation*}
g (pt_{1,1}) = \left\{ \begin{array}{ll}
8 - pt_{1,1} & pt_{1,1} \in [ 4 , 8 )\\
0 & pt_{1,1} \notin [ 4 , 8 )
\end{array} \right.
\end{equation*}
The function $pt^{\circ}_{1,1}(t_{1}) = \arg \min_{pt_{1,1}} \{ f (pt_{1,1} + t_{1}) + g (pt_{1,1}) \} $, with $4 \leq pt_{1,1} \leq 8$, is determined by applying lemma~\ref{lem:xopt}. It is (see figure~\ref{fig:esS3_tau_0_1_2_pt_1_1})
\begin{equation*}
pt^{\circ}_{1,1}(t_{1}) = \left\{ \begin{array}{ll}
x_{\mathrm{s}}(t_{1}) &  t_{1} < -14.5\\
x_{1}(t_{1}) & -14.5 \leq t_{1} < -5.5\\
x_{\mathrm{e}}(t_{1}) & t_{1} \geq -5.5
\end{array} \right. \qquad \text{with} \quad x_{\mathrm{s}}(t_{1}) = \left\{ \begin{array}{ll}
8 &  t_{1} < -15.5\\
-t_{1} - 7.5 & -15.5 \leq t_{1} < -14.5
\end{array} \right. \  , 
\end{equation*}
\begin{equation*}
\qquad x_{1}(t_{1}) = \left\{ \begin{array}{ll}
8 &  -14.5 \leq t_{1} < -7.5\\
-t_{1} + 0.5 & -7.5 \leq t_{1} < -5.5
\end{array} \right. \  , \  \text{and} \quad x_{\mathrm{e}}(t_{1}) = \left\{ \begin{array}{ll}
8 &  -5.5 \leq t_{1} < -3\\
-t_{1} + 5 & -3 \leq t_{1} < 1\\
4 & t_{1} \geq 1
\end{array} \right.
\end{equation*}

\begin{figure}[h!]
\centering
\psfrag{f(x)}[Bl][Bl][.8][0]{$pt^{\circ}_{1,1}(t_{1})$}
\psfrag{x}[bc][Bl][.8][0]{$t_{1}$}
\psfrag{0}[tc][Bl][.8][0]{$0$}
\includegraphics[scale=.2]{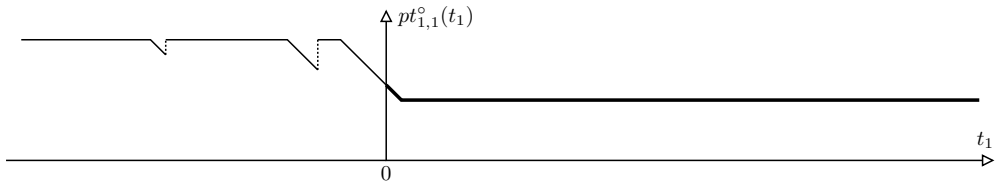}
\caption{Optimal processing time $pt^{\circ}_{1,1}(t_{1})$, under the assumption $\delta_{1} = 1$ in state $[ 0 \; 1 \; 2 \; t_{1}]^{T}$.}
\label{fig:esS3_tau_0_1_2_pt_1_1}
\end{figure}

The conditioned cost-to-go $J^{\circ}_{0,1,2} (t_{1} \mid \delta_{1}=1) = f ( pt^{\circ}_{1,1}(t_{1}) + t_{1} ) + g ( pt^{\circ}_{1,1}(t_{1}) )$, illustrated in figure~\ref{fig:esS3_J_0_1_2_min}, is provided by lemma~\ref{lem:h(t)}. It is specified by the initial value 1.5, by the set \{ --15.5, --14.5, --7.5, --5.5, --3, 8.5, 10.5, 11, 12.75, 14.5, 15.5, 16.5 \} of abscissae $\gamma_{i}$, $i = 1, \ldots, 12$, at which the slope changes, and by the set \{ 1, 0, 1, 0.5, 1, 1.5, 2, 2.5, 1.5, 3.25, 3.75, 5.25 \} of slopes $\mu_{i}$, $i = 1, \ldots, 12$, in the various intervals.

{\it Case ii)} in which it is assumed $\delta_{2} = 1$ (and $\delta_{1} = 0$).

In this case, it is necessary to minimize, with respect to the (continuos) decision variable $\tau$ which corresponds to the processing time $pt_{2,2}$, the following function
\begin{equation*}
\alpha_{2,2} \, \max \{ t_{1} + st_{2,2} + \tau - dd_{2,2} \, , \, 0 \} + \beta_{2} \, ( pt^{\mathrm{nom}}_{2} - \tau ) + sc_{2,2} + J^{\circ}_{0,2,2} (t_{2})
\end{equation*}
that can be written as $f (pt_{2,2} + t_{1}) + g (pt_{2,2})$ being
\begin{equation*}
f (pt_{2,2} + t_{1}) = \max \{ pt_{2,2} + t_{1} - 24 \, , \, 0 \} + J^{\circ}_{0,2,2} (pt_{2,2} + t_{1})
\end{equation*}
\begin{equation*}
g (pt_{2,2}) = \left\{ \begin{array}{ll}
1.5 \cdot (6 - pt_{2,2}) & pt_{2,2} \in [ 4 , 6 )\\
0 & pt_{2,2} \notin [ 4 , 6 )
\end{array} \right.
\end{equation*}
The function $pt^{\circ}_{2,2}(t_{1}) = \arg \min_{pt_{2,2}} \{ f (pt_{2,2} + t_{1}) + g (pt_{2,2}) \} $, with $4 \leq pt_{2,2} \leq 6$, is determined by applying lemma~\ref{lem:xopt}. It is (see figure~\ref{fig:esS3_tau_0_1_2_pt_2_2})
\begin{equation*}
pt^{\circ}_{2,2}(t_{1}) = x_{\mathrm{e}}(t_{1}) \qquad \text{with} \quad x_{\mathrm{e}}(t_{1}) = \left\{ \begin{array}{ll}
6 &  t_{1} < 8.5\\
-t_{1} + 14.5 & 8.5 \leq t_{1} < 10.5\\
4 & t_{1} \geq 10.5
\end{array} \right.
\end{equation*}

\begin{figure}[h!]
\centering
\psfrag{f(x)}[Bl][Bl][.8][0]{$pt^{\circ}_{2,2}(t_{1})$}
\psfrag{x}[bc][Bl][.8][0]{$t_{1}$}
\psfrag{0}[tc][Bl][.8][0]{$0$}
\includegraphics[scale=.2]{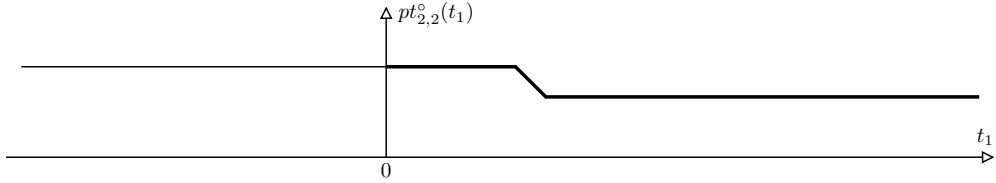}
\caption{Optimal processing time $pt^{\circ}_{2,2}(t_{1})$, under the assumption $\delta_{2} = 1$ in state $[ 0 \; 1 \; 2 \; t_{1}]^{T}$.}
\label{fig:esS3_tau_0_1_2_pt_2_2}
\end{figure}

The conditioned cost-to-go $J^{\circ}_{0,1,2} (t_{1} \mid \delta_{2}=1) = f ( pt^{\circ}_{2,2}(t_{1}) + t_{1} ) + g ( pt^{\circ}_{2,2}(t_{1}) )$, illustrated in figure~\ref{fig:esS3_J_0_1_2_min}, is provided by lemma~\ref{lem:h(t)}. It is specified by the initial value 1, by the set \{ --7.5, --6, --5, --1.5, 8.5, 10.5, 11.5, 12.5, 14.5, 20, 25.25, 30 \} of abscissae $\gamma_{i}$, $i = 1, \ldots, 12$, at which the slope changes, and by the set \{ 1, 0, 0.5, 1, 1.5, 1.75, 2.25, 3.25, 4.25, 5.25, 4.25, 5.25 \} of slopes $\mu_{i}$, $i = 1, \ldots, 12$, in the various intervals.

\begin{figure}[h!]
\centering
\psfrag{J1}[br][Bc][.8][0]{$J^{\circ}_{0,1,2} (t_{1} \mid \delta_{1} = 1)$}
\psfrag{J2}[tl][Bc][.8][0]{$J^{\circ}_{0,1,2} (t_{1} \mid \delta_{2} = 1)$}
\psfrag{0}[cc][tc][.7][0]{$0$}\includegraphics[scale=.6]{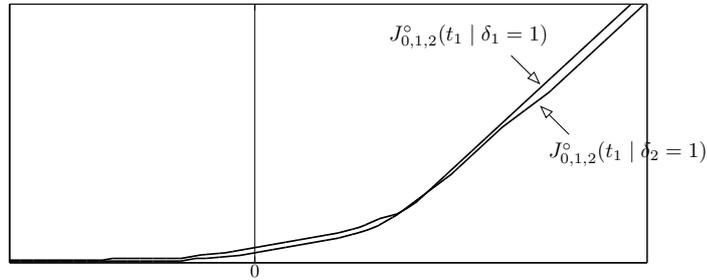}%
\vspace{-12pt}
\caption{Conditioned costs-to-go $J^{\circ}_{0,1,2} (t_{1} \mid \delta_{1} = 1)$ and $J^{\circ}_{0,1,2} (t_{1} \mid \delta_{2} = 1)$ in state $[0 \; 1 \; 2 \; t_{1}]^{T}$.}
\label{fig:esS3_J_0_1_2_min}
\end{figure}

In order to find the optimal cost-to-go $J^{\circ}_{0,1,2} (t_{1})$, it is necessary to carry out the following minimization
\begin{equation*}
J^{\circ}_{0,1,2} (t_{1}) = \min \big\{ J^{\circ}_{0,1,2} (t_{1} \mid \delta_{1} = 1) \, , \, J^{\circ}_{0,1,2} (t_{1} \mid \delta_{2} = 1) \big\}
\end{equation*}
which provides, in accordance with lemma~\ref{lem:min}, the continuous, nondecreasing, piecewise linear function illustrated in figure~\ref{fig:esS3_J_0_1_2}.

\begin{figure}[h!]
\centering
\psfrag{0}[cc][tc][.8][0]{$0$}
\includegraphics[scale=.8]{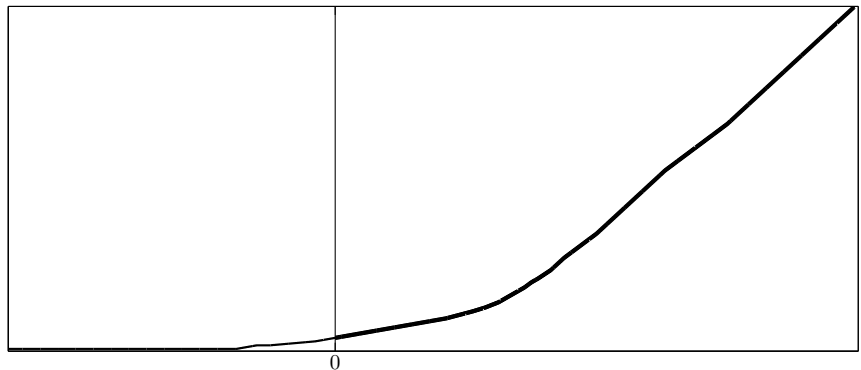}%
\vspace{-12pt}
\caption{Optimal cost-to-go $J^{\circ}_{0,1,2} (t_{1})$ in state $[0 \; 1 \; 2 \; t_{1}]^{T}$.}
\label{fig:esS3_J_0_1_2}
\end{figure}

The function $J^{\circ}_{0,1,2} (t_{1})$ is specified by the initial value 1, by the set \{ --7.5, --6, --5, --1.5, 8.5, 10.5, 11.5, 12.5, 14.5, 15, 15.5, 16.5, 17.5, 20, 25.25, 30 \} of abscissae $\gamma_{i}$, $i = 1, \ldots, 16$, at which the slope changes, and by the set \{ 1, 0, 0.5, 1, 1.5, 1.75, 2.25, 3.25, 4.25, 3.25, 3.75, 5.25, 4.25, 5.25, 4.25, 5.25 \} of slopes $\mu_{i}$, $i = 1, \ldots, 16$, in the various intervals.

Since $J^{\circ}_{0,1,2} (t_{1} \mid \delta_{1} = 1)$ is the minimum in $[15,17.5)$, and $J^{\circ}_{0,1,2} (t_{1} \mid \delta_{2} = 1)$ is the minimum in $(-\infty, 15)$ and in $[17.5,+\infty)$, the optimal control strategies for this state are
\begin{equation*}
\delta_{1}^{\circ} (0,1,2, t_{1}) = \left\{ \begin{array}{ll}
0 &  t_{1} < 15\\
1 &  15 \leq t_{1} < 17.5\\
0 & t_{1} \geq 17.5
\end{array} \right. \qquad \delta_{2}^{\circ} (0,1,2, t_{1}) = \left\{ \begin{array}{ll}
1 &  t_{1} < 15\\
0 &  15 \leq t_{1} < 17.5\\
1 & t_{1} \geq 17.5
\end{array} \right.
\end{equation*}
\begin{equation*}
\tau^{\circ} (0,1,2, t_{1}) = \left\{ \begin{array}{ll}
6 &  t_{1} < 8.5\\
-t_{1} + 14.5 & 8.5 \leq t_{1} < 10.5\\
4 & t_{1} \geq 10.5
\end{array} \right.
\end{equation*}
The optimal control strategy $\tau^{\circ} (0,1,2, t_{1})$ is illustrated in figure~\ref{fig:esS3_tau_0_1_2}.

\begin{figure}[h!]
\centering
\psfrag{f(x)}[Bl][Bl][.8][0]{$\tau^{\circ} (0,1,2, t_{1})$}
\psfrag{x}[bc][Bl][.8][0]{$t_{1}$}
\psfrag{0}[tc][Bl][.8][0]{$0$}
\includegraphics[scale=.2]{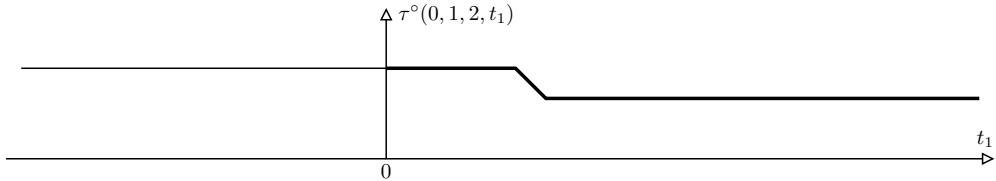}
\caption{Optimal control strategy $\tau^{\circ} (0,1,2, t_{1})$ in state $[ 0 \; 1 \; 2 \; t_{1}]^{T}$.}
\label{fig:esS3_tau_0_1_2}
\end{figure}

%%
%% S1 - [1 0 1 t1]
%%

{\bf Stage $1$ -- State $\boldsymbol{[1 \; 0 \; 1 \; t_{1}]^{T}}$ ($S1$)}

In state $[1 \; 0 \; 1 \; t_{1}]^{T}$, the cost function to be minimized, with respect to the (continuos) decision variable $\tau$ and to the (binary) decision variables $\delta_{1}$ and $\delta_{2}$ is
\begin{equation*}
\begin{split}
&\delta_{1} \big[ \alpha_{1,2} \, \max \{ t_{1} + st_{1,1} + \tau - dd_{1,2} \, , \, 0 \} + \beta_{1} \, ( pt^{\mathrm{nom}}_{1} - \tau ) + sc_{1,1} + J^{\circ}_{2,0,1} (t_{2}) \big] +\\
&+ \delta_{2} \big[ \alpha_{2,1} \, \max \{ t_{1} + st_{1,2} + \tau - dd_{2,1} \, , \, 0 \} + \beta_{2} \, ( pt^{\mathrm{nom}}_{2} - \tau ) + sc_{1,2} + J^{\circ}_{1,1,2} (t_{2}) \big]
\end{split}
\end{equation*}

{\it Case i)} in which it is assumed $\delta_{1} = 1$ (and $\delta_{2} = 0$).

In this case, it is necessary to minimize, with respect to the (continuos) decision variable $\tau$ which corresponds to the processing time $pt_{1,2}$, the following function
\begin{equation*}
\alpha_{1,2} \, \max \{ t_{1} + st_{1,1} + \tau - dd_{1,2} \, , \, 0 \} + \beta_{1} \, ( pt^{\mathrm{nom}}_{1} - \tau ) + sc_{1,1} + J^{\circ}_{2,0,1} (t_{2})
\end{equation*}
that can be written as $f (pt_{1,2} + t_{1}) + g (pt_{1,2})$ being
\begin{equation*}
f (pt_{1,2} + t_{1}) = 0.5 \cdot \max \{ pt_{1,2} + t_{1} - 24 \, , \, 0 \} + J^{\circ}_{2,0,1} (pt_{1,2} + t_{1})
\end{equation*}
\begin{equation*}
g (pt_{1,2}) = \left\{ \begin{array}{ll}
8 - pt_{1,2} & pt_{1,2} \in [ 4 , 8 )\\
0 & pt_{1,2} \notin [ 4 , 8 )
\end{array} \right.
\end{equation*}
The function $pt^{\circ}_{1,2}(t_{1}) = \arg \min_{pt_{1,2}} \{ f (pt_{1,2} + t_{1}) + g (pt_{1,2}) \} $, with $4 \leq pt_{1,2} \leq 8$, is determined by applying lemma~\ref{lem:xopt}. It is (see figure~\ref{fig:esS3_tau_1_0_1_pt_1_2})
\begin{equation*}
pt^{\circ}_{1,2}(t_{1}) = \left\{ \begin{array}{ll}
x_{\mathrm{s}}(t_{1}) &  t_{1} < -12\\
x_{1}(t_{1}) & -12 \leq t_{1} < -3\\
x_{\mathrm{e}}(t_{1}) & t_{1} \geq -3
\end{array} \right. \qquad \text{with} \quad x_{\mathrm{s}}(t_{1}) = \left\{ \begin{array}{ll}
8 &  t_{1} < -13\\
-t_{1} - 5 & -13 \leq t_{1} < -12
\end{array} \right. \  , 
\end{equation*}
\begin{equation*}
\qquad x_{1}(t_{1}) = \left\{ \begin{array}{ll}
8 &  -12 \leq t_{1} < -5\\
-t_{1} + 3 & -5 \leq t_{1} < -3
\end{array} \right. \  , \  \text{and} \quad x_{\mathrm{e}}(t_{1}) = \left\{ \begin{array}{ll}
8 &  -3 \leq t_{1} < -0.5\\
-t_{1} + 7.5 & -0.5 \leq t_{1} < 3.5\\
4 & t_{1} \geq 3.5
\end{array} \right.
\end{equation*}

The conditioned cost-to-go $J^{\circ}_{1,0,1} (t_{1} \mid \delta_{1}=1) = f ( pt^{\circ}_{1,2}(t_{1}) + t_{1} ) + g ( pt^{\circ}_{1,2}(t_{1}) )$, illustrated in figure~\ref{fig:esS3_J_1_0_1_min}, is provided by lemma~\ref{lem:h(t)}. It is specified by the initial value 0.5, by the set \{ --13, --12, --5, --3, --0.5, 7, 11, 11.5, 12, 13.5, 15.1$\overline{6}$, 15.5, 17.$\overline{3}$, 20, 21 \} of abscissae $\gamma_{i}$, $i = 1, \ldots, 15$, at which the slope changes, and by the set \{ 1, 0, 1, 0.5, 1, 1.5, 2, 3, 5, 6, 4.5, 6, 4.5, 5, 6.5 \} of slopes $\mu_{i}$, $i = 1, \ldots, 15$, in the various intervals.

\newpage

\begin{figure}[h!]
\centering
\psfrag{f(x)}[Bl][Bl][.8][0]{$pt^{\circ}_{1,2}(t_{1})$}
\psfrag{x}[bc][Bl][.8][0]{$t_{1}$}
\psfrag{0}[tc][Bl][.8][0]{$0$}
\includegraphics[scale=.2]{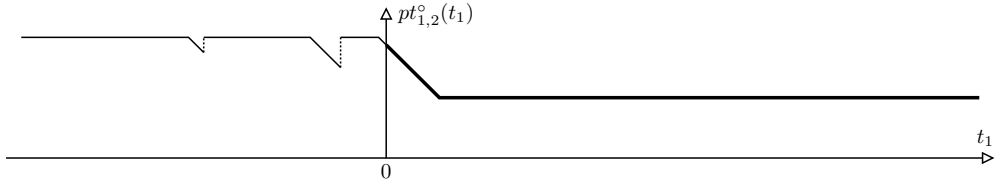}
\caption{Optimal processing time $pt^{\circ}_{1,2}(t_{1})$, under the assumption $\delta_{1} = 1$ in state $[ 1 \; 0 \; 1 \; t_{1}]^{T}$.}
\label{fig:esS3_tau_1_0_1_pt_1_2}
\end{figure}

{\it Case ii)} in which it is assumed $\delta_{2} = 1$ (and $\delta_{1} = 0$).

In this case, it is necessary to minimize, with respect to the (continuos) decision variable $\tau$ which corresponds to the processing time $pt_{2,1}$, the following function
\begin{equation*}
\alpha_{2,1} \, \max \{ t_{1} + st_{1,2} + \tau - dd_{2,1} \, , \, 0 \} + \beta_{2} \, ( pt^{\mathrm{nom}}_{2} - \tau ) + sc_{1,2} + J^{\circ}_{1,1,2} (t_{2})
\end{equation*}
that can be written as $f (pt_{2,1} + t_{1}) + g (pt_{2,1})$ being
\begin{equation*}
f (pt_{2,1} + t_{1}) = 2 \cdot \max \{ pt_{2,1} + t_{1} - 20 \, , \, 0 \} + 0.5 + J^{\circ}_{1,1,2} (pt_{2,1} + t_{1} + 1)
\end{equation*}
\begin{equation*}
g (pt_{2,1}) = \left\{ \begin{array}{ll}
1.5 \cdot (6 - pt_{2,1}) & pt_{2,1} \in [ 4 , 6 )\\
0 & pt_{2,1} \notin [ 4 , 6 )
\end{array} \right.
\end{equation*}
The function $pt^{\circ}_{2,1}(t_{1}) = \arg \min_{pt_{2,1}} \{ f (pt_{2,1} + t_{1}) + g (pt_{2,1}) \} $, with $4 \leq pt_{2,1} \leq 6$, is determined by applying lemma~\ref{lem:xopt}. It is (see figure~\ref{fig:esS3_tau_1_0_1_pt_2_1})
\begin{equation*}
pt^{\circ}_{2,1}(t_{1}) = x_{\mathrm{e}}(t_{1}) \qquad \text{with} \quad x_{\mathrm{e}}(t_{1}) = \left\{ \begin{array}{ll}
6 &  t_{1} < 6.5\\
-t_{1} + 12.5 & 6.5 \leq t_{1} < 8.5\\
4 & t_{1} \geq 8.5
\end{array} \right.
\end{equation*}

\begin{figure}[h!]
\centering
\psfrag{f(x)}[Bl][Bl][.8][0]{$pt^{\circ}_{2,1}(t_{1})$}
\psfrag{x}[bc][Bl][.8][0]{$t_{1}$}
\psfrag{0}[tc][Bl][.8][0]{$0$}
\includegraphics[scale=.2]{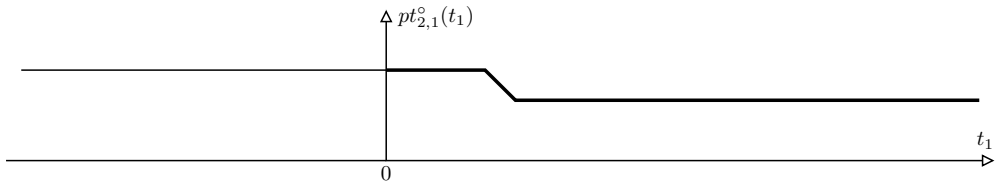}
\caption{Optimal processing time $pt^{\circ}_{2,1}(t_{1})$, under the assumption $\delta_{2} = 1$ in state $[ 1 \; 0 \; 1 \; t_{1}]^{T}$.}
\label{fig:esS3_tau_1_0_1_pt_2_1}
\end{figure}

The conditioned cost-to-go $J^{\circ}_{1,0,1} (t_{1} \mid \delta_{2}=1) = f ( pt^{\circ}_{2,1}(t_{1}) + t_{1} ) + g ( pt^{\circ}_{2,1}(t_{1}) )$, illustrated in figure~\ref{fig:esS3_J_1_0_1_min}, is provided by lemma~\ref{lem:h(t)}. It is specified by the initial value 1.5, by the set \{ --6.5, --5, --4, --0.5, 6.5, 11.5, 13.5, 15, 16, 21.25, 25 \} of abscissae $\gamma_{i}$, $i = 1, \ldots, 11$, at which the slope changes, and by the set \{ 1, 0, 0.5, 1, 1.5, 2.5, 3.5, 4.5, 6.5, 5.5, 6.5 \} of slopes $\mu_{i}$, $i = 1, \ldots, 11$, in the various intervals.

\begin{figure}[h!]
\centering
\psfrag{J1}[br][Bc][.8][0]{$J^{\circ}_{1,0,1} (t_{1} \mid \delta_{1} = 1)$}
\psfrag{J2}[tl][Bc][.8][0]{$J^{\circ}_{1,0,1} (t_{1} \mid \delta_{2} = 1)$}
\psfrag{0}[cc][tc][.7][0]{$0$}\includegraphics[scale=.6]{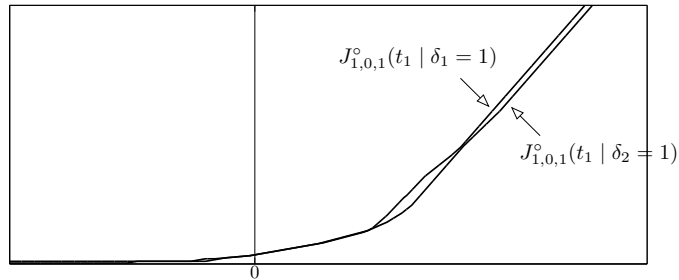}%
\vspace{-12pt}
\caption{Conditioned costs-to-go $J^{\circ}_{1,0,1} (t_{1} \mid \delta_{1} = 1)$ and $J^{\circ}_{1,0,1} (t_{1} \mid \delta_{2} = 1)$ in state $[1 \; 0 \; 1 \; t_{1}]^{T}$.}
\label{fig:esS3_J_1_0_1_min}
\end{figure}

In order to find the optimal cost-to-go $J^{\circ}_{1,0,1} (t_{1})$, it is necessary to carry out the following minimization
\begin{equation*}
J^{\circ}_{1,0,1} (t_{1}) = \min \big\{ J^{\circ}_{1,0,1} (t_{1} \mid \delta_{1} = 1) \, , \, J^{\circ}_{1,0,1} (t_{1} \mid \delta_{2} = 1) \big\}
\end{equation*}
which provides, in accordance with lemma~\ref{lem:min}, the continuous, nondecreasing, piecewise linear function illustrated in figure~\ref{fig:esS3_J_1_0_1}.

\begin{figure}[h!]
\centering
\psfrag{0}[cc][tc][.8][0]{$0$}
\includegraphics[scale=.8]{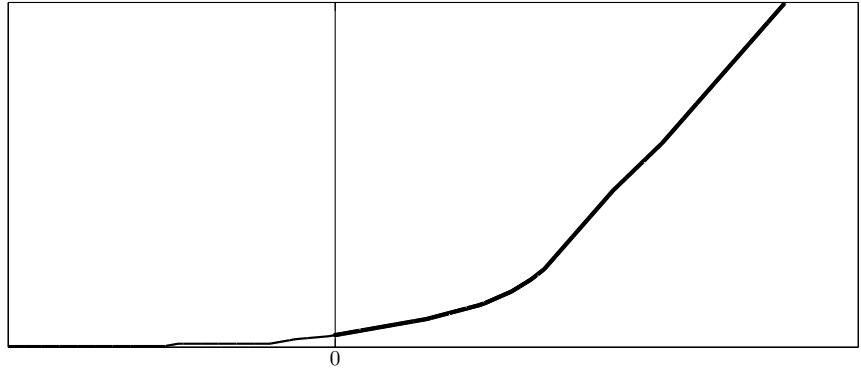}%
\vspace{-12pt}
\caption{Optimal cost-to-go $J^{\circ}_{1,0,1} (t_{1})$ in state $[1 \; 0 \; 1 \; t_{1}]^{T}$.}
\label{fig:esS3_J_1_0_1}
\end{figure}

The function $J^{\circ}_{1,0,1} (t_{1})$ is specified by the initial value 0.5, by the set \{ --13, --12, --5, --3, --0.5, 7, 11, 11.5, 13.5, 15, 16, 21.25, 25 \} of abscissae $\gamma_{i}$, $i = 1, \ldots, 13$, at which the slope changes, and by the set \{ 1, 0, 1, 0.5, 1, 1.5, 2, 2.5, 3.5, 4.5, 6.5, 5.5, 6.5 \} of slopes $\mu_{i}$, $i = 1, \ldots, 13$, in the various intervals.

Since $J^{\circ}_{1,0,1} (t_{1} \mid \delta_{1} = 1)$ is the minimum in $(-\infty, 12)$ and in $[-5,11.5)$, and $J^{\circ}_{1,0,1} (t_{1} \mid \delta_{2} = 1)$ is the minimum in $[-12, -5)$ and in $[11.5,+\infty)$, the optimal control strategies for this state are
\begin{equation*}
\delta_{1}^{\circ} (1,0,1, t_{1}) = \left\{ \begin{array}{ll}
1 &  t_{1} < -12\\
0 &  -12 \leq t_{1} < -5\\
1 &  -5 \leq t_{1} < 11.5\\
0 & t_{1} \geq 11.5
\end{array} \right. \qquad \delta_{2}^{\circ} (1,0,1, t_{1}) = \left\{ \begin{array}{ll}
0 &  t_{1} < -12\\
1 &  -12 \leq t_{1} < -5\\
0 &  -5 \leq t_{1} < 11.5\\
1 & t_{1} \geq 11.5
\end{array} \right.
\end{equation*}
\begin{equation*}
\tau^{\circ} (1,0,1, t_{1}) = \left\{ \begin{array}{ll}
8 &  t_{1} < -13\\
-t_{1} - 5 & -13 \leq t_{1} < -12\\
6 & -12 \leq t_{1} < -5\\
-t_{1} + 3 & -5 \leq t_{1} < -3\\
8 &  -3 \leq t_{1} < -0.5\\
-t_{1} + 7.5 & -0.5 \leq t_{1} < 3.5\\
4 & t_{1} \geq 3.5
\end{array} \right.
\end{equation*}
The optimal control strategy $\tau^{\circ} (1,0,1, t_{1})$ is illustrated in figure~\ref{fig:esS3_tau_1_0_1}.

\begin{figure}[h!]
\centering
\psfrag{f(x)}[Bl][Bl][.8][0]{$\tau^{\circ} (1,0,1, t_{1})$}
\psfrag{x}[bc][Bl][.8][0]{$t_{1}$}
\psfrag{0}[tc][Bl][.8][0]{$0$}
\includegraphics[scale=.2]{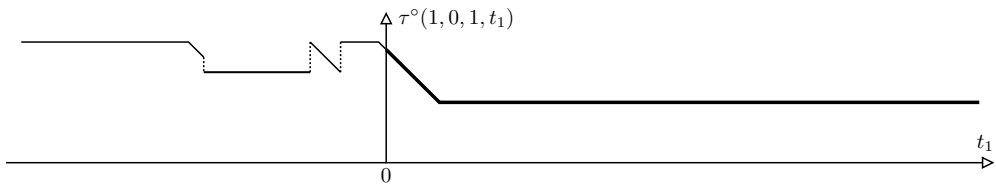}
\caption{Optimal control strategy $\tau^{\circ} (1,0,1, t_{1})$ in state $[ 1 \; 0 \; 1 \; t_{1}]^{T}$.}
\label{fig:esS3_tau_1_0_1}
\end{figure}

%%
%% S0 - [0 0 0 t0]
%%

{\bf Stage $0$ -- State $\boldsymbol{[0 \; 0 \; 0 \; t_{0}]^{T}}$ ($S0$)}

In the initial state $[0 \; 0 \; 0 \; t_{0}]^{T}$, the cost function to be minimized, with respect to the (continuos) decision variable $\tau$ and to the (binary) decision variables $\delta_{1}$ and $\delta_{2}$ is
\begin{equation*}
\begin{split}
&\delta_{1} \big[ \alpha_{1,1} \, \max \{ t_{0} + st_{0,1} + \tau - dd_{1,1} \, , \, 0 \} + \beta_{1} \, ( pt^{\mathrm{nom}}_{1} - \tau ) + sc_{0,1} + J^{\circ}_{1,0,1} (t_{1}) \big] +\\
&+ \delta_{2} \big[ \alpha_{2,1} \, \max \{ t_{0} + st_{0,2} + \tau - dd_{2,1} \, , \, 0 \} + \beta_{2} \, ( pt^{\mathrm{nom}}_{2} - \tau ) + sc_{0,2} + J^{\circ}_{0,1,2} (t_{1}) \big]
\end{split}
\end{equation*}

{\it Case i)} in which it is assumed $\delta_{1} = 1$ (and $\delta_{2} = 0$).

In this case, it is necessary to minimize, with respect to the (continuos) decision variable $\tau$ which corresponds to the processing time $pt_{1,1}$, the following function
\begin{equation*}
\alpha_{1,1} \, \max \{ t_{0} + st_{0,1} + \tau - dd_{1,1} \, , \, 0 \} + \beta_{1} \, ( pt^{\mathrm{nom}}_{1} - \tau ) + sc_{0,1} + J^{\circ}_{1,0,1} (t_{1})
\end{equation*}
that can be written as $f (pt_{1,1} + t_{0}) + g (pt_{1,1})$ being
\begin{equation*}
f (pt_{1,1} + t_{0}) = 0.75 \cdot \max \{ pt_{1,1} + t_{0} - 19 \, , \, 0 \} + J^{\circ}_{1,0,1} (pt_{1,1} + t_{0})
\end{equation*}
\begin{equation*}
g (pt_{1,1}) = \left\{ \begin{array}{ll}
8 - pt_{1,1} & pt_{1,1} \in [ 4 , 8 )\\
0 & pt_{1,1} \notin [ 4 , 8 )
\end{array} \right.
\end{equation*}
The function $pt^{\circ}_{1,1}(t_{0}) = \arg \min_{pt_{1,1}} \{ f (pt_{1,1} + t_{0}) + g (pt_{1,1}) \} $, with $4 \leq pt_{1,1} \leq 8$, is determined by applying lemma~\ref{lem:xopt}. It is (see figure~\ref{fig:esS3_tau_0_0_0_pt_1_1})
\begin{equation*}
pt^{\circ}_{1,1}(t_{0}) = \left\{ \begin{array}{ll}
x_{\mathrm{s}}(t_{0}) &  t_{0} < -20\\
x_{1}(t_{0}) & -20 \leq t_{0} < -11\\
x_{\mathrm{e}}(t_{0}) & t_{0} \geq -11
\end{array} \right. \qquad \text{with} \quad x_{\mathrm{s}}(t_{0}) = \left\{ \begin{array}{ll}
8 &  t_{0} < -21\\
-t_{0} - 13 & -21 \leq t_{0} < -20
\end{array} \right. \  , 
\end{equation*}
\begin{equation*}
\qquad x_{1}(t_{0}) = \left\{ \begin{array}{ll}
8 &  -20 \leq t_{0} < -13\\
-t_{0} - 5 & -13 \leq t_{0} < -11
\end{array} \right. \  , \  \text{and} \quad x_{\mathrm{e}}(t_{0}) = \left\{ \begin{array}{ll}
8 &  -11 \leq t_{0} < -8.5\\
-t_{0} -0.5 & -8.5 \leq t_{0} < -4.5\\
4 & t_{0} \geq -4.5
\end{array} \right.
\end{equation*}

\begin{figure}[h!]
\centering
\psfrag{f(x)}[Bl][Bl][.8][0]{$pt^{\circ}_{1,1}(t_{0})$}
\psfrag{x}[bc][Bl][.8][0]{$t_{0}$}
\psfrag{0}[tc][Bl][.8][0]{$0$}
\includegraphics[scale=.2]{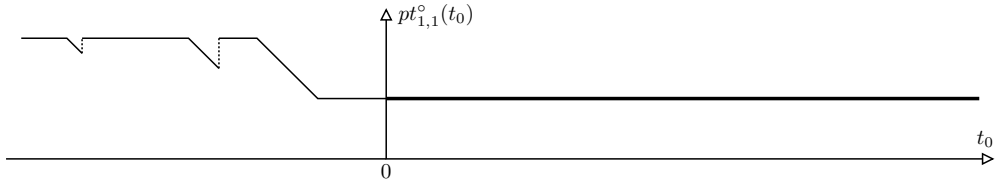}
\caption{Optimal processing time $pt^{\circ}_{1,1}(t_{0})$, under the assumption $\delta_{1} = 1$ in the initial state $[ 0 \; 0 \; 0 \; t_{0}]^{T}$.}
\label{fig:esS3_tau_0_0_0_pt_1_1}
\end{figure}

The conditioned cost-to-go $J^{\circ}_{0,0,0} (t_{0} \mid \delta_{1} = 1) = f ( pt^{\circ}_{1,1}(t_{0}) + t_{0} ) + g ( pt^{\circ}_{1,1}(t_{0}) )$, illustrated in figure~\ref{fig:esS3_J_0_0_3_min}, is provided by lemma~\ref{lem:h(t)}. It is specified by the initial value 0.5, by the set \{ --21, --20, --13, --11, --8.5, 3, 7, 7.5, 9.5, 11, 12, 15, 17.25, 21 \} of abscissae $\gamma_{i}$, $i = 1, \ldots, 14$, at which the slope changes, and by the set \{ 1, 0, 1, 0.5, 1, 1.5, 2, 2.5, 3.5, 4.5, 6.5, 7.25, 6.25, 7.25 \} of slopes $\mu_{i}$, $i = 1, \ldots, 14$, in the various intervals.

{\it Case ii)} in which it is assumed $\delta_{2} = 1$ (and $\delta_{1} = 0$).

In this case, it is necessary to minimize, with respect to the (continuos) decision variable $\tau$ which corresponds to the processing time $pt_{2,1}$, the following function
\begin{equation*}
\alpha_{2,1} \, \max \{ t_{0} + st_{0,2} + \tau - dd_{2,1} \, , \, 0 \} + \beta_{2} \, ( pt^{\mathrm{nom}}_{2} - \tau ) + sc_{0,2} + J^{\circ}_{0,1,2} (t_{1})
\end{equation*}
that can be written as $f (pt_{2,1} + t_{0}) + g (pt_{2,1})$ being
\begin{equation*}
f (pt_{2,1} + t_{0}) = 2 \cdot \max \{ pt_{2,1} + t_{0} - 21 \, , \, 0 \} + J^{\circ}_{0,1,2} (pt_{2,1} + t_{0})
\end{equation*}
\begin{equation*}
g (pt_{2,1}) = \left\{ \begin{array}{ll}
1.5 \cdot (6 - pt_{2,1}) & pt_{2,1} \in [ 4 , 6 )\\
0 & pt_{2,1} \notin [ 4 , 6 )
\end{array} \right.
\end{equation*}
The function $pt^{\circ}_{2,1}(t_{0}) = \arg \min_{pt_{2,1}} \{ f (pt_{2,1} + t_{0}) + g (pt_{2,1}) \} $, with $4 \leq pt_{2,1} \leq 6$, is determined by applying lemma~\ref{lem:xopt}. It is (see figure~\ref{fig:esS3_tau_0_0_0_pt_2_1})
\begin{equation*}
pt^{\circ}_{2,1}(t_{0}) = x_{\mathrm{e}}(t_{0}) \qquad \text{with} \quad x_{\mathrm{e}}(t_{0}) = \left\{ \begin{array}{ll}
6 & t_{0} < 2.5\\
-t_{0} + 8.5 & 2.5 \leq t_{0} < 4.5\\
4 & t_{0} \geq 4.5
\end{array} \right.
\end{equation*}

\begin{figure}[h!]
\centering
\psfrag{f(x)}[Bl][Bl][.8][0]{$pt^{\circ}_{2,1}(t_{0})$}
\psfrag{x}[bc][Bl][.8][0]{$t_{0}$}
\psfrag{0}[tc][Bl][.8][0]{$0$}
\includegraphics[scale=.2]{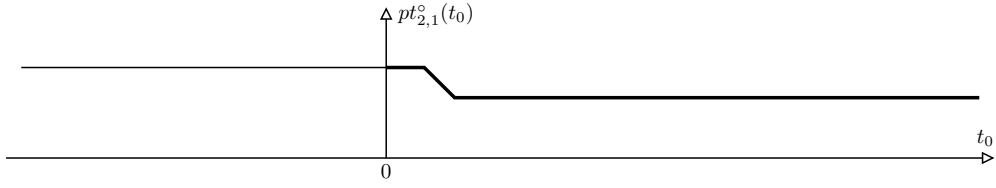}
\caption{Optimal processing time $pt^{\circ}_{2,1}(t_{0})$, under the assumption $\delta_{2} = 1$ in the initial state $[ 0 \; 0 \; 0 \; t_{0}]^{T}$.}
\label{fig:esS3_tau_0_0_0_pt_2_1}
\end{figure}

The conditioned cost-to-go $J^{\circ}_{0,0,0} (t_{0} \mid \delta_{2} = 1) = f ( pt^{\circ}_{2,1}(t_{0}) + t_{0} ) + g ( pt^{\circ}_{2,1}(t_{0}) )$, illustrated in figure~\ref{fig:esS3_J_0_0_3_min}, is provided by lemma~\ref{lem:h(t)}. It is specified by the initial value 1, by the set \{ --13.5, --12, --11, --7.5, 2.5, 6.5, 7.5, 8.5, 10.5, 11, 11.5, 12.5, 13.5, 16, 17, 21.25, 26 \} of abscissae $\gamma_{i}$, $i = 1, \ldots, 17$, at which the slope changes, and by the set \{ 1, 0, 0.5, 1, 1.5, 1.75, 2.25, 3.25, 4.25, 3.25, 3.75, 5.25, 4.25, 5.25, 7.25, 6.25, 7.25 \} of slopes $\mu_{i}$, $i = 1, \ldots, 17$, in the various intervals.

\begin{figure}[h!]
\centering
\psfrag{J1}[br][Bc][.8][0]{$J^{\circ}_{0,0,0} (t_{0} \mid \delta_{1} = 1)$}
\psfrag{J2}[tl][Bc][.8][0]{$J^{\circ}_{0,0,0} (t_{0} \mid \delta_{2} = 1)$}
\psfrag{0}[cc][tc][.7][0]{$0$}
\includegraphics[scale=.6]{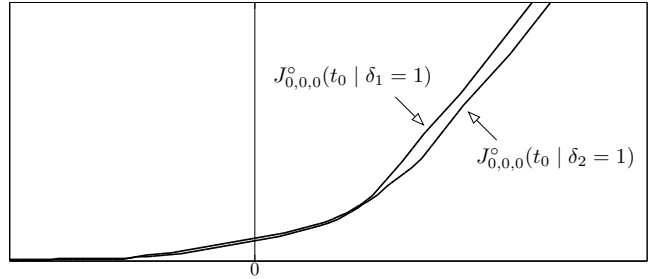}%
\vspace{-12pt}
\caption{Conditioned costs-to-go $J^{\circ}_{0,0,0} (t_{0} \mid \delta_{1} = 1)$ and $J^{\circ}_{0,0,0} (t_{0} \mid \delta_{2} = 1)$ in the initial state $[0 \; 0 \; 0 \; t_{0}]^{T}$.}
\label{fig:esS3_J_0_0_3_min}
\end{figure}

In order to find the optimal cost-to-go $J^{\circ}_{0,0,0} (t_{0})$, it is necessary to carry out the following minimization
\begin{equation*}
J^{\circ}_{0,0,0} (t_{0}) = \min \big\{ J^{\circ}_{0,0,0} (t_{0} \mid \delta_{1} = 1) \, , \, J^{\circ}_{0,0,0} (t_{0} \mid \delta_{2} = 1) \big\}
\end{equation*}
which provides, in accordance with lemma~\ref{lem:min}, the continuous, nondecreasing, piecewise linear function illustrated in figure~\ref{fig:esS3_J_0_0_3}.

\begin{figure}[h!]
\centering
\psfrag{0}[cc][tc][.8][0]{$0$}
\includegraphics[scale=.8]{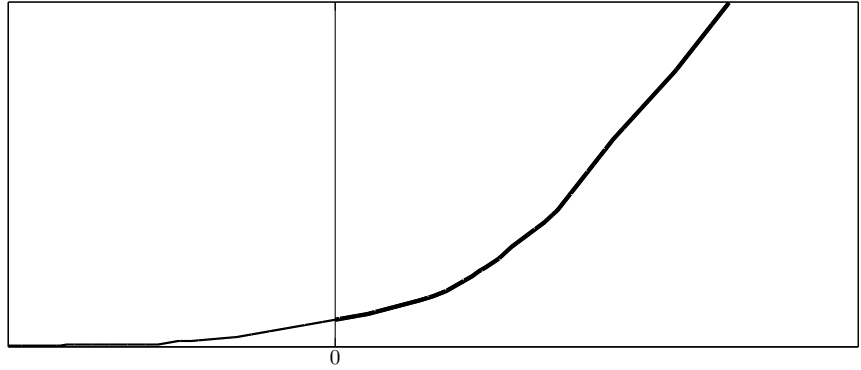}%
\vspace{-12pt}
\caption{Optimal cost-to-go $J^{\circ}_{0,0,0} (t_{0})$ in the initial state $[0 \; 0 \; 0 \; t_{0}]^{T}$.}
\label{fig:esS3_J_0_0_3}
\end{figure}

The function $J^{\circ}_{0,0,0} (t_{0})$ is specified by the initial value 0.5, by the set \{ --21, --20.5, --13.5, --12, --11, --7.5, 2.5, 6.5, 7.5, 8.5, 10.5, 11, 11.5, 12.5, 13.5, 16, 17, 21.25, 26 \} of abscissae $\gamma_{i}$, $i = 1, \ldots, 19$, at which the slope changes, and by the set \{ 1, 0, 1, 0, 0.5, 1, 1.5, 1.75, 2.25, 3.25, 4.25, 3.25, 3.75, 5.25, 4.25, 5.25, 7.25, 6.25, 7.25 \} of slopes $\mu_{i}$, $i = 1, \ldots, 19$, in the various intervals.

Since $J^{\circ}_{0,0,0} (t_{0} \mid \delta_{1} = 1)$ is the minimum in $(-\infty,-20.5)$, and $J^{\circ}_{0,0,0} (t_{0} \mid \delta_{2} = 1)$ is the minimum in $[-20.5,+\infty)$ (see again figure~\ref{fig:esS3_J_0_0_3_min}), the optimal control strategies for the initial state are
\begin{equation*}
\delta_{1}^{\circ} (0,0,0, t_{0}) = \left\{ \begin{array}{ll}
1 & t_{0} < -20.5\\
0 & t_{0} \geq -20.5
\end{array} \right. \qquad \delta_{2}^{\circ} (0,0,0, t_{0}) = \left\{ \begin{array}{ll}
0 & t_{0} < -20.5\\
1 & t_{0} \geq -20.5
\end{array} \right.
\end{equation*}
\begin{equation*}
\tau^{\circ} (0, 0, 0, t_{0}) = \left\{ \begin{array}{ll}
8 &  t_{0} < -21\\
-t_{0} - 13 & -21 \leq t_{0} < -20.5\\
6 & -20.5 \leq t_{0} < 2.5\\
-t_{0} + 8.5 & 2.5 \leq t_{0} < 4.5\\
4 & t_{0} \geq 4.5
\end{array} \right.
\end{equation*}

The optimal control strategy $\tau^{\circ} (0, 0, 0, t_{0})$ is illustrated in figure~\ref{fig:esS3_tau_0_0_0}.

\begin{figure}[h!]
\centering
\psfrag{f(x)}[Bl][Bl][.8][0]{$\tau^{\circ} (0, 0, 0, t_{0})$}
\psfrag{x}[bc][Bl][.8][0]{$t_{0}$}
\psfrag{0}[tc][Bl][.8][0]{$0$}
\includegraphics[scale=.2]{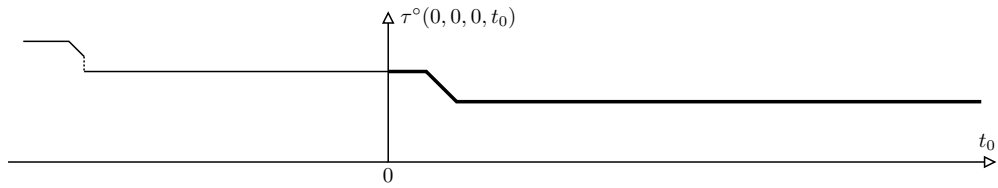}
\caption{Optimal control strategy $\tau^{\circ} (0, 0, 0, t_{0})$ in the initial state $[ 0 \; 0 \; 0 \; t_{0}]^{T}$.}
\label{fig:esS3_tau_0_0_0}
\end{figure}

\vspace{80pt}
$ $

\bibliographystyle{plain}
\bibliography{PaperScheduling}

\end{document}